\documentclass[12pt,a4paper,reqno]{amsart}
\usepackage[margin=3.5cm]{geometry}
\usepackage[utf8]{inputenc}
\usepackage{amsmath,nccmath}
\usepackage{enumerate}   
\usepackage{amsthm}
\usepackage{amsfonts}
\usepackage[square, comma, sort&compress,numbers]{natbib}
\usepackage{amssymb}
\usepackage{setspace}
\usepackage{graphicx}
\usepackage{color}
\usepackage[marginal]{footmisc}
\usepackage[section]{placeins}
\usepackage{mathtools}
\usepackage{indentfirst}
\usepackage{bm}
\usepackage[font={small}]{caption}
\usepackage{ifmtarg}
\usepackage[backref=page]{hyperref}
\hypersetup{
     colorlinks=true,    
     linktocpage=true,
    }
\renewcommand*{\backref}[1]{}
\renewcommand*{\backrefalt}[4]{~~{\tiny%
    \ifcase #1 Not cited.%
          \or [Cited on page~#2.]%
          \else [Cited on pages #2.]%
    \fi%
    }}
\usepackage[shortlabels]{enumitem} 

\usepackage{float}

\theoremstyle{plain}
\newtheorem{prop}{Proposition}[section]
\newtheorem*{prop*}{Proposition}
\newtheorem{cor}[prop]{Corollary}
\newtheorem{lem}[prop]{Lemma}

\newtheorem{thms}[prop]{Theorem}
\newtheorem{ques}[prop]{Question}
\newtheorem*{ques*}{Question}
\newtheorem{conj}[prop]{Conjecture}
\newtheorem*{conj*}{Conjecture}

\newtheorem*{lem*}{Lemma}
\newtheorem*{thm*}{Theorem}
\newtheorem*{claim*}{Claim}

\newtheorem*{fact*}{Fact}

\newtheorem{thma}{Theorem}

\newtheorem{cora}[thma]{Corollary}

\newtheorem{lema}{Lemma}

\newtheorem{manualtheoreminner}{Theorem}
\newenvironment{manualtheorem}[1]{%
  \renewcommand\themanualtheoreminner{#1}%
  \manualtheoreminner
}{\endmanualtheoreminner}

\newtheorem{manualcorinner}{Corollary}
\newenvironment{manualcor}[1]{%
  \renewcommand\themanualcorinner{#1}%
  \manualcorinner
}{\endmanualcorinner}

\theoremstyle{definition}
\newtheorem{defi}[prop]{Definition}
\newtheorem*{defi*}{Definition}

\newtheorem{rem}[prop]{Remark}

\numberwithin{equation}{section}

\renewcommand{\geq}{\geqslant}
\renewcommand{\leq}{\leqslant}

\allowdisplaybreaks

\DeclareMathOperator{\tr}{tr}

\def\eps{\varepsilon}
\def\pp{\varphi}

\definecolor{darkred}{rgb}{.56,0,0}

\makeatletter
\newcommand*{\rom}[1]{\expandafter\@slowromancap\romannumeral #1@}

\newcommand\primeitem{%
 \item[(\roman{enumi}\textquotesingle)]\def\@currentlabel{(\roman{enumi}\textquotesingle)}}
\newcommand\pprimeitem{%
 \item[\arabic{enumi}\textquotesingle]\def\@currentlabel{\arabic{enumi}\textquotesingle}}

\newcommand{\setword}[2]{%
  \phantomsection
  #1\def\@currentlabel{\unexpanded{#1}}\label{#2}%
} 

\makeatother

\makeatletter
\newcommand{\xMapsto}[2][]{\ext@arrow 0599{\Mapstofill@}{#1}{#2}}
\def\Mapstofill@{\arrowfill@{\Mapstochar\Relbar}\Relbar\Rightarrow}
\makeatother

\title{{\bf Symplectic blenders near whiskered tori and  persistence of saddle-center  homoclinics}}
\author{Dongchen Li}
\address{Shanghai Center for Mathematical Sciences, Fudan University, China}
\email{dongchenli@fudan.edu.cn}

\author{Dmitry Turaev}	
\address{Department of Mathematics, Imperial College London, UK}
\email{d.turaev@imperial.ac.uk}

\begin{document}
\def\p{\partial}
\def\thebeta{\alpha}
\def\thebetai{\alpha^{-1}}
\def\D{\mathrm{D}}
\def\bModa{\;\mathrm{mod}_0\;}
\def\Moda{\quad\mathrm{mod}_0}
\def\bMod{\mathrm{mod}\;}
\def\Modb{\quad\mathrm{mod}^-\;}
\def\d{\mathrm{d}}
\def\diff{\mathrm{Diff}}
\def\symp{\mathrm{Symp}}
\def\sympe{\mathrm{Symp}_\gamma}
\def\ham{\text{Ham}}
\def\inds{\mbox{ind}^s}
\def\indu{\mbox{ind}^u} 
\def\ind{\mbox{ind}}
\def\rc{\mathrm{KAM}}
\def\s{\mathrm{s}}
\def\ss{\mathrm{ss}}
\def\cs{\mathrm{cs}}
\def\u{\mathrm{u}}
\def\uu{\mathrm{uu}}
\def\cu{\mathrm{cu}}
\def\c{\mathrm{c}}
\def\loc{\mathrm{loc}}
\def\T{\mathrm{T}}
\def\cc{{\hat{\lambda}}}
\def\dd{\beta}
\def\A{\mathbb{A}}
\def\const{\mathrm{const}}
\def\orb{{\mathcal{O}(\gamma)}}
\def\orba{{\mathcal{O}(\mathbb{A})}}
\def\per{{\mathrm{per}(\gamma)}}
\def\Me{\mathcal{M}_{\mathrm{e}}}
\def\M{\mathcal{M}}
\def\DD{\mathbb{D}}
\def\P{\mathrm{P}}
\def\new{\mathrm{new}}
\def\O{\mathcal{O}}
\def\aa{\hat p(\delta)}

\subjclass[2020]{37C29, 37G25, 37J11,37D30}

\begin{abstract}
A blender is a hyperbolic basic set such that the projection of its stable/unstable set onto some center subspace has a higher topological dimension than the set itself.
We prove that, for any   $C^s$ symplectic diffeomorphism (where $s=2,\dots\infty,\omega$), if it has a one-dimensional  whiskered torus  with a homoclinic orbit, then a symplectic blender can be created by an  arbitrarily $C^s$-small perturbation.  Using this result, we show that the non-transverse homoclinic intersection between the invariant manifolds of a saddle-center periodic point is persistent, in the sense that the original system lies in the $C^s$-closure of a $C^1$-open set of symplectic diffeomorphisms where those having  saddle-center homoclinics are dense. Our results also hold in the corresponding continuous-time settings.
\end{abstract}
\maketitle
{\setlength{\parskip}{0pt}
\tableofcontents
}

\section{Introduction}\label{sec:intro}
In this paper, we develop a theory of hyperbolic dynamics and bifurcations  near homoclinic orbits to one-dimensional whiskered tori of symplectic diffeomorphisms. We establish the existence of  symplectic blenders near such tori and, as an application, prove the persistence of homoclinics to saddle-center periodic orbits. The results also hold in the corresponding settings for Hamiltonian flows. We also discuss some potential use of our results in celestial mechanics and in the stable ergodicity problem for symplectic diffeomorphisms.

\subsection{Symplectic blenders near whiskered tori}

A blender is a hyperbolic set such that the projection of its stable/unstable set onto some center subspace has non-empty interior; in particular, this projection has a higher topological dimension than the set itself. Since this property is $C^1$-robust, it allows for an unremovable non-transverse intersection of low-dimensional manifolds with the blender's stable/unstable set; see Section~\ref{sec:blender} for the precise definition. The notion of a blender  was introduced by Bonatti and D\'iaz \citep{BonDia:96} as a tool for producing robust transitivity for non-hyperbolic  diffeomorphisms. Since then, the theory of blenders has been used to show that various non-trivial  phenomena that seem fragile at first glance are in fact persistent: the persistence of heterodimensional cycles \citep{BonDia:08,BonDiaKir:12,LiTur:24},  abundance of $C^1$-robust homoclinic tangencies \citep{BonDia:12b,Li:24b}, typicality of Newhouse phenomenon in the space of families of diffeomorphisms \citep{Ber:16}, robustly fast growth of number of periodic orbits \citep{AsaShiTur:17,AsaShiTur:21,Ber:21},  robust existence of nonhyperbolic ergodic measures \citep{BocBonDia:16}, $C^1$ density of stable ergodicity \citep{HerHerTahUre:11,AviCroWil:21}, robust transitivity in Hamiltonian dynamics \citep{NasPuj:12} and robust bifurcations in complex dynamics \citep{Duj:17,Bie:19,Taf:21}, among others. 

A  symplectic variation of blenders was  proposed by Nassiri and Pujals \citep{NasPuj:12}. 
It was shown in \citep{NasPuj:12} that symplectic blenders emerge after small perturbations of a direct product of an integrable twist diffeomorphism and a symplectic diffeomorphism with a hyperbolic basic set. A non-trivial generalization of the Nassiri-Pujals construction, applicable to a wide class of near-integrable systems, was developed by Guardia and Paradela in the recent paper \citep{GuaPar:25}.
We show that the emergence of blenders is, in fact, a general phenomenon of partially-hyperbolic symplectic dynamics.

We begin with the description of the setting of the problem; a detailed explanation is in Section~\ref{sec:kam}. Let  $\M$ be a compact  symplectic  manifold ($C^\infty$ or real analytic) of dimension $2N$, with $N\geqslant 2$, equipped with a symplectic form $\Omega$. Denote by $\symp^s(\M)$ the space of $C^s$ symplectic diffeomorphisms of $\M$ ($1\leq s\leq \infty$, or $s=\omega$ meaning the real-analytic case).

Recall that a compact invariant set $\Lambda$ of a diffeomorphism $f$ is hyperbolic if its tangent bundle admits an invariant splitting $T_\Lambda \mathcal{M}=E^\s\oplus E^\u$, where the vectors in $E^\s$ and $E^\u$ are, respectively, uniformly contracted and expanded by the differential $\D f$. In the symplectic setting $\dim 
E^\s = \dim E^\u = N$.
Every point of $\Lambda$ has a smooth local stable and unstable manifolds tangent to $E^s$ and, respectively, $E^u$. They comprise locally-invariant continuous fibrations $W^\s_{\loc}(\Lambda)$
and $W^\u_{\loc}(\Lambda)$, which are extended globally by iteration.
A {\em basic set} is a zero-dimensional, transitive (i.e., containing a dense orbit),  locally maximal hyperbolic set (where  local maximality means the set contains all orbits that never leave its small neighborhood). Basic sets (along with their stable and unstable fibrations) persist at $C^1$-small perturbations, and depend continuously on the map $f$.

A compact invariant set is called {\em partially hyperbolic} if its tangent bundle admits an invariant splitting $T_\Lambda \mathcal{M}=E^\ss\oplus E^\c\oplus E^\uu$, where the vectors in $E^\ss$ and $ E^\uu$ are uniformly contracted and expanded, respectively, by the differential $\D f$, while the possible contraction and expansion in $E^\c$ are dominated by those in $E^\ss$ and $E^\uu$. This splitting extends to a small neighborhood of $\Lambda$. In this neighborhood, there exist a strong-stable locally invariant foliation $\mathcal{F}^{\ss}$ and a strong-unstable locally invariant foliation $\mathcal{F}^{\uu}$, which are tangent to $E^\ss$ and $E^\uu$, respectively. The leaves of these foliations that pass through the points of $\Lambda$ are defined uniquely and are extended globally by iterations.

Let a smooth curve  $\gamma\cong \mathbb{S}^1$ 
be invariant with respect to an iteration of a symplectic diffeomorphism $f$, i.e., $f^n(\gamma)=\gamma$ for some $n\geq 1$; we denote the minimal such $n$ as $\per$. We call $\gamma$
a {\em one-dimensional whiskered torus} if the orbit $\mathcal{O}(\gamma)$ of $\gamma$ (the union of $\per$ smooth curves)  is partially hyperbolic with a two-dimensional center, i.e.,  $T_\orb \mathcal{M}=E^\ss\oplus E^\c\oplus E^\uu$ with $\dim E^\c=2$ and the tangent to $\gamma$ belonging to $E^c$ at every point of $\gamma$. One can choose a sufficiently small tubular neighborhood of $\gamma$ such that the partially-hyperbolic splitting extends to it continuously. Then, there is a smooth two-dimensional manifold $\A \supset \gamma$, 
tangent to $E^\c$ and locally invariant with respect to $f^\per$.
The leaves of the local foliations 
$\mathcal{F}^{\ss}$ and $\mathcal{F}^{\uu}$ are $(N-1)$-dimensional,  smoothly embedded discs. The leaves of $\mathcal{F}^{\ss}$ and
$\mathcal{F}^{\uu}$ through points of $\gamma$ form the $N$-dimensional local stable and local unstable manifolds (the ``whiskers'') of 
$\gamma$, denoted by $W^\s_\loc(\gamma)$ and $W^\u_\loc(\gamma)$. The global stable or unstable manifold of $\gamma$ is defined as the union of iterations of the corresponding local 
manifold.

\begin{rem}\label{rem:exact}
The restriction of the symplectic form $\Omega$ to 
a small tubular neighborhood of $\orb$ is always exact, since  $H^2_{\mathrm{dR}}=0$ for such a neighborhood. Moreover, the fact that $\gamma$ is $f^\per$-invariant implies that $f$ is exact in this neighborhood. 
\end{rem} 

We always assume that $f^\per$ preserves the orientation on $\gamma$, so the rotation number $\rho(\gamma)$ is defined.
Recall that an orbit in the intersection $W^\s(\gamma)\cap W^\u(\gamma)$ is called  {\em homoclinic}. Our main result is

\begin{thma}\label{thm:main_blender_allcases}
Let $f\in \symp^s(\M)$,  $s=2,\dots, \infty,\omega$, have a one-dimensional whiskered torus $\gamma$ of class $C^s$ with a homoclinic orbit $\Gamma$, and let $\rho(\gamma)$ be irrational. Given any neighborhood $\hat V$ of $\Gamma\cup\mathcal{O}(\gamma)$, there exists $g\in \symp^{s}(\M)$, arbitrarily $C^s$-close to $f$, such that $g$ has a symplectic blender $\Lambda$, which is connected via $\hat V$ to a non-degenerate whiskered KAM-torus $\gamma_g$ of class $C^s$, arbitrarily $C^s$-close to $\gamma$. When the smoothness $s$ is finite,  both  $g$ and $\gamma_g$ can be taken $C^\infty$.
\end{thma}

Theorem~\ref{thm:main_blender_allcases} is proved in Section~\ref{sec:thmA}. 
By the connection we mean that the symplectic blender and  the whiskered torus are homoclinically related (i.e., $W^{\u/\s}(\Lambda)$
intersects $W^{\s/\u}(\gamma)$) in a certain special way; since the precise definition is somewhat technical, we postpone  it to Section~\ref{sec:blencon} (see Definition~\ref{defi:sympblen_conn}). In what follows, we state the key properties of $\Lambda$ and $\gamma_g$, as well as the consequences of their connection (see Proposition~\ref{prop:persis}). 

The whiskered torus $\gamma_g$  is $C^\infty$ if $s\leq \infty$ and $C^\omega$ in the real-analytic case. The KAM non-degeneracy means that the rotation number $\rho(\gamma_g)$ is Diophantine and $g^{\per}|_{\A}$ is close to an integrable nonlinear rotation with a non-zero twist (see Definition~\ref{defi:kamtori} and the discussion in Section~\ref{sec:quadraunfold}). In particular, the KAM theory gives the persistence of $\gamma_g$ at symplectic perturbations that are small in sufficiently high regularity \citep{Mos:73}.

The symplectic blender $\Lambda$ is a hyperbolic basic set endowed with a  partially hyperbolic structure:
\begin{equation*}
T_\Lambda \mathcal{M}=E^\s\oplus E^\u=(E^\ss\oplus E^\mathrm{ws})\oplus( E^\mathrm{wu}\oplus E^\uu)=E^\ss\oplus E^\c\oplus E^\uu,
\end{equation*}
where $\dim E^\mathrm{ws} =\dim E^\mathrm{wu}=1$ and $\dim E^\s=\dim E^\u = N-1$.  It has the following {\em blender property}:  for every map that is  $C^1$-close to $g$,  if an $(N-1)$-dimensional disc $L^\u$ is $C^1$-close 
to a strong-unstable leaf\footnote{We mean that they are $C^1$-close as two embeddings of a unit disc in $\mathbb{R}^{N-1}$. We shall use this convention throughout the paper.
} 
from $W^\u_\loc(\gamma_g)$, then it intersects $W^\s(\Lambda)$, and if an $(N-1)$-dimensional disc $L^\s$ is $C^1$-close 
 to a  strong-stable leaf from $W^\s_\loc(\gamma_g)$, then it intersects $W^\u(\Lambda)$.

The blender property is in line with the one that firstly appeared in \citep{BonDia:96}:  
although the intersection of the $(N-1)$-dimensional disc $L^{\u/\s}$ with the $N$-dimensional stable/unstable manifold of any individual point of  $\Lambda$ is non-transverse (and hence can be removed  by a small perturbation),  the non-transverse intersection of $L^{\u/\s}$ with the whole set $W^{\s/\u}(\Lambda)$ is $C^1$-robust, that is, it persists under all $C^1$-small perturbations.

Recall that the stable and unstable manifolds of any given point of the basic set $\Lambda$ are dense in $W^\s(\Lambda)$
and, respectively, $W^\u(\Lambda)$.  As a result, one can, by a $C^s$-small perturbation, create an intersection between $L^{\u/\s}$ and $W^{\s/\u}(P)$ for any particular point $P\in\Lambda$. 
In fact, the specific type of connection given by Theorem \ref{thm:main_blender_allcases} between the blender $\Lambda$ and a  non-degenerate  whiskered KAM-torus $\gamma_g$ leads to a stronger statement:  one can establish the persistence of intersections between $L^{\u/\s}$ and $W^{\s/\u}(P)$ within parametric families satisfying explicitly formulated genericity conditions,  thus allowing one to locate,
in a wide set of applications, persistent intersections of manifolds whose dimension does not allow transversality. 

Let $s=\infty,\omega$, and let  $\{f_\eps\}$ be
a  $C^s$ family (with at least two parameters) of diffeomorphisms in $\symp^s(\M)$ such that $f_0$ has a symplectic blender $\Lambda$ 
connected, as in Definition~\ref{defi:sympblen_conn}, to a non-degenerate whiskered KAM-torus $\gamma$ of class $C^s$. As we mentioned, the blender and the KAM-torus persist at small $\eps$.

\begin{prop}[Persistent intersections]\label{prop:persis}
Let $\{L^\u_\eps\}$ and $\{L^\s_\eps\}$ be two  $C^1$ families of $(N-1)$-dimensional embedded discs such that 
the disc $L^\u_0$ is close to some local strong-unstable leaf
in $W^\u_{\loc}(\gamma)$, and $L^\s_0$ is close to some
strong-stable leaf in $W^\s_{\loc}(\gamma_\eps)$. If  the distances from $L^\u_\eps$ and $L^\s_\eps$ to $\gamma_\eps$ change independently with $\eps$ (see condition \eqref{eq:gcpersis}), then there exists a neighborhood $\mathcal{E}$ of $\eps=0$ such that
\begin{itemize}[nosep]
\item for any pair of points $P$ and $P'$ of $\Lambda$,
$$L^\u_\eps \cap W^\s(P_\eps)\neq \emptyset \;\;\mbox{and}\;\;   L^\s_\eps \cap W^\u(P'_\eps)\neq \emptyset$$ for a dense subset of  $\mathcal{E}$;
\item an iteration of $L^\u_\eps$ by $f_\eps$ intersects $ L^\s_\eps$ for a dense subset of  $\mathcal{E}$.
\end{itemize}
\end{prop}

This result is proved in Section~\ref{sec:unfoldblen}. The KAM property of $\gamma_\eps$ is  essential for the first claim; the second claim follows from the first one by   taking $P=P'$ a periodic point of $\Lambda$. One can see from the proof that Proposition~\ref{prop:persis} also holds  for all sufficiently large finite $s$.

 The proof of Theorem~\ref{thm:main_blender_allcases} is based on a somewhat counter-intuitive result that  a small perturbation of transverse homoclinics to $\gamma$ creates tangencies between $W^\u(\gamma)$ and $W^\s(\gamma)$ (see Proposition~\ref{prop:A2}), and then follows from our analysis of bifurcations of
orbits of tangency of whiskered KAM-tori in two-parameter families. 

More specifically, we consider {\em partially-hyperbolic homoclinic tangencies} (see Definition~\ref{defi:phtangency}), implying, in particular, the partial hyperbolicity of the union of  $\orb$ and the homoclinic orbits. Let $\{f_\eps\}\subset \symp^s(\M)$  be a two-parameter family such that $f_0=f$ and   $f_{\eps}$ is jointly $C^s$ with respect to parameters and coordinates. We investigate three cases:
\begin{itemize}[nosep]
\item[(1)]   $\{f_\eps\}$ unfolds generically a  partially-hyperbolic cubic homoclinic tangency of $W^\u(\gamma)$ and $W^\s(\gamma)$;
\item[(2)]   $\{f_\eps\}$ unfolds generically two  partially-hyperbolic  quadratic homoclinic tangencies of $W^\u(\gamma)$ and $W^\s(\gamma)$; and
\item[(3)]   $\{f_\eps\}$ unfolds {\em properly} a partially-hyperbolic quadratic homoclinic tangency of $W^\u(\gamma)$ and $W^\s(\gamma)$.
\end{itemize}

The natural genericity/propriety conditions for these two-parameter unfoldings  are given by~\eqref{eq:unfoldcubic},~\eqref{eq:unfold2quadra} and~\eqref{eq:unfoldquadra}, respectively. 
In the third case, the unfolding involves the change of the hyperbolicity rate of the tangency. 
Examples of such unfoldings are given in Section~\ref{sec:unfoldingfam}. 
As part of the proof of Theorem~\ref{thm:main_blender_allcases}, we also explicitly construct  perturbations that turn a whiskered torus with an irrational rotation number into a non-degenerate KAM-torus and make arbitrary homoclinic orbits partially hyperbolic.

\begin{thma}\label{thm:main_2punfolding}
Let $f \in \symp^\infty(\M)$ have a  non-degenerate whiskered KAM-torus $\gamma$ of class $C^\infty$. Let $\{f_\eps\}\subset \symp^\infty(\M)$ be any  two-parameter  family of type  (1), (2) or (3), where  each $f_\eps$ is exact (see Remark~\ref{rem:exact}) in a small tubular neighborhood of $\orb$. Then, for any small neighborhood $\hat V$ of the union of  $\mathcal{O}(\gamma)$ and the homoclinic orbit(s) involved, there exist arbitrarily small values of $\eps$ for which  $f_\eps$  has  a symplectic blender $\Lambda_\eps$, connected via $\hat V$ to the KAM-continuation of $\gamma$.
\end{thma}

\begin{rem}\label{rem:exact2}
We only use the   exactness of $f_\eps$  to apply the KAM theory near $\orb$ to achieve persistence of KAM-curves. 
When $\gamma$  lies in a simply connected subset of $\M$,  the required local exactness is satisfied automatically for all maps in $\symp^s(\M)$ by the Poincar\'e lemma.
\end{rem}

By Definition~\ref{defi:sympblen_conn}, the set of $\eps$ values in Theorem~\ref{thm:main_2punfolding} is open.
The detailed setting of this theorem and an outline of the proof are given in Section~\ref{sec:quadraunfold}, and the complete proof is presented in Section~\ref{sec:proofsymp}. It relies on the key fact that center-stable and center-unstable blenders, from which we build symplectic blenders, naturally arise  {\em without perturbations} near a whiskered torus with a cubic homoclinic tangency (see Theorem~\ref{thm:main_blender_cubic_nonpara}).

One can see from the proof that Theorem~\ref{thm:main_2punfolding} also holds when the family $f_\eps$ and the whiskered torus $\gamma$ have sufficiently large finite regularity. The order of regularity can be estimated in terms of the Diophantine properties of the rotation number $\rho(\gamma)$ (see Remark~\ref{rem:smoothness}).  
The genericity/propriety conditions we impose on $\{f_\eps\}$ are open and dense in the space of families of any sufficiently high regularity, including $C^\omega$. As a result, the theorem is applicable to  real analytic families.
  
Thus, Theorem~\ref{thm:main_2punfolding} holds for any sufficiently smooth family of perturbations satisfying 
explicitly formulated and verifiable genericity conditions, which allows for a broad set of applications.  In particular, we discuss in Section~\ref{sec:introSC} the setting of saddle-center periodic points, where whiskered KAM-tori exist naturally and Theorem~\ref{thm:main_2punfolding} and Proposition \ref{prop:persis} apply.

\subsection{Saddle-center periodic points with homoclinics}\label{sec:introSC}
Let $f\in\symp^s(\mathcal{M})$  have a {\em saddle-center} periodic point $O$, that is, $O$ has exactly two multipliers\footnote{eigenvalues of the differential of the period map at $O$} on the unit circle: 
$\lambda=e^{\pm i\rho}$, $\rho\in (0,\pi)$, and, by symplecticity, there are $N-1$ multipliers outside and inside the unit circle. The point $O$ has a two-dimensional, locally-invariant, normally-hyperbolic, symplectic center manifold $W^\c(O)$, which has the same smoothness as $f$ if $s$ is finite and can have arbitrarily large finite smoothness if $f$ is $C^\infty$. The local invariance means that there exists a neighborhood $V$ of $O$ such that the orbit of any point in $W^\c(O)$ stays in $W^\c(O)$ as long as it lies in $V$. Normal hyperbolicity means that the tangent bundle $T_{W^\c(O)}\mathcal{M}$  admits a partially-hyperbolic splitting with the center bundle $E^\c$ equal to ${T} {W^\c(O)}$. The partial hyperbolicity implies that every point of $W^\c(O)$ has a strong-stable and a strong-unstable leaf, which are both $(N-1)$-dimensional. In particular,  the leaves through $O$  are the  stable and  unstable invariant manifolds $W^\s(O)$ and $W^\u(O)$ (see Figure~\ref{fig:sc}). We assume that these two manifolds have a homoclinic intersection, which is automatically non-transverse by counting the dimensions. Because of the non-transversality, this homoclinic intersection can  be removed by an arbitrarily small perturbation of $f$ (note that $O$ persists under small perturbations since it does not have 1 as its multiplier, and it remains a saddle-center when the perturbations are symplectic).  However, we show that  {\em arbitrarily close to  $f$ in the space of symplectic diffeomorphisms there are  open regions where maps having a saddle-center with a homoclinic orbit are dense}, as given by the Theorem \ref{thm:main_SC_para} below. 

\begin{figure}[!h]
\begin{center}
\includegraphics[scale=.85]{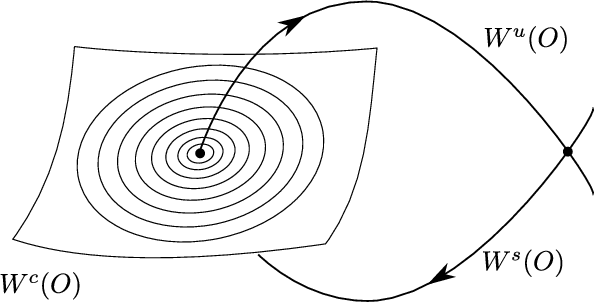}
\end{center}
\caption{The schematic picture for a saddle-center periodic point with a homoclinic orbit.}
\label{fig:sc}
\end{figure}

We say that the saddle-center $O$ and its homoclinic orbit $\Gamma$ are {\em generic} if conditions \ref{word:C1}--\ref{word:C3} of Section~\ref{sec:para_SC}  are fulfilled. These conditions ensure that, first, Moser's conditions \citep{Mos:73} are met, so that the restriction $f^{\mathrm{per}(O)}|_{W^\c(O)}$ (where $\mathrm{per}(O)$ denotes the period) has a large measure set of KAM-curves on $W^\c(O)$. Second, the so-called scattering map (see \eqref{eq:scattering_1}) of $W^\c(O)$ can be defined along the strong-stable and strong-unstable leaves near $W^\s(O)$ and $W^\u(O)$, and the scattering map is not equal to a linear rotation (cf. \citep{LerMar:15}).  Note that all  the genericity conditions  can be achieved by an arbitrarily small perturbation, see Section~\ref{sec:gc2}.
 
A family $\{f_\eps\}\subset \symp^s(\mathcal{M})$ is called  a {\em proper unfolding of the homoclinic intersection at $\Gamma$} if it has at least {\em four parameters} and condition \ref{word:H1} of Section~\ref{sec:para_SC} is fulfilled. This condition means that, when  $\eps$ varies,  certain four quantities change independently, which control 
\begin{enumerate}[nosep]
\item the positions of the KAM-curves on $W^\c(O)$,
\item the distortion of the image of these KAM-curves by the scattering map, 
\item the position of a piece of $W^\u(O)$ near $W^\s_\loc(O)$, and
\item the position of a piece of $W^\s(O)$ near $W^\u_\loc(O)$, respectively.
\end{enumerate} 
An  example of such unfolding family is provided in   Section~\ref{sec:famC}.

\begin{thma}\label{thm:main_SC_para}
Let $f\in \symp^{\infty}(\mathcal{M})$  have a generic  saddle-center with a generic homoclinic orbit. For any   proper unfolding family $\{f_\eps\}\subset  \symp^{\infty}(\mathcal{M})$,  there exists a sequence $\{\mathcal{E}_j\}$ of open sets converging to $\eps=0$ such that  the set of $\eps$ for which the saddle-center has a homoclinic orbit is dense in $\bigcup_j \mathcal{E}_j$.
\end{thma}

This theorem is proven in Section~\ref{sec:homounfold}. The main ingredient of the proof is the creation of a symplectic blender near  the homoclinic orbit to the saddle-center.
The fundamental observation  here is that the KAM-curves in the two-dimensional manifold $W^\c(O)$  are whiskered tori: they  possess $N$-dimensional stable and unstable invariant manifolds due to the normal hyperbolicity of $W^\c(O)$. 
For the whiskered tori close to $O$, their stable and unstable manifolds are close to those of $O$. If $O$ has a generic homoclinic orbit, then each of the torus has transverse  homoclinic orbits \citep{LerMar:15}. 
However, we cannot immediately apply Theorem~\ref{thm:main_blender_allcases} as the perturbations we construct in its proof are not done within the context of finite-parameter families.

Therefore, in order to proceed with the proper unfolding families of Theorem~\ref{thm:main_SC_para}, we further show that the whiskered tori near $O$ have persistent heteroclinic tangencies (see Lemma~\ref{lem:tangency_KAM}). For the proper unfolding family $\{f_\eps\}$, bifurcations of these heteroclinic tangencies, as $\eps$ varies, create two quadratic  homoclinic tangencies for some whiskered torus near $O$ (see Proposition~\ref{prop:tangenfamily}). At this point, the existence of symplectic blenders  within the proper unfolding family becomes possible by case (2) of Theorem~\ref{thm:main_2punfolding} (by Remark~\ref{rem:exact2}, the local exactness required by Theorem~\ref{thm:main_2punfolding} is ensured by noticing that the  whiskered tori here are contractible to $O$). Namely, we consider two-parameter sub-families which {\em do not destroy} the homoclinic intersection of $W^\u(O)$ and $W^\s(O)$, but still control
the positions of the KAM-curves on $W^\c(O)$ and the distortion of the image of these KAM-curves by the scattering map. We call such two-parameter families {\em tangency unfolding}, see condition~\ref{word:H2} in Section~\ref{sec:properunfolding}. Using Theorem~\ref{thm:main_2punfolding}, we obtain

\begin{thma}\label{thm:main_SCblender}
Let $f\in \symp^\infty(\mathcal{M})$ have a generic  saddle-center $O$ with a generic homoclinic orbit $\Gamma$. 
 For any  tangency-unfolding family $\{f_\eps\}$ and any small neighborhood $V'$ of $\mathcal{O}(O)\cup\Gamma$,  there exists a sequence $\{\eps_j\}$ converging to $\eps=0$ such that $f_{\eps_j}$ has a symplectic blender connected in $V'$ to some whiskered KAM-torus that lies in $W^\c(O)$.
\end{thma}

This theorem is proven in Section~\ref{sec:SC}. We next show that, within the proper unfolding family, one can find pieces $W^\s$ and $W^\u$ of $W^\s(O)$ and $W^\u(O)$ that are  arbitrarily close to the local invariant manifolds of the whiskered KAM-torus in Theorem~\ref{thm:main_SCblender}.  Theorem~\ref{thm:main_SC_para} then follows from  Proposition~\ref{prop:persis}, by taking  $L^{\s/\u}=W^{\s/\u}$ (see Proposition~\ref{prop:homofamily}). 

\begin{rem}
It follows from the proof that  $\mathcal{E}_j$ in Theorem~\ref{thm:main_SC_para} and  $\eps_j$ in Theorem~\ref{thm:main_SCblender}  depend continuously on the family $\{f_\eps\}$.
\end{rem}

It is obvious from the proof that the above results  also hold  for sufficiently large finite $s$.
The regularity requirement   can be relaxed when we do not restrict to parameterized perturbations.
Given any $f\in\symp^s(\mathcal{M})$ with $1\leq s<\infty$, we can approximate it in the $C^s$ topology by a map $g\in\symp^\infty(\mathcal{M})$ \citep{Zeh:76}. If $f$ has a saddle-center periodic orbit with a homoclinic, so does $g$.  The genericity conditions \ref{word:C1}--\ref{word:C3} become fulfilled after a further perturbation of $g$. Embedding the resulting map into a proper unfolding family of Theorem~\ref{thm:main_SC_para}, we obtain

\begin{cora}\label{cor:SC_nonpara}
Let $f\in \symp^{s}(\mathcal{M})$,  $s=1,\dots,\infty$, have a   saddle-center $O$ with a homoclinic orbit $\Gamma$, and let $ V'$ be any neighborhood of $\mathcal{O}(O)\cup\Gamma$. Then
 there exist a map $g\in \symp^{\infty}(\mathcal{M})$, arbitrarily $C^s$ close to $f$, a $C^1$ neighborhood $\mathcal{U}\subset \symp^{\infty}(\mathcal{M})$ of $g$ and a $C^\infty$-dense subset $\mathcal{U'}$ of $\mathcal{U}$ such that 
\begin{itemize}[nosep]
\item $g$ has a symplectic blender $\Lambda$ connected to a whiskered KAM-torus  via $V'$,
\item $W^\u(O_h)\cap W^\s(\Lambda_h)\neq \emptyset$ and  $W^\s(O_h)\cap W^\u(\Lambda_h)\neq \emptyset$ for every $h\in\mathcal{U}$, and
\item $O_h$ has a homoclinic orbit for every $h\in \mathcal{U}'$.
\end{itemize}
If $f\in \symp^{\omega}(\mathcal{M})$, then the same result holds with $\omega$ in place of $\infty$ for the regularity.
\end{cora}

Note that our approach allows for a straightforward transmission of the above results to the Hamiltonian setting, i.e., the case of continuous time, see Section~\ref{sec:conti}.

\subsection{Further application directions}
\subsubsection{Planar elliptic restricted  three-body problem}
The restricted    three-body problem is the special case   where one of the bodies has negligible mass. The problem is to study the motion of this mass under the gravitational pull of the other two bodies with large masses. A classical model is the  Sun-Jupiter-asteroid system.
The problem is further  called circular if the two massive bodies move along  circular orbits (with a common center of mass) that solve the Kepler problem, and is called elliptic if they move on  ellipces. Here we consider the case where the three bodies are coplanar, reducing the problem to a two-dimensional case. 

Let us first recall some basic results of the circular problem (see e.g. \citep[Section 4]{MeyOff:17}). The system has 2 degrees of freedom and  its Hamiltonian in the rotating coordinates is given by 
$$
H(x,y)= \dfrac{y^2}{2} - x^T \begin{pmatrix}
0 &1\\
-1 & 0
\end{pmatrix}
y
-\left( \frac{\mu}{d_1(x,y)}+\frac{1-\mu}{d_2(x,y)}\right),
$$
where $x,y\in\mathbb{R}^2$, $d_1$ and $d_2$ are the distances from the infinitesimal body to the  massive body  at $(1-\mu,0)$ with mass $\mu$ and, respectively, to the one at $(-\mu,0)$ with mass  $1-\mu$.
 This Hamiltonian has five critical points, corresponding to five equilibria  in the four-dimensional phase space. They are called the Lagrange points and denoted by $L_i$ $(i=1,\dots,5)$. 
 Each of the collinear points $L_{1,2,3}$ has two purely imaginary eigenvalues and two real ones, so they are saddle-centers. We are in particular interested in $L_1$ (the one between  the Sun and Jupiter), since it possesses a homoclinic orbit \citep{LliMarSim:85} (it was called $L_2$ in \citep{LliMarSim:85}).

In order to apply the results from Section~\ref{sec:introSC}, it is necessary to produce from the saddle-center equilibrium $L_1$  a saddle-center periodic point of some symplectic map. To this aim, we allow the two massive bodies move on ellipses with small eccentricities. It is well-known that the corresponding  elliptic problem  is  a small time-periodic perturbation of the circular one. Specifically, in the pulsating coordinates, its Hamiltonian is given by 
$$
R_{e,\mu}(x,y,\tau)= \dfrac{y^2}{2} - x^T \begin{pmatrix}
0 &1\\
-1 & 0
\end{pmatrix}
y - g_e(\tau)\left( \frac{\mu}{d_1(x,y)}+\frac{1-\mu}{d_2(x,y)}\right) + \dfrac{1-g_e(\tau)}{2}x^T x,
$$
where $e$ is the eccentricity, $\tau$ is the time and $g_e(\tau)=1/(1+e\cos \tau)$ (see e.g. \citep[Section 8.10]{MeyOff:17}). 
 In the perturbed system with small $e\neq 0$, the Lagrange point $L_2$ becomes a periodic orbit. The intersection point $O_e=L_1\cap \{\tau=0\}$ is a saddle-center fixed point of the time-$2\pi$  map of the four-dimensional cross-section $\{\tau=0\}$. This map 
  is an analytic symplectic diffeomorphism. Also, it is reasonable to believe that the homoclinic orbit of $L_1$ given by \citep{LliMarSim:85} at $e=0$ (in the circular problem) persists for small $e$. 
 
\begin{conj}
For every small $e>0$ there exists $\mu(e)$ for which
the orbit $L_1$ in the planar elliptic restricted three-body problem has a homoclinic orbit. Moreover,  there is a neighborhood $U$ of the curve $\mu=\mu(e)$ such that for each $(e,\mu) \in U$ the  time-$2\pi$ map of $R_{e,\mu}$  has a symplectic blender, and the orbit $L_1$ has a homoclinic orbit for a dense subset of $U$.
\end{conj}

As can be seen from the proof of Theorem~\ref{thm:main_SC_para}, in general one needs four parameters to achieve the persistence of the saddle-center homoclinics. However, the Hamiltonian $R_{e,\mu}$ is time-reversible. We believe that the reversibility allows for a reduction of the number of parameters to two as in the above conjecture. For other applications of symplectic blenders to the three-body problem, see \citep{GuaPar:25}.

\subsubsection{Stable ergodicity problem}
It was pointed out to us by Xue, Zhang and Avila that our results can be potentially applied to  the stable ergodicity problem for symplectic diffeomorphisms. In what follows, we briefly review the problem and explain how our result could be relevant.

Let $\mathcal{N}$ be any compact connected Riemannian manifold. 
Denote by $\mathrm{PH}^s_m(\mathcal{N})$ the space of  partially-hyperbolic, volume-preserving $C^s$ diffeomorphisms of $\mathcal N$.
The works on stable ergodicity in the past thirty years were mostly motivated by the conjecture of Pugh and Shub \citep{PugShu:97}: {\em for any $s>1$, stable ergodicity is $C^s$-dense in $\mathrm{PH}^s_m(\mathcal{N})$.} Here a map is stably ergodic if it is ergodic 
along with every close map in $\mathrm{PH}^s_m(\mathcal{N})$. The conjecture was proven for the case where $\dim E^c=1$ by F. Rodriguez-Hertz, M. A. Rodriguez-Hertz and Ures \citep{HerHerUre:08}, and remains open in other cases. However, a $C^1$ version of the conjecture has recently been fully proven by Avila, Crovisier and Wilkinson \citep{AviCroWil:21}: {\em for any $s>1$, stable ergodicity is $C^1$-dense in $\mathrm{PH}^s_m(\mathcal{N})$} (see also \citep{HerHerTahUre:11}). One could then ask the natural question:

{\em Is stable ergodicity  $C^1$-dense in $\mathrm{PH}^s_\Omega(\mathcal{N})$ for $s>1$?}

\noindent Here $\mathrm{PH}^s_\Omega(\mathcal{N})$ denotes the space of partially-hyperbolic $C^s$ symplectic diffeomorphisms of a $2N$-dimensional manifold $\mathcal{N}$ equipped with a symplectic form $\Omega$, and the ergodicity is with respect to the volume form $\Omega^N$. 

In the seminal work \citep{AviBocWil:09}, Avila, Bochi and Wilkinson proved that  generic maps in $\mathrm{PH}^1_\Omega(\mathcal{N})$ are ergodic. Their strategy is to first create  local ergodicity and then spread it to the whole manifold by accessibility and nonuniform center bunching. The key step in the local ergodicity part is to create saddle-center periodic point:

\begin{thms}[{\citep[Theorem 3.5]{AviBocWil:09}}]\label{thm:ABW}
Let $f\in \mathrm{PH}^1_\Omega(\mathcal{N})$ have the splitting $T\mathcal{N}=E^\s\oplus E^\c\oplus  E^\u$ such that $d^\c:=\dim E^\c$ is minimal (so that  any map close to $f$ has a center  of the same dimension). There exists a $C^1$-small perturbation $\tilde{f}$ that has a periodic point with $d^\c$ eigenvalues of modulus 1.
\end{thms}

The center manifold of the periodic point is a $d^\c$-dimensional disc. For a further perturbation $\hat f$, which implements the Anosov-Katok example \citep{AnoKat:70}, the disc becomes an ergodic component of $\hat f$ \citep[Lemma 3.8]{AviBocWil:09}, and providing the sought local source of ergodicity.  Since the Anosov-Katok construction does not persist under perturbations, only the genericity of ergodicity was established in \citep{AviBocWil:09}, instead of the density of stable ergodicity.

On the other hand, in \citep{HerHerTahUre:11,AviCroWil:21}, stable ergodicity was successfully achieved by using blenders to let the system robustly satisfy   the ergodicity criterion of \citep{HerHerTahUre:11}. Thus, one might expect to replicate this method in the symplectic setting. An immediate difficulty is that the creation of blenders in \citep{HerHerTahUre:11,AviCroWil:21} requires  nonuniform hyperbolicity (e.g. a full measure set of points which do not have zero Lyapunov exponents), but this property cannot be dense in $\mathrm{PH}^s_\Omega(\mathcal{N})$ by Bochi \citep{Boc:10}: all central Lyapunov exponents vanish at almost every point. By Theorem~\ref{thm:main_SCblender}, this difficulty of creating blenders is overcome for the space of
 partially-hyperbolic symplectic diffeomorphisms with two-dimensional center (denoted by $\widetilde{\mathrm{PH}}^s_\Omega(\mathcal{N})$). Note that $\widetilde{\mathrm{PH}}^s_\Omega(\mathcal{N})=\mathrm{PH}^s_\Omega(\mathcal{N})$ when $\dim \mathcal{N}=4$.
 
  Recall that for $C^s$ diffeomorphisms with $s>1$, one can define the Pesin stable and unstable manifolds for points in a full measure set (the Oseledets regular points). Let us show that Corollary \ref{cor:SC_nonpara} implies

\begin{cora}\label{cor:ergodic}
There exists a $C^1$-open and $C^1$-dense subset of $\widetilde{\mathrm{PH}}^s_\Omega(\mathcal{N})$, $s\geq~1$, where every map has a symplectic blender $\Lambda$. Moreover, if $s>1$, then the manifolds $W^\u(\Lambda)$ and $W^\s(\Lambda)$ intersect the Pesin stable and, respectively, unstable manifolds of almost every point.
\end{cora}

\begin{proof}
By Dolgopyat and Wilkinson \citep{DolWil:03}, there exists a $C^1$-open and $C^1$-dense set $\mathcal{U}$ of $\widetilde{\mathrm{PH}}^s_\Omega(\mathcal{N})$ $(s\geq 1)$, where every map has the accessibility property. By Theorem~\ref{thm:ABW}, we find a dense subset  $\mathcal{U}'\subset \mathcal{U}$ where every map has a saddle-center periodic orbit.

Now take any $f\in \widetilde{\mathrm{PH}}^s_\Omega(\mathcal{N})$. Let us assume $s\geq 2$, otherwise we approximate $f$ by a  $C^\infty$ one. By the density of $\mathcal{U}'$, we  find $f_1\in \mathcal{U}'$, $C^1$-close to $f$. By the connecting Lemma of Arnaud, Bonatti and Crovisier \citep[Th\'eor\`eme 2 and Remarque 1.4]{ArnBonCro:05} (see also \citep{Tak:72}), we further obtain  $f_2\in \mathcal{U}'$, $C^s$-close to $f_1$, such that $O$ has a homoclinic orbit. Finally, applying Corollary~\ref{cor:SC_nonpara} gives a   map $f_3$, $C^s$-close to $f_2$, with a symplectic blender $\Lambda$ connected to a whiskered torus $\gamma$.   The $C^1$-openness required in the corollary follows from the $C^1$-robustness of blenders.

Note that the strong-stable and strong-unstable leaves in the local invariant manifolds of $\gamma$ now belong to the global foliations $\mathcal{F}^\ss$ and $\mathcal{F}^\uu$ of $f_3$, which admit unique continuations. The first item of Proposition~\ref{prop:persis} then implies that there is a $C^1$ neighborhood $\mathcal{U}''\subset \mathcal{U}$ of $f_3$ such that,  for every  $g\in\mathcal{U}'' $, the foliation $\mathcal{F}^\ss_g$ contains an open subset where every leaf intersects $W^\u(\Lambda)$; similarly for $\mathcal{F}^\uu_g$.

Recall that, by Brin \citep{Bri:75} (see also \citep{BurDolPes:02}), for volume-preserving $C^2$ diffeomorphisms with the accessibility property,  almost every point has a dense orbit. 
Note that by construction $\mathcal{U}''\subset \widetilde{\mathrm{PH}}^2_\Omega(\mathcal{N})$. 
Thus, for every  $g\in\mathcal{U}''$, almost every orbit has two points $P_1$ and $P_2$ such that the strong-stable leaf through $P_1$ intersects $W^\u(\Lambda)$  and strong-unstable leaf through $P_2$ intersects $W^\s(\Lambda)$. The corollary follows immediately, since the Pesin stable/unstable manifold of a point contains the strong-stable/strong-unstable leaf through it. 
\end{proof}

Although the existence of blenders is ensured by the above corollary, one cannot directly follow \citep{HerHerTahUre:11,AviCroWil:21} to achieve stable ergodicity. It is because that the ergodicity criterion used in their proof is based on Hopf argument which relies on transverse intersections between the Pesin manifolds of points outside and inside the blender. Such transversality is guaranteed in \citep{HerHerTahUre:11,AviCroWil:21} due to the  nonuniform hyperbolicity (so that the size of the manifolds of points outside the blender is large enough), which is missing in the symplectic setting by \citep{Boc:10}. Thus, one can ask
\begin{ques}
Can  the symplectic blenders given by Corollary~\ref{cor:ergodic} produce ergodicity and hence lead to the $C^1$ density of stable ergodicity in $\widetilde{\mathrm{PH}}^s_\Omega(\mathcal{N})$?
\end{ques}

In the general case without constraint on the center dimension, a similar result to Corollary~\ref{cor:ergodic} can be expected, based on a generalization of Theorem~\ref{thm:main_blender_allcases} for whiskered tori of any possible dimension. We hope this can be done by a suitable modification of the method developed in this paper.

\subsection{Organization of the paper}

In Section~\ref{sec:intro2}, we give a precise description of the setting of Theorems~\ref{thm:main_blender_allcases} and \ref{thm:main_2punfolding}, including the notions of blenders, whiskered tori, KAM-curves,  etc. We state a non-perturbative result (Theorem~\ref{thm:main_blender_cubic_nonpara}) on the creation of cu/cs-blenders near cubic tangencies, which is the main ingredient in the proof of Theorem~\ref{thm:main_2punfolding}. We next introduce the perturbative setting for the creation of symplectic blenders, and sketch the proof of  Theorem~\ref{thm:main_2punfolding}.

In Section~\ref{sec:para}, we derive  formulas for the iterations near $\gamma$ and for those along the homoclinic orbit $\Gamma$, i.e., those taking a small neighborhood of some point in $\Gamma\cap W^\u_\loc(O)$ to a small neighborhood of some point in $\Gamma\cap W^\s_\loc(O)$. We give the core technical results on the estimates for these formulas.

Section~\ref{sec:proofcu} is devoted to the proof of Theorem~\ref{thm:main_blender_cubic_nonpara}. We will first derive formulas for the first-return maps (with different return times) along the homoclinic orbit. We prove  the hyperbolicity for the map induced from a  finite collection of the first-return maps, and further show that the induced map, and hence the original map, has a blender.

In Section~\ref{sec:proofsymp}, we first establish case (1) of Theorem~\ref{thm:main_2punfolding}, using Theorem~\ref{thm:main_blender_cubic_nonpara}. Then we show that a quadratic tangency can be perturbed  to produce two secondary quadratic tangencies, and these can be further perturbed to yield a cubic tangency. This proves  cases (2) and (3) of Theorem~\ref{thm:main_2punfolding}.   We conclude this section with a discussion on the unfolding of the non-transverse intersections produced by  blenders.

In Section~\ref{sec:thmA}, we prove Theorem~\ref{thm:main_blender_allcases} for the general case, where, in particular, the homoclinic of the whiskered torus may be transverse. The key ingredient is to create homoclinic tangencies from transverse homoclinic intersections so that Theorem~\ref{thm:main_2punfolding} becomes applicable. 

Finally, we prove Theorems~\ref{thm:main_SC_para} and~\ref{thm:main_SCblender} in Section~\ref{sec:SC}, following the plan stated in Section~\ref{sec:introSC}.

\noindent\textbf{Acknowledgments.}
We thank Artur Avila, Jinxin Xue and Zhiyuan Zhang  for pointing out a potential application of our results to the problem of stable ergodicity. We also thank Marcel Guardia and Jaime Paradela for useful discussions and for communicating their result on symplectic blenders. We are grateful to Gabriella Pinzari for her  encouragement of this research when D.L. was a postdoc at the University of Padova. 
This work was supported by the Leverhulme Trust.
The research of D.L. was also supported by the Science Fund Program for Excellent
Young Scientists (Overseas), the New Cornerstone Science Foundation, and the ERC project 677793 StableChaoticPlanetM.

\section{Definitions and detailed setting}\label{sec:intro2}
In this section, we first give a precise description of blenders. Then we present  the results on the creation of (symplectic) blenders in different cases, depending on whether the homoclinic orbit corresponds to a cubic or quadratic tangency, or is transverse. 

\subsection{Blenders}\label{sec:blender}
Depending on the purpose, the definition of a blender varies in different works, see e.g. \citep{BonDiaVia:05, NasPuj:12,BonDia:12b,BarKiRai:14,BocBonDia:16}. The essential feature shared by all these variants is the generation of robust non-transverse intersections. In this section, we first give a definition that is most convenient for our construction, and then define the connection between blenders and whiskered tori mentioned in the introduction. 

Recall that,  given a compact invariant set $\Lambda$ of a diffeomorphism $g$,   it is uniformly hyperbolic if there exists a pair of continuous cone fields $(\mathcal{C}^{\s},\mathcal{C}^{\u})$ in a small neighborhood $U$ of $\Lambda$ such that the following holds. The cone field $\mathcal{C}^{\s}$ is strictly backward-invariant (i.e., for any point $M\in U \cap g(U)$ and any  vector $v\in \overline{\mathcal{C}^{\s}_M} \subset  T_M U $), we have $\D g^{-1}(v)\in \mathrm{int}(\mathcal{C}^{\s}_{g^{-1}(M)})$, and $\D g$ uniformly contracts vectors   in $\mathcal{C}^{\s}$. The cone field $\mathcal{C}^{\u}$ is strictly forward-invariant (i.e., for any point $M\in U \cap g^{-1}(U)$ and any vector $v\in \overline{\mathcal{C}^{\u}_M} \subset  T_M U $), we have $\D g(v)\in \mathrm{int}(\mathcal{C}^{\u}_{g(M)})$, and $\D g$ uniformly expands vectors   in $\mathcal{C}^{\u}$. The existence of these cone fields yields two continuous invariant subbundles of the tangent bundle of $\Lambda$: the contracting bundle $E^{\s}\subset \mathcal{C}^{\s}$ and expanding bundle $ E^{\u}\subset \mathcal{C}^{\u}$  such that $ T {\Lambda}=E^{\s}\oplus E^{\u}$. 

The hyperbolic set $\Lambda$ can also carry a partially-hyperbolic structure, i.e., at least one of the cone fields  $\mathcal{C}^{\s}$ and $\mathcal{C}^{\u}$ contains a strictly invariant (backwards or, respectively, forwards) subfields of cones of smaller dimension, $\mathcal{C}^{\ss}$, or, respectively, $\mathcal{C}^{\uu}$, which correspond to a stronger contraction, or, respectively, expansion. These subfields give rise to invariant subspaces $E^{\ss}\subsetneq E^{\s}$ and $E^{\uu}\subsetneq E^{\u}$. We denote  $d^{\s}:=\dim E^\s,d^{\ss}:=\dim E^{\ss},d^{\u}:=\dim E^\u,d^{\uu}:=\dim E^{\uu}$.

The blenders we construct near a whiskered torus are {\em basic} sets, i.e., hyperbolic sets that  are 0-dimensional,  compact, transitive and locally maximal. Recall that transitivity means the existence of a dense orbit. The local maximality of the set $\Lambda$ means that 
there exists an open set $U$ such that $\Lambda$ consists of all points whose orbits never leave $U$ (we call $U$ an isolating neighborhood of
$\Lambda$). Basic sets vary continuously for $C^1$-small perturbations of the map.

Let a diffeomorphism  $g$ have a  hyperbolic basic set $\Lambda$ of index $d^\u$. 
The local maximality of  $\Lambda$ means that 
there exists an open set $U$, the {\em isolating neighborhood of $\Lambda$}, such that $\Lambda$ consists of all points whose orbits never leave $U$.  Given an open subset $U'\subset U$, for every point $P\in\Lambda\cap U'$, we define the {\em local stable manifold} $W^\s_{\loc,U'}(P)$ as the connected piece of $W^\s(P)\cap U'$ that contains $P$; the {\em local unstable manifold} $W^\u_{\loc,U'}(P)$ is defined similarly. We also define 
\begin{equation}\label{eq:localmanifolds}
W^{\u/\s}_{\loc,U'}(\Lambda):=\bigcup_{P\in\Lambda}W^{\u/\s}_{\loc,U'}(P).
\end{equation}
We will omit the subscript $U'$ when there is no ambiguity. 
Since $\Lambda$ is 0-dimensional, one can take $U$ sufficiently small so that
\begin{equation}\label{eq:UofLambda}
\Lambda=\bigcap_{i\in\mathbb Z} g^i(U),\qquad
W^\s_{\loc,U}(\Lambda)=\bigcap_{i\leqslant 0}g^i(U),\qquad
 W^\u_{\loc,U}(\Lambda)=\bigcap_{i\geqslant 0}g^i(U).
\end{equation}

\begin{defi}[Blenders]\label{defi:blender}
 The set $\Lambda$ is called
\begin{itemize}[nosep]
\item a {\em center-unstable (cu) blender} if the unstable bundle $E^\u$ contains a non-trivial subbundle $E^{\uu}$ and
there exists a pair of $C^1$-open sets $\mathcal{D}$ and $\mathcal{D}'\subset \mathcal{D}$ of embeddings  of a closed $d^{\uu}$-dimensional disc in $U$ such that
\begin{itemize}[nosep]
\item every embedded disc $D\in\mathcal{D}$ is tangent to the cone field $\mathcal{C}^{\uu}$,
\item  the image  $g(D)$ for every $D\in\mathcal{D}$ contains a disc from $\mathcal{D}'$, and
\item the set $\mathcal{D'}$ lies inside $\mathcal{D}$ at a non-zero $C^1$ distance to the boundary of $\mathcal{D}$;
\end{itemize}
\item a {\em center-stable (cs) blender}  if $\Lambda$ is a cu-blender of $g^{-1}$;
\item a {\em double-blender} if $\Lambda$ is simultaneously a cs-blender  and a cu-blender.
\end{itemize}
The differences $(d^{\s} - d^{\ss})$ and $(d^{\u} - d^{\uu})$ are called the stable, and, respectively, unstable central dimensions of the blender.
\end{defi}

It immediately follows from the above definition that every embedded disc from $\mathcal{D}$ for cs-blenders intersects $W^\u_{\loc}(\Lambda)$, and that every embedded disc from $\mathcal{D}$ for cu-blenders intersects $W^\s_{\loc}(\Lambda)$. By the dimension count, one can see that such an intersection is not transverse, but it still persists for an open set of embedded discs, as shown below. 

Recall that the hyperbolic basic set $\Lambda$ admits a unique continuation $\Lambda_h$ for  any diffeomorphism $h$ sufficiently $C^1$-close to $g$, along with the partially-hyperbolic structure,  with the cone fields and the isolating neighborhood $U$ staying the same. 

\begin{cor}
Up to shrinking $\mathcal{D}$, every embedded disc from $\mathcal{D}$ intersects $W^\u_\loc(\Lambda_h)$ or $W^\s_\loc(\Lambda_h)$  (depending on the type of blenders) for every $h$ close to $g$ in $C^1$. 
\end{cor}
\begin{proof}
For a $C^1$-neighbourhood $\mathcal{U}$ of $g$, we  take   $\hat{\mathcal{D}}=\bigcap_{h\in\mathcal{U}} \{D\in \mathcal{D}:g^{-1}\circ h(D)\in \mathcal{D}\}$.  Since $\mathcal{D}'$ is at a non-zero distance to $\partial \mathcal{D}$, we have
 $\mathcal{D}'\subset \hat{\mathcal{D}}\subset \mathcal{D}$ if $\mathcal{U}$ is sufficiently small. Hence $\Lambda_h$ is a blender for all $h\in\mathcal{U}$, with the pair $(\mathcal{D}',\hat{\mathcal{D}})$. The corollary follows by replacing $\mathcal{D}$ with 
 $\hat{\mathcal{D}}$.
\end{proof}

\begin{defi}[Symplectic blenders {\citep{NasPuj:12}}]\label{defi:blender_symp}
A double-blender for a symplectic diffeomorphism is called a {\em symplectic blender}. 
\end{defi}
Our symplectic blenders are always constructed from cu- and cs-blenders with central dimension one, so we omit the reference to their central dimension in the remainder of the paper.

\begin{rem}\label{rem:blender}
If there are finitely many, say $K$, disjoint open sets $U_i$, and
 a sequence  $\{n_i\}_{i=1}^K$ of positive integers such that  the induced map $\hat f$ of $\bigcup^K_{i=1} U_i$ defined as 
 $\hat f(M)=f^{n_i}|_{U_i}(M)$ for $M\in U_i$ has a blender $\hat\Lambda\subset \bigcup U_i$, then the hyperbolic set $ \Lambda= \bigcup_{i=1}^{K}\bigcup_{j=1}^{n_i-1}f^j(\hat\Lambda) $ is a blender of $f$. 
\end{rem}

\subsubsection{Blenders connected to whiskered tori}\label{sec:blencon}
Let $\Lambda$ be a hyperbolic basic set, and let $U'$ be an open subset of its isolating neighborhood. Denote by  $\ind(\Lambda)$ the index of $\Lambda$. 
Let $W$ be a smooth manifold of codimension $\ind(\Lambda)$.  We say that $W$ is {\em locally transverse to $W^{\u}(\Lambda)$ in the open set $U'$} if $W^{\u}_{\loc,U'}(P)\cap W \neq \emptyset$ for every $P\in\Lambda\cap U'$, and all these intersections are transverse (see the definition of local manifolds in \eqref{eq:localmanifolds}). We write $W \pitchfork_{U'} W^\u(\Lambda)$. The local transversality to $W^\s(\Lambda)$ is defined similarly.

\begin{defi}[Symplectic blenders connected to whiskered tori] \label{defi:sympblen_conn}
Let $f\in \symp^1(\mathcal{M})$ have  a one-dimensional whiskered torus $\gamma$ and a hyperbolic basic set $\Lambda$ of index $N$. 
We say that $\Lambda$ is a  {\em symplectic blender connected to $\gamma$} in an open set $\hat V$  if the following  properties are satisfied:
\begin{itemize}[topsep=0pt]	
\item {\em (partial hyperbolicity)} the set $\hat V$ contains an isolating neighborhood $V$ of $\orb$ and an isolating neighborhood $U$ of $\Lambda$ with
$U\cap V\neq\emptyset$, and the partially-hyperbolic invariant splitting  $T_\orb \mathcal{M}=E^\ss\oplus E^\c\oplus E^\uu$ extends continuously to $\hat V$
in a way compatible with the hyperbolic structure of $\Lambda$, i.e., 
$$
T_\Lambda \mathcal{M}=E^\s\oplus E^\u=(E^\ss\oplus E^\mathrm{ws})\oplus( E^\mathrm{wu}\oplus E^\uu)=E^\ss\oplus E^\c\oplus E^\uu,
$$
where $ E^\mathrm{ws}$ and $ E^\mathrm{ws}$ are one-dimensional invariant bundles;
\item {\em (local transversality)} there exist  hyperbolic basic subsets $\Lambda^{\cu}$ and $\Lambda^{\cs}$ of $\Lambda$ and  open subsets $V^\u$  and $V^\s$ of $V\cap U$  such that 
$$W^\u_\loc(\gamma) \pitchfork_{V^\u} W^\s(\Lambda^{\cu})
\quad\mbox{and}\quad
W^\s_\loc(\gamma) \pitchfork_{V^\s} W^\u(\Lambda^{\cs});$$
\item {\em (blender property)} 
there exist $\delta>0$ , an integer $K>0$ and a $C^1$-open neighborhood $\mathcal{U}$ of $f$ such that, for every $g\in\mathcal{U}$,
\begin{itemize}
\item if an $(N-1)$-dimensional manifold is $\delta$-$C^1$-close
 to a local strong-unstable leaf from $W^\u_\loc(\gamma)$, then it intersects $W^\s(\Lambda_g)$: it contains a point of intersection with $g^{-k}(W^\s_{\loc,V^\u}(P))$  for some  $P\in \Lambda^{\cu}_g\cap V^\u$ and $0\leq k\leq K$, and the first $k$ iterations of the intersection points stay in the isolating neighborhood $V$;
\item  if an $(N-1)$-dimensional manifold is $\delta$-$C^1$-close to a local strong-stable leaf from $W^\s_\loc(\gamma)$, then it intersects $W^\u(\Lambda_g)$: it contains a point of intersection with $g^{k}(W^\u_{\loc,V^\s}(P))$ for some $P\in \Lambda^{\cs}_g\cap V^\s$ and $0\leq k\leq K$, and the first $k$ backward iterations of the intersection point stay in $V$.
\end{itemize}
\end{itemize} 
\end{defi} 

While the whiskered torus $\gamma$ might be destroyed after a small perturbation\footnote{e.g., if the perturbation is not small in a higher regularity class.}, the blender property is $C^1$-robust. 
Specifically, there exist two $C^1$-open sets $\mathcal{D}^\s$ and $\mathcal{D}^\u$ of $(N-1)$-dimensional embedded discs in $V$ (these are the sets of discs sufficiently close to the leaves of $\mathcal{F}^{\ss}$ and, respectively, $\mathcal{F}^{\uu}$) such that all discs in $\mathcal{D}^\s$  and $\mathcal{D}^\u$ admit non-transverse intersections with, respectively, $W^\s(\Lambda_g)$ and $W^\s(\Lambda_g)$ for every $g$ which is $C^1$-close to $f$. 

The purpose of including the information on the partial hyperbolicity and the connection to a whiskered torus in Definition~\ref{defi:sympblen_conn} is to make the application of the symplectic blender results more convenient. In particular, we use the properties included in this definition to prove Theorem~\ref{thm:main_SC_para}.

As mentioned, we construct symplectic blenders from cu- and cs-blenders.

\begin{defi}[Blenders connected to whiskered tori]\label{defi:blenConn}
We say that a cu-blender $\Lambda$ of index $N$  is {\em connected to a one-dimensional whiskered torus $\gamma$} via a neighborhood $\hat V$ of $\mathcal{O}(\gamma)\cup \Lambda$ if  
\begin{itemize}[nosep]	
\item {\em (partial hyperbolicity)}  the set $\hat V$ contains an isolating neighborhood $V$ of $\orb$ and an isolating neighborhood $U$ of $\Lambda$ with $U\cap V \neq \emptyset$, and the partially-hyperbolic structure on $\orb$  extends to $\hat V$ in the sense of Definition~\ref{defi:sympblen_conn};
\item {\em (local transversality)} there exists  an open subset $V^\u \subset V$  such that 
$W^\u_\loc(\gamma) \pitchfork_{V^\u} W^\s(\Lambda);$
\item {\em (blender property)} the first conclusion in the blender property of Definition~\ref{defi:sympblen_conn} holds with taking $\Lambda_g^{\cu}=\Lambda_g$.
\end{itemize}
The connection of a cs-blender to $\gamma$ is defined similarly.
\end{defi}

\begin{lem}\label{lem:blentoblenCon}
Let $\Lambda$ be a cu-blender of Definition~\ref{defi:blender},  satisfying the partial hyperbolicity and local transversality conditions of Definition~\ref{defi:blenConn}. If $\hat\Lambda :=\Lambda\cap V^\u$ is a cu-blender of an induced map $\hat f$ (see Remark~\ref{rem:blender}) defined  in an open subset of $V^\u$, and if there exists  a local strong-unstable leaf $\ell^\uu\subset W^\u_\loc(\gamma)$  containing a disc in the set $\mathcal{D}_{\hat \Lambda}$ associated with $\hat\Lambda$, then $\Lambda$ satisfies the blender property of Definition~\ref{defi:blenConn}.
\end{lem}

\begin{proof}
Take, an isolating neighborhood $U_{\hat \Lambda}\subset V^\u$ associated with $\hat\Lambda$. By Definition~\ref{defi:blender}, every disc in $\mathcal{D}_{\hat \Lambda}$ intersects the local stable manifold $W^\s_{\loc,U_{\hat \Lambda}}(P)\subset W^\s_{\loc,V^\u}(P)$ for some point $P\in \hat \Lambda\subseteq \Lambda$.
By the irrationality of the rotation number of $f^\per|_\gamma$, there exists $K>0$ such that for any $P'\in\gamma$, $f^{k}(P')$ is, for some $0\leq k\leq K$, sufficiently close to $\ell^\uu\cap \gamma$, so that the strong-unstable leaf through $f^{k}(P')$ contains a disc from the collection $\mathcal{D}_{\hat \Lambda}$ (recall that
$\mathcal{D}_{\hat \Lambda}$ is $C^1$-open by definition). The same is then true for the image by $f^k$ of any $(N-1)$-dimensional manifold sufficiently $C^1$-close to the local strong-unstable leaf of $P'$. 
\end{proof}

Note that the cu-blender in Definition~\ref{defi:blenConn} is also required to be partially-hyperbolic in the stable bundle, i.e., it admits the additional splitting $E^\s=E^\mathrm{ws}\oplus E^\ss$. Similarly, we have the partially-hyperbolic splitting for a cs-blender connected to $\gamma$. Thus, in order to obtain a symplectic blender connected to $\gamma$ in the sense of
Definition~\ref{defi:sympblen_conn} , it suffices to construct in the partially-hyperbolic neighborhood $\hat V$ a cu-blender and a cs-blender that are connected to $\gamma$ and are homoclinically related to each other through orbits in $\hat V$. 

\subsection{Whiskered tori} \label{sec:whiskeredtori}

 Let $f\in \symp^s(\mathcal{M})$  have   a periodic smooth  curve $\gamma\cong \mathbb{S}^1$, i.e., $f^\ell(\gamma)=\gamma$ for some positive integer $\ell$. We denote by $\mathrm{per(\gamma)}$    the smallest of such $\ell$. Assume that $\gamma$  is  symmetrically partially-hyperbolic  with a two-dimensional center. 
This means that there exists a continuous, $f$-invariant splitting ${T}_\orb \mathcal{M}=E^\ss \oplus E^\c \oplus E^\uu$, where  $\dim E^\c=2$, $\dim E^\ss =\dim E^\uu =N-1$ and $E^\c$ contains the tangent bundle ${T} \orb$.  Moreover, for some appropriate norm, there exist  constants $\cc,\lambda\in \mathbb R$ satisfying 
\begin{equation}\label{eq:intro:1aa}
0 < \lambda < 1 < \cc < \lambda^{-1},
\end{equation}
such that, for all points $P\in \orb$,
\begin{equation}\label{eq:intro:1}
\|\D f|_{E^\c_P}\|< \cc, \quad \|(\D f)^{-1}|_{E^\c_P}\|< \cc, \quad
\|\D f|_{E^\ss_P}\|<\lambda, \quad \|(\D f)^{-1}|_{E^\uu_P}\|<\lambda.
\end{equation}

By the theory of normal hyperbolicity (see e.g. \cite{Fen:71,HirPugShu:77,ChoLiuYi:00,DeldlLSea:06,GelTur:17}), there exists  a two-dimensional  symplectic  manifold $\A\cong \mathbb R\times \gamma$ in a sufficiently small tubular neighborhood $V$ of $\gamma$, which is locally invariant under $f^\per$. The set $\orba=\bigcup_{n=0}^{\per-1}f^n(\mathbb{A})$ contains $\orb$ and is tangent to $E^\c_P$ for every $P\in\orb$. For sufficiently small neighborhood $V$, the  splitting ${T}_P \mathcal{M}= E^\ss_P \oplus E^\c_P \oplus E^\uu_P $ extends to all $P\in V$  with conditions  \eqref{eq:intro:1aa} and \eqref{eq:intro:1}  fulfilled by the same $\lambda$ and $\cc$, and $\orba$ is tangent to $E^\c_P$ for every $P\in\orba$. 

The {\em local center-stable manifold} $W^{\s}_{\loc}(\orba)$ is an $(N+1)$-dimensional locally forward-invariant smooth
 manifold, which contains $\orba$  and is tangent to $E^\ss\oplus E^\c$ at the points of $\orba$. Similarly, there exists an 
  $(N+1)$-dimensional {\em local center-unstable manifold} $W^{\u}_{\loc}(\orba)$, which is  tangent to $E^\c\oplus E^\uu$ at the points of $\orba$. The manifold $W^{\s}_{\loc}(\orba)$ contains all forward orbits that never leave $V$, and $W^{\u}_{\loc}(\orba)$ contains all backward orbits that never leave $V$.

There exist strong-stable and strong-unstable invariant  foliations $\mathcal{F}^{\ss}_\loc$ and $\mathcal{F}^{\uu}_\loc$ on the manifolds $W^\s_\loc(\orba)$ and $W^\u_\loc(\orba)$, respectively, and they are both  continuous  and have $(N-1)$-dimensional $C^s$-smooth leaves. For each point $P\in \orba$,  the leaf of $\mathcal{F}^{\ss}_\loc$ through $P$ is tangent to $E^{\ss}_P$ and, if the forward orbit  of $P$ stays in $V$, then the leaf is exactly the set of points whose forward orbits converge to the forward orbit of $P$ exponentially with a rate at least $\lambda$; similarly, the strong-unstable leaf of $P$ is tangent to $E^{\uu}_P$ and, if the backward orbit of $P$ stays in $V$, it is the set of points whose backward orbits converge to the backward orbit of $P$ with a rate at least $\lambda$. The union of the strong-stable leaves of  $\gamma$ forms its local stable manifold $W^\s_\loc(\gamma)$, and the union of the strong-unstable leaves of $\gamma$ forms its local unstable manifold $W^\u_\loc(\gamma)$. Similarly, we define the stable manifold $W^\s_\loc(\A)$ and unstable manifold $W^\u_\loc(\A)$. 
Since $\gamma$ is $f^\per$-invariant, its  global stable and unstable manifolds are well-defined as follows: $W^\s(\gamma)=\bigcup_{i\geqslant 0}f^{-i}(W^\s_{\loc}(\gamma))$ and $W^\u(\gamma)=\bigcup_{i\geqslant 0}f^{i}(W^\u_{\loc}(\gamma))$.

\begin{defi}[Whiskered tori]\label{defi:tori}
The  curve $\gamma$ described above is called a {\em one-dimensional whiskered torus} of $f$. 
\end{defi}

The order of smoothness of $\A$ is, in general, equal to $\min\{s,-\ln\lambda/\ln \cc\}$. However, in the present paper, we always assume that $f^{\per}|_\gamma$ is topologically conjugate to rotation. In this case,  the Lyapunov exponent along the tangent direction to $\gamma$ is zero. Moreover, since $f^{\per}|_A$ preserves the symplectic form, the other central Lyapunov exponent must also be zero. So, the gap between the contraction in $E^\s$ and in $E^\c$ is infinitely large at the points of $\gamma$, and the same is true for the gap between the expansion in $E^\u$ and in $E^\c$. This implies that the norm in the tangent spaces can be chosen such that $\cc$ in \eqref{eq:intro:1} gets  arbitrarily close to 1, and, in particular, the ratio $(-\ln\lambda/\ln \cc)$  becomes arbitrarily large. Thus, when the isolating neighborhood $V$ of $\orb$ is sufficiently small, the manifolds $\A,W^\s_\loc(\A),W^\u_\loc(\A)$ are of the same smoothness class $C^s$ as the map $f$ when $s$ is finite, and of any given finite smoothness when $f\in \symp^\infty(\mathcal{M})$. 
The invariant foliations $\mathcal{F}_\loc^{\ss}$ and $\mathcal{F}^{\uu}_\loc$ are $C^{s-1}$ (or, of any given finite smoothness if $s=\infty$), while the order of  smoothness  of $W^{\s/\u}_\loc(\gamma)$ equals to the minimum of $(s-1)$ and that of $\gamma$. In the case where both $f$ and $\gamma$ are $C^\infty$, the invariant manifolds of $\gamma$ are also $C^\infty$.

\subsection{Hyperbolic homoclinic tangencies}\label{sec:phtangency}
Let $\Gamma$ be a homoclinic orbit of the whiskered torus $\gamma$, i.e., $\Gamma\subset W^\s(\gamma)\cap W^\u(\gamma)$.  For any point $M\in \Gamma$, let $n^\pm \in \mathbb{Z}$ be such that $M$ belongs to both  $ W^\s_{n^+}(\gamma):=\bigcup_{i=0}^{n^+}f^{-i}(W^\s_{\loc}(\gamma))$
and $W^\u_{n^-}(\gamma):=\bigcup_{i=0}^{n^-}f^{i}(W^\u_{\loc}(\gamma))$.
 We define $W^\s(\A)$ and $W^\u(\A)$ to be  small neighborhoods of $W^\s_{n^+}(\gamma)$ in $\bigcup_{i=0}^{n^+}f^{-i}(W^\s_{\loc}(\A))$ and, respectively, of $W^\u_{n^-}(\gamma)$  in $\bigcup_{i=0}^{n^-}f^{i}(W^\u_{\loc}(\A))$. The strong-stable foliation $\mathcal{F}^\ss$ of $W^\s(\A)$   and strong-unstable  foliation $\mathcal{F}^\uu$ of $W^\u(\A)$ are obtained by iterations of the local foliations  $\mathcal{F}^\ss_\loc$ and $\mathcal{F}^\uu_\loc$.
 
\begin{defi}[Partially-hyperbolic homoclinics]\label{defi:phorbit}
Let $\ell^{\ss}$ and $\ell^{\uu}$ be the strong-stable and, respectively, the strong-unstable leaves through $M$. 
We say that $\Gamma$ is a {\em partially-hyperbolic homoclinic orbit}  if the set  $\gamma\cup\Gamma$ is, i.e., the following is satisfied:
\begin{equation}\label{eq:condition:trans1}
 T_M \ell^{\ss} \oplus  T_M \ell^{\uu} \oplus  T_M (W^\s(\A)\cap W^\u(\A)) = \mathbb{R}^{2N},
\end{equation}
or, equivalently,  $\ell^{\ss}$ and $\ell^{\uu}$ are transverse at the point $M$ to $W^\u(\A)$ and $W^\s(\A)$, respectively.
\end{defi}
We can then take a small neighborhood $\Sigma$ of $M$ in  $W^\u(\A)\cap W^\s(\A)$, which is a symplectically embedded  two-dimensional disc  transverse to the foliations $\mathcal{F}^{\ss}$ in $W^\s(\A)$  and $\mathcal{F}^{\uu}$  in $W^\u(\A)$, respectively (see \citep[Section 5]{GelTur:17}). 
Let $\Sigma^+\subset W^\s_\loc(\A)$ and  $\Sigma^-\subset W^\u_\loc(\A)$  be some forward, and, respectively, backward iterations of $\Sigma$, so that one has $\Sigma^+=f^n (\Sigma^-)$ for some integer $n>0$. By the transversality of  
$\mathcal{F}^{\ss}$ to $\Sigma$, the holonomy map $\pi^\s:\Sigma^+ \to \pi^\s(\Sigma^+)\subset \A$ defined by the leaves of $\mathcal{F}^{\ss}$ is a diffeomorphism (of class $C^{s-1}$). Similarly,  there is a holonomy map  $\pi^\u :\Sigma^- \to \A $,  defined by  the leaves of  $\mathcal{F}^{\uu}$. 
The scattering map is  defined on $\A$ as
\begin{equation}\label{eq:scattering_1}
S:={\pi}^\s \circ f^n \circ ({\pi}^\u)^{-1}: {\pi}^\u(\Sigma^-) \to {\pi}^\s(\Sigma^+),
\end{equation}
which is a  symplectic $C^{s-1}$-diffeomorphism (see \citep{DeldlLSea:06,GelTur:17}). 

We further assume that $\Gamma$ is an orbit of {\em homoclinic tangency}, that is,  $W^\s(\gamma)$ intersects 
$W^\u(\gamma)$ non-transversely  at the points of $\Gamma$.
By \eqref{eq:condition:trans1}, the intersections $W^\s(\gamma)\cap \Sigma$ and $W^\u(\gamma)\cap \Sigma$ are one-dimensional curves and, by the non-transversality assumption, they are tangent at the point $M$, see Figure~\ref{fig:tangency}.
Because the invariant manifolds of $\gamma$ consist of the leaves of the foliations $\mathcal{F}^{\ss}$ and $\mathcal{F}^{\uu}$, we have that $\pi^\s(W^\s(\gamma)\cap \Sigma)$ and $\pi^\u(W^\u(\gamma)\cap \Sigma)$ are arcs of $\gamma$ near the points $\pi^\s(M)$ and $\pi^\u(M)$, respectively. Therefore, the curve $S(\gamma)$ is tangent to $\gamma$ at $\pi^\s(M)$, i.e., the derivative $\D S$ takes a tangent $v$ to $\gamma$ at the point $\pi^\u(M)$ to a tangent $\D S(v)$ to $\gamma$ at the point  $\pi^\s(M)$. We denote by $\alpha$   the ratio of the signed length of these two vectors, i.e.,
\begin{equation}\label{beta} 
\D S(v) =\alpha v.
\end{equation}
So, $|\alpha|$ measures the contraction or expansion by $\D S$ along $\gamma$ (see formula \eqref{eq:a11}).   

Note that $\alpha$ depends on the choice of coordinates. In this paper, we focus on the case where $f^\per|_\gamma$ is 
an irrational rotation in some smooth coordinates. These coordinates are unique up to a rotation, so we define $\alpha$ in these coordinates: it becomes a well-defined quantuty and the following definition makes sense.

\begin{defi}[Hyperbolic homoclinic tangencies]\label{defi:phtangency}
A  homoclinic tangency of $\gamma$ is called  {\em  partially hyperbolic} if $\Gamma$ is partially hyperbolic (i.e.,  \eqref{eq:condition:trans1} is satisfied). It is called {\em  hyperbolic} if, in addition,
$$|\alpha|\neq 1;$$
in this case, it is said to be  {\em contracting} if $|\alpha|<1$ and {\em expanding} if $|\alpha|>1$. 
\end{defi}
 
\begin{figure}[!h]
\begin{center}
\includegraphics[scale=1]{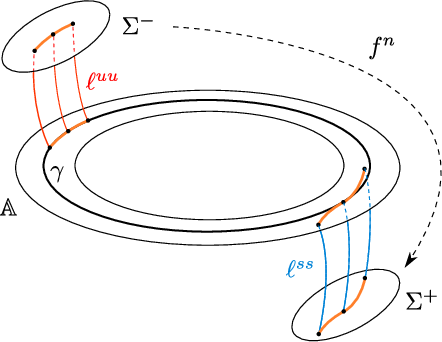}
\end{center}
\caption{
A schematic picture for a partially-hyperbolic cubic  homoclinic tangency, where the space above $\A$  represents $W^\u_{\loc}(\A)$ and  the space below $\A$  represents $W^\s_{\loc}(\A)$.  Here $\ell^{\ss}$ and $\ell^{\uu}$ are the strong-stable and strong-unstable leaves, respectively. The orange piece of $\gamma$  is mapped by ${{S}}$ to a curve tangent to $\gamma$ at $\pi^\s(M)$.
}
\label{fig:tangency}
\end{figure}

\subsection{Blenders near a whiskered torus with a 2-flat homoclinic tangency}\label{sec:kam}

If $f$ is at least $C^3$, then the stable and unstable foliations are $C^2$. So, the scattering map $S$ is $C^2$ and, when 
$\gamma$ is $C^2$ as well, the stable and unstable manifolds of 
$\gamma$ are $C^2$. Thus we can distinguish quadratic tangencies between $W^\s(\gamma)$ and $W^\u(\gamma)$ from those with higher degeneracy. We call a tangency {\em 2-flat} if the first two derivatives of $S$ along $\gamma$ vanish at the tangency point. 
For example, a cubic tangency is 2-flat, but a 2-flat one is not necessarily cubic (see Section~\ref{sec:transtangency} for a precise description.)

\begin{thma}\label{thm:main_blender_cubic_nonpara}
Let $f\in \symp^3(\mathcal{M})$ have a one-dimensional whiskered torus $\gamma$ of class $C^2$ such that $f^\per|_\gamma$ is $C^1$-conjugate to an irrational rotation. Let $\gamma$ have an orbit $\Gamma$ of a  hyperbolic  2-flat homoclinic tangency. Then, given any neighborhood $\hat V$ of $\orb\cup \Gamma$, there exists a blender $\Lambda \subset \hat V$ connected to $\gamma$, center-stable if the tangency is contracting and center-unstable if expanding.
\end{thma}

\begin{rem}\label{rem:herman}
By Herman-Yoccoz Theorem \citep{Yoc:84}, the requirement of $C^1$ conjugacy is automatically satisfied when $\gamma$ is $C^3$ and 
$\rho(\gamma)$ is Diophantine of order $<1$.
\end{rem}

Theorem~\ref{thm:main_blender_cubic_nonpara} implies that if the cylinder $\A$ is taken sufficiently small, then for any point in $\A$ either, when  $|\thebeta| <1$, its strong-stable leaf intersects $W^\u(\Lambda)$ or, when $|\thebeta| >1$, its strong-unstable leaf intersects $W^\s(\Lambda)$. Moreover, the intersections persist for all $C^1$-close maps and all $(N-1)$-dimensional discs which are $C^1$-close to the strong-stable (for $|\alpha| <1$) or strong-unstable  (for $|\alpha| >1$) leaves through $\A$ -- even though the whiskered torus $\gamma$ does not necessarily persist at $C^1$-small perturbations.

The theorem is proved in Section~\ref{sec:proofcu}. The proof is based on the  analysis of the dynamics of a pair of maps $T_0$ and $T_1$. The {\em local map} $T_0$ is  the restriction of $f^\per$ to a small neighborhood $V$ of $\A$. The map $T_1$ is the {\em transition map} along the homoclinic orbit $\Gamma$, i.e.,  the iteration of $f$  which takes a small neighborhood $\Pi^-$ of some homoclinic point $M^-\in \Gamma\cap W^\u_{\loc}(\gamma)$  to  a small neighborhood $\Pi^+$ of  a homoclinic point $M^+\in \Gamma\cap W^\s_{\loc}(\gamma)$, see Figure~\ref{fig:returnmap}. 
Because the rotation number $\rho(\gamma)$ is irrational, there are infinitely many $k$ values for which  $T^k_0(\Pi^+)\cap \Pi^-\neq \emptyset$ so the first-return maps $T_k:=T^k_0\circ T_1$ are  well-defined for such $k$, and $T_k(\Pi^-)\cap \Pi^-\neq \emptyset$. A detailed discussion on these maps is given in Section \ref{sec:para}.

\begin{figure}[!h]
\begin{center}
\includegraphics[scale=1.3]{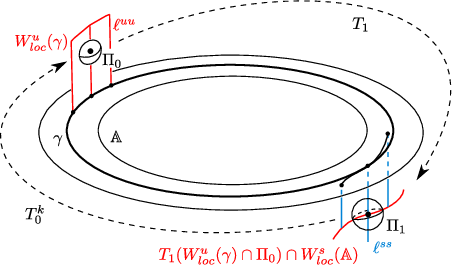}
\end{center}
\caption{
A schematic picture for the return maps, where the vertical lines $\ell^{\ss}$ and $\ell^{\uu}$ are the strong-stable and strong-unstable leaves, respectively. The two points in $\Pi^-$ and $\Pi^+$ are $M^+$ and $M^-$, respectively, and the image of $W^\u_{\loc}(\gamma)$ by $T_1$ intersects  $W^\s_{\loc}(\A)$ along the red curve, which is tangent to $W^\s_{\loc}(\gamma)$ at $M^+$ and whose projection by $\pi^\s$ is tangent to $\gamma$ in $\A$ at $\pi^\s(M^+)$.
}
\label{fig:returnmap}
\end{figure}

It follows from the partial hyperbolicity of $\orb\cup \Gamma$ that the dynamics of the first-return maps $T_k$ are essentially reduced to the central dynamics, projected to the two-dimensional cylinder $\A$. We analyze the central dynamics by rescaling the coordinates in $\Pi^-$ and $\Pi^+$ to a factor vanishing as $k\to\infty$. In Section \ref{sec:scaling}, we show that, at large $k$, the rescaled maps $T_k$ projected to the central coordinates  are close to the affine maps 
\begin{equation}\label{eq:ifs}
(R,\Phi)\mapsto(\bar R,\bar \Phi)= ( \alpha^{-1}R,  \alpha \Phi + c(k)),
\end{equation}
for some constants $c(k)$. Obviously, these maps are hyperbolic for
$|\alpha|\neq 1$. Using the irrationality of the rotation number $\rho$ and the Dirichlet approximation theorem, we find a finite set $\mathcal{K}$ of $k$ values for which  the constants $c(k)$ form a sufficiently dense grid. This ensures the fulfilment, by the iterated function system \eqref{eq:ifs}, of a version of the  so-called covering property, which is  crucial for the creation of blenders (see e.g. \citep{BonDiaVia:05,BarKiRai:14}). In this way, a blender is found for the induced map ${\hat T}$:
$${\hat T}(P)=T_k(P) \quad\mbox{if}\quad P\in T^{-1}_k(\Pi^-)\cap\Pi^+
\quad\mbox{for}\quad k\in \mathcal{K},
$$
which by Remark~\ref{rem:blender} gives a blender of $f$. 

\subsection{Symplectic blenders at the unfolding of homoclinic tangencies to a whiskered torus}\label{sec:quadraunfold}
Note that Theorem~\ref{thm:main_blender_cubic_nonpara} has a non-perturbative nature: the existence of a hyperbolic
cubic homoclinic tangency to  a whiskered torus implies the existence of a cs- or cu-blender. In this subsection, 
we describe the setting for the  perturbative result, Theorem~\ref{thm:main_2punfolding}, which shows that symplectic blenders can emerge in two-parameter families.
The proof of this theorem is based on Theorem~\ref{thm:main_blender_cubic_nonpara}. So, it is important that  the local structure involved in Theorem~\ref{thm:main_blender_cubic_nonpara} persists under small perturbations. To this aim, we introduce
\begin{defi}[Whiskered KAM-tori]\label{defi:kamtori}
A one-dimensional whiskered torus  $\gamma$ of $f$  is called a {\em whiskered KAM-torus}    if it is a KAM-curve  for the restriction  $f^\per|_\A$. We further call $\gamma$ {\em non-degenerate} if it is a non-degenerate KAM-curve.
\end{defi}

In order to explain this definition, let us collect some facts on circle maps and the  KAM theory. Let $\tilde f$ be a $C^s$ exact symplectic  diffeomorphism of a two-dimensional  cylinder, with an invariant curve $\gamma\cong \mathbb{S}^1$ of class $C^{s'}$ for some number $s'\leq s$. 

Recall that an irrational  number $\rho$
 is called $(c,\tau)$-Diophantine for some $c>0$ and $\tau\geqslant 2$ if
 \begin{equation}\label{eq:diophantine} 
 \left|\rho - \dfrac{p}{q}\right| > \dfrac{c}{q^{\tau}},\quad p\in\mathbb Z,\;q\in \mathbb Z\setminus \{0\}.
 \end{equation}
By Yoccoz \citep{Yoc:84}, there is $s''>0$ depending on $\tau$ such that, if the rotation number $\rho$ of $\tilde f|_\gamma$ is $(c,\tau)$-Diophantine and $s'\geq s'' \geq 3$, then  $\tilde f|_\gamma$ is smoothly conjugate to a rigid rotation with the same rotation number. Moreover, the smoothness of the conjugacy increases with that of the map $\tilde f|_\gamma$, and is $C^\infty$ or real analytic if $\tilde f|_\gamma$ is.

One extracts from the above the existence of smooth symplectic coordinates 
$(r,\varphi)\in \mathbb{R}^2$  in a small  neighborhood $\A$  of 
$\gamma$  such that the curve $\gamma$ is given by $\{r=0\}$ and  
$\tilde f|_\A$ assumes the form 
\begin{equation}\label{eq:kam_low}
\bar r = r + O(r^2),\qquad
\bar \varphi =\varphi + \rho+ O(r),
\end{equation}
where the terms $O(r^2)$ and $O(r)$  are functions of $r$ and $\pp$, periodic in $\pp$ with period 1. The smoothness of the coordinates can be made arbitrarily high by taking $s'$ and $s$ sufficiently large and 
$\A$ sufficiently small. 

\begin{defi}[KAM-curves]\label{defi:kam}
The invariant curve $\gamma$  is called a {\em KAM-curve} of $\tilde f$  if  there exist smooth symplectic coordinates  such that the restriction $\tilde f|_\A$  takes the form \eqref{eq:kam_low}  with    Diophantine $\rho$. We further call the KAM-curve {\em non-degenerate} if $\int_{\pp\in \mathbb{S}^1}\frac{d\bar\pp(0,\pp)}{d r} d\pp\neq 0$ (i.e., the twist condition is satisfied).
\end{defi}
The KAM theory  establishes that, if the map $\tilde f$ and the non-degenerate KAM-curve $\gamma$ are sufficiently smooth, then  $\gamma$ is accumulated (from both sides) by a large measure set of non-degenerate KAM-curves \citep{Mos:62}. In particular,  we single out a large measure subset  of these KAM-curves (including $\gamma$) whose rotation numbers are characterized by the same $c,\tau$ in \eqref{eq:diophantine}  as $\rho$. 
Every such curve persists for all $C^s$-small exact symplectic perturbations of $\tilde f$ if $s$ is sufficiently large, in the sense that the perturbed map has a KAM-curve which is close to the original one and has the same rotation number. Thus, the whole structure described above persists under small perturbations. See the discussion in Section~\ref{sec:kam_strait}.

Returning to the higher-dimensional case, we have
\begin{lem}
Let $f$   have a one-dimensional whiskered torus $\gamma$. If  $\rho(\gamma)$ is Diophantine and $f$ and $\gamma$ are sufficiently smooth, then $\gamma$ is a whiskered KAM-torus.
\end{lem}
\begin{proof}
Under the conditions of the lemma, $f^\per|_\gamma$ is smoothly conjugate to a rotation. Thus, as discussed after Definition~\ref{defi:tori}, the cylinder $\A$ defined in Section~\ref{sec:kam} is of the same smoothness as $f$.
This gives enough smoothness to bring $\tilde f:= f^\per|_\A$ to the form \eqref{eq:kam_low}, up to shrinking $\A$.
\end{proof}

Let us take $\gamma$ to be a non-degenerate whiskered KAM-torus of sufficiently high smoothness.
By the normal hyperbolicity of $\A$, any  $C^s$ symplectic diffeomorphism $g$ that is sufficiently close to $f$ admits a  continuation  $\A_g$ of $\A$. The restriction $g^\per|_{\A_g}$ is a small   perturbation of $f^\per|_{\A}$. 
When $g$ is locally exact near $\orba$ (see Remark~\ref{rem:exact}), all KAM-curves (of the same Diophantine class as $\gamma$) persist  for $g^\per|_{\A_g}$.

Thus, we further take $\A$ to be a subcylinder that contains $\gamma$ and is bounded by two persistent KAM-curves. Then $\A$ is invariant with respect to $f^\per$, which implies \citep{Fen:71}  that $\A$ is uniquely defined, along with its  continuation as $f$ varies within the class of symplectic diffeomorphisms that are locally exact near $\orba$.

Since the orbit $\orba$ is now invariant, rather than locally invariant, the center-stable manifold  and center-unstable manifold  can be defined globally by $W^\s(\A)=\bigcup_{i\geqslant 0}f^{-i}(W^\s_{\loc}(\A))$ and $W^\u(\A)=\bigcup_{i\geqslant 0}f^{i}(W^\u_{\loc}(\A))$, where the local manifolds are defined in Section~\ref{sec:kam}. Similarly, the foliations $\mathcal{F}^\ss$ and $\mathcal{F}^\uu$ can be extended globally. Like in Section~\ref{sec:phtangency}, we only operate with  the subsets of $W^{\s}(\A)$ and $W^{\u}(\A)$ that  correspond to finitely many iterations of  the local manifolds.

Let us  assume that $\gamma$ has an orbit $\Gamma$ of partially-hyperbolic homoclinic tangency between $W^\s(\gamma)$ and $W^\u(\gamma)$. As in Section \ref{sec:kam}, the strong-stable and strong-unstable foliations near $\Gamma$ define the scattering map $\hat S $ as in \eqref{eq:scattering_1}.
The map $S$  takes a small piece $I$ of $\gamma$ to a curve ${S}(I)$ that has a  tangency to $\gamma$. In certain coordinates $(r,\pp)\in \A$, the curve ${S}(I)$ has the form $r=\beta \pp^2+o(\pp^2)$  if the tangency is quadratic and  $r=\ \beta \pp^3 + o(\pp^3)$ if  the tangency is cubic, where $\beta\neq 0$ is some constant; here $\gamma:\{r=0\}$ and $(r,\pp)=0$ is the tangency point. 
For a sufficiently smooth family $\{f_\eps\}$ with $f_0=f$,  all objects above  depend smoothly on $\eps$. So, for all $\eps$ sufficiently close to $0$, the equation of  ${S}(I)$ takes the  general form
$$
r=\mu +  \beta (\pp-\pp^*)^2 + o((\pp-\pp^*)^2),
$$
for quadratic tangencies, or
$$
r=\mu +  \nu (\pp-\pp^*) + \beta (\pp-\pp^*)^3 +o( (\pp-\pp^*)^3),
$$
for cubic tangencies, where  all coefficients depend smoothly on $\eps$, and $\pp^*$ satisfies $\pp^*(\eps)=0$ and is chosen to kill the linear or quadratic term in the corresponding equation. We call $\mu$ and $\nu$  the {\em splitting parameters}; in the quadratic case $\mu$  can be interpreted as the signed distance between ${S}(I)$ and $\gamma$.

In case (1) of  Theorem~\ref{thm:main_2punfolding}, where $\gamma$ has a cubic  tangency, the genericity condition imposed on the two-parameter unfolding family $\{f_\eps\}$ is 
\begin{equation}\label{eq:unfoldcubic}
\det\left.\dfrac{\partial (\mu,\nu)}{\partial (\eps_1,\eps_2)}\right|_{\eps=0}\neq 0.
\end{equation}
In case (2) of two quadratic tangencies,  we take the splitting parameters $\mu_1$ and $\mu_2$ for each of the two homoclinic tangencies, and the genericity condition is
\begin{equation}\label{eq:unfold2quadra}
\det\left.\dfrac{\partial (\mu_1,\mu_2)}{\partial (\eps_1,\eps_2)}\right|_{\eps=0}\neq 0.
\end{equation}
Finally, in case (3), in addition to the splitting parameter $\mu$, we use $\alpha$ defined in \eqref{beta}, which measures the contraction/expansion of ${S}(I)$. The family $\{f_\eps\} $ is called {\em proper}  if
\begin{equation}\label{eq:unfoldquadra}
\det\left.\dfrac{\partial (\mu,\alpha)}{\partial (\eps_1,\eps_2)}\right|_{\eps=0}\neq 0.
\end{equation}

We now outline the proof of Theorem~\ref{thm:main_2punfolding}.
In case~(1), we show that unfolding a partially-hyperbolic cubic tangency can always produce an expanding cubic tangency (see Lemma~\ref{lem:a11}).
Applying Theorem~\ref{thm:main_blender_cubic_nonpara} then yields a cu-blender $\Lambda_1$ connected to $\gamma$.
Since a blender is $C^1$-robust, we may, while keeping $\Lambda_1$, apply the same argument to $f^{-1}$ to obtain a cu-blender $\Lambda_2$ of $f^{-1}$, which corresponds to a cs-blender of $f$.
We next show that these two blenders are homoclinically related, so that a locally maximal invariant hyperbolic set containing $\Lambda_1\cup\Lambda_2$ is a symplectic blender connected to $\gamma$. This concludes case (1). 
 The other two cases are reduced to case (1) by the following
\begin{prop}
In case (2) or case (3), there exists a sequence of $\eps\to 0$ for which $f_\eps$ has a partially-hyperbolic cubic  homoclinic tangency of $\gamma$ that unfolds generically as $\eps$ varies.
\end{prop}

The proof is given by Lemmas~\ref{lem:1to2} and~\ref{lem:cubic}.
In case (3), we unfold the original quadratic tangency to obtain  values of $\mu$ and $\alpha$ for which there exist two coexisting partially-hyperbolic quadratic tangencies that unfold independently as $\mu$ and $\alpha$ vary, reducing case~(3)  to case~(2)  (see Lemma~\ref{lem:1to2}). This is achieved by using an inclination lemma (Lemma~\ref{lem:incli}) for whiskered KAM-tori. Finally, in Lemma~\ref{lem:cubic}, we show that the unfolding of the two quadratic tangencies leads to a partially-hyperbolic cubic tangency -- this is similar in spirit to results of \cite{GonTurShi:07}.

\section{Local map and transition maps for  a whiskered torus with homoclinics}\label{sec:para}

In this section, we  derive  formulas for the local map $T_0$ and its iterations $T_0^k$, and for the the transition map $T_1$.
 Since all our results, except for Theorem~\ref{thm:main_blender_cubic_nonpara}, involve perturbations, we also consider a family $\{f_\eps\}\subset \symp^s(\mathcal{M})$ where $f_0=f$ and   $f_\eps$ is jointly $C^\s$ with respect to $\eps$ and coordinates. Consequently, the maps $T_0$ and  $T_1$ also depend on $\eps$.

\subsection{Local map}\label{sec:local}
Recall that the local map $T_0$ is defined as $f^\per|_V$ for some neighborhood $V$ of the normally-hyperbolic cylinder $\A$ containing 
$\gamma$. The restriction of $T_0$ to $\gamma$ is smoothly conjugate to the rigid rotation $\pp\mapsto \pp+\rho$, which implies that the normal hyperbolicty of the cylinder $\A$ is strong enough, so it has the same smoothness $s$ as the map $f$, or it can be taken to have any given finite smoothness $s$ when $f\in C^\infty$, as discussed in Section~\ref{sec:whiskeredtori}. The invariant foliations 
$\mathcal{F}^\ss$ and $\mathcal{F}^\uu$ are $C^{s-1}$.

\subsubsection{Fenichel coordinates}\label{sec:innermap}
We, first, introduce $C^s$ coordinates in $V$ such that $\A$ and its local stable and unstable manifolds get {\em straightened}, i.e., these are coordinates  $(r,\varphi,x,y)\in \mathbb R\times \mathbb R\times \mathbb{R}^{N-1} \times \mathbb{R}^{N-1}$  such that
\begin{equation}\label{eq:strait}
\A=\{x=0,y=0\},\qquad W^\s_{\loc}(\A)=\{y=0\}, \qquad W^\u_{\loc}(\A)=\{x=0\}.
\end{equation}
Note that $\varphi$ is an angular vatriable, so we identify the points corresponding to $\varphi$ and $\varphi+1$ -- this is accompanied by gluing the $(x,y)$ variables in a way that respects the symplectic form.

The symmetrically normally-hyperbolic manifold $\A$ is symplectic (see e.g. \citep{GelTur:17}), implying that the symplectic form  $\Omega|_\A$ is given by $q(r,\pp;\eps) d r \wedge d\pp$ for some nowhere vanishing  function $q$ of class $C^{s-1}$. Applying the Darboux Theorem, we  make $q\equiv 1$ by a $C^{s-1}$ coordinate transformation so that
\begin{equation}\label{8222}
\Omega|_\A= d r \wedge d\pp.
\end{equation} 

We always assume that the curve $\gamma$ is at least $C^{s-1}$, so one can choose the $C^{s-1}$ coordinates such that
\begin{equation}\label{eq:strait1}
\gamma=\{r=0,x=0,y=0\},
\end{equation}
while keeping the symplectic form $\Omega|_\A$ in the standard form
(\ref{8222}). 

After that, we straighten the foliations 
$\mathcal{F}^\ss$ and $\mathcal{F}^\uu$ by a $C^{s-1}$ change of coordinates, identical on $\A$. Formulas (\ref{8222}) and (\ref{eq:strait1}) hold, and we also obtain that the leaves of
$\mathcal{F}^\ss$ and $\mathcal{F}^\uu$ are given by 
\begin{equation}\label{eq:strait2a}
\ell^{\ss}=\{(r,\varphi)={\const},y=0\},\qquad \ell^{\uu}=\{(r,\varphi)={\const},x=0\},
\end{equation}
implying
\begin{equation}\label{eq:strait3}
W^\s_{\loc}(\gamma)=\{r=0, y=0\},\qquad W^\u_{\loc}(\gamma)=\{r=0, x=0\}.
\end{equation}

We call the coordinates for which (\ref{eq:strait})--(\ref{eq:strait3}) are satisfied the {\em Fenichel coordinates}.
Note that assumptions on the whiskered torus differ in Theorem~\ref{thm:main_blender_cubic_nonpara} and everywhere else. We distinguish them as the {\em low regularity case} and the {\em KAM case}, respectively. In the low regularity case we have $s=3$, so the Fenichel coordinates are at least $C^2$. 

We do not introduce a parameter dependence in Theorem~\ref{thm:main_blender_cubic_nonpara}, so we do not analyze it in the low regularity case. However, we do study the parameter dependence in the KAM case. Here, we assume that the regularity class of $f_\eps$ is high enough and that $\gamma$ is a sufficiently smooth non-degenerate KAM-curve. Thus, it persists under small perturbations 
as a non-degenerate, $C^{s-1}$-smooth KAM-curve for all small $\eps$,
implying that the Fenichel coordinates can be chosen such that
$\gamma_\eps$ is given by (\ref{eq:strait1}) (so $W^{\s/\u}_{\loc}(\gamma_\eps)$ are given by (\ref{eq:strait3})) for all small $\eps$.\\

Using the terminology of \citep{DeldlLSea:06,GelTur:17}, the map 
$$F:=T_0|_\A:(r,\pp)\mapsto (F_1(r,\pp), F_2(r,\pp))$$
 is called the {\em inner map}.
In the Fenichel coordinates the local map  $T_0:(r,\pp,x,y)\mapsto (\bar r,  \bar \pp,\bar x,\bar y)$ assumes the following form (there is no dependence on $\eps$ in the low regularity case):
\begin{equation}\label{eq:T0}
\begin{aligned}
\bar r &= F_1(r,\pp,\eps) + \hat g_1(r,\varphi,x,y,\eps),\qquad
\bar \varphi =F_2(r,\pp,\eps) + \hat g_2(r,\varphi,x,y,\eps),\\
\bar x &=  \hat g_3(r,\varphi,x,y,\eps),\qquad
\bar y =  \hat g_4(r,\varphi,x,y,\eps),	
\end{aligned}
\end{equation}
where  the functions $\hat g_i$ are at least $C^2$, and they satisfy 
\begin{equation}\label{eq:T0_nonlinearities}
\begin{aligned}
&\hat g_{1,2}(r,\varphi,0,y,\eps)\equiv 0, \quad \hat g_{1,2}(r,\varphi,x,0,\eps) \equiv 0,\\
&\hat g_3(r,\varphi,0,y,\eps)\equiv 0,\quad \hat g_4(r,\varphi,x,0,\eps)\equiv 0.
\end{aligned}
\end{equation}
These identities correspond to straightened invariant manifolds $W_{loc}^{\s/\u}(\A)$ and foliations $\mathcal{F}^{\ss}_{\loc}$ and  $\mathcal{F}^{\uu}_{\loc}$.

It will be convenient for us to consider the so-called {\em cross-form} of this map. Since the restriction of $T_0$ to the locally invariant manifold $W^\u_\loc(\A):\{x=0\}$ is a diffeomorphism, the matrix ${\partial \hat g_4}/{\partial y} $ is invertible. Hence, the variable $y$ can be found as a function of $(r,\pp,\bar y,x,\eps)$ from the fourth equation of \eqref{eq:T0}. Substituting the resulting expression for $y$ into the remaining equations, yields that $T_0(r,\varphi,x,y)=(\bar r,\bar\varphi,\bar x,\bar y)$ if and only if the points satisfy
\begin{equation}\label{eq:T0_xform}
\begin{aligned}
\bar r &=F_1(r,\varphi,\eps)+g_1(r,\varphi,x,\bar y,\eps),\qquad
\bar \varphi = F_2(r,\varphi,\eps)+g_2(r,\varphi,x,\bar y,\eps),\\
\bar x &=  g_3(r,\varphi,x,\bar y,\eps),\qquad
 y =  g_4(r,\varphi,x,\bar y,\eps),	
\end{aligned}
\end{equation}
where $g$ are smooth functions satisfying
\begin{equation}\label{eq:ggg}
\begin{aligned}
&g_{1,2}(r,\varphi,0,\bar y,\eps)\equiv 0, \quad g_{1,2}(r,\varphi,x,0,\eps) \equiv 0, \\
&g_3(r,\varphi,0,\bar y,\eps)\equiv 0,\quad g_4(r,\varphi,x,0,\eps)\equiv 0.
\end{aligned}
\end{equation}

Note that by \eqref{eq:intro:1} we have
\begin{equation}\label{eq:T0_xform_deri}
\left\|\dfrac{\partial(F_1,F_2)}{\partial (r,\varphi)}\right\|<\cc, \quad
\left\|\left(\dfrac{\partial(F_1,F_2)}{\partial (r,\varphi)}\right)^{-1}\right\|<\cc, \quad
\left\| \dfrac{\partial g_3}{\partial x} \right\| < \lambda, \quad 
\left\| \dfrac{\partial g_4}{\partial \bar y} \right\| < \lambda,
\end{equation}
where $\hat\lambda$ can be taken arbitrarily close to 1 (see remarks after Definition \ref{defi:tori}).

We use these formulas to obtain the following
\begin{lem}\label{lem:derisyst}
There exists $C>0$ such that, if $V$ is sufficiently small, then 
for every small $(x_0,y_k)$ and  
for every point $(r_0,\pp_0)\in \A$ such that 
its orbit $\{F^j(r_0,\pp_0)\}_{j=0}^k$ by the inner map stays in the interior of $\A$ there exist uniquely defined $r_k,\varphi_k,x_k,y_0$ such that 
$(r_k,\varphi_k,x_k,y_k)= T_0^k (r_0,\varphi_0,x_0,y_0)$. The coordinates $r_k,\varphi_k,x_k,y_0$ are smooth functions of $r_0,\varphi_0,x_0,y_k$
and of $\eps$ (if there is a parameter dependence), and satisfy
 \begin{equation}\label{eq:derisyst:1}
\left\|\dfrac{\partial^{|i|} (r_k, \varphi_k)}
{\partial (r_0,\varphi_0,x_0,y_k,\eps)^i}
-\dfrac{\partial^{|i|} F^k(r_0,\varphi_0,\eps)}
{\partial (r_0,\varphi_0,x_0,y_k,\eps)^i }\right\|\leqslant C\lambda^{\frac{k}{2}}, 
\end{equation}
\begin{equation}\label{eq:derisyst:2}
\left\| \dfrac{\partial^{|i|} (x_k,y_0)}{\partial(r_0,\varphi_0,x_0,y_k,\eps)^{i}} \right\|\leqslant C\lambda^k,
\end{equation}
for multi-indices $i$ satisfying $0\leq |i| \leq s-2$. We also have, for all $k\geq 0$, that
\begin{align*}
&x_k=0 \;\;\mbox{when}\;\; x_0=0, \qquad
 y_0 =0 \;\;\mbox{when}\;\; y_k=0,\\
 &(r_k,\varphi_k) = F^k(r_0,\varphi_0,\eps)\;\;\mbox{when}\;\; x_0=0 \;\; \mbox{or} \;\; y_k=0.
\end{align*}
\end{lem}

The proof is given in Section~\ref{sec:T0k}. Essentially,
the result shows that, for any map $T_0$ satisfying \eqref{eq:T0_xform}, \eqref{eq:ggg} and \eqref{eq:T0_xform_deri},
the orbits of  $T_0$ are well approximated by the orbits of the inner map.
Since partial derivatives of $F^k$ with respect to $(x_0,y_k)$ are zero, the lemma in particular implies 
$$
\left\| \dfrac{\partial^{|i|+|j|} (r_k,\varphi_k)}{\partial(r_0,\varphi_0,\eps)^i(x_0,y_k)^{j}} \right\|\leqslant C\lambda^{\frac{k}{2}}
\quad \mbox{if} \quad |j|\geq 1.
$$

\subsection{Estimates for iterations of the inner map: the low regularity case}\label{sec:est}
By Lemma \ref{lem:derisyst}, we obtain estimates on the iterations of the local map from estimates on the iterations of the inner map $F$. Since the invariant curve $\gamma$ of $F$ is straightened (i.e. $\gamma=\{r=0\}$) and the standard symplectic form is preserved by $F$, this map is written as
\begin{equation}\label{o-1}
\bar r  = F_1(r,\pp), \qquad \bar\pp = F_2(r,\pp),
\end{equation}
where
\begin{equation}\label{o0}
F_1(0,\pp)=0, \qquad 
\dfrac{\p F_1(0,\pp)}{\partial r}
\cdot \dfrac{\p F_2(0,\pp)}{\partial \pp}  =1.
\end{equation}

In the low regularity case (the case of Theorem~\ref{thm:main_blender_cubic_nonpara}), the Fenichel coordinates are $C^2$, so the functions $F_{1,2}$ in (\ref{o-1}) and (\ref{o0}) are $C^2$, and we can write the inner map $F$ as 
\begin{equation}\label{o1}
\bar r = G^\prime(\pp)^{-1} r + p(r,\pp), \qquad 
\bar\pp = G(\pp) + q(r,\pp),
\end{equation}
where $G(\pp): = F_2(0,\pp)$, and functions $p$ and $q$ are $C^2$ and satisfy
\begin{equation}\label{o3} 
p=O(r^2), \quad \dfrac{\p p}{\partial (r,\pp)}   = O(r),\quad 
q=O(r), \quad\dfrac{\p q}{\partial r}  = O(1), \quad \dfrac{\p q}{\partial \pp}  = O(r).
\end{equation}

In Theorem~\ref{thm:main_blender_cubic_nonpara} we assume that the restriction $F|_\gamma$ is $C^1$-conjugate to a rigid rotation. This means that there exists a $C^1$ diffeomorphism $\psi$ of $\mathbb{S}^1$ such that 
\begin{equation}\label{c0d}
\psi(G(\pp)) = \psi(\pp) + \rho
\end{equation}
for a constant $\rho$ and all $\pp\in\mathbb{S}^1$. Note that
\begin{equation}\label{c00dd}
\psi^\prime(G(\pp))G^\prime(\pp) = \psi^\prime(\pp).
\end{equation}
Denoting the $k$-th iteration of the map $\pp\mapsto G(\pp)$ as $G_k(\pp)$, we have
\begin{equation}\label{c1d}
\psi(G_k(\pp)) = \psi(\pp) + k\rho,\qquad\psi^\prime(G_k(\pp)) G_k^\prime(\pp) = \psi^\prime(\pp).
\end{equation}
In particular, $G_k^\prime(\pp)$ is uniformly bounded away from zero and infinity for all $k$, which gives that
\begin{equation}\label{c2d}
G_k^{\prime\prime}(\pp) = O(k),
\end{equation}
because
$$G_k^{\prime\prime}(\pp) = \frac{d}{d\pp} \prod_{j=0}^{k-1} G^\prime(G_j(\pp)) = G_k^\prime(\pp) \sum_{j=0}^{k-1} \frac{G^{\prime\prime}(G_j(\pp))}{G^\prime(G_j(\pp))} G^\prime_j(\pp).$$

\begin{lem}\label{lemc1} Let $(r_k,\pp_k)=F^k(r_0,\varphi_0)$. Then
\begin{equation}\label{lrmap}
r_k = G^\prime_k(\pp_0)^{-1} r_0 + p_k(r_0,\pp_0), \qquad \pp_k = G_k(\pp_0) +q_k(r_0,\pp_0),
\end{equation}
where
\begin{equation}\label{estlr}
\begin{aligned}
p_k &= O(k^2r_0^2), \qquad &\dfrac{\p p_k}{\partial {r_0}} &= O(k^2r_0),&
\qquad \dfrac{\p p_k}{\partial {\pp_0} } &= O(kr_0),\\
q_k &= O(kr_0), \qquad &\dfrac{\p q_k}{\partial {r_0}}
  &= O(k),&
\qquad \dfrac{\p q_k}{\partial {\pp_0} }  &= O(k^2r_0),
\end{aligned}
\end{equation}
uniformly for all $k = o(r_0^{-1/2})$.
\end{lem}
\begin{proof}
We have by (\ref{c0d})
$$\psi(\bar\pp) = \psi(G(\pp) + O(r))=\psi(\pp) + \rho + O(r).$$
Let us assume that $r_k$ remains of order $r_0$ for all $k = o(r_0^{-1/2})$, which will be verified in the end of the proof.
Then iterating $k$ times gives
$$\psi(\pp_k) =\psi(\pp_0) + k\rho + O(kr_0).$$
Thus, by (\ref{c1d}),
\begin{equation}\label{phik}
\pp_k =\psi^{-1}(\psi(\pp_0) + k\rho) + O(kr_0) = G_k(\pp_0)+ O(kr_0).
\end{equation}

Our next goal is to estimate the first derivative 
$\displaystyle \partial (r_k,\phi_k)/\p {(r_0,\phi_0)}$, from which we will also get the expression for $r_k$ in \eqref{lrmap}.
Differentiating (\ref{o1}), and using (\ref{o3}) and (\ref{phik}), we obtain
\begin{align*}
d r_{k+1}& = (G^\prime(G_k(\pp_0))^{-1} + O(kr_0)) dr_k + O(r_0) d\pp_k,\\
d \pp_{k+1}& = O(1) dr_k + (G^\prime(G_k(\pp_0)) + O(kr_0)) d\pp_k.
\end{align*}

Denote
$$a_k = \psi^{\prime}(G_k(\pp_0))^{-1} dr_k, \qquad  
b_k = \psi^{\prime}(G_k(\pp_0)) d\pp_k.$$
By (\ref{c00dd}), 
$\psi^{\prime}(G_{k+1}(\pp_0))=G^\prime(G_k(\pp_0))^{-1}
\psi^{\prime}(G_k(\pp_0))$,
so we have
\begin{align*}
a_{k+1} = (1 + O(kr_0)) a_k + O(r_0) b_k,\qquad
b_{k+1} = O(1) a_k + (1 + O(kr_0)) b_k,
\end{align*}
for all $k$ such that $\p r_k/\p {r_0}  = O(1)$. Then, the induction on $k$ gives, for $k=o(r_0^{-1/2})$,
\begin{align*}
a_k = (1 + O(k^2r_0)) a_0 + O(kr_0) b_0,\qquad
b_k= O(k) a_0 + (1 + O(k^2r_0)) b_0,
\end{align*}
i.e.,
\begin{align}
\dfrac{\p r_k}{\p r_0}
 &= \psi^\prime(G_k(\pp_0))\psi^\prime(\pp_0)^{-1} +
O(k^2r_0)=O(1), \qquad 
\dfrac{\p r_k}{\p \pp_0} = O(kr_0),\label{0929r_k}\\
\dfrac{\p \pp_k}{\p r_0}&= O(k),\qquad
\dfrac{\p \pp_k}{\p \pp_0}  = \psi^\prime(G_k(\pp_0))^{-1}\psi^\prime(\pp_0) +
O(k^2r_0).
\label{0929pp_k}
\end{align}

Since $r_k=0$ for all $k$ when $r_0=0$, integrating \eqref{0929r_k} with respect to $r_0$ gives
\begin{equation}\label{0929r_k2}
r_k =  \psi^\prime(G_k(\pp_0))\psi^\prime(\pp_0)^{-1} r_0 + O(k^2r^2_0),
\end{equation}
which, along with (\ref{c1d}), gives the expression for $r_k$ in \eqref{lrmap} and the estimate for $p_k$. The derivatives $\p p_k/\p r_0$ and $\p p_k/\p \pp_0$ are obtained  from \eqref{0929r_k}, using (\ref{c1d}) and (\ref{c2d}).  The estimates for $q_k$ are immediate from \eqref{0929pp_k} and (\ref{c1d}).  This finishes the proof of the lemma under the assumption that $r_k=O(r_0)$. But this assumption always holds with a margin of safety, by \eqref{0929r_k2} for $k=o(r_0^{1/2})$.
\end{proof}

We are now in the position to introduce the coordinates that  will be used in the proof of Theorem~\ref{thm:main_blender_cubic_nonpara}.

\begin{lem}\label{lem:innerlowfinal}
Take any two points $P^+$ and $P^-$ in $\mathbb{S}^1$.  There exist  $C^2$-smooth symplectic coordinates near $\gamma$ such that, if the following  are satisfied:
\begin{enumerate}[nosep]
\item $k=o(r_0^{-1/2})$,
\item $\pp_0-\pp^+$ is sufficiently close to an integer, and
\item $\pp^+ + k\rho-\pp^-$ is sufficiently close to an integer,
\end{enumerate}
where $\pp^\pm$ are the coordinates of $P^\pm$ in $\mathbb{S}^1$, then formula \eqref{lrmap} for $F^k$ takes the form 
\begin{equation}\label{innerlowfinal}
\begin{aligned}
r_k &= r_0 +\hat p_k(r_0,\pp_0),\qquad
\pp_k = \pp_0 +k\rho + \hat q_k(r_0,\pp_0),
\end{aligned}
\end{equation}
where 
\begin{equation}\label{innerlowfinalest}
\begin{aligned}
&\hat p_k=O(\pp_0 -\pp^+)r_0 + O(\pp_0 +k\rho - \pp^-)r_0+O(k^2r_0^2),\\
&\dfrac{\p \hat p_k}{\partial {r_0}}=O(\pp_0 -\pp^+) + O(\pp_0 +k\rho - \pp^-)+O(k^2r_0),\qquad
\dfrac{\p p_k}{\partial {\pp_0} }=O(k r_0),\\
&\hat q_k=O((\pp_0 -\pp^+)^2 )+ O((\pp_0 +k\rho - \pp^-)^2)+O(k r_0),\\
&\dfrac{\p \hat q_k}{\partial {r_0}}=O(k),\qquad
\dfrac{\p \hat q_k}{\partial {\pp_0} }=O(\pp_0 -\pp^+) +O(\pp_0 +k\rho - \pp^-)+O(k^2 r_0).
\end{aligned}
\end{equation}
\end{lem}

\begin{proof}
Take any $C^3$ diffeomorphism $\hat\psi$ of $\mathbb{S}^1$ such that 
$\hat\psi(\pp^\pm)=\psi(\pp^\pm)$ and $\hat\psi^\prime(\pp^\pm)=\psi^\prime(\pp^\pm)$. By (\ref{c1d}), one has
\begin{equation}\label{c1dhat}
\begin{aligned}
&\hat\psi(G_k(\pp)) = \hat\psi(\pp) + k\rho + O((\pp-\pp^+)^2)+O((G_k(\pp)-\pp^-)^2),\\
&\hat\psi^\prime(G_k(\pp)) G_k^\prime(\pp) = \hat\psi^\prime(\pp)+O(\pp-\pp^+)+ O(G_k(\pp)-\pp^-),
\end{aligned}
\end{equation}
for $\pp$ close to $\pp^+$ and $k$ such  that $G_k(\pp)$ is close to $\pp^-$. Make the $C^2$-smooth symplectic coordinate transformation
$$
r^{\mathrm{new}}=\hat\psi^\prime(\pp)^{-1} r, \qquad
\pp^{\mathrm{new}}=\hat\psi(\pp).$$
Denote $G^{\mathrm{new}}_k:=\hat\psi\circ G_k\circ \hat\psi^{-1}$ and  $\pp^{\pm,\mathrm{new}}:=\hat\psi(\pp^\pm)$.
By (\ref{c1dhat}), it holds in the new coordinates that
\begin{equation*}\label{c2dform}
\begin{aligned}
& G^{\mathrm{new}}_k(\pp_0) 
=\pp^{\mathrm{new}}_0 + k\rho + O((\pp^{\mathrm{new}}_0-\pp^{+,\mathrm{new}})^2)
+O((\pp^{\mathrm{new}}_0+k\rho -\pp^{-,\mathrm{new}})^2), \\
& \dfrac{d}{d\pp^{\mathrm{new}}}G_k^{\mathrm{new}}(\pp^{\mathrm{new}}_0)= 1 + O(\pp^{\mathrm{new}}_0-\pp^{+,\mathrm{new}})
+O(\pp^{\mathrm{new}}_0+k\rho -\pp^{-,\mathrm{new}}),
\end{aligned}
\end{equation*}
where to obtain the $O(\cdot)$ terms we used the fact that $\hat\psi^{-1}(\pp^{\mathrm{new}})-\pp^\pm=O(\pp^{\mathrm{new}}-\hat\psi(\pp^\pm))$.

With the above expressions for $G^{\mathrm{new}}_k$ and its derivative, one  easily rewrites (\ref{lrmap}) in the new coordinates as \eqref{innerlowfinal} (after dropping the superscript). Since the coordinate transformation has bounded derivatives, the estimates for $\hat p_k$ and $\hat q_k$ follow immediately from \eqref{estlr}.
\end{proof}

\begin{rem}\label{rem:C1conju}
Let $\hat \psi$ be the diffeomorphism  in the proof of Lemma~\ref{lem:innerlowfinal}. We see from the proof that the restriction $F|_\gamma$ in the coordinates of Lemma~\ref{lem:innerlowfinal} is $C^1$-conjugate to a rigid rotation via the conjugacy $\tilde \psi:=\psi\circ \hat\psi^{-1}$ satisfying $\tilde \psi'(\pp^\pm)=1$.
\end{rem}

\subsection{Estimates for iterations of the inner map: the KAM case}\label{sec:estkam}

As opposed to the previous case, in the setting of Theorem~\ref{thm:main_2punfolding} the map $f$ and the curve $\gamma$
have high regularity, so the cylinder $\A$, as well as the manifolds $W^\s(\A),W^\u(\A)$ and the foliations $\mathcal{F}^\ss,\mathcal{F}^\uu$, have sufficiently high smoothness, also with respect to parameters $\eps$.
Therefore, the inner map $F$ can be brought to the form (\ref{o1}) where the functions $G$, $p$ and $q$ now depend on $\eps$ and are sufficiently smooth.

Since the rotation  number $\rho=\rho(\gamma)$ is assumed to be Diophantine in the KAM case, the conjugacy map $\psi$ that brings $F|_\gamma$ to a rigid rotation also has high regularity, both in $\pp$ and $\eps$. So, we can do the symplectic transformation 
$(r,\pp) \mapsto (\psi^\prime(\pp)^{-1} r,\; \psi(\pp))$ and bring the inner map (\ref{o1}) to the form
$$\bar r = r + g_1(r,\pp,\eps), \qquad 
\bar\pp = \pp + \rho + g_2(r,\pp,\eps),
$$
where $g_{1,2}$ are sufficiently smooth functions such that
$$g_1(0,\pp,\eps)=0,\qquad \dfrac{\p g_1(0,\pp,\eps)}{\p r}=0, \qquad g_2(0,\pp,\eps)=0.$$

By averaging (see \citep{Bir:36} and also \citep[Proposition 2.6]{BerTur:25}), one establishes for any integer $m>1$ that if the map is sufficiently smooth, then there are sufficiently smooth symplectic coordinates such that the inner map $F$  assumes the form
$$\bar r = r + O(r^{m}),\qquad
\bar \varphi =\varphi + \rho + \hat\rho_1(\eps) r +\ldots+ 
\hat\rho_{m-1}(\eps) r^{m-1} + O(r^m).$$
Since  $\hat\rho_1(\eps)\neq 0$, by the non-degeneracy condition in   Definition~\ref{defi:kam}, we can normalize the coefficient of $r$ by taking $r^{\mathrm{new}}=\hat\rho_1(\eps)r$. After the normalization, we have that
in some symplectic $C^{m'}$ coordinates (with $m'\geq m$) the following formula holds for the inner map:
\begin{equation}\label{eq:kam_new}
\bar r=F_1(r,\pp,\eps) = r + \xi(r,\pp,\eps),\qquad
\bar \varphi =F_2(r,\pp,\eps)=\varphi +\rho+ r+ \hat\rho(r,\eps)+ \eta(r,\pp,\eps),
\end{equation}
where $\hat\rho(r,\eps)=\sum_{i=2}^{m-1}\hat\rho_i(\eps) r^i$ (with coefficients $\rho_i(\eps)$ possibly different  from the above ones), and  $\xi,\eta$ are periodic in $\pp$ with period $1$ and satisfy
\begin{equation}\label{eq:kam_new_deri}
\dfrac{\p^{i+|j|} (\xi,\eta)}{\p r^i \p (\pp,\eps)^j}=  O(r^{m-i}), 
\end{equation}
for $i\leq m$ and $0\leq i+|j| \leq  m'$.

In Section~\ref{sec:lem:F^k}, we use the above formulas to prove the following
\begin{lem}\label{lem:F^k}
If $F$ satisfies (\ref{eq:kam_new}) and (\ref{eq:kam_new_deri}) with $m'\geq m\geq 2$, then, for all sufficiently small $r_0$ and $k=o(r_0^{1-m})$, the iteration $F^k:(r_0,\pp_0)\mapsto(r_k,\pp_k)$  is given by
\begin{equation}\label{F^k}
r_k = r_0 +  \xi_k(r_0,\pp_0,\eps),\qquad \pp_k = \pp_0 + k\rho + kr_0 +k\hat\rho(r_0,\eps)+ k \eta_k(r_0,\pp_0,\eps),
\end{equation}
where $\rho$ is the polynomial in \eqref{eq:kam_new}, and the functions $\xi_k$ and $\eta_k$ are periodic in $\pp$ with period 1 and satisfy
\begin{equation}\label{eq:gg}
\dfrac{\p^{i+|j|} (\xi_k,\eta_k)}{\p r^i \p (\pp,\eps)^j}=  O(kr^{m-i}_0)
\end{equation}
for $0\leq i+|j|\leq \min\{m,m'-m\}$.
\end{lem}

\subsubsection{Straightening the KAM-fibration}\label{sec:kam_strait}
In the KAM case, the rotation number $\rho$ of $\gamma$ is $(c,\tau)$-Diophantine for some $c,\tau$, and the inner map $F$ (see \eqref{eq:kam_new}) is a perturbation of an integrable twist map.  Then, for every $(c,\tau)$-Diophantine number $\rho'$ close to $\rho$, there exists a KAM-curve  in $\A$ with rotation number $\rho'$ \citep{Mos:62}.  Moreover, by Lazutkin \citep{Laz:73}, for any given integer $m'>0$, if the order of smoothness of $F$ is large enough (depending on $\tau$ and $m'$), then there exist symplectic $C^{m'}$ coordinates in $\A$ such that all $(c,\tau)$-Diophantine KAM curves are straightened. Thus, we may assume that in \eqref{eq:kam_new} the functions $\eta$ and $\xi$ satisfy
\begin{equation}\label{8241}
\xi(r,\pp,\eps)=\eta(r,\pp,\eps) = 0
\quad\mbox{if}\;
\rho+r+\hat\rho(r,\eps)
\;\mbox{is}\;(c,\tau)\mbox{-Diophantine.}
\end{equation}
From now on, in the KAM case, we only consider the $C^{m'}$ Fenichel coordinates where \eqref{8241} is satisfied. 

It follows that the $C^{m'}$ map $F$ has an invariant fibration in $\A$ by non-degenerate $(c,\tau)$-Diophantine KAM-circles $\gamma': \{r={\const}\}$ such that (see e.g. \citep{Mos:62,Laz:73,Pos:82}) {\em every point $P$ of any of these KAM-curves is a Lebesgue density point of the union of these curves.}
The invariant fibration persists under $C^s$-small exact symplectic perturbations if $s$ is large enough.

\begin{rem}\label{rem:smoothness}
In the proof of Theorem~\ref{thm:main_2punfolding}, we need to deal with cubic tangencies and hence need estimates for derivatives up to order three. Thus, by Lemma~\ref{lem:F^k}, we need $\min\{m,m'-m\}\geq 3$ (see also Lemma~\ref{lem:derisyst2}).  One can see from the proof that this is enough.
The resulting required smoothness $s$ for $f$ can, in principle, be computated    via formulas in  \citep{Laz:73,Pos:82}.
\end{rem}

\subsection{Transition map for homoclinic orbits}\label{sec:transition}

Recall that, for a homoclinic orbit $\Gamma$ to a whiskered torus 
$\gamma$, the transition map $T_1$ takes a small neighborhood $\Pi^-\subset V$ of $M^-\in W^\u_{\loc}(\gamma)\cap \Gamma$ to  a small neighborhood $\Pi^+\subset V$ of $M^+\in W^\s_{\loc}(\gamma)\cap \Gamma$.  So, $T_1$ is defined as $f^n|_{\Pi^-}$ for some $n$ such that $f^n(M^-)=M^+$.

\subsubsection{A partially-hyperbolic homoclinic orbit}\label{sec:phorbit}
In the straightened coordinates of Section~\ref{sec:innermap}, we can write $M^-=(0,\varphi^-,0,y^-)$ and $M^+=(0,\varphi^+,x^+,0)$ for some constants $\pp^\pm,x^+,y^-$.  The transition map  $T_1:(r,\varphi,x,y)\mapsto (\tilde r,\tilde \varphi,\tilde x,\tilde y)$ can be written as
\begin{equation}\label{eq:T1}
\begin{aligned}
\tilde r &= \hat{a}_{11} r + \hat a_{12} (\varphi-\varphi^-) + \hat a_{13} x + \hat a_{14}(y-y^-)+\dots, \\
\tilde \varphi - \varphi^+ &= \hat{a}_{21} r + \hat a_{22} (\varphi-\varphi^-) + \hat a_{23} x + \hat a_{24}(y-y^-)+\dots, \\
\tilde x-x^+ &= \hat{a}_{31} r + \hat a_{32} (\varphi-\varphi^-) + \hat a_{33} x + \hat a_{34}(y-y^-)+\dots, \\
\tilde y &= \hat{a}_{41} r + \hat a_{42} (\varphi-\varphi^-) + \hat a_{43} x + \hat a_{44}(y-y^-)+\dots, 
\end{aligned}
\end{equation}
where the dots denote  the Taylor remainder of order two.  Observe that the transversality in Definition~\ref{defi:phorbit}  holds at every point of the homoclinic orbit, and, in particular, we have the transverse intersections of $\ell^\uu$ with $T_1^{-1}(W^\s_\loc(\A))$ at $M^-$, and of $\ell^\ss$ with $T_1(W^\u_\loc(\A))$ at $M^+$. This implies that  
\begin{equation}\label{8281}
\hat a_{44}\neq 0.
\end{equation}
To see this, take any non-zero vector $(\Delta \tilde r, \Delta \tilde \varphi, \Delta \tilde x,0)\in {T}_{M^+} W^\s(\A_\loc)$ and denote its preimage by the derivative $\D T_1$ as $(\Delta  r, \Delta  \varphi,  \Delta x,\Delta y)$.  The intersection of $\ell^\uu$ with $T_1^{-1}(W^\s_\loc(\A))$ is not transverse if and only if the preimage lies in $T_{M^-} \ell^{\uu}$, i.e., $(\Delta  r, \Delta  \varphi,  \Delta x)=0$. By   \eqref{eq:T1}, we have 
\begin{equation*}
\Delta \tilde r =  \hat a_{14}\Delta y, \quad
\Delta \tilde \varphi =  \hat a_{24}\Delta y,  \quad
\Delta \tilde x=  \hat a_{34}\Delta y,  \quad
0=  \hat a_{44}\Delta y. 
\end{equation*}
Hence,  the non-transversality is equivalent to the existence of non-zero solutions to this system, i.e., to  $\hat a_{44} = 0$.   This proves \eqref{8281}. 

We then rewrite \eqref{eq:T1} as
\begin{equation}\label{eq:T1_general}
\begin{aligned}
\tilde r &= \tilde{a}_{11} r + \tilde a_{12} (\varphi-\varphi^-) + \tilde a_{13} x + \tilde a_{14}\tilde y+\dots, \\
\tilde \varphi - \varphi^+ &= \tilde{a}_{21} r + \tilde a_{22}(\varphi-\varphi^-)+ \tilde a_{23} x + \tilde a_{24}\tilde y+\dots, \\
\tilde x-x^+ &= \tilde{a}_{31} r + \tilde a_{32} (\varphi-\varphi^-) + \tilde a_{33} x + \tilde a_{34}\tilde y+\dots, \\
 y-y^- &= \tilde{a}_{41} r + \tilde a_{42}(\varphi-\varphi^-) + \tilde a_{43} x + \tilde a_{44}\tilde y+\dots, 
\end{aligned}
\end{equation}
for some constants $\tilde a_{ij}$ with $\tilde a_{44}\neq 0$.

\subsubsection{Scattering map}\label{sec:scattering}
We now take the disc $\Sigma^+$ used to define the scattering map  ${S}$  in  \eqref{eq:scattering_1} such that it contains $M^+$, and take $\Sigma^-:=T_1^{-1}(\Sigma^+)$. The  scattering map  assumes the form
\begin{equation}\label{eq:scattering_3}
S=\pi^\s\circ T_1|_{\Sigma^-}\circ ( \pi^\u)^{-1}:{\pi^\u(\Sigma^-)}\to \pi^\s(\Sigma^+),
\end{equation}
where $\pi^\s$ and $\pi^\u$ are the holonomy maps.  Note that the scattering map $\hat S$ defined by a different choice of $\Sigma^+$ and $\Sigma^-$ is a composition of $S$ with iterations of $T_0$, i.e.,
\begin{equation}\label{eq:changesigma}
\hat S= T_0^{n_1} \circ S\circ T_0^{n_2},
\end{equation}
for some integers $n_1$ and $n_2$.
 By \citep[Proposition 6]{GelTur:17}, the scattering map preserves the symplectic form on $\A$, i.e., $\det \D S(\pi^{\u}(M^-)) =1$.

\begin{lem}\label{lem:scattering_deri}
$$\D S(\pi^{\u}(M^-))=\D T_1|_{\Sigma^-}(M^-)= 
\begin{pmatrix}
\tilde{a}_{11}  & \tilde a_{12}\\
\tilde{a}_{21}  & \tilde a_{22}
\end{pmatrix},$$ 
where $\tilde a_{ij}$ are the coefficients from (\ref{eq:T1_general}).
In particular, $ \det \D T_1|_{\Sigma^-}(M^-)=1$.
\end{lem}
\begin{proof}
Recall that in coordinates \eqref{eq:strait}, one has $W^\s_\loc(\A)=\{y=0\}$ and $W^\u_\loc(\A)=\{x=0\}$.
By condition \eqref{eq:condition:trans1}, $\Sigma^+\subset W^\s_\loc(\A)$ and $\Sigma^-\subset W^\u_\loc(\A)$  are  given by 
\begin{equation}\label{eq:sigmapm}
\{\tilde x = h^+(\tilde r,\tilde \pp), \tilde y=0 \}
 \quad \mbox{and} \quad 
 \{x=0, y=h^-(r,\pp)\},
\end{equation} 
for some smooth functions $h^\pm$, respectively.  Substituting $x=0$ and $\tilde y=0$ into the first two equations of \eqref{eq:T1_general}, yields the formula for $T_1|_{\Sigma^-}:\Sigma^-\to \Sigma^+$ as 
\begin{equation*}
\begin{aligned}
\tilde r &= \tilde{a}_{11} r + \tilde a_{12} (\varphi-\varphi^-) + \tilde a_{13} h^-_1(r,\pp) + \tilde a_{14}h_2^+(\tilde r,\tilde \pp)+\dots, \\
\tilde \varphi - \varphi^+ &= \tilde{a}_{21} r + \tilde a_{22}(\varphi-\varphi^-)+ \tilde a_{23} h^-_1(r,\pp) + \tilde a_{24}h_2^+(\tilde r,\tilde \pp)+\dots,
\end{aligned}
\end{equation*}
where the dots denote terms of at least second order that are functions of $r$ and $\pp$.  The statement of the lemma then follows from  \eqref{eq:scattering_3}.
\end{proof}

\subsubsection{An orbit of homoclinic tangency}\label{sec:transtangency}
Let us now consider the case where $\Gamma$ is an  orbit of tangency, namely, the tangent space $T_{M^+}W^\s_{\loc}(\gamma)$ has a non-zero intersection with $\D T_1({T}_{M^-}W^\u_{\loc}(\gamma))$. This means that $(0,\Delta \tilde\pp,\Delta \tilde x,0)=\D T_1(0,\Delta \pp,0,\Delta  y)$ for some non-zero vector $(0,\Delta \tilde\pp,\Delta \tilde x,0)$, which immediately gives  $\tilde a_{12}=0$.
It then follows from Lemma~\ref{lem:scattering_deri} that $\tilde a_{11}$ and $\tilde a_{22}$ are non-zero and $\tilde a_{11}^{-1}=\tilde a_{22}$. Moreover, if the symplectic coordinates are chosen such that $T_0|_\gamma$ is the rigid rotation, then 
\begin{equation}\label{eq:a11}
\tilde a_{11}^{-1}=\tilde a_{22}=\thebeta,
\end{equation}
where $\thebeta$ is given by \eqref{beta}. Since $ d(T_0|_\gamma)/d\pp=1$ in this case, relation \eqref{eq:changesigma} implies that the value of $\alpha$ does not depend on the choice of the discs $\Sigma^+$ and $\Sigma^-$. 

Note that relation (\ref{eq:a11}) automatically holds for a more general choice of symplectic coordinates, where $T_0|_\gamma$ is conjugate to a rotation via  a diffeomorphism whose derivative at  $\pp=\pp^\pm$ equals 1. Such are the coordinates we choose for the proof of Theorem~\ref{thm:main_blender_cubic_nonpara}, see Remark~\ref{rem:C1conju}. 

Since $\alpha\neq 0 $, we can resolve \eqref{eq:T1_general} with respect to $\pp$, which together with $\tilde a_{12}=0$ yields
\begin{equation}\label{eq:T1_xform0}
\begin{aligned}
\tilde r &= \thebetai r +     a_{13} x +  a_{14}\tilde y+\dots, \\
\varphi-\varphi^-&= {a}_{21} r + \thebetai (\tilde \varphi - \varphi^+) +  a_{23} x +  a_{24}\tilde y+\dots, \\
\tilde x-x^+ &= a_{31} \tilde r +  a_{32} (\tilde \varphi - \varphi^+) +  a_{33} x +  a_{34}\tilde y+\dots, \\
y-y^- &= a_{41} r +  a_{42}(\tilde \varphi - \varphi^+)+  a_{43} x +  a_{44}\tilde y+\dots,
\end{aligned}
\end{equation}
where $a_{ij}$ are some coefficients, and the dots represent the Taylor remainder of  order two. 
 The formula  implies that the common tangent space of $W^\s_{\loc}(\gamma)$ with $T_1(W^\u_{\loc}(\gamma))$ is one-dimensional.

This is the general formula for the transition map along a partially-hyperbolic homoclinic tangency,  without information on the order of the tangency. Recall that two manifolds with a one-dimensional common tangent space, $W^\s_{\loc}(\gamma)$ and $T_1(W^\u_{\loc}(\gamma))$ in our case, have a tangency of order $\ell$ if there exist local coordinates $(u,v)\in \mathbb R^N\times \mathbb R^N$ near $M^+$ such that $W^\s_{\loc}(\gamma)$ remains straightened as $\{u_1=\dots=u_N=0\}$ and $T_1(W^\u_{\loc}(\gamma))$ is given by $\{u_1=h(v_1),v_2=\dots=v_N=0\}$ for some function $h$ vanishing along with its derivatives up to order $\ell$  at $0$, and $h^{(\ell+1)}(0)\neq 0$. 

Substituting the equation $\{r=0,x=0\}$ for $W^\u_{\loc}(\gamma)$ into \eqref{eq:T1_xform0} gives the equation of 
$T_1(W^\u_{\loc}(\gamma))$ as 
\begin{equation*}
\begin{aligned}
\tilde r &=     a_{14}\tilde y+\dots=:h_1(\tilde \varphi,\tilde y), \\
\tilde x-x^+ &=   a_{31}\tilde r + a_{32} (\tilde \varphi - \varphi^+)  +  a_{34}\tilde y+\dots=:h_2(\tilde \varphi,\tilde y).
\end{aligned}
\end{equation*}
Denote by $a^{(i)}$ the coefficients  of the terms $(\tilde \varphi - \varphi^+)^i$ in $h_1(\tilde \varphi,\tilde y)$.
One readily sees that the coordinate transformation   $\tilde r^{\mathrm{new}}=\tilde r-h_1(\tilde \varphi,\tilde y)+h_1(\tilde \varphi,0), \; \tilde x^{\mathrm{new}}=\tilde x-h_2(\tilde \varphi,\tilde y)$ leads to the standard coordinates for defining the order of the tangency.  We therefore have that {\em the partially-hyperbolic homoclinic tangency of $\gamma$ has order $\ell$} if $a_{12}^{(i)}=0$ for $i=1,\dots,\ell$ and $a_{12}^{(\ell+1)}\neq 0$. Thus, for such a tangency, $T_1(r,\varphi,x,y)=(\tilde r,\tilde \varphi,\tilde x,\tilde y)$ if and only if
\begin{equation}\label{eq:T1_xform_n}
\begin{aligned}
\tilde r &=\thebetai r + \dd (\tilde \varphi - \varphi^+)^{\ell+1} +  a_{13} x +  a_{14}\tilde y+\dots, \\
\varphi-\varphi^- &= a_{21} r +  \thebetai(\tilde \varphi - \varphi^+) +  a_{23} x +  a_{24}\tilde y+\dots, \\
\tilde x-x^+ &= a_{31}  r +  a_{32} (\tilde \varphi - \varphi^+) +  a_{33} x +  a_{34}\tilde y+\dots, \\
y-y^- &= a_{41} r +  a_{42} (\tilde \varphi - \varphi^+) +  a_{43} x +  a_{44}\tilde y+\dots,
\end{aligned}
\end{equation}
where $\dd :=a^{(\ell+1)}\neq 0$.
Here the dots represent the remaining terms of  the  Taylor polynomial of order $\ell+1$, plus the Taylor remainder of higher order. Note that there are no terms $(\tilde \varphi - \varphi^+)^i$ for $i\leqslant \ell+1$ in the first equation. 

Note that, to define a tangency of order $\ell$, we need the map $T_1$ to be at least $C^{\ell+1}$ (so that the coefficient $\beta$ is defined). 
In the case where the map  is $C^{\ell+1}$ but not necessarily $C^{\ell+2}$, we define {\em $(\ell+1)$-flat} tangencies by
requiring the vanishing of the first $\ell+1$ derivatives of $\tilde r$ with respect to $\tilde \pp$ in the  transition map.
In this case, the transition map takes the form \eqref{eq:T1_xform_n} with  $\dd =0$. 

Under the assumptions of Theorem \ref{thm:main_blender_cubic_nonpara}, only $C^2$  coordinates can be guaranteed, and hence it is not always possible to define cubic tangencies -- but we can still have 2-flat tangencies, whose transition map takes the form
\begin{equation}\label{eq:T1_qudra+0}
\begin{aligned}
\tilde r &=\thebetai r  +  a_{13} x +  a_{14}\tilde y+ p_1(\tilde \varphi )+q_1(r,\tilde \pp,x,\tilde y), \\
\varphi-\varphi^- &= a_{21} r +  \thebetai(\tilde \varphi - \varphi^+) +  a_{23} x +  a_{24}\tilde y+b_2(\tilde \varphi - \varphi^+)^2\\
&\qquad\qquad\qquad\qquad\qquad\qquad\qquad\qquad\qquad\qquad
 +p_2(\tilde \varphi )+q_2(r,\tilde \pp,x,\tilde y), \\
\tilde x-x^+ &= a_{31}  r+  a_{32} (\tilde \varphi - \varphi^+) +  a_{33} x +  a_{34}\tilde y+b_3(\tilde \varphi - \varphi^+)^2\\
&\qquad\qquad\qquad\qquad\qquad\qquad\qquad\qquad\qquad\qquad
+
p_3(\tilde \varphi )+q_3(r,\tilde \pp,x,\tilde y), \\
y-y^- &= a_{41} r +  a_{42} (\tilde \varphi - \varphi^+) +  a_{43} x +  a_{44}\tilde y + b_4(\tilde \varphi - \varphi^+)^2\\
&\qquad\qquad\qquad\qquad\qquad\qquad\qquad\qquad\qquad\qquad
+p_4(\tilde \varphi )+q_4(r,\tilde \pp,x,\tilde y).  
\end{aligned}
\end{equation}
Here $b_{2,3,4}$ are constants,  and the functions $p$ and $q$ satisfy
\begin{equation}\label{eq:aa0}\begin{array}{l}
p=o((\tilde \varphi-\pp^+ )^{2}),\qquad
\dfrac{\p p}{\p \tilde\pp} = o(\tilde \varphi-\pp^+ ),\\
q=O(r^2+x^2+\tilde y^2+|\tilde \varphi-\pp^+|(|r|+\|x\|+\|\tilde y\|)),\\
\dfrac{\p q}{\p \tilde\pp} = O(|r|+
\|x\| + \|\tilde y\|),\qquad
\dfrac{\p q}{\p (r, x, \tilde y)} = O(|r|+|\tilde \varphi-\pp^+|
+\|x\| + \|\tilde y\|).\end{array}
\end{equation}

For the convenience in later computations, we substitute
\begin{align*}
r&=\alpha(\tilde r -  a_{13} x -  a_{14}\tilde y- p_1(\tilde \varphi )-q_1(r,\tilde \pp,x,\tilde y)),\\
\tilde \varphi - \varphi^+ &=\alpha
(\varphi-\varphi^- -a_{21} r -  a_{23} x -  a_{24}\tilde y-b_2(\tilde \varphi - \varphi^+)^2\\
&\qquad\qquad\qquad\qquad\qquad\qquad\qquad\qquad\qquad
-
p_2(\tilde \varphi )-q_2(r,\tilde \pp,x,\tilde y)
)
\end{align*}
into the terms $a_{31}r$, $a_{42} (\tilde \varphi - \varphi^+)$, $b_4(\tilde \varphi - \varphi^+)^2$ in \eqref{eq:T1_qudra+0} and obtain
\begin{equation}\label{eq:T1_qudra+}
\begin{aligned}
\tilde r &=\thebetai r  +  a_{13} x +  a_{14}\tilde y+ p_1(\tilde \varphi )+q_1(r,\tilde \pp,x,\tilde y), \\
\varphi-\varphi^- &= a_{21} r +  \thebetai(\tilde \varphi - \varphi^+) +  a_{23} x +  a_{24}\tilde y+b_2(\tilde \varphi - \varphi^+)^2\\
&\qquad\qquad\qquad\qquad\qquad\qquad\qquad\qquad\qquad\qquad
+
p_2(\tilde \varphi )+q_2(r,\tilde \pp,x,\tilde y), \\
\tilde x-x^+ &= a_{31} \tilde r+  a_{32} (\tilde \varphi - \varphi^+) +  a_{33} x +  a_{34}\tilde y+b_3(\tilde \varphi - \varphi^+)^2\\
&\qquad\qquad\qquad\qquad\qquad\qquad\qquad\qquad\qquad\qquad
+
p_3(\tilde \varphi )+q_3(r,\tilde \pp,x,\tilde y), \\
y-y^- &= a_{41} r +  a_{42} (\varphi - \varphi^-) +  a_{43} x +  a_{44}\tilde y + b_4(\varphi - \varphi^-)^2\\
&\qquad\qquad\qquad\qquad\qquad\qquad\qquad\qquad\qquad\qquad
+p_4(\tilde \varphi )+q_4(r,\tilde \pp,x,\tilde y),
\end{aligned}
\end{equation}
whith modified coefficients $a_{3i}$, $a_{4i}$, $b_4$ and functions $p_{3,4}$ and $q_{3,4}$, which satisfy the same estimates \eqref{eq:aa0}.
~\\

In Theorem~\ref{thm:main_2punfolding}, we consider a family of systems depending on parameters $\eps$ and assume that at $\eps=0$ the system has a whiskered KAM-torus $\gamma$ with a hyperbolic homoclinic tangency. In this case the transition maps are sufficiently smooth (at least $C^3$) so that we can define quadratic and cubic tangencies. 

Suppose first that the tangency is quadratic. We take the Fenichel coordinates \eqref{eq:strait} which depend smoothly on parameters. The local map $T_0$ is given in the form  \eqref{eq:T0_xform}  for all small $\eps$, with coefficients  depending on $\eps$. Similarly, $T_1$ keeps the form \eqref{eq:T1_xform_n}  except that there are extra terms in the $\tilde r$-equation:
$$\tilde r = \thebetai r + \hat\mu_1  + \hat \mu_2(\tilde \varphi - \varphi^+)+ \dd (\tilde \varphi - \varphi^+)^2 +  a_{13} x +  a_{14}\tilde y+\dots,$$
where $\hat \mu_1$ and $\hat \mu_2$ vanish at $\eps=0$.  Here and below we  omit the explicit dependence of coefficients on $\eps$.

Since $\dd\neq 0$, the term $\hat \mu_2 (\tilde \varphi - \varphi^+ )$ can be killed for all small $\eps$ by adding a small correction to $\pp^+$ depending on $\eps$. Thus, the formula \eqref{eq:T1_xform_n} for $T_1$ can be rewritten as
\begin{equation}\label{eq:T1_quadra}
\begin{aligned}
\tilde r &=\mu+\thebetai  r +  \dd (\tilde \varphi - \varphi^+)^2 +  a_{13} x +  a_{14}\tilde y+\dots, \\
\varphi-\varphi^- &= a_{21} r +  \thebetai (\tilde \varphi - \varphi^+) +  a_{23} x +  a_{24}\tilde y+\dots, \\
\tilde x-x^+ &= a_{31}  r +  a_{32} (\tilde \varphi - \varphi^+) +  a_{33} x +  a_{34}\tilde y+\dots, \\
y-y^- &= a_{41} r +  a_{42} (\tilde \varphi - \varphi^+) +  a_{43} x +  a_{44}\tilde y+\dots,
\end{aligned}
\end{equation}
where  the dots are terms of order at least two, and  there are no  $(\tilde \varphi - \varphi^+)^2$ term in the first equation. 
By \eqref{eq:scatter_new}, the coefficient $\mu$ in this formula is the splitting parameter introduced in Section~\ref{sec:quadraunfold}.

Similarly, in the case of a cubic tangency, $T_1$ takes the form
\begin{equation}\label{eq:T1_cubic}
\begin{aligned}
\tilde r &=\mu+ \thebetai r +\nu(\tilde \varphi - \varphi^+)+  \dd (\tilde \varphi - \varphi^+)^3 +  a_{13} x +  a_{14}\tilde y+\dots, \\
\varphi-\varphi^- &= a_{21} r +  \thebetai(\tilde \varphi - \varphi^+) +  a_{23} x +  a_{24}\tilde y+\dots, \\
\tilde x-x^+ &= a_{31}  r +  a_{32} (\tilde \varphi - \varphi^+) +  a_{33} x +  a_{34}\tilde y+\dots, \\
y-y^- &= a_{41} r +  a_{42} (\tilde \varphi - \varphi^+) +  a_{43} x +  a_{44}\tilde y+\dots,
\end{aligned}
\end{equation}
where  the dots are terms of order at least two, and  there are no  $(\tilde \varphi - \varphi^+)^2$ and  $(\tilde \varphi - \varphi^+)^3$ terms in the dots of the first equation.  Again by \eqref{eq:scatter_new}, the coefficient $\nu$ in \eqref{eq:T1_cubic} is the other splitting parameter introduced in Section~\ref{sec:quadraunfold}.

\subsection{Transition maps for heteroclinic orbits}\label{transhe}

In the proofs, we also consider partially-hyperbolic {\em heteroclinic} orbits connecting two whiskered tori  $\gamma_1$ and $\gamma_2$ in $\A$, i.e., heteroclinic orbits satisfying \eqref{eq:condition:trans1}. A transition map can be defined in the same way as in the homoclinic case. More specifically, we  consider coordinates where both $\gamma_1$ and $\gamma_2$ are straightened, i.e., $\gamma_i=\{r=r_i,x=y=0\}$. Let $\mathcal{O}$ be a heteroclinic orbit and take points $M^-=(r_1,\pp^-,0,y^-)\in \mathcal{O}\cap W^\u_{\loc}(\gamma_1)$ and $M^+=(r_2,\pp^+,x^+,0)\in \mathcal{O}\cap W^\s_{\loc}(\gamma_2)$ satisfying $M^+=f^n(M^-)$ for some $n\in\mathbb N$. The transition map from a small neighborhood of $M^-$ to a small neighborhood of $M^+$ takes the form
\begin{equation}\label{eq:T1_hetero}
\begin{aligned}
\tilde r - r_2 &=\thebetai (r-r_1) + \dd (\tilde \varphi - \varphi^+)^{\ell+1} +  a_{13} x +  a_{14}\tilde y+\dots, \\
\varphi-\varphi^- &= {a}_{21} (r-r_1) +  \thebetai(\tilde \varphi - \varphi^+) +  a_{23} x +  a_{24}\tilde y+\dots, \\
\tilde x-x^+ &= {a}_{31} (r-r_1) +  a_{32} (\tilde \varphi - \varphi^+) +  a_{33} x +  a_{34}\tilde y+\dots, \\
y-y^- &= {a}_{41} (r-r_1) +  a_{42} (\tilde \varphi - \varphi^+) +  a_{43} x +  a_{44}\tilde y+\dots,
\end{aligned}
\end{equation}
where $\alpha,\dd,a_{44}$ are non-zero, and $\ell\geq 0$ is the order of the tangency (with $\ell=0$ for transverse heteroclinics).
\begin{rem}\label{rem:scat_trans}
Similar to the homoclinic case, the scattering map $S$ can also be defined along any partially-hyperbolic heteroclinic orbit connecting two whiskered tori, by the same formula \eqref{eq:scattering_3} with the transition map $T_1$ given by \eqref{eq:T1_hetero} and the holonomy maps $\pi^\u$ and $\pi^\s$ defined, respectively, in $W^\u_{\loc}(\gamma_1)$
and $W^\s_{\loc}(\gamma_2)$.  
\end{rem}

Remark~\ref{rem:scat_trans} allows us to establish a correspondence between the tangencies of  $T_1(W^\u_{\loc}(\gamma_1))$ with $W^\s_{\loc}(\gamma_2)$ and those of  $S(\gamma_1)$ with $ \gamma_2$. To see this, note that the holonomy maps used to define $S$ in \eqref{eq:scattering_3} are projections to $(r,\pp)$-coordinates, due to the straightening of the foliations. It follows from \eqref{eq:T1_hetero} that the scattering map $S$ restricted to a small neighborhood of $\pi^\u(M^-)\in\gamma_1$ is given by
\begin{equation}\label{eq:scatter_new}
\begin{aligned}
\tilde r - r_2 &=\thebetai (r-r_1) + \dd (\tilde \varphi - \varphi^+)^{\ell+1} +\dots, \\
\varphi-\varphi^- &= {a}_{21} (r-r_1) +  \thebetai(\tilde \varphi - \varphi^+) +\dots,
\end{aligned}
\end{equation}
which implies that the image $S\circ  \pi^\u(\Sigma^-\cap W^\u_{\loc}(\gamma_1))\subset S(\gamma_1)$ is given by
$$\tilde r - r_2 =    \dd (\tilde \varphi - \varphi^+)^{\ell+1}  +\dots.$$
Thus, with the obvious definition of tangencies in $\mathbb{R}^2$, we obtain
\begin{lem}\label{lem:equitangency}
For two whiskered tori $\gamma_1$ and $\gamma_2$, a partially-hyperbolic  tangency between $T_1(W^\u_{\loc}(\gamma_1))$ and $W^\s_{\loc}(\gamma_2)$ at $M^+=T_1(M^-)$ corresponds to
a tangency between $S(\gamma_1)$ and $\gamma_2$ at $(r_2,\pp^+)=S(r_1,\pp^-)$, respecting the order of the tangency.
The same is true for transverse intersections (tangencies of order 0).
\end{lem}

\section{Existence of cs/cu-blenders}\label{sec:proofcu}
In this section, we prove Theorem~\ref{thm:main_blender_cubic_nonpara}. 
Recall that this is a non-perturbative result, so  no parameters are involved.   We first  rewrite the formulas for $T_0$ and $T_1$ obtained in Section~\ref{sec:para} in certain rescaled coordinates and derive a formula for the first-return map $T_k=T^k_0\circ T_1$. After that, we prove Theorem~\ref{thm:main_blender_cubic_nonpara} along the lines sketched  in Section~\ref{sec:kam}. Since the case $|\thebeta|<1$ reduces to $|\thebeta|>1$ by considering the inverse map $f^{-1}$ instead of $f$, it is enough to prove Theorem~\ref{thm:main_blender_cubic_nonpara} for the case 
$|\thebeta|>1$ only. Thus, we
assume $|{\thebeta}|>1$ throughout this Section.

\subsection{First-return map in rescaled coordinates}\label{sec:scaling}

\subsubsection{Rescaled transition map}
Recall that $\Pi^-$ and $\Pi^+$ are small neighborhoods of the points 
$M^-=(0,\pp^-,0,y^-)$ and  $M^+=(0,\pp^+,x^+,0)$, respectively.   The  transition map $T_1:\Pi^- \to \Pi^+$ near the 2-flat tangency is given by \eqref{eq:T1_qudra+}. Since the functions $p_i$ in
\eqref{eq:T1_qudra+} are such that $p_i(\pp^++\delta)=o(\delta^2)$
as $\delta\to 0$, we can take sufficiently slowly decreasing to zero
positive functions $\aa$ and $\tilde p(\delta)$ such that 
\begin{equation}\label{eq:aa}
\tilde p(\delta) = o(\aa) 
\quad\mbox{and}\quad p_i(\pp^++\delta) = \tilde p(\delta)\cdot o(\delta^2)
\quad\mbox{as}\quad 
\delta\to 0.
\end{equation}
 
For sufficiently small $\delta>0$, consider the following coordinate change in $\Pi^-$:
\begin{equation}\label{eq:scaling1}
\begin{array}{l}
r=  \aa \cdot \delta^2 R,\quad \varphi-\varphi^-=\delta\Phi,\quad x=\tilde p(\delta) \cdot \delta^2 X, \\[5pt]
 y-y^-= \tilde p(\delta) \cdot \delta^2 Y+a_{41}r+ a_{42}\thebeta(\varphi-\pp^-)+
 a_{43}x+ b_4(\varphi-\pp^-)^2,
\end{array}
\end{equation}
where $a$'s and $b$'s are coefficients from \eqref{eq:T1_qudra+}. Similarly, in $\Pi^+$ we consider the rescaling
\begin{equation}\label{eq:scaling2}
\begin{array}{l}
\tilde r= \aa\cdot \delta^2 \tilde R,\quad \tilde \varphi-\varphi^+=\delta \tilde \Phi,\quad \tilde y= \tilde p(\delta) \cdot \delta^2 \tilde Y,\\[5pt]
\tilde x-x^+=\tilde p(\delta) \cdot \delta^2 \tilde X + a_{31}\tilde r + a_{32}(\tilde \varphi - \varphi^+)
+ a_{34}\tilde y + b_3(\tilde \varphi - \varphi^+)^2.
\end{array}
\end{equation}

In the new coordinates,  formula \eqref{eq:T1_qudra+} for $T_1:(R,\Phi,X,Y)\mapsto(\tilde R,\tilde \Phi,\tilde X,\tilde Y)$ recasts as
\begin{equation}\label{eq:T1_new}
\begin{aligned}
&\tilde R= \thebetai R+\tilde h_1(R,\tilde\Phi,X,\tilde{Y}), \qquad
\Phi=\thebetai\tilde \Phi +\tilde h_2(R,\tilde\Phi,X,\tilde{Y}), \\
&\tilde X=a_{33}X+\tilde h_3(R,\tilde\Phi,X,\tilde{Y}), \qquad
Y =a_{44}\tilde Y +\tilde h_4(R,\tilde\Phi,X,\tilde{Y}),
\end{aligned}
\end{equation}
where $\|\tilde h\|_{C^1}=o(1)_{\delta\to 0}$ by (\ref{eq:aa0}) and (\ref{eq:aa}).

\subsubsection{Rescaled local map}\label{sec:rescalelocl}
Let $\{(r_j,\varphi_j,x_j,y_j)\}_{j=0}^k\subset V$ be an orbit segment of $T_0$ such that
\begin{equation}\label{eq:asign}
(\tilde r,\tilde \varphi,\tilde x,\tilde y):=(r_0,\varphi_0,x_0,y_0)\in \Pi^+\quad\mbox{and}\quad
(\bar r,\bar \varphi,\bar x,\bar y):=(r_k,\varphi_k,x_k,y_k)\in \Pi^-.
\end{equation}
The point $(\tilde r,\tilde \varphi,\tilde x,\tilde y)$ corresponds to rescaled coordinates $(\tilde R, \tilde \Phi, \tilde X, \tilde Y)$ as defined by
\eqref{eq:scaling2} and the point $(\bar r,\bar \varphi,\bar x,\bar y)$
corresponds to rescaled coordinates $(\bar R, \bar \Phi, \bar X, \bar Y)$ as defined by \eqref{eq:scaling1}. 

Our goal is to estimate the map $T_0^k: (\tilde R, \tilde \Phi, \tilde X, \tilde Y)\mapsto (\bar R, \bar \Phi, \bar X, \bar Y)$. We will use formula \eqref{innerlowfinal} -- for that, we need to ensure that the three conditions in Lemma~\ref{lem:innerlowfinal} are satisfied. The second condition is automatic for bounded values of $\tilde \Phi$ when the scaling coefficient $\delta$ in \eqref{eq:scaling1} and \eqref{eq:scaling2} is sufficiently small. Also, 
we will only consider $k$ values such that 
\begin{equation}\label{eq:kr}
k\in[c_1\delta^{-1}, c_2\delta^{-1}],
\end{equation}
for some constants $c_2>c_1>0$. This implies that for bounded $R$ in \eqref{eq:scaling1} we have $k=o(r^{-1/2})$, and hence the first condition in Lemma~\ref{lem:innerlowfinal} is fulfilled as well. For the last condition, we further assume that 
\begin{equation}\label{condition}
|(\varphi^+ -\varphi^-+k\rho) \bModa 1| =O(\delta),
\end{equation}
where we used the notation 
\begin{equation}\label{mod}
a \bModa b = ((a+ b/2)\; \bMod b) - b/2.
\end{equation}
For such choice of $k$, the map $T_0^k$ takes an $O(\delta)$ neighborhood of the homoclinic point $M^+$ to an $O(\delta)$
neighborhood of the homoclinic point $M^-$.

Then, combining  Lemma~\ref{lem:derisyst} and Lemma~\ref{lem:innerlowfinal}, we find the following formula for $T_0^k$ in the non-rescaled coordinates:
\begin{equation}\label{eq:T0^kaaa}
\begin{aligned}
r_k &=r_0 + \hat h_1(r_0,\pp_0,x_0,y_k), \\ 
\varphi_k &=\varphi_0 + k\rho + \hat h_2(r_0,\pp_0,x_0,y_k),\\
 x_k &= \hat h_3(r_0,\pp_0,x_0,y_k),\qquad  y_0 = \hat h_4(r_0,\pp_0,x_0,y_k),
\end{aligned}
\end{equation}
where the functions $\hat h$ satisfy
\begin{equation}\label{eq:derisyst2:11}
\begin{array}{l}
\hat h_1 = O(k^2r_0^2 + |r_0|\cdot |\pp-\pp^+| + \lambda^{k/2}), 
\qquad \hat h_2 = O(k |r_0| + (\pp-\pp^+)^2 + \lambda^{k/2}),\\
\dfrac{\p \hat h_1}{\p {r_0}}  = O(k^2 |r_0| + |\pp-\pp^+| + \lambda^{k/2}), \qquad 
\dfrac{\p \hat h_1}{\p {\pp_0}} = O(k |r_0| + \lambda^{k/2}),\\
\dfrac{\p \hat h_2}{\p {r_0}} = O(k), \qquad 
\dfrac{\p \hat h_2}{\p {\pp_0}} = O(k^2 |r_0| + |\pp-\pp^+| + \lambda^{k/2}),\\
\dfrac{\p \hat h_{1,2}}{\p (x_0,y_k)}
  = O(\lambda^{k/2}), \qquad \|\hat h_{3,4}\|_{C^1}=O(\lambda^k).
\end{array}
\end{equation}

After the scalings (\ref{eq:scaling1}) and (\ref{eq:scaling2}), for $k$ satisfying (\ref{eq:kr}), the map $T_0^k:(\tilde R,\tilde \Phi,\tilde X,\tilde Y)\mapsto(\bar R,\bar \Phi,\bar X,\bar Y)$  takes the form
\begin{equation}\label{eq:T0^k_new}
\begin{aligned}
\bar R&= \tilde  R +\bar h_1(\tilde R,\tilde \Phi,\tilde X,\bar Y), \\
\bar \Phi&=\delta^{-1}(\varphi^+ + k\rho -\varphi^-) +\tilde \Phi + \bar h_2(\tilde R,\tilde \Phi,\tilde X,\bar Y), \\
\bar X&=\bar h_3(\tilde R,\tilde \Phi,\tilde X,\bar Y),\qquad
\tilde Y=\bar h_4(\tilde R,\tilde \Phi,\tilde X,\bar Y),
\end{aligned}
\end{equation}
where $\|\bar h\|_{C^1} =o(1)_{k\to +\infty}= o(1)_{\delta\to 0}$.

\subsubsection{Rescaled first-return map}
The region
\begin{equation}\label{eq:Tk_domain}
\Pi: \{R \in [ - 1, 1], \Phi \in [-1,1], X\in [-1,1]^{N-1}, Y \in [-1,1]^{N-1}\} \subset \Pi^-
\end{equation}
is an $O(\delta)$ neighborhood (in the original, non-rescaled coordinates) of the homoclinic point $M^-$. Combining (\ref{eq:T1_new}) and (\ref{eq:T0^k_new}), we obtain
\begin{lem}\label{lem:normalform}
Let conditions \eqref{eq:kr} and \eqref{condition} be satisfied. For all sufficiently small $\delta$ (hence large $k$),  the points 
 $(R,\Phi,X,Y)$ and $(\bar R,\bar \Phi,\bar X,\bar Y)$ in $\Pi$ are related by the map
$T_k:= T_0^k\circ T_1$ if and only if
\begin{equation}\label{eq:Tk}
\begin{aligned}
\bar R&= {\thebetai}R +h_1( R,\Phi, X,\bar Y), \\
\bar\Phi&=\delta^{-1}((\varphi^+-\varphi^- +k\rho) \bModa 1) +\thebeta \Phi +h_2( R,\Phi, X,\bar Y), \\
\bar X&=h_3( R,\Phi, X,\bar Y),\qquad
Y=h_4(R,\Phi, X,\bar Y),
\end{aligned}
\end{equation}
where $\bModa 1$ is defined in \eqref{mod} and the functions $h$ satisfy
\begin{equation}\label{eq:est1}
\|h\|_{C^1}=o(1)_{\delta\to 0}.
\end{equation}
\end{lem}

\subsection{Cone  lemma}
We now establish  the existence of invariant cone fields in $\Pi$, which implies the uniform hyperbolicity, and the partial hyperbolicity, of the system of first-return maps $T_k$. We denote by $(\Delta R,\Delta \Phi,\Delta X,\Delta Y)$ vectors in the tangent space. 

\begin{lem}\label{lem:conefields} Take any  small $L>0$.
For all sufficiently small $\delta$ and all $k$ satisfying \eqref{eq:kr} and \eqref{condition} such that $T_k$ assumes the form \eqref{eq:Tk}, the cone fields
\begin{align}
\mathcal{C}^{\u} &= \{(\Delta R,\Delta \Phi,\Delta X,\Delta Y): \max\{|\Delta R|,\|\Delta X\|\} \leqslant L(|\Delta \Phi|+\|\Delta Y\|)\},\label{eq:conefields:u}\\
\mathcal{C}^{\uu} &= \{(\Delta R,\Delta \Phi,\Delta X,\Delta Y): \max\{|\Delta R|,|\Delta \Phi|,\|\Delta X\|\}\leqslant L  \|\Delta Y\|\} \subset \mathcal{C}^{\u} ,\label{eq:conefields:uu}
\end{align}
are strictly forward-invariant in the sense that if a point $M\in \Pi$ has its image $\bar M=T_k(M)$ in $\Pi$, 
then the cone at $M$ is mapped into the interior of the cone at $\bar M$ by $\D T_k$; 
and the cone fields
\begin{align}
\mathcal{C}^{\s}&=\{(\Delta R,\Delta \Phi,\Delta X,\Delta Y): \max\{ |\Delta \Phi|,\|\Delta Y\|\}\leqslant L(|\Delta R|+\|\Delta X\|)\},\label{eq:conefields:s}\\
\mathcal{C}^{\ss} &= \{(\Delta R,\Delta \Phi,\Delta X,\Delta Y): \max\{|\Delta R|,|\Delta \Phi|,\|\Delta Y\|\}\leqslant L  \|\Delta X\|\} \subset \mathcal{C}^{\s} , \nonumber 
\end{align}
are strictly backward-invariant in the sense that if a point $\bar{M}\in\Pi$ has its pre-image $M=T^{-1}_{k}(\bar M)$ in $\Pi$, 
then the cone at $\bar{M}$ is mapped into the interior of the cone at $M$ by $\D T^{-1}_k$. Moreover, vectors in $\mathcal{C}^{\u}$ are uniformly expanded by $\D T_{k}$, and vectors in $\mathcal{C}^{\s}$ are uniformly contracted by $\D T_{k}$.
\end{lem}

\begin{proof} Letting $\delta=0$ in (\ref{eq:Tk}), we get 
$$\Delta\bar R = {\thebetai}\Delta R, \qquad
\Delta \bar\Phi = \thebeta \Delta \Phi, \qquad
\Delta \bar X = 0, \qquad \Delta Y=0,$$
and the lemma obviously holds in the limit $\delta=0$ (recall that we consider here the case $|\alpha|>1$).
As the cone property is $C^1$-open, the 
result holds true for all sufficiently small $\delta$ as well.
\end{proof}

\subsection{Covering property}\label{sec:cover}
Our goal is to show that for certain $k$ values the first-return map 
$T_k$ has the so-called covering property (given by Lemma \ref{cor:image} below) for the sub-cube
\begin{equation}\label{eq:Pi_rho}
\Pi_d:=[-1,1]\times[-d,d]\times[-1,1]^{N-1}\times[-1,1]^{N-1}\subset \Pi,
\end{equation}
where $\Pi$ is defined in \eqref{eq:Tk_domain}, and $0 <d <1$.
 
Note that since the rotation number $\rho$ is irrational, the Dirichlet Approximation Theorem implies that there exist arbitrarily large co-prime integers $p$ and $q>0$ such that 
\begin{equation}\label{eq:qcon}
\rho = \frac{p}{q} + \frac{C}{q^2} \quad\mbox{with}\;\; |C|<1.
\end{equation}
\begin{lem}\label{prop:covering}
Let $q$ satisfying \eqref{eq:qcon} be sufficiently large and
\begin{equation}\label{eq:deltaq}
\delta=\dfrac{5}{2q}.
\end{equation}
For any $\Phi\in [-d,d]$ there exists 
integer $k\in[q,2q-1]$ such that 
\begin{equation}\label{coverp}
|\delta^{-1}((\varphi^+-\varphi^- +k\rho) \bModa 1) +\thebeta \Phi| < d,
\end{equation}
\end{lem}
\begin{proof}
For any given $\varphi^+,\varphi^-$  and $\Phi\in[-d,d]$, one can write 
\begin{equation}\label{eq:cover:4}
\varphi^+ - \varphi^- +\delta\thebeta \Phi=\dfrac{s}{q} + \dfrac{C_0}{q} 
\quad\mbox{with}\;\; |C_0|<\dfrac{1}{2}. 
\end{equation}
Since $p$ and $q$ are co-prime, the Diophantine equation
\begin{equation}\label{eq:cover:5}
kp - nq = - s
\end{equation}
has an integer solution $(k',n')$. Moreover, any pair $(k' + iq,n'+ip)$ with $i\in \mathbb{Z}$ is a solution too. Hence, for any 
value of $s$, we can always find a solution $(k,n)$ such that $k\in[q,2q-1]$. 

Take such $(k,n)$. Then, by \eqref{eq:qcon} and \eqref{eq:cover:4}, we have 
\begin{equation}\label{eq:cover6}
\varphi^+ -\varphi^-+k\rho  + \thebeta\delta\Phi  
=n+ \dfrac{C_0}{q}+\dfrac{k C}{q^2},
\end{equation}
so
$$|(\varphi^+ -\varphi^-+k\rho  + \thebeta\delta\Phi) \bModa 1| < \frac{5}{2q}.$$
When $\delta$ is small enough, this gives
\begin{equation}\label{eq:cover7}
|(\varphi^+ -\varphi^-+k\rho)  \bModa 1 + \thebeta\delta\Phi| < \frac{5}{2q} = d \delta,
\end{equation}
which proves the lemma.
\end{proof}

We now define $\mathcal{K}_q$ as the set 
of all integers $k\in  [q,2q-1]$ for which \eqref{eq:cover6} 
is satisfied for at least one value of $\Phi\in [-d,d]$.

\begin{rem}\label{rem:n=0} 
Since $|\Phi|\leq  d$ in \eqref{eq:cover7}, it follows that
\begin{equation}\label{coverest}
|(\varphi^+-\varphi^- +k\rho) \bModa 1| < d (1+\thebeta) \delta.
\end{equation}
for all $k\in\mathcal{K}_q$. Thus, condition \eqref{condition} is satisfied for all $k\in\mathcal{K}_q$. 
Condition \eqref{eq:kr} is satisfied as well for $\delta=O(1/q)$ and $q\leq k \leq 2q-1$, so  
the maps $T_k$, for all $k\in\mathcal{K}_q$, are given by \eqref{eq:Tk} with \eqref{eq:est1} satisfied.
\end{rem}

\begin{rem}\label{rem:Ak}
Since $\delta=5/(2q)$ and $|\alpha|>1$, the variation of $\alpha\delta \Phi$ as $\Phi$ runs the interval $[-1,1]$ is at least $5/q$, and the variation in $\alpha\delta \Phi - \frac{C_0(\Phi)}{q}$ is at least $4/q$
(where $C_0$ is given by (\ref{eq:cover:4})). It then follows from \eqref{eq:cover6} that, for all sufficiently large $q$, the maximal difference between the values of $(\varphi^+-\varphi^- +k\rho)\bModa 1$ for different $k\in \mathcal{K}_q$ is at least $3/q$, implying that there always exists $k\in \mathcal{K}_q$  such that  $\delta^{-1}|(\varphi^+-\varphi^- +k\rho)\bModa 1|\geq 3/5$.
\end{rem}

Take a small $\kappa>0$ and consider the cube 
$$\Pi_d' := [-(1-\kappa),1-\kappa]\times[-d(1-\kappa),d(1-\kappa)]\times[-\dfrac{1}{2},\dfrac{1}{2}]^{N-1}\times[-1,1]^{N-1}\subset \Pi_d.$$
\begin{defi}[Crossing] 
Consider a cube $Q = Q_1\times Q_2$ where $Q_i$ are closed subsets of $\mathbb{R}^{n_i}$ ($i=1,2$), diffeomorphic to discs. An $n_2$-dimensional disc $D$ is said to {\em cross} $Q$ in the direction of $Q_2$ if $D\cap Q$ is the graph of a function $Q_2\to int(Q_1)$.
\end{defi}
\begin{lem}\label{cor:image}
The following holds for all sufficiently large $q$ satisfying \eqref{eq:qcon} and $\delta=5/(2q)$: for any $(N-1)$-dimensional disc $D$  which crosses $\Pi_d$ along the $Y$-direction and is tangent to $\mathcal{C}^{\uu}$,  there exists $k\in \mathcal K_q$ such that $T_k(D\cap \Pi_d)$ is a disc that crosses $\Pi_d'$ along the $Y$-direction and is tangent to $\mathcal{C}^{\uu}$. 
\end{lem}

\begin{figure}[!h]
\begin{center}
\includegraphics[scale=.7]{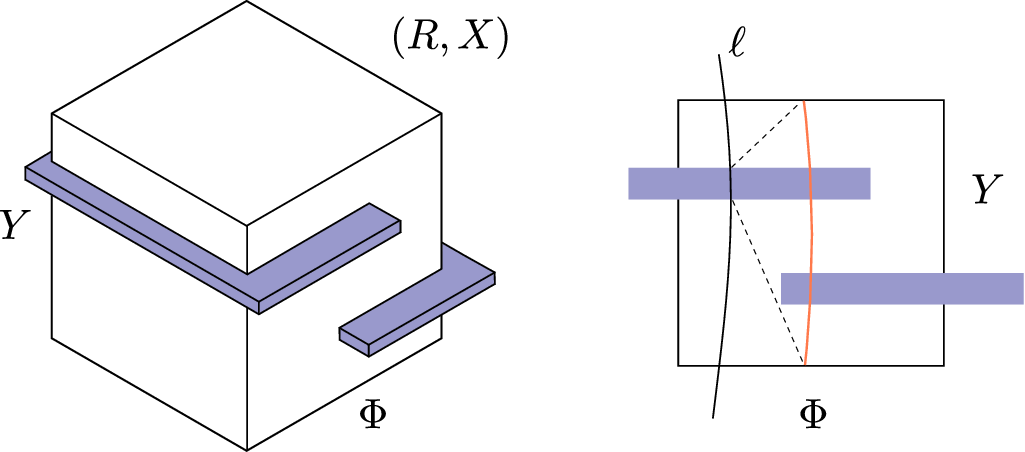}
\end{center}
\caption{The case where $N=1$, and $\mathcal{K}_q$ has two elements. The purple boxes are the two preimages of $\Pi$.
The one-dimensional disc $\ell$ intersects at least one of the preimages, and hence its image under $T_k$ for some $k\in \mathcal{K}_q$ again contains a piece (the orange one) intersecting one of the preimages.}
\label{fig:cover}
\end{figure}

\begin{proof}
By Remark \ref{rem:n=0}, for $k\in \mathcal{K}_q$ the map $T_k$ assumes the form \eqref{eq:Tk} with estimates \eqref{eq:est1}, and
Lemma \ref{lem:conefields} is applicable. So, the cone field $\mathcal{C}^{\uu}$ is forward-invariant, implying that
the image $T_k (D)$ is tangent to $\mathcal{C}^{\uu}$ as required.

Let $(R,\Phi,X)=(R(Y),\Phi(Y),X(Y))$  be the equation of $D$, with the derivative of the functions
$R(Y),\Phi(Y),X(Y)$ being bounded by the constant $L$ from Lemma \ref{lem:conefields}. In particular,
\begin{equation}\label{phl}
|\Phi(y) - \Phi(0)| \leq L \|y\| = o(1)_{\delta\to 0},
\end{equation}
because the cone constant $L$ can be taken arbitrarily small as $\delta\to 0$.
By \eqref{eq:Tk}, for each
$\bar Y\in [-1,1]^{N-1}$ there exist unique $\bar R, \bar\Phi, \bar X$ such that $(\bar R, \bar\Phi, \bar X, \bar Y)\in T_k (D)$
and 
\begin{equation}\label{phiy0}
\begin{aligned}
\bar R&= {\thebetai}R(y) +h_1(R(y),\Phi(Y), X(Y),\bar Y), \\
\bar\Phi&=\delta^{-1}((\varphi^+-\varphi^- +k\rho) \bModa 1) +\thebeta \Phi(Y) +h_2(R(y),\Phi(Y), X(Y),\bar Y), \\
\bar X&=h_3(R(Y),\Phi(Y), X(Y),\bar Y),
\end{aligned}
\end{equation}
where
$$Y=h_4(R(Y),\Phi(Y), X(Y),\bar Y)$$
(the value of $Y=o(1)_{k\to+\infty}$ is uniquely defined from the last equation for all $\bar Y\in [-1,1]^{N-1}$ by the contraction mapping principle).

By \eqref{eq:est1}, we have $\|\bar X\|<1/2$ if $k$ is large enough; since $|\thebetai|<1$, we also have $|\bar R|<1-\kappa$ if $|R|\leq 1$. Thus, to prove the lemma, it remains to show that $k\in \mathcal{K}_q$ can be chosen such that
$|\bar \Phi|< (1-\kappa) d$ for all $\bar Y \in [-1,1]^{N-1}$. For that, we choose $k$ given by Lemma \ref{prop:covering}
with $\Phi$ equal to $\Phi(0)$ from \eqref{phiy0}. Then, by \eqref{coverp} and \eqref{phl}, we have
\begin{align*}
&\max_{\|Y\|\leq 1} |\delta^{-1}((\varphi^+-\varphi^- +k\rho) \bModa 1) +\thebeta \Phi(Y)|\\ 
&\qquad\qquad\qquad\qquad\qquad \leq 
\max_{|\Phi|\leq d} |\delta^{-1}((\varphi^+-\varphi^- +k\rho) \bModa 1) +\thebeta \Phi| + L< d.
\end{align*}
Since $h_2\to 0$ as $q\to +\infty$, this implies $\max_{\|Y\|\leq 1} |\bar \Phi|<d$, which completes the proof of the lemma.
\end{proof}

\subsection{A blender for an induced map}\label{sec:cublender}

Choose
\begin{equation}\label{eq:rho}
d := \dfrac{|\alpha|-1}{|\alpha|+1}<1
\end{equation}
and $\delta=5/(2q)$ with a sufficiently large integer $q$ that satisfies \eqref{eq:qcon}.
Consider the induced map $\hat T$ defined by
\begin{equation}\label{eq:induce2}
\hat T(P)=T_k(P) \quad\mbox{if}\quad P\in T^{-1}_k(\Pi^-)\cap\Pi^-
\quad\mbox{for}\quad k\in \mathcal{K}_q.
\end{equation}
In what follows, we first find a hyperbolic set of $\hat T$, and then show that it is indeed a cu-blender, using Lemma~\ref{cor:image}.

Let us restrict $\hat T$ to the cube $\Pi$ given by \eqref{eq:Tk_domain}.
We say that a point $M$ has the coding $\{k_n\in \mathcal{K}_q\}_{n\in\mathbb{Z}}$ if there is a sequence of points 
$\{M_n\in\Pi\}_{n\in\mathbb{Z}}$ with $M_0=M$ such that $T_{k_n}(M_n)=M_{n+1}$. 
\begin{lem}\label{lem:hyp}
If $q$ is sufficiently large, the map $\hat T$ has a hyperbolic basic set $\hat\Lambda_q\subset int(\Pi)$ which is in one-to-one correspondence with the set of all possible codings $\{k_n\in \mathcal{K}_q\}_{n\in\mathbb{Z}}$.
\end{lem}
\begin{proof}
It is enough to show that for any coding $\{k_n\in \mathcal{K}_q\}_{n\in\mathbb{Z}}$
there is a unique point in $int(\Pi)$ with this coding. Once this is proven, we define 
$\hat\Lambda_q$ as the set of all points whose entire orbit under $\hat T$ never leaves $int(\Pi)$;
by Lemma~\ref{lem:conefields} the set $\hat\Lambda_q$ is uniformly hyperbolic.
By the definition of codings, ${\hat T}|_{\hat\Lambda_q}$ is conjugate to the full shift of $\mathrm{Card}\,\mathcal{K}_q$ symbols and, hence, $\hat\Lambda_q$ is a transitive locally-maximal set, i.e., it is indeed a hyperbolic basic set of $\hat T$. 

Let us take any sequence $\{k_n\in \mathcal{K}_q\}_{n\in\mathbb{Z}}$ and  find the point $M\in int(\Pi)$ with this 
coding sequence. Note that one can express $\Phi$ from formula \eqref{eq:Tk} as a function of $(R,\bar\Phi, X,\bar Y)$:
$$\Phi =\thebetai (\bar\Phi - \delta^{-1}((\varphi^+-\varphi^- +k\rho) \bModa 1)) + o(1)_{\delta\to 0}.$$

Consider now  the {\em cross-map} $T^\times_k$  which takes 
$(R,\bar \Phi,X,\bar Y)$ to $(\bar R, \Phi,\bar X, Y)$ if and only if 
$(\bar R,\bar \Phi,\bar X,\bar Y)= T_k(R, \Phi,X, Y)$. It is a map defined on $\Pi$, and we claim that $T^\times_k(\Pi) \subset int(\Pi)$ for all
$k\in \mathcal{K}_q$ with sufficiently large $q$. Indeed, we note from 
 \eqref{coverest} and \eqref{eq:rho} that 
$|\Phi|< |\thebetai| (1 + d (1+|\alpha|)) = 1$ for large enough $q$, and from  \eqref{eq:Tk} that $|\bar R|= |\thebetai R| + o(1)_{\delta\to 0}<1$ and
$\|(\bar X,Y)\|=o(1)_{\delta\to 0}<1$. 

By the definition of the cross-map, a   point
$P_n=(R_n,\Phi_n,X_n,Y_n) \in \Pi$ satisfies $T_{k_n}(P_n)=P_{n+1}$ if and only if 
$$(R_{n+1},\Phi_n,X_{n+1},Y_n)=T^\times_{k_n}(R_n,\Phi_{n+1},X_n,Y_{n+1}).$$
Thus, the orbit corresponding to the coding sequence $\{k_n\}$ is the fixed point  of the operator 
$$T^\infty_{_{\{k_n\}}}:\{( R_n, \Phi_n, X_n, Y_n)\}_{n\in\mathbb Z}\mapsto \{(\bar R_n,\bar \Phi_n,\bar X_n,\bar Y_n)\}_{n\in\mathbb Z},$$
which acts on the space of sequences of points in $\Pi$ by the rule
$$(\bar R_{n+1},\bar \Phi_n,\bar X_{n+1},\bar Y_n)=T^\times_{k_n}(R_n,\Phi_{n+1},X_n,Y_{n+1}).$$

Since $|\alpha^{-1}|<1$, it follows from \eqref{eq:Tk} that the cross-maps $T_k$ are contractions for $k\in \mathcal{K}_q$.
Therefore, by the Shilnikov lemma on the fixed point in a direct product of metric spaces \citep[Theorem 6.2]{Shi:67}, the operator $T^\infty_{_{\{k_n\}}}$ has a unique fixed point -- the sought orbit in $int(\Pi)$ with the coding $\{k_n\}$.  
\end{proof}

\begin{rem}\label{invfol} The local stable and unstable manifolds of a point $P\in \hat \Lambda_q$ is the set of all points whose coding sequence $\{k_n\}$ coincides with the coding of $P$ for all $n\geq 0$ and, respectively, $n\leq 0$. It is a standard consequence of the cross-map construction (see e.g. \cite[Lemma 1]{Tur:24}) that $W^\s_\loc(P)$ is the graph of a smooth function
$w^\s_P: \{|R|\leq 1, \; \|X\|\leq 1\} \to \{|\Phi|< 1, \; \|Y\| < 1\}$ and $W^\u_\loc(P)$ is the graph of a smooth function
$w^\u_P:  \{|\Phi|\leq 1, \; \|Y\|\leq 1\} \to  \{|R|< 1, \; \|X\| < 1\}$. Since the local stable and unstable manifolds are tangent to the cones $\mathcal{C}^\s$ and, respectively, $\mathcal{C}^\u$, the norms of the derivatives of $w^\s_P$ and $w^\u_P$ are bounded by the constant $L$ of Lemma~\ref{lem:conefields}, which satisfies $L=o(1)_{\delta\to 0}=o(1)_{q\to\infty}$.
\end{rem}

\begin{prop}\label{prop:cublender}
The  hyperbolic basic set $\hat\Lambda_q\subset \Pi$ found in Lemma~\ref{lem:hyp} is a cu-blender of $\hat T$.
\end{prop}

\begin{proof}
Let $\mathcal{D}$ be the set of all $(N-1)$-dimensional discs that cross $\Pi_d\subset \Pi$ and are tangent to $\mathcal{C}^{\uu}$, and let $\mathcal{D}'$ be the set of all $(N-1)$-dimensional discs that cross $\Pi_d'\subset \Pi_d$ and tangent to $\mathcal{C}^{\uu}$, where $\Pi_d$ and $\Pi_d'$ are the cubes from Lemma~\ref{cor:image}. By construction,
all requirements of Definition~\ref{defi:blender} are satisfied by $\mathcal{D}$ and $\mathcal{D}'$.
\end{proof}

\subsection{Connecting the blender to $\gamma$: proof of Theorem~\ref{thm:main_blender_cubic_nonpara}}
The blender $\hat\Lambda_q$ of $\hat T$ which is given by Proposition~\ref{prop:cublender} corresponds to 
a cu-blender $\Lambda_q$ of $f$:
\begin{equation}\label{eq:blenderf}
\Lambda_q:=\bigcup_{k\in\mathcal{K}_q}\bigcup_{n=0}^{k}f^n(\hat\Lambda_q).
\end{equation} 
We complete the proof of Theorem~\ref{thm:main_blender_cubic_nonpara} for the case $|\alpha|>1$ by showing that
$\Lambda_q$ is connected to $\gamma$ in the sense of Definition~\ref{defi:blenConn}. The partial hyperbolicity requirement of this definition is an automatic consequence of Lemma~\ref{lem:conefields}.
The local transversality is achieved in $V^\u:=\Pi$, which follows from the facts that 
$W^\u_{\loc}(\gamma)\cap \Pi=\{R=0,X=0\}$ by \eqref{eq:strait3} and \eqref{eq:scaling1}, and that for every point $P\in \Lambda_q\cap \Pi$, one has $W^\s_\loc(P)=
\{(\Phi,Y)=w^\s_P(R,X)\}$ by Remark~\ref{invfol}. Finally, the blender property is a consequence of Lemma~\ref{lem:blentoblenCon} plus the fact that  the strong-unstable leaf through $M^-$ contains a disc in the set $\mathcal{D}$ in the proof of Proposition~\ref{prop:cublender}. \qed

\section{Creation of symplectic blenders}\label{sec:proofsymp}
In this section we prove Theorem~\ref{thm:main_2punfolding} and Proposition~\ref{prop:persis}. Recall that we consider sufficiently smooth two-parameter families of symplectic maps $f_\eps$ which have, for all small $\varepsilon$, a non-degenerate whiskered KAM-curve $\gamma$ with rotation number $\rho$ (independent of $\eps$). At $\eps=0$, the curve $\gamma$
has a partially-hyperbolic homoclinic tangency. We start with the proof of the theorem with case (1) -- the cubic tangency case. Then we show that the unfolding of a quadratic tangency can give rise to two secondary quadratic tangencies, which further lead to a cubic tangency, thus proving cases (2) and (3). In the last subsection, we consider the unfolding of the non-transverse intersections with the invariant manifolds of a symplectic blender, proving Proposition~\ref{prop:persis}. We start with preliminary results, common for all the proofs.

\subsection{Iterations of the local map and an inclination lemma}
By assumption, $\gamma$ lies in a normally-hyperbolic, two-dimensional invariant cylinder $\A$.
We use $C^{m'}$ Fenichel coordinates in a neighborhood of $\A$, where the inner map assumes the form (\ref{eq:kam_new}) with some integer $m$ satisfying $m'\geq m\geq 2$ (the numbers $m'$ and $m$ can be taken arbitrarily large when $f_\eps$ is sufficiently smooth). Combining 
Lemma~\ref{lem:F^k} and Lemma~\ref{lem:derisyst} (with $s\geq m'+2$), we immediately obtain 

\begin{lem}\label{lem:derisyst2} If  $V$ (the neighborhood of $\A$ where the local map $T_0$ is defined) is sufficiently small, 
then, for any $(r_0,\pp_0,x_0,y_0) \in V$ with sufficiently small $r_0$, and for all $k=o(r_0^{1-m})$ one has $(r_k,\pp_k,x_k,y_k)=T_0^k(r_0,\pp_0,x_0,y_0)\in V$ if and only if
\begin{equation}\label{eq:derisyst2:4}
\begin{aligned}
r_k & = r_0 +  h_1(r_0,\pp_0,x_0,y_k,\eps), \\
\varphi_k & =\varphi_0 +k\rho + kr_0+k\hat\rho(r_0,\eps)+ k h_2(r_0,\pp_0,x_0,y_k,\eps),\\
 x_k & =  h_3(r_0,\pp_0,x_0,y_k,\eps),\quad
 y_0  = h_4(r_0,\pp_0,x_0,y_k,\eps),
\end{aligned}
\end{equation}
where $\rho = O(r_0^2)$ is a degree $(m-1)$ polynomial in $r_0$, and
the functions $h$ satisfy the following properties:
$$h_s=h_{s1}(r_0,\pp_0,\eps) + h_{s2}(r_0,\pp_0,x_0,y_k,\eps), \qquad s=1,2,$$
where
\begin{equation}\label{eq:derisyst2:1}
\begin{array}{l}
\dfrac{\p^{i+|j|+|\ell|} (h_{11},h_{21})}{\p r^i_0 \p (\pp_0,\eps)^j}=
O(kr_0^{m-i}), \;\; 0\leq i+|j| \leq \min\{m,m'-m\},\\
\|(h_{12},h_{22})\|_{C^{m'}}= O(\lambda^{\frac{k}{2}}), \qquad
\|(h_3,h_4)\|_{C^{m'}}=O(\lambda^k).
\end{array}
\end{equation}
We also have
\begin{align*}
&h_{11}=0\;\;\mbox{for}\;\; r_0=0, \qquad
(h_{12},h_{22},h_3)=0\;\;\mbox{for}\;\; x_0=0, \\
& (h_{12},h_{22},h_4)=0\;\;\mbox{for}\;\; y_k=0.
\end{align*}
\end{lem}

We will use the following version of the ``inclination lemma'' for whiskered KAM-curves (cf. \citep{CreWig:15,Sab:15}). 
\begin{lem}\label{lem:incli0}
Let $m'\geq m \geq 3$ and let $m^*:=\min\{m'-m,m\}$. If a smooth $N$-dimensional manifold $W$ intersects  
$W^\s_{\loc}(\gamma)$ transversely, then the sequence of manifolds $\{T_0^k(W)\}_{k\in\mathbb{N}}$ accumulates,
in the $C^{m^*}$ topology (jointly with respect to variables and parameters),
on $W^\u_{\loc}(\gamma)$ from both sides -- from $r>0$ and $r<0$. Likewise, if $W$ intersects  $W^\u_{\loc}(\gamma)$ transversely, then the sequence $\{T_0^{-k}(W)\}_{k\in\mathbb{N}}$ accumulates on $W^\s_{\loc}(\gamma')$ from both sides.
\end{lem}

\begin{proof}
By symmetry, it suffices to only consider the case where $W$ intersects $W^\s_\loc(\gamma)$.
Let $M$ be the point of the transverse intersection of $W$ and $W^\s_\loc(\gamma)$. 
Recall that $W^\s_\loc(\gamma)$ is given by the equation $\{r=0, y=0\}$. So,
by the transversality, $W$ near $M$ is the graph of some $C^{m'}$ function 
$(\pp_0,x_0)=w(r_0,y_0,\eps)$ defined for all small $r_0,y_0,\eps$.  

For integer $k$, let $\delta_k=k^{-3/4}$. Let $W_k^+$ be the part of $W$ corresponding to $r_0\in [\delta_k/2,\delta_k]$ and 
$y\in D$ where $D$ is a small $(N-1)$-dimensional ball around zero, and $W_k^-$ be the part of $W$ corresponding to $r_0\in [-\delta_k,-\delta_k/2]$ and $y\in D$. To prove the lemma, we will show that $T^k_0(W^+_k)$ and 
$T^k_0(W^-_k)$ accumulate on $W^\u_{\loc}(\gamma): \{r=0, x=0\}$ as $k\to+\infty$, from the side of positive and, respectively, negative $r$. We only consider the sequence $T^k_0(W^+_k)$; the computations for $T^k_0(W^-_k)$
are the same. 

So, our goal is to show that $T^k_0(W^+_k)$ is the graph of a function $(\varphi,y)\mapsto (r,x)$ defined for 
$y\in D$ and $\varphi$ from an interval of length larger than $1$, such that the $C^{m^*}$ norm of this function tends to zero as $k\to +\infty$. 
Since $m\geq 3$,
one has $k<|r_0|^{-4/3}=o(r_0^{1-m})$ for $|r_0| \leq \delta_k$. Thus, Lemma~\ref{lem:derisyst2} is applicable for sufficiently large $k$, and we have from \eqref{eq:derisyst2:1} that $(r_k, \varphi_k, x_k, y_k) \in T^k_0(W^+)$ 
if and only if, for some $r_0 \in [\delta/2,\delta]$, $y_0\in D$,
\begin{equation}\label{eq:T0incline2}
\begin{aligned}
r_k &=r_0 +  \hat h_1(r_0,y_0,y_k,\eps), \\ 
\varphi_k &=w_1(r_0,y_0,\eps) +k\rho+ kr_0 +k\hat\rho(r_0,\eps) +  k \hat h_2(r_0,y_0,y_k,\eps),\\
 x_k &=  \hat h_3(r_0,y_0,y_k,\eps),\qquad  
 y_0 =  \hat h_4(r_0,y_0,y_k,\eps),
\end{aligned}
\end{equation}
where $w_1$ is the $\pp_0$-component of $w$ and $\hat h(r_0,y_0,y_k,\eps)=h(r_0,w(r_0,y_0,\eps),y_k,\eps)$. We have
$$\dfrac{\p^{i+|j|} \hat h_{1,2}}{\p r_0^i\p(y_0, y_k, \varepsilon)^j}=O(\delta^{m-{\frac{4}{3}}-i}),$$
for $0\leq i+|j|\leq m^*$, and 
$$\|\hat h_{3,4}\|_{C^{m^*}}= O(\lambda^k).$$

Since the range of values of $w_1$ is bounded, $\rho= O(r_0^2)=O(\delta^2)$, and $\hat h_2=O(\delta^{m-(4/3)})=
O(\delta^{5/3})$, it follows from the second equation in \eqref{eq:T0incline2} that the range of $\pp_k$ covers
an interval whose length is of order $k\delta \gg 1$ as $r_0$ runs from $\delta/2$ to $\delta$. 
Also, since $\hat h_1=O(\delta^{5/3})=o(\delta)$, we have from the first equation in \eqref{eq:T0incline2}
that $r_k>0$ when $r_0\in [\delta/2, \delta]$ and $\delta$ is sufficiently small.

Hence, we are left to show that for each $\pp_k$ from this interval, and $y_k\in D$, the corresponding values of $x_k$ and $r_k>0$ are uniquely defined from \eqref{eq:T0incline2}, and that they tend to zero as $k\to +\infty$, along with the derivatives up to order $m^*$. For that, we express $(r_0,y_0)$ as a function of $(r_k,y_k)$ from the 
$r_0$ and $y_0$-equations in \eqref{eq:T0incline2} ans substitute the result into the rest of the equations. This gives
$$r_0 =r_k +  \tilde h_1(r_k,y_k,\eps), \qquad y_0 =  \tilde h_4(r_k,y_k,\eps), \qquad x_k =\tilde h_3(r_k,y_k,\eps),$$
$$\frac{1}{k} \varphi_k - \rho =  r_k + \hat\rho(r_k,\eps) + \frac{1}{k} w_1(r_k,y_k,\eps) +  \hat h_2(r_k,y_k,\eps),$$
where 
$\frac{\p^{i+|j|} \tilde h_{1,2}}{\p r_k^i\p (y_k, \varepsilon)^j}=O(\delta^{m-(4/3)-i})$
for $0\leq i+|j|\leq m^*$, and 
$\|\tilde h_{3,4}\|_{C^{m^*}}= O(\lambda^k)$.
Now, expressing $r_k$ as a function of $(\pp_k,y_k)$ from the last equation, and noticing that 
$k^{-1}= \delta^{4/3}$, we find that, 
$$x_k=O(\lambda^k), \qquad r_k = O(\delta^{m-{\frac{4}{3}}})= O(k^{-{\frac{5}{4}}}),$$
along with all derivatives with respect to $(y_k,\pp_k,\eps)$ up to the order $m^*$. 
\end{proof}

Recall that $\gamma$ is surrounded in $\A$ by a Cantor set of non-gegenerate KAM circles whose rotation numbers
have the same Diophantine properties as $\rho$ (they all are $(c,\tau)$-Diophantine). These curves are circles
$r=\const$ (see (\ref{8241})); so for any such curve $\gamma'$ we can make it equation $\{r=0\}$ by a shift in $r$,
and all the formulas for the local map in $\A$ remain the same, i.e., the theory we developed for the curve $\gamma$
remains true for the curve $\gamma'$. In particular, the above lemma is applicable to any such curve. This gives us

\begin{lem}\label{lem:incli}
For any pair of the $(c,\tau)$-Diophantine KAM-curves $\gamma_1,\gamma_2$ in $\A$, if 
$W^\u(\gamma_1)$ has a transverse intersection with $W^\s(\gamma_2)$, 
then the images of $W^\u(\gamma_1)$ by $T_0^k$ and of $W^\s(\gamma_2)$ by $T_0^{-k}$, $k \in \mathbb{N}$
accumulate (from both sides) on $W^\u_{\loc}(\gamma_2)$ and, respectively,  $W^\s_{\loc}(\gamma_1)$, in the $C^{m^*}$ topology (jointly with variables and parameters),  where $m^*=\min\{m'-m,m\}$.
\end{lem}

\begin{figure}[!h]
\begin{center}
\includegraphics[scale=.7]{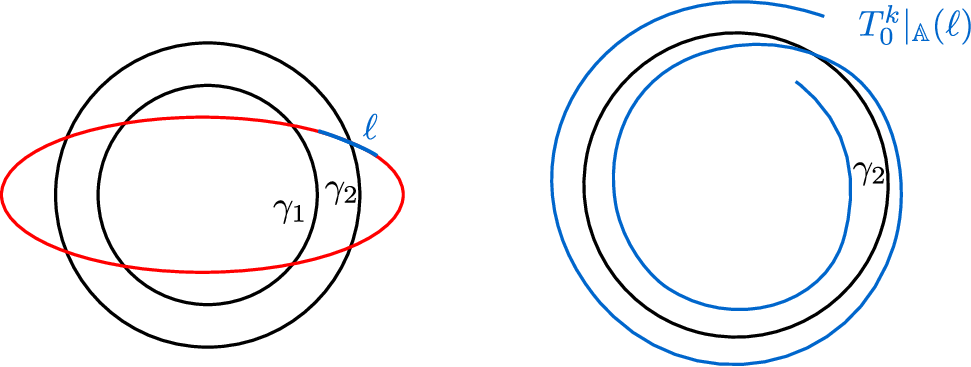}
\end{center}
\caption{Illustration to Lemma \ref{lem:incli} ($(r,\pp)$ projection). For any small piece $\ell$ of $S(\gamma_1)$ containing the point of a transverse intersection of $S(\gamma_1)$ with $\gamma_2$, there exists $k$ such that the $k$-th iterate of $\ell$ approaches $\gamma_2$ from both sides. The scattering map $S$ is defined for partially-hyperbolic heteroclinic intersections in Remark \ref{rem:scat_trans}.}
\label{fig:incline}
\end{figure}

\subsection{Case of a cubic tangency: proof of case (1) of Theorem~\ref{thm:main_2punfolding}}\label{sec:proofcubic}

We will use two key lemmas:

\begin{lem}\label{lem:a11}
Let $\min\{m'-m,m\}\geq 3$. Let the non-degenerate whiskered KAM-curve $\gamma$ have an orbit $\Gamma$
of a partially-hyperbolic cubic homoclinic tangency. Let $\hat V$ be a small neighborhood of $\orb\cup \Gamma$. Then, for a generic two-parameter unfolding family given by \eqref{eq:unfoldcubic}, there exists a sequence $\eps_k\to 0$ such that the  has, at each $\eps=\eps_k$, a secondary cubic homoclinic tangency in $\hat V$. The tangency is hyperbolic (expanding), see Definition \ref{defi:phtangency}, with the expansion factor $|\alpha|$ tending to infinity as $k\to \infty$. Moreover, this tangency unfolds generically as $\eps$ varies from $\eps_k$.
\end{lem}

\begin{lem}\label{lem:trans1}
Let $\min\{m'-m,m\}\geq 3$. Let the non-degenerate whiskered KAM-curve $\gamma$ have an orbit $\Gamma$ of cubic hyperbolic homoclinic tangency and let $\hat V$ be a small neighborhood of $\orb\cup \Gamma$. Then, the blender $\Lambda\subset \hat V$ given by Theorem~\ref{thm:main_blender_cubic_nonpara} is homoclinically related to $\gamma$. Namely, $W^\u(\Lambda)$ has a transverse intersection with $W^\s(\gamma)$
and $W^\s(\Lambda)$ has a transverse intersection with $W^\u(\gamma)$. Moreover, the corresponding heteroclinic orbits stay in $\hat V$.  
\end{lem}

These two lemmas are proved in Section~\ref{sec:a11} and Section~\ref{sec:trans2}, respectively.  The results imply case (1) of Theorem~\ref{thm:main_2punfolding} as follows.
Applying Lemma~\ref{lem:a11}, we find arbitrarily close to $\eps=0$ values of $\eps= \eps_k$ for which  $\gamma$ has a hyperbolic cubic homoclinic tangency with $|\alpha|>1$. 
By Theorem~\ref{thm:main_blender_cubic_nonpara}, a cu-blender $\Lambda^{\cu}$ exists for $\eps=\eps_k$. By 
Lemma~\ref{lem:trans1}, the blender is homoclinically related to $\gamma$.
The blender persists for all $C^1$-small perturbations by definition, so it exists for an open set of $\eps$ values around $\eps^k$; the homoclinic relation to $\gamma$ is also persistent. 

Since the newly found cubic homoclinic tangency unfolds generically as $\eps$ varies from $\eps_k$, we can apply the same arguments to it -- but now we do it for the family $f^{-1}_{\eps}$. Thus, arbitrarily close to $\eps_k$
(so, in the region of the existence of $\Lambda^{\cu}$), we find an open region of $\eps$
where the map $f_{\eps}^{-1}$ has a cu-blender homoclinically related to $\gamma$. The cu-blender for $f^{-1}$ is a cs-blender for $f_\eps$. Thus, arbitrarily close to $\eps=0$ we have regions of $\eps$ values for which $f_\eps$
has a cu-blender $\Lambda^{\cu}$ and a cs-blender $\Lambda^{\cs}$, both connected to $\gamma$.

Now, by Lemma~\ref{lem:incli0}, $W^\u(\Lambda^\cu)$ accumulates on $W^\u(\gamma)$ and, hence, has a transverse intersection with $W^\s(\Lambda^\cs)$. Similarly, $W^\u(\Lambda^\cs)$ has a transverse intersection with $W^\s(\Lambda^\cu)$. So, the two blenders $\Lambda^\cu$ and $\Lambda^\cs$ are homoclinically related. 
By construction, the blenders and the heteroclinic orbits that connect them belong to a small neighborhood $\hat V$
of $\mathcal{O}(O)\cup \Gamma$ (where $\Gamma$ is the orbit of homoclinic tangency to $O$ that exists at $\eps=0$).
Thus, they inherit the partial hyperbolicity of $\Gamma$. This implies that a hyperbolic basic set set 
$\Lambda$ containing $\Lambda^\cu \cup \Lambda^\cs$ (such exists because $\Lambda^\cu \cup \Lambda^\cs$
are homoclinically related) carries the partially-hyperbolic structure required by Definition~\ref{defi:sympblen_conn}. 
Since $\Lambda$ contains a cu-blender and a cs-blender, it also obviously satisfies the other conditions in the definition, and hence it is a symplectic blender connected to $\gamma$, as required. \hfill $\Box$

\subsubsection{Secondary cubic homoclinic tangencies: proof of Lemma~\ref{lem:a11}}\label{sec:a11}

We need the following auxiliary result:
\begin{lem}\label{lem:transinter} Let $m'\geq m+1 \geq 4$.
If $\gamma$ has an orbit $\Gamma_1$ of a partially-hyperbolic cubic homoclinic tangency, then there exists a sequence of homoclinic orbits of transverse intersection of $W^{\u}(\gamma)$ and $W^{\s}(\gamma)$, accumulating on $\Gamma_1$. Moreover, the homoclinics are partially-hyperbolic in the sense of \eqref{eq:condition:trans1}.
\end{lem}
\begin{proof}
Take two points $M^+\in W^\s_\loc(\gamma)$ and $M^-\in W^\s_\loc(\gamma)$ in $\Gamma_1$, so $W^\u$ is tangent to  $W^\s_\loc(\gamma)$ at $M^+$ and $W^\s$ is tangent to  $W^\u_\loc(\gamma)$ at $M^-$.
By Lemma~\ref{lem:equitangency}, there are cubic tangencies between $S(\gamma)$ and $\gamma$, and between 
$\gamma$ and $S^{-1}(\gamma)$ as well, where $S$ is the scattering map.
Recall that $\gamma$ is accumulated, in $C^{m'}$ topology with $m' > 3$, by non-degenerate $(c,\tau)$-Diophantine KAM-curves (see Section \ref{sec:kam_strait}). Since $S(\gamma)$ and $S^{-1}(\gamma)$ have a cubic tangency to 
$\gamma$, it is obvious that they intersect all these KAM curves (except for $\gamma$ itself) transversely. Thus, we can take a $(c,\tau)$-Diophantine curve $\gamma'$ such that 
$$W^\u(\gamma)\pitchfork W^\s_\loc(\gamma')\neq \emptyset,\qquad 
W^\s(\gamma)\pitchfork W^\u_\loc(\gamma')\neq \emptyset.$$
Note that these intersections can be found arbitrarily close to $\Gamma_1$ by taking $\gamma'$ sufficiently close to $\gamma$. Applying Lemma~\ref{lem:incli} with $m^*=1$ gives the accumulation of $W^\s(\gamma)$ on $W^\s_\loc(\gamma')$ in the $C^1$ topology, and hence the sought transverse intersections of $W^\u(\gamma)$ and $W^\s(\gamma)$. By construction, the corresponding homoclinic orbits lie in a small neighborhood of $\orb\cup \Gamma_1$, so they inherit the partial hyperbolicity from $\Gamma_1$.
\end{proof}
 
Let $\Gamma_2$ be a transverse homoclinic orbit given by the preceding lemma. Take $n_1$ and $n_2$ such that 
$M^-_1\in W^\u_{\loc}(\gamma)\cap \Gamma_1$ and $M^+_1=f^{n_1}(M^-_1)\in W^\s_{\loc}(\gamma)\cap \Gamma_1$, and $M^-_2\in W^\u_{\loc}(\gamma)\cap \Gamma_2$ and $M^+=f^{n_2}(M^-_2)\in W^\s_{\loc}(\gamma)\cap \Gamma_2$.  In the Fenichel coordinates, we can write (see \eqref{eq:strait}):
\begin{align*}
M^-_1=(0,\varphi^-_1,0,y^-_1),\qquad
M^+_1=(0,\varphi^+_1,x^+_1,0),\\
M^-_2=(0,\varphi^-_2,0,y^-_2),\qquad
M^+_2=(0,\varphi^+_2,x^+_2,0),
\end{align*}
for some $\pp^\pm_{1,2},x^+_{1,2},y^-_{1,2}$.

Since $\Gamma_2$ is partially hyperbolic in the sense of \eqref{eq:condition:trans1}, it follows from \eqref{eq:T1_general} that the transition map $T_2:(r,\varphi,x,y)\mapsto (\tilde r,\tilde \varphi,\tilde x,\tilde y)$ along the transverse homoclinic from a neighbourhood of $M^-_2$ to a neighborhood of $M^+_2$ takes the form
\begin{equation*}
\begin{aligned}
\tilde r &= {a}_{11} r +  a_{12} (\varphi -\varphi^-_2) +  a_{13} x +  a_{14}\tilde y+\dots, \\
\tilde \varphi - \varphi^+_2&= {a}_{21} r +  a_{22} (\varphi -\varphi^-_2 ) +  a_{23} x +  a_{24}\tilde y+\dots, \\
\tilde x-x^+_2 &= {a}_{31} r +  a_{32} (\varphi -\varphi^-_2 ) +  a_{33} x +  a_{34}\tilde y+\dots, \\
y-y^-_2 &= {a}_{41} r +  a_{42} (\varphi -\varphi^-_2 ) +  a_{43} x +  a_{44}\tilde y+\dots,
\end{aligned}
\end{equation*}
where $a_{44}\neq 0$. By the transversality of the homoclinic intersection, we  also have $a_{12}\neq 0$.
This allows us to find $(\varphi-\varphi^-_2)$ from the first equation and rewrite the above formula in the following cross-form:
\begin{equation}\label{homo:T1}
\begin{aligned}
\varphi -\varphi^-_2&= {a}_{11} r +  a_{12} \tilde r +  a_{13} x +  a_{14}\tilde y+\dots, \\
\tilde \varphi - \varphi^+_2&= {a}_{21}  r +  a_{22} \tilde r +  a_{23} x +  a_{24}\tilde y+\dots, \\
\tilde x-x^+_2 &= {a}_{31}  r+  a_{32} \tilde r +   a_{33} x +  a_{34}\tilde y+\dots, \\
y-y^-_2 &= {a}_{41}  r +  a_{42} \tilde r+  a_{43} x +  a_{44}\tilde y+\dots,
\end{aligned}
\end{equation}
with new coefficients $a$ such that $a_{12}\neq 0$ and $a_{44}\neq 0$. When we change parameters $\eps$, the transverse homoclinic persists; the coefficients $a_{ij}$ and $\pp^\pm_{1,2},x^+_{1,2},y^-_{1,2}$ then are 
at least $C^{m'-1}$ functions of $\eps$.

The transition map $T_1:(\hat r,\hat \varphi,\hat x,\hat y)\mapsto (\bar r,\bar \varphi,\bar x,\bar y)$ along the cubic tangency takes a neighbourhood of $M^-1$ to a neighborhood of $M^+_1$. A formula for $T_1$ for a two-parameter unfolding $\{f_\eps\}$ is given by \eqref{eq:T1_cubic}. To avoid confusion with the formula  \eqref{homo:T1} for the map $T_2$, we replace $a_{ij}$ in \eqref{eq:T1_cubic} by $b_{ij}$ and $\alpha^{-1}$ by $b$, and write $T_1$ as
\begin{equation}\label{homo:T2}
\begin{aligned}
\bar r &= \mu+\nu(\bar\varphi - \varphi^+_1)+{b} \hat r +  \dd (\bar \varphi - \varphi^+_1)^3 +  b_{13} \hat x +  b_{14}\bar y+\dots, \\
\hat \varphi -\varphi^-_1 &= {b}_{21}\hat  r +  {b} (\bar \varphi - \varphi^+_1) +  b_{23} \hat x +  b_{24}\bar y+\dots, \\
\bar x-x^+_1 &= {b}_{31} \hat r +  b_{32} (\bar \varphi - \varphi^+_1) +  b_{33} \hat x +  b_{34}\bar y+\dots, \\
\hat y-y^-_1 &= {b}_{41} \hat r +  b_{42} (\bar \varphi - \varphi^+_1) +  b_{43}\hat  x +  b_{44}\bar y+\dots,
\end{aligned}
\end{equation}
where $b\neq 0$, $\dd\neq 0$, and also $\mu=\nu=0$ at $\eps=0$. Note that the coefficients are at least 
$C^{m'-3}$ functions of $\eps$.

\begin{proof}[Proof of Lemma~\ref{lem:a11}]
According to \eqref{eq:unfoldcubic}, we can use $\mu$ and $\nu$ as parameters. We consider the composition 
$$(r,\pp,x,y) \xmapsto{T_2} (\tilde r,\tilde \pp,\tilde x,\tilde y) \xmapsto{T^k_0}  (\hat r,\hat \pp,\hat x,\hat y)
\xmapsto{T_1} (\bar r,\bar \pp,\bar x,\bar y),$$
 and show that changing $(\mu,\nu)$ leads to a secondary cubic tangency between $T_1\circ T^k_0\circ T_2(W^\u_\loc(\gamma))$ and $W^\s_\loc(\gamma)$.
 
We set
\begin{equation}\label{a11:trans}
\begin{array}{l}
(r,\tilde r,\hat r,\bar r)=k^{-\frac{3}{2}}(R,\tilde R,\hat R,\bar R),\\
(\varphi - \varphi^-_2,\tilde \varphi - \varphi^+_2) = k^{-\frac{3}{2}}(\Phi,\tilde \Phi),\qquad
(\hat \varphi - \varphi^-_1,\bar \varphi - \varphi^+_1) = k^{-\frac{1}{2}}(\hat \Phi,\bar \Phi),\\
(x,\tilde y)=k^{-\frac{5}{2}}(X,\tilde Y),\qquad 
(\tilde x - x^+_2, y-y^-_2)= k^{-\frac{3}{2}}(\tilde X,Y),\\
(\hat x, \bar y)=k^{-\frac{5}{2}}(\hat X,\bar Y),\qquad
(\bar x - x^+_1,\hat y-y^-_1)=k^{-\frac{1}{2}}(\bar X,\hat Y).
\end{array}
\end{equation}
Since we consider finite values of $\tilde R$, we have $k=O(\tilde r^{-2/3})=o(\tilde r^{1-m})$. Hence Lemma~\ref{lem:derisyst2} is applicable with  $(\tilde r,\tilde \pp,\tilde x,\tilde y):=(r_0,\pp_0,x_0,y_0)$ and $(\hat r,\hat \pp,\hat x,\hat y):=(r_k,\pp_k,x_k,y_k)$. 
In the new coordinates, the formulas \eqref{homo:T2}, \eqref{homo:T1}
and \eqref{eq:derisyst2:4} (with $m'\geq m+3\geq 6$) become\\

\noindent $T_1:(\hat R, \hat \Phi,\hat X,\hat Y)\mapsto (\bar R, \bar \Phi,\bar X,\bar Y)$:
\begin{equation}\label{homo:3}
\begin{aligned}
\bar R &= k^{\frac{3}{2}}\mu+k\nu\bar\Phi+b \hat R + \dd \bar\Phi^3 +O(k^{-1/2}), \qquad
\hat\Phi =  b \bar\Phi +O(k^{-\frac{1}{2}}), \\
\bar X &=   b_{32}\bar\Phi+O(k^{-\frac{1}{2}}), \qquad \hat Y =  b_{42} \bar\Phi+O(k^{-\frac{1}{2}}),
\end{aligned}
\end{equation}

\noindent$T_2:(R,\Phi,X,Y)\mapsto (\tilde R, \tilde \Phi,\tilde X,\tilde Y)$:
\begin{equation}\label{homo:1}
\begin{aligned}
\Phi &= a_{11}  R +  a_{12} \tilde R+ O(k^{-\frac{3}{2}}), \qquad
\tilde \Phi= a_{21}  R +  a_{22}\tilde R +  O(k^{-\frac{3}{2}}), \\
\tilde X &= a_{31} R +  a_{32}\tilde R+ O(k^{-\frac{3}{2}}), \qquad Y = a_{41} R +  a_{42}\tilde R + O(k^{-\frac{3}{2}}),
\end{aligned}
\end{equation}

\noindent $T^k_0:(\tilde R, \tilde \Phi,\tilde X,\tilde Y)\mapsto (\hat R, \hat \Phi,\hat X,\hat Y)$:
\begin{equation}\label{homo:2}
\begin{aligned}
\hat R &= \tilde R + O(k^{-\frac{3(m-1)}{2}}) ,\qquad
\hat \Phi= A(k)+ \tilde R+ O(k^{-\frac{3}{2}}) ,\\
\hat X &=  O(k^{\frac{5}{2}}\lambda^k),\qquad \tilde Y = O(k^{\frac{5}{2}} \lambda^k),
\end{aligned}
\end{equation}
where
\begin{equation}\label{ak}
A(k)= k^{\frac{1}{2}} ((\varphi^+_2 - \varphi^-_1+k \rho) \bModa  1),
\end{equation}
and the $O(\cdot)$ estimates in these formulas are with $m^*\geq 3$ derivatives at least.

By the Dirichlet Approximation Theorem, there exist arbitrarily large co-prime integers $p$ and $q>0$ such that 
$\rho = p/q + O(q^{-2})$. Let $s$ be an integer such that
$|\varphi^+_2 - \varphi^-_1 - s/q| <q^{-1}$. Take $k\in \{q,\dots, 2q-1\}$ such that
$k p + s = nq$ for an integer $n$. Then 
$$\varphi^+_2 - \varphi^-_1+k \rho  = n + O(k^{-1}),$$
implying that
$$A(k) = O(k^{-1/2})$$
in \eqref{ak}, i.e., we can choose a sequence of $k$ values such that $A(k)\to 0$.

Combining formulas \eqref{homo:3}--\eqref{homo:2} gives that gives that the map  
$T_1\circ T^k_0\circ T_2: (R,\Phi,X,Y)\mapsto(\bar R, \bar \Phi,\bar X,\bar Y)$ 
has the following form for these $k$:
$$\begin{aligned}
\bar R &= k^{\frac{3}{2}}\mu+(k\nu + b^2)\bar\Phi + \dd \bar\Phi^3 +O(k^{-\frac{1}{2}}), \quad
\bar\Phi = b^{-1} a_{12}^{-1} (a_{11}  R  -  \Phi) + O(k^{-\frac{1}{2}}), \\
\bar X &=   b_{32}\bar\Phi+O(k^{-\frac{1}{2}}), \qquad Y = a_{41} R +  a_{42} b \bar \Phi + O(k^{-\frac{1}{2}}).
\end{aligned}
$$

We have $W^\u_\loc(\gamma)=\{R=0,X=0\}$ and $W^\s_\loc(\gamma)=\{\bar R=0,\bar Y=0\}$. So,
the equation of $T_1\circ T^k_0\circ T_2 (W^\u_\loc(\gamma))$ in these coordinates is
$$\bar R = k^{\frac{3}{2}}\mu+(k\nu + b^2)\bar\Phi + \dd \bar\Phi^3 +O(k^{-\frac{1}{2}}), \qquad
\bar X =   b_{32}\bar\Phi+O(k^{-\frac{1}{2}}).$$
As the $O(k^{-1/2})$ terms are at least $C^3$ small, we immediately obtain the existence of the sought cubic
tangency with $W^\s_\loc(\gamma)$ for some
$$\mu = O(k^{-2}), \qquad \nu = - k^{-1} b^2 + O(k^{-\frac{3}{2}});$$
as required, the tangency unfolds generically as $(\mu,\nu)$ vary.

This tangency is partially hyperbolic, as every orbit lying in a small neighborhood of $\orb\cup \Gamma_1$
is partially hyperbolic. The quantity $\alpha$ for this tangency is given by $\alpha=\p \bar \varphi/\p \varphi= k \cdot 
\p \bar \Phi/\p \Phi = - k b^{-1} a_{12}^{-1} + O(k^{1/2}) \to \infty$.
\end{proof}

\subsubsection{Blenders of Theorem~\ref{thm:main_blender_cubic_nonpara} for  cubic tangencies: proof of Lemma~\ref{lem:trans1}}\label{sec:trans2}

By symmetry of the problem, it suffices to prove  Lemma~\ref{lem:trans1} for a cu-blender of Theorem~\ref{thm:main_blender_cubic_nonpara}. Recall that the cu-blender is obtained from the blender $\hat\Lambda_q$ of the induced map $\hat T$ given by Proposition~\ref{prop:cublender}, where $\hat T$ is defined by  \eqref{eq:induce2}. Since a fixed point of $T_k=T_0^k\circ T_1$ for some $k\in\mathcal{K}_q$ is a fixed point of $\hat T$ in $\hat\Lambda_q$ by \eqref{lem:hyp}, it further suffices to prove   the following
\begin{lem}\label{lem:trans2}
Let $\min\{m'-m,m\}\geq 3$.
For all sufficiently large $q$, there exists $k\in \mathcal{K}$ such that the unstable manifold  of the fixed point $P$ of $T_k$ intersects $W^\s(\gamma)$ transversely.
\end{lem}

We will prove this lemma by showing that  $T_1(W^\u_\loc(P))$ intersects $W^\s_\loc(\gamma)$ transversely. For that, we need to consider the new transition map corresponding to the cubic tangency and find a formula for $T_k$ with scalings slightly different from \eqref{eq:scaling1} and \eqref{eq:scaling2}.

We begin with the transition map, which is now given by \eqref{eq:T1_xform_n} with $\ell= 2$ (here our map is at least $C^{3}$). We rewrite it as the following form similar to \eqref{eq:T1_qudra+0}:
\begin{equation*}\label{eq:T1cubic}
\begin{aligned}
\tilde r &=\thebetai r  + \beta(\tilde \pp-\pp^+)^3+ a_{13} x +  a_{14}\tilde y+ p_1(\tilde \varphi )+q_1(r,\tilde \pp,x,\tilde y), \\
\varphi-\varphi^- &= a_{21} r +  \thebetai(\tilde \varphi - \varphi^+) +  a_{23} x +  a_{24}\tilde y+
p_2(\tilde \varphi )+q_2(r,\tilde \pp,x,\tilde y), \\
\tilde x-x^+ &= a_{31} \tilde r+  a_{32} (\tilde \varphi - \varphi^+) +  a_{33} x +  a_{34}\tilde y+b_3(\tilde \varphi - \varphi^+)^2 \\
&\qquad\qquad\qquad \qquad \qquad\qquad\qquad 
+c_3(\tilde \varphi - \varphi^+)^3+
p_3(\tilde \varphi )+q_3(r,\tilde \pp,x,\tilde y), \\
y-y^- &= a_{41} r +  a_{42} (\varphi - \varphi^-) +  a_{43} x +  a_{44}\tilde y + b_4(\varphi - \varphi^-)^2\\
&\qquad\qquad\qquad \qquad \qquad\qquad\qquad 
+c_4(\tilde \varphi - \varphi^+)^3+
p_4(\tilde \varphi )+q_4(r,\tilde \pp,x,\tilde y),
\end{aligned}
\end{equation*}
where $\beta\neq 0$ and 
\begin{equation*}\label{eq:T1cubicEst}
\begin{array}{l}
p_i=o((\tilde \varphi-\pp^+ )^3),\qquad 
\dfrac{\p p_i}{\p \tilde\pp} = o((\tilde \varphi-\pp^+ )^2 ),\qquad i=1,3,4, \\
p_2=O((\tilde \varphi-\pp^+ )^{2}),\qquad
\dfrac{\p p_2}{\p \tilde\pp} = O(\tilde \varphi-\pp^+ ),\\
q=O(r^2+x^2+\tilde y^2+|\tilde \varphi-\pp^+|(|r|+\|x\|+\|\tilde y\|)),\\
\dfrac{\p q}{\p \tilde\pp} = O(|r|+
\|x\| + \|\tilde y\|),\qquad
\dfrac{\p q}{\p (r, x, \tilde y)} = O(|r|+|\tilde \varphi-\pp^+|
+\|x\| + \|\tilde y\|).\end{array}
\end{equation*}
We consider the following scalings: 
\begin{equation}\label{eq:scalingnew}
\begin{array}{l}
r=   \delta^3 R,\quad \varphi-\varphi^-=\delta\Phi,\quad x= \delta^{\frac72} X,\\[5pt]
 y-y^-=   \delta^{\frac72} Y+a_{41}r+ a_{42}\thebeta(\varphi-\pp^-)+
 a_{43}x+ b_4(\varphi-\pp^-)^2,\\[5pt]
 \tilde r= \delta^3 \tilde R,\quad \tilde \varphi-\varphi^+=\delta \tilde \Phi,\quad \tilde y=  \delta^{\frac72} \tilde Y,\\[5pt]
\tilde x-x^+= \delta^{\frac72} \tilde X + a_{31}\tilde r + a_{32}(\tilde \varphi - \varphi^+)
+ a_{34}\tilde y + b_3(\tilde \varphi - \varphi^+)^2.
\end{array}
\end{equation}
In the rescaled coordinates, the above transition map assumes the form
\begin{equation}\label{eq:T1cubicNew}
\begin{aligned}
&\tilde R= \thebetai R+ \beta\tilde\Phi^3+ \tilde h_1(R,\tilde\Phi,X,\tilde{Y}), \qquad
\Phi=\thebetai\tilde \Phi +\tilde h_2(R,\tilde\Phi,X,\tilde{Y}), \\
&\tilde X=a_{33}X+\tilde h_3(R,\tilde\Phi,X,\tilde{Y}), \qquad
Y =a_{44}\tilde Y +\tilde h_4(R,\tilde\Phi,X,\tilde{Y}),
\end{aligned}
\end{equation}
where $\|\tilde h\|_{C^1}=O(\delta^{\frac12})$.

For the iterations $T^k_0:\Pi^+ \to \Pi^-$ of the local map, we  use  formula~\eqref{eq:derisyst2:4} with $\min\{m'-m,m\}\geq 3$.  With the assignment~\eqref{eq:asign}, we find that $T^k_0$ assumes the same form as~\eqref{eq:T0^k_new}, but with $\|\bar h\|_{C^1}=O(\delta)$ (here we also used the fact that $k\in \mathcal{K}_q$ satisfy~\eqref{eq:kr} and~\eqref{condition}).
Thus, we obtain the return map $T_k$ as 
\begin{equation}\label{eq:Tkcubic}
\begin{aligned}
\bar R&= {\thebetai}R + \beta (\alpha\Phi)^3 +h_1( R, \Phi, X,\bar Y), \\
\bar\Phi &= \delta^{-1}((\varphi^+-\varphi^- +k\rho) \bModa 1) +\alpha\Phi +  h_2( R, \Phi, X,\bar Y), \\
\bar X&=h_3( R, \Phi, X,\bar Y),\qquad
Y=h_4(R, \Phi, X,\bar Y),
\end{aligned}
\end{equation}
where $\bModa 1$ is defined in~\eqref{mod} and the functions $h$ satisfy $\|h\|_{C^1}=O(\delta^{\frac12})$.

\begin{proof}[Proof of Lemma~\ref{lem:trans2}]
By \eqref{condition}, the limit 
\begin{equation}\label{eq:Ak}
A(k):=\lim_{\delta\to 0} \delta^{-1}((\varphi^+-\varphi^- +k\rho) \bModa 1),
\end{equation}
for any $k\in \mathcal{K}_q$ is finite. Hence in the limit case $\delta=0$, the map $T_k$ assumes the form
\begin{equation}
\bar R  = \alpha^{-1}R +\beta \alpha^3 \Phi^3,\qquad
\bar \Phi = A(k)+\alpha\Phi,\qquad \bar X=0,\qquad Y=0.
\end{equation}
We easily find from this the fixed point $P=(R^*,\Phi^*,X^*,Y^*)$ of   the original map, where
\begin{align*}
&R^*=\dfrac{-\beta \alpha^4 A(k)^3}{(1-\alpha)^4}+O(\delta^{\frac12}),\qquad
\Phi^*=\dfrac{A(k)}{1-\alpha}+O(\delta^{\frac12}),\\
&X^*=O(\delta^{\frac12}),\qquad
Y^*=O(\delta^{\frac12}),
\end{align*}
and  two invariant manifolds of $T_k$ are given by 
\begin{equation}\label{eq:invmfdP}
\begin{aligned}
W^\s_\loc(P):&\qquad \Phi=\Phi^*+w^\s_1(R,X),\qquad Y=w^\s_2(R,X),\\
W^\u_\loc(P):& \qquad R=a_3\Phi^3+a_2\Phi^2+a_1\Phi+a_0+w^\u_1(\Phi,Y),\qquad X=w^\u_2(\Phi,Y)
\end{aligned}
\end{equation}
where functions $w^{\s/\u}_{1,2}$ along with their first derivatives are estimated as $O(\delta^{\frac12})$, and 
\begin{equation}\label{eq:thm1cubic:0}
\begin{aligned}
&a_3=\dfrac{\alpha^4\beta}{\alpha^4-1}\neq 0,\qquad
a_2=\dfrac{3\alpha^3 a_3A(k)}{1-\alpha^3},\qquad
a_1=\dfrac{2\alpha^2a_2A(k)+3\alpha^2a_3A(k)^2}{1-\alpha^2},\\
&a_0=\dfrac{a_1A(k)+a_2A(k)^2+a_3A(k)^3}{\alpha^{-1}-1}.
\end{aligned}
\end{equation}
It is obvious from~\eqref{eq:Tkcubic} that points on $W^\s_\loc(P)$ converge to $P$, so it is indeed the local stable manifold of $P$. Since $a_3\neq 0$, $W^\u_\loc(P)$ is transverse to $W^\s_\loc(P)$ at $P$, and hence  it is  the local unstable manifold of $P$ by uniqueness.

Let us denote $g(\Phi):=a_3\Phi^3+a_2\Phi^2+a_1\Phi+a_0$.
By \eqref{eq:T1cubicNew}, the image $T_1(W^\u_\loc(P))$ is given by
\begin{equation*}
\tilde R= \alpha^{-1} g(\alpha^{-1}\tilde \Phi)+ \beta\tilde\Phi^3+ O(\delta^{\frac12}),\qquad
\tilde X = O(\delta^{\frac12}).
\end{equation*}
Since $W^\s_\loc(\gamma)\cap \Pi^-=\{\tilde R=0,\tilde Y=0\}$ by \eqref{eq:strait3} and \eqref{eq:scalingnew}, we see that $T_1(W^\u_\loc(P))$ is at least topologically transverse  $W^\s_\loc(\gamma)$ for all sufficiently small $\delta$ (hence all sufficiently large $q$ by~\eqref{eq:deltaq}). To ensure that this intersection is indeed transverse and hence to prove the lemma, it suffices to show that there exists $k\in\mathcal{K}_q$ such that no real numbers $b_1,b_2$  satisfy
\begin{equation}\label{eq:thm1cubic:1}
\alpha^{-1}g(\tilde \Phi)+ \beta(\alpha\tilde\Phi)^3 = b_1(\tilde \Phi-b_2)^3.
\end{equation}
In what follows we prove this.

Since $W^\s_\loc(P)$ is invariant, $\bar R$ and $\bar X$ found from \eqref{eq:Tkcubic} also satisfies $\bar R=g(\bar \Phi)+w^\u_1(\bar \Phi,\bar Y)$ and $\bar X=w^\u_2(\bar \Phi,\bar Y)$. Taking $\delta\to 0$, we find the following functional equation for $g$:
$$
g(\alpha\Phi+A(k))=\alpha^{-1}g(\Phi) +\beta(\alpha \Phi)^3.
$$
If \eqref{eq:thm1cubic:1} holds for some $b_1$ and $b_2$, then the above equation implies that $g(\Phi)=c_1(\Phi-c_2)^3$ for some constants $c_1$ and $c_2$. As a result, we have
$$
\alpha^{-1}c_1(\tilde \Phi-c_2)^3+ \beta(\alpha\tilde\Phi)^3 = b_1(\tilde \Phi-b_2)^3,
$$
which holds only if $c_2=0$, and hence only if $A(k)=0$ by \eqref{eq:thm1cubic:0}. Thus, we only need to take $k\in\mathcal{K}$ such that $A(k)\neq 0$, whose existence is guaranteed   by Remark~\ref{rem:Ak}.
\end{proof}

\subsection{Perturbations of quadratic tangencies: proofs for cases (2) and (3) of Theorem~\ref{thm:main_2punfolding}}\label{sec:proofquadra}

We now conclude the proof of Theorem~\ref{thm:main_2punfolding} by observing that cases~(2) and~(3) can be reduced to case~(1) by  the following two results:

\begin{lem}\label{lem:1to2}
Let $\gamma$ have a partially-hyperbolic  quadratic homoclinic tangency, and consider any proper two-parameter unfolding family given by \eqref{eq:unfoldquadra}.
There exists a sequence $\{\eps_j\}$ converging to $0$ such that   the continuation of  $\gamma$ at each $\eps=\eps_j$ has two partially-hyperbolic quadratic  homoclinic tangencies. Moreover, these two tangencies unfold independently in the sense of  \eqref{eq:unfold2quadra} as $\eps$ varies from $\eps=\eps_j$.
\end{lem}

\begin{lem}\label{lem:cubic}
Let $\gamma$ have two partially-hyperbolic  quadratic homoclinic tangencies, and consider any generic two-parameter unfolding family given by \eqref{eq:unfold2quadra}.
There exists a sequence $\{\eps_k\}$ converging to $0$ such that the continuation of  $\gamma$ at each $\eps=\eps_k$ has a partially-hyperbolic cubic homoclinic tangency. Moreover, 
the found  tangency   unfolds generically in the sense of \eqref{eq:unfoldcubic} as $\eps$ varies  from $\eps=\eps_k$.
\end{lem}

\subsubsection{Creating two quadratic tangencies from one: proof of Lemma~\ref{lem:1to2}} \label{sec:1to2}
Let $\gamma$ have a partially-hyperbolic quadratic homoclinic tangency. In the Fenichel coordinates, we have \eqref{eq:strait}
satisfied, and the transition map from a small neighborhood of $M^-=(0,\pp^-,0,y^-)$ to a small neighborhood of $M^+=(0,\pp^+,x^+,0)$  is given by \eqref{eq:T1_quadra}.  Let $S$ be the scattering map defined by \eqref{eq:scattering_3}. Since  $\gamma$ is given by $\{r=0,x=0,y=0\}$,  Lemma~\ref{lem:equitangency} implies that $S(\gamma)$ has a quadratic tangency with $\gamma$ at $S(0,\pp^-)=(0,\pp^+)$ in the cylinder $\A:\{x=0,y=0\}$. Let $\ell_0$ and $\ell_1$ be two small arcs of 
$\gamma$ that contain $(0,\pp^-)$ and $(0,\pp^+)$, respectively. The splitting parameter $\mu$ in \eqref{eq:T1_quadra} measures the signed distance between  $S(\ell_0)$ and $\ell_1$. By \eqref{eq:unfoldquadra}, we can take $(\eps_1,\eps_2)=(\alpha-\alpha_0,\mu)$, where $\alpha_0$ is the value of $\alpha$ for the original quadratic tangency at $\eps=0$.

Adding the splitting parameter $\mu$ to \eqref{eq:scatter_new} and scaling $r$, we may write the scattering map $S$ near 
$(0,\pp^-)$ as
\begin{equation}\label{eq:1to2:scat0}
\tilde r = \mu + \thebetai r - (\tilde \varphi - \varphi^+)^2 +\dots, \qquad
\varphi-\varphi^- = c\; r +  \thebetai(\tilde \varphi - \varphi^+) +\dots,
\end{equation}
where $c$ is a constant, and the dots denote higher order terms. So, the image $S(\ell_0)$ is the parabola-like curve
\begin{equation}\label{8345}
\tilde r = \mu - (\tilde \varphi - \varphi^+)^2 +\dots.
\end{equation}

Recall that $\gamma$ is accumulated, from both sides, by a set of KAM-curves that lie in the cylinder $\A$.
Moreover, these curves are straightened, i.e.,
the coordinates are chosen such that  (\ref{8241}) is satisfied, so these KAM-curves are circles     of constant
$r$. In particular, there are two sets $\mathcal{G}^+$ and $\mathcal{G}^-$ of positive and, respectively, negative $r$ values  such that the KAM-curves are given by $r\in \mathcal{G}^\pm$ and $r=0$ is a Lebesgue density point for both of $\mathcal{G}^\pm$. 

For  $r^*\in \mathcal{G}^\pm$, we denote by $\ell^*$ a small arc in the KAM-curve $\gamma^*:\{r=r^*\}$ that is centered at $(r^*,\pp^-)$ and near $\ell_0$. By \eqref{eq:1to2:scat0}, the image $S(\ell^*)$ is the curve
\begin{equation}\label{8346}
\tilde r = \mu + \thebetai r^* - (\tilde \varphi - \varphi^+)^2 +\dots.
\end{equation}
In what follows, we consider the cases $\alpha>0$ and $\alpha<0$ separately. 

\noindent{\bf (1) The case of $\alpha>0$.} First note that we can choose  $r^+\in \mathcal{G}^+,r^-\in \mathcal{G}^-$ and take $\mu=r^++o(r^+)>0$ such that
\begin{itemize}[nosep]
\item $S(\ell_0)$ given by (\ref{8345}) has a tangency with the KAM-curve $\gamma^+:\{r=r^+\}$, and
\item $S(\ell^-)$ given by (\ref{8346}) with $r^*=r^-<0$ has a tangency with $\gamma$.
\end{itemize}
Such 
$r^+$ and $r^-$ exist arbitrarily close to $r=0$: by~(\ref{8346}) they must satisfy $r^+ + o(r^+) = - \thebetai r^- + o(r^-)$,
and this equation always has solutions near $0$, since $r=0$ is the density point for $\mathcal{G}^+$ and $\mathcal{G}^-$.
Obviously, these two tangencies are quadratic, and unfold generically as $\mu$ and $\alpha$ vary. By Lemma~\ref{lem:equitangency}, they correspond to two partially-hyperbolic quadratic heteroclinic tangencies: one is between $T_1(W^\u_{\loc}(\gamma))$ and $W^\s_{\loc}(\gamma^+)$, and the other is between $T_1(W^\u_{\loc}(\gamma^-))$ and $W^\s_{\loc}(\gamma)$, where $\gamma^-$ is the KAM-curve $\{r=r^-\}$.

Since $r^+$ and $\mu=r^+ +o(r^+)$ are both positive, one readily finds from~(\ref{8346}) with $r^*=r^+>0$ that $S(\ell^+)$ intersects $\gamma$ transversely, so Lemma~\ref{lem:equitangency} gives a transverse intersection of $T_1(W^\u_{\loc}(\gamma^+))$ with $W^\s_{\loc}(\gamma)$.
Similarly, by~(\ref{8345}), we have a transverse intersection of $S(\ell_0)$ with $\gamma^-$, which gives a transverse intersection of $T_1(W^\u_{\loc}(\gamma))$ with $W^\s_{\loc}(\gamma^-)$. 
Lemma~\ref{lem:incli} then implies $W^\s(\gamma)\to W^\s_{\loc}(\gamma^+)$  and
$W^\u(\gamma) \to W^\u_{\loc}(\gamma^-)$, where  ``$\to$''  means that the former manifold accumulates on the latter in the $C^1$ topology from both sides. It follows  that arbitrarily small changes in $\mu$ and $\alpha$ that unfold the heteroclinic tangencies of $T_1(W^\u_{\loc}(\gamma))$ with $W^\s_{\loc}(\gamma^+)$ and of $T_1(W^\u_{\loc}(\gamma^-))$ with $W^\s_{\loc}(\gamma)$ create the sought pair of homoclinic tangencies between $W^\u(\gamma)$ and $W^\s(\gamma)$
(we need the tangencies to be quadratic and to unfold generically -- this requires the reqularity $m^*\geq 2$ in Lemma~\ref{lem:incli}). This finishes the proof of the lemma for $\alpha>0$; see Figure~\ref{fig:1to2a} for an illustration.

\begin{figure}[!h]
\begin{center}
\includegraphics[scale=1.2]{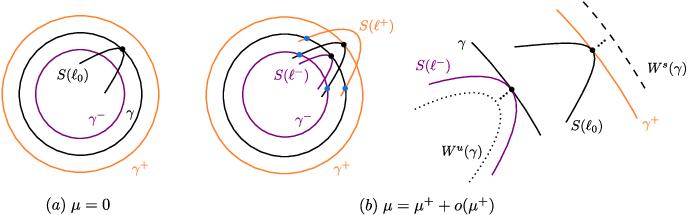}
\end{center}
\caption{(a) The original homoclinic tangency at $\mu=0$. (b) The creation of two heteroclinic tangencies (black dots) at $\mu>0$, the transverse intersections (blue dots), and the accumulations. Here the dotted curve represents the $(r,\pp)$-coordinate projection of a piece of $W^\u(\gamma)$ close to $T_1(W^\u_\loc(\gamma^-))$   and the dashed curve represents the $(r,\pp)$-coordinate projection of a piece of $W^\s(\gamma)$ close $W^\s_{\loc}(\gamma^+))$.}
\label{fig:1to2a}
\end{figure}

\noindent{\bf (2) The case of $\alpha>0$.} We note that at $\mu=0$ there exist $r^-\in \mathcal{G}^-$ and $r^+\in \mathcal{G}^+$ 
such that the image $S(\ell^-)$ of an arc of the KAM-curve $\gamma^-:\{r=r^-<0\}$ has a quadratic tangency to the
KAM-curve $\gamma^+:\{r=r^+>0\}$. Indeed, by (\ref{8346}) with $r^*=r^-$, this happens when 
$r^+= \thebetai r^- + o(r^-)$, 
and the existence of arbitrarily small solutions to this equation follows from the density of $\mathcal{G}^-$ and $ \mathcal{G}^+$ at $r=0$ (when $\alpha$ changes, the tangency splits with a non-zero velocity). Thus, by Lemma~\ref{lem:equitangency}, we have at $\mu=0$ a partially-hyperbolic quadratic heteroclinic tangency
between $T_1(W^\u_{\loc}(\gamma^-))$ and $W^\s_{\loc}(\gamma^+)$, in addition to the homoclinic tangency of 
the stable and unstable manifolds of $\gamma$.

We also have from (\ref{8345}) and (\ref{8346}) that 
\begin{itemize}[nosep]
\item $S(\ell_0)$ intersects    $\gamma^-$ transversely,
\item $S(\ell^*)$,  for every KAM-curve 
$\gamma^*:\{r=r^*\in\mathcal{G}^-\}$, intersects $\gamma$ transversely, and
\item $S(\ell^+)$ intersects transversely KAM-curves $\gamma^*$ with $r^* < \thebetai r^+ + o(r^+)$.
\end{itemize}
By Lemma~\ref{lem:equitangency}, this gives a transverse intersection between $T_1(W^\u_{\loc}(\gamma))$ and $W^\s_{\loc}(\gamma^-)$, a transverse intersection of $T_1(W^\u_{\loc}(\gamma^*))$ with $W^\s_{\loc}(\gamma)$, and a transverse intersection of $T_1(W^\u_{\loc}(\gamma^+))$ with $W^\s_{\loc}(\gamma^*)$.  Thus, by Lemma~\ref{lem:incli}, we have $W^\u(\gamma) \to W^\u(\gamma^-)$, $W^\s(\gamma) \to W^\s(\gamma^*)$ and $W^\s(\gamma^*) \to W^\s(\gamma^+)$, where the last two further imply $W^\s(\gamma) \to W^\s(\gamma^+)$. As a result, an arbitrarily small change in $\alpha$ that splits the heteroclinic tangency of $T_1(W^\u_{\loc}(\gamma^-))$ with $W^\s_{\loc}(\gamma^+)$ creates a homoclinic tangency between $W^\u(\gamma)$ and $W^\s(\gamma)$ in addition to the primary tangency between them (since we do not change $\mu$). This completes the proof of the lemma for $\alpha<0$; see Figure~\ref{fig:1to2b}. 

\begin{figure}[!h]
\begin{center} 
\includegraphics[scale=1.3]{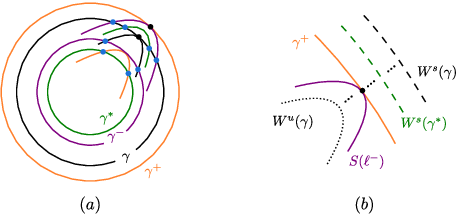}
\end{center}
\caption{(a) The homoclinic tangency, heteroclinic tangency (black dots) and transverse intersections (blue dots). (b) The accumulation, where the dotted and dashed curves represent  projections of pieces of the corresponding manifolds.}
\label{fig:1to2b}
\end{figure}

\subsubsection{Creating cubic tangencies: proof of Lemma~\ref{lem:cubic}}\label{sec:cubic}
Now let $\gamma$ have two partially-hyperbolic quadratic homoclinic tangencies.
Take two pairs of points $(M^-_i,M^+_i)$ $(i=1,2)$  in the two orbits of tangency, such that $M^-_i\in W^\u_{\loc}(\gamma)$, $M^+_i\in W^\s_{\loc}(\gamma)$, with $M^+_i=f^{n_i}(M^-_i)$ for some integers $n_i>0$. We define transition maps as $T_1:=f^{n_1}$ from a neighborhood of $M^-_1$ to a neighborhood of $M^+_1$ and $T_2:=f^{n_2}$ from a neighborhood of $M^-_2$ to a neighborhood of $M^+_2$. 
In the Fenichel coordinates, equations \eqref{eq:strait} hold, and one has $M^+_i=(0,\pp^+_i,x^+_i,0)$ and $M^-_i=(0,\pp^-_i,0,y^-_i)$.
By \eqref{eq:T1_quadra}, for any  family $\{f_{\eps}\}$ with $f_0=f$, the maps $T_1:(r,\pp,x,y)\mapsto (\tilde r,\tilde \pp,\tilde x,\tilde y)$ and $T_2:(\hat r,\hat \pp,\hat x,\hat y)\mapsto (\bar r,\bar \pp,\bar x,\bar y)$ take the following form (with a slight change in the notation):
\begin{equation}\label{eq:T1new}
\begin{aligned}
\tilde r &= \mu_1+{a} r +  \dd_1 (\tilde \varphi - \varphi^+_1)^2 +  a_{13} x +  a_{14}\tilde y+\dots, \\
\varphi -\varphi^-_1 &= {a}_{21} r +  {a} (\tilde \varphi - \varphi^+_1) +  a_{23} x +  a_{24}\tilde y+\dots, \\
\tilde x-x^+_1 &= {a}_{31} r +  a_{32} (\tilde \varphi - \varphi^+_1) +  a_{33} x +  a_{34}\tilde y+\dots, \\
y-y^-_1 &= {a}_{41} r +  a_{42} (\tilde \varphi - \varphi^+_1) +  a_{43} x +  a_{44}\tilde y+\dots,
\end{aligned}
\end{equation}
and,
\begin{equation}\label{eq:T2new}
\begin{aligned}
\bar r &= \mu_2+{b} \hat r + \dd_2 (\bar \varphi - \varphi^+_2)^2 +  b_{13} \hat x +  b_{14}\bar y+\dots, \\
\hat \varphi -\varphi^-_2 &= {b}_{21}\hat  r +  {b} (\bar \varphi - \varphi^+_2) +  b_{23} \hat x +  b_{24}\bar y+\dots, \\
\bar x-x^+_2 &= {b}_{31} \hat r +  b_{32} (\bar \varphi - \varphi^+_2) +  b_{33} \hat x +  b_{34}\bar y+\dots, \\
\hat y-y^-_2 &= {b}_{41} \hat r +  b_{42} (\bar \varphi - \varphi^+_2) +  b_{43}\hat  x +  b_{44}\bar y+\dots,
\end{aligned}
\end{equation}
where all the coefficients are functions of $\eps$; note that $\mu_{1,2}=0$ at $\eps=0$, and ${a},{b},\dd_{1,2}$ are non-zero.
The genericity condition \eqref{eq:unfold2quadra} allows us to use $\mu_1$ and $\mu_2$ as parameters.

Let us consider the map $T_2\circ T^k_0 \circ T_1:(r,\varphi,x,y)\mapsto (\bar r,\bar \varphi,\bar x,\bar y)$ from a neighborhood of $M^-_1$ to a neighborhood of $M^+_2$, where $T^k_0$ is given by \eqref{eq:derisyst2:1}.   Similarly to the proof of Lemma~\ref{lem:a11}, we look for homoclinic intersections of $W^\u(\gamma)=\{r=x=0\}$ and $W^\s_{\loc}(\gamma)=\{r=y=0\}$ as solutions to the equation $T_2\circ T^k_0 \circ T_1(0,\varphi,0,y)=(0,\bar\varphi,\bar x,0)$. This is done by  the cross-form $(r,\bar\varphi,x,\bar y)\mapsto(\bar r, \varphi, \bar x, y)$ of the map. We will further show that for some parameter values there is a solution which satisfies
\begin{equation}\label{cubic:derivative}
\dfrac{\partial\bar r}{\partial\bar\varphi}=0,\qquad 
\dfrac{\partial^2\bar r}{\partial\bar\varphi^2}=0, \qquad 
\dfrac{\partial^3\bar r}{\partial\bar\varphi^3}\neq 0,
\end{equation}
and hence  corresponds to a partially-hyperbolic cubic homoclinic tangency of $\gamma$.

\noindent{\bf (1) Rescaling of parameters and variables.}
We set
\begin{equation}\label{cubic:trans1}
\begin{array}{l}
\mu_i = k^{-2}\mu'_i, \qquad (r,\tilde r,\hat r,\bar r)=k^{-2}(R,\tilde R,\hat R,\bar R),\\
(\varphi - \varphi^-_1,
\tilde \varphi - \varphi^+_1,
\hat \varphi - \varphi^-_2,
\bar \varphi - \varphi^+_2)
 = k^{-1}(\Phi,\tilde \Phi,\hat \Phi,\bar \Phi),\\
(x,\hat x,\tilde y,\bar y)=k^{-3}(X,\hat X,\tilde Y,\bar Y),\\
(\tilde x - x^+_1,\bar x - x^+_2,\hat y-y^-_2,y-y^-_1)=k^{-1}(\tilde X,\bar X, Y,\hat Y).
\end{array}
\end{equation}
In the new coordinates, the two transition maps \eqref{eq:T1new} and \eqref{eq:T2new} satisfy the following relations:  

\noindent$T_1:(R,\Phi,X,Y)\mapsto (\tilde R, \tilde \Phi,\tilde X,\tilde Y):$
\begin{equation*}
\begin{aligned}
\tilde R &= \mu'_1+{a} R+  \dd_1 \tilde \Phi^2 + O(k^{-1}), \qquad
\Phi=   a \tilde{\Phi} + O(k^{-1}), \\
\tilde X &= a_{32} \tilde{\Phi} + O(k^{-1}), \qquad
Y = a_{42} \tilde{\Phi} + O(k^{-1}),
\end{aligned}
\end{equation*}
$T_2:(\hat R, \hat \Phi,\hat X,\hat Y)\mapsto (\bar R, \bar \Phi,\bar X,\bar Y):$
\begin{equation*}
\begin{aligned}
\bar R &= \mu'_2+{b} \hat R +  \dd_2 \bar \Phi^2 + O(k^{-1}), \qquad
\hat \Phi= b \bar \Phi + O(k^{-1}), \\
\bar X &= b_{32} \bar \Phi + O(k^{-1}), \qquad
\hat Y = b_{42} \bar \Phi + O(k^{-1}),
\end{aligned}
\end{equation*}
where the  $O(k^{-1})$ terms are $C^3$ functions of $(R,\tilde \Phi, X,\tilde Y)$ in the first set of equations, and  of $(\hat R,\bar \Phi, \hat X,\bar Y)$ in the second set, and, moreover, the derivatives up to third order of these terms are also estimated as $O(k^{-1})$. 

Since we  consider finite values of $\tilde R$ and $\hat{R}$,  the scaling \eqref{cubic:trans1} implies that the condition $k=o(r_0^{1-m})$ of  Lemma~\ref{lem:derisyst2} is satisfied when $m\geq 2$.  Thus, by taking $(r_0,\pp_0,x_0,y_0)=(\tilde r,\tilde \pp, \tilde x,\tilde y)$ and $(r_k,\pp_k,x_k,y_k)=(\hat r,\hat \pp, \hat x,\hat y)$ in \eqref{eq:derisyst2:4} and setting  $m\geq 3$ and $m'-m\geq 3$ in \eqref{eq:derisyst2:1}, the map  $T^k_0:(\tilde R, \tilde \Phi,\tilde X,\tilde Y)\mapsto (\hat R, \hat \Phi,\hat X,\hat Y)$ satisfies 
\begin{equation*}
\begin{aligned}
\hat R &= \tilde R + O(k^{-1}) ,\qquad
\hat \Phi= A(k)+ \tilde{\Phi} + \tilde R+ O(k^{-1}) ,\\
\hat X &=  O(k^{-1}),\qquad
\tilde Y = O(k^{-1}),
\end{aligned}
\end{equation*}
 for all sufficiently large $k$, where
\begin{equation}\label{Ak2}
A(k)= k((\varphi^+_1 -\varphi^-_2+k\rho_0)\bModa 1),
\end{equation}
and the $O(\cdot)$ terms are functions of $(\tilde R, \tilde \Phi,\tilde X,\hat Y)$, with their  joint derivatives with respect to variables and parameters  up to the third order estimated as $O(k^{-1})$.

We consider the maps $T_1$ and $T_2$ in small neighborhoods of $M_1^-$ and, respectively, $M_2^-$, which correspond to bounded values of the rescaled variables. Note that  there are infinitely many $k$ values for which $A(k)$ in \eqref{Ak2} stays uniformly bounded, so that the map 
$T_2\circ T^k_0\circ T_1$  is well-defined (i.e., $T^k_0$ takes the image by $T_1$ of the small neighborhood of $M_1^-$ to the domain of $T_2$). Indeed, we can always take arbitrarily large co-prime integers $p$ and $q>0$ such that 
$\rho = p/q + O(q^{-2})$. Then, if $s$ is an integer such that
$|\varphi^+_1 - \varphi^-_2 - s/q| <q^{-1}$, we take $k\in \{q,\dots, 2q-1\}$ such that
$k p + s = 0 \bmod q$, whic gives 
$$\varphi^+_2 - \varphi^-_1+k \rho  \bModa 1 = O(k^{-1}).$$

\noindent{\bf (2) Cross-form for $T_2\circ T^k_0\circ T_1$.} 
Let us further simplify the above formulas. Since $\dd_1,\dd_2,{b}$ are non-zero, we can take
\begin{equation}\label{cubic:trans2}
\begin{array}{l}
\mu''_1 = \dd_1 \mu'_1 + \beta_1A(k),\quad 
\mu''_2 = \dfrac{{b}\dd_1^2}{\dd_2}\mu'_2 -\dfrac{{b}^3\dd_1}{\dd_2}A(k),\\[10pt]
\tilde R=\dfrac{1}{\dd_1}\tilde R^{\mathrm{new}} - A(k),\quad  
\hat R=\dfrac{1}{\dd_1}\hat R^{\mathrm{new}} - A(k), \quad
\bar R=\dfrac{\dd_2}{\dd_1^2{b}^2}\bar R^{\mathrm{new}}, \\[10pt]
(\tilde \Phi,\hat \Phi)=\dfrac{1}{\dd_1}(\tilde \Phi,\hat \Phi)^{\mathrm{new}},\quad
\bar \Phi=\dfrac{1}{\dd_1{b}}\bar \Phi^{\mathrm{new}},
\end{array}
\end{equation}
The three maps now assume the following forms:

\noindent$T_1:(R,\Phi,X,Y)\mapsto (\tilde R, \tilde \Phi,\tilde X,\tilde Y):$
\begin{equation}\label{cubic:1}
\begin{aligned}
\tilde R &= \mu''_1+a\beta_1 R+ \tilde \Phi^2 + O(k^{-1}), \qquad
\Phi= \dfrac{{a}}{\dd_1}    \tilde{\Phi} + O(k^{-1}), \\
\tilde X &= \dfrac{a_{32}}{\dd_1}\tilde{\Phi} + O(k^{-1}), \qquad
Y = \dfrac{a_{42}}{\dd_1}  \tilde{\Phi} + O(k^{-1}),
\end{aligned}
\end{equation}
where the right-hand sides are functions of $(R,\tilde \Phi,X,\tilde{Y})$;\\

\noindent $T^k_0:(\tilde R, \tilde \Phi,\tilde X,\tilde Y)\mapsto (\hat R, \hat \Phi,\hat X,\hat Y):$
\begin{equation}\label{cubic:2}
\begin{aligned}
\hat R &= \tilde R + O(k^{-1}) ,\qquad
\hat\Phi=  \tilde{\Phi} +\tilde R+ O(k^{-1}),\\
\hat X &=  O(k^{-1}),\qquad
\tilde Y = O(k^{-1}),  
\end{aligned}
\end{equation}
where the right-hand sides are functions of $(\tilde R,\tilde \Phi,\tilde X,\hat{Y})$;\\

\noindent$T_2:(\hat R, \hat \Phi,\hat X,\hat Y)\mapsto (\bar R, \bar \Phi,\bar X,\bar Y):$
\begin{equation}\label{cubic:3}
\begin{aligned}
\bar R &= \mu''_2+B \hat R + \bar \Phi^2 + O(k^{-1}), \qquad
\hat \Phi=  \bar \Phi + O(k^{-1}), \\
\bar X &= \dfrac{b_{32} }{\dd_1{b}}\bar \Phi + O(k^{-1}), \qquad
\hat Y =\dfrac{ b_{42}}{\dd_1{b}} \bar \Phi + O(k^{-1}), 
\end{aligned}
\end{equation}
where the right-hand sides are functions of $(\hat R,\bar \Phi,\hat X,\bar{Y})$, with $B={{b}^3\dd_1}/{\dd_2}$.

Combing  \eqref{cubic:1}, \eqref{cubic:2} and \eqref{cubic:3}, we see that
one has
\begin{equation}\label{arrowmap}
(R,\Phi,X,Y)
\xrightarrow{T_1} (\tilde R,\tilde \Phi,\tilde X,\tilde Y)
\xrightarrow{T^k_0}  (\hat R,\hat \Phi,\hat X,\hat Y)
\xrightarrow{T_2} (\bar R,\bar \Phi,\bar X,\bar Y),
\end{equation}
if and only if 
\begin{equation}\label{cubic:comp}
\begin{aligned}
\bar R &= \mu''_2 + B( \mu''_1+a\beta_1R+\tilde \Phi^2)+(\mu''_1+a\beta_1R+\tilde\Phi+\tilde\Phi^2)^2+O(k^{-1}),\\
\Phi &= \dfrac{a}{\dd_1} \tilde{\Phi} + O(k^{-1}),\\
\bar X &= \dfrac{b_{32} }{\dd_1{b}}(\mu''_1+R+\tilde\Phi+\tilde\Phi^2)+O(k^{-1}),\qquad
Y=\dfrac{a_{42}}{\dd_1} \tilde{\Phi} + O(k^{-1}),
\end{aligned}
\end{equation}
and
\begin{equation}\label{cubic:eta}
\bar\Phi = \mu''_1+a\beta_1 R+\tilde\Phi+\tilde\Phi^2+O(k^{-1}),
\end{equation}
where the  $O$ terms are functions of $(R,\tilde \Phi,X,\bar Y)$ such that their  joint derivatives with respect to variables and parameters up to the third order are estimated as $O(k^{-1})$.

\noindent{\bf (3) Creation of a partially-hyperbolic cubic  homoclinic tangency.} 
After the coordinate transformations \eqref{cubic:trans1} and \eqref{cubic:trans2}, we have $W^\u_\loc(\gamma)=\{ R=0,X=0\}$ and $W^\s_\loc(\gamma)=\{\bar R=0,\bar Y=0\}$. Thus, a homoclinic intersection between $T_2\circ T^k_0\circ T_1(W^\u_\loc(\gamma))$ and $W^\s_\loc(\gamma)$ corresponds to 
  a solution  of \eqref{cubic:comp} with $R,\bar R,X,\bar Y$ vanishing. 
  To find such a solution, it suffices to set $R,\bar R,X,\bar Y$ to 0  in the first equation of \eqref{cubic:comp} and solve for $\tilde\Phi$, since other coordinates can be obtained once $\tilde\Phi$ is found. This gives
\begin{equation*}
0= \mu''_2 + B( \mu''_1+\tilde \Phi^2)+( \mu''_1+\tilde\Phi+\tilde\Phi^2)^2+O(k^{-1}),
\end{equation*}
or,
\begin{equation}\label{eq:cubic11}
0=\mu''_2 + B \mu''_1+(\mu''_1)^2+2\mu''_1\tilde{\Phi} + (1+B+2\mu''_1)\tilde\Phi^2+2\tilde\Phi^3+\tilde\Phi^4+O(k^{-1}),
\end{equation}
where the term $O(k^{-1})$ is a function of $\tilde{\Phi}$. 

Obviously, if $\p\tilde \Phi/\p\Phi\neq 0$, then the conditions for the homoclinic intersection to be a cubic tangency as stated in \eqref{cubic:derivative}  for the non-scaled coordinates are equivalent to $\partial\bar R/\partial \tilde \Phi=0, \partial^2\bar R/\partial \tilde \Phi^2=0,\partial^3\bar R/\partial \tilde \Phi^3\neq 0$. By~\eqref{cubic:comp}, they read
\begin{align}
0&=\mu''_1 +(1+B+2\mu''_1)\tilde\Phi+3\tilde\Phi^2+2\tilde\Phi^3+O(k^{-1}),\label{eq:cubic12}\\
0&=1+B+2\mu''_1+6\tilde\Phi+6\tilde\Phi^2+O(k^{-1}),\label{eq:cubic12'}\\
0&\neq 1+2\tilde\Phi +O(k^{-1}).\label{eq:cubic12''}
\end{align} 
Thus, a cubic tangency corresponds to  a solution $(\tilde{\Phi},\mu_1'',\mu_2'')$ of the system consisting of  \eqref{eq:cubic11} and \eqref{eq:cubic12} that also satisfies $\p\tilde \Phi/\p\Phi\neq 0$.

Expressing $\mu_1''$ and $\mu''_2$ as functions of $\tilde{\Phi}$ through \eqref{eq:cubic11} and   \eqref{eq:cubic12}, and substituting them into \eqref{eq:cubic12'}, reduces the problem to 
\begin{equation}\label{cubic:root}
(2\tilde\Phi+1)^3=-B+O(k^{-1}).
\end{equation}
Obviously, it always has a simple real solution 
 for every sufficiently large $k$. 
Since $B={b}^3\dd_1/\dd_2\neq 0$ by assumption, the  inequality  \eqref{eq:cubic12''} holds automatically. 
Thus, we have found a cubic homoclinic tangency of $\gamma$; it is  partially hyperbolic since it is created in a small neighborhood of  partially hyperbolic orbits. 

\noindent{\bf (4) Genericity of the unfolding.}
To finish the proof of the lemma, let us verify that changing $\mu_1$ and $\mu_2$  unfolds the   cubic tangency  in the sense of \eqref{eq:unfoldcubic}. Note that the parameters $\mu$ and $\nu$ in \eqref{eq:unfoldcubic} are just the values of  $\bar r$ and $\partial  \bar r/\partial \bar \pp$, respectively, evaluated at the cubic tangency point. Then it suffices to show that the  matrix of derivatives   of the right-hand sides of \eqref{eq:cubic11} and \eqref{eq:cubic12} with respect to $\mu''_{1,2}$ is invertible. This is immediate since the matrix is
$$\begin{pmatrix}
B+2\tilde{\Phi}+2\tilde{\Phi}^2+2\mu''_1 &1\\
1+2\tilde{\Phi}&0
\end{pmatrix},$$
and  $\tilde \Phi\neq -1/2$ at the tangency point.

\subsection{Persistent intersections: proof of Proposition~\ref{prop:persis}}\label{sec:unfoldblen}

Consider any $f\in \symp^s(\mathcal{M})$ having a symplectic blender $\Lambda$ connected to a whiskered KAM-torus $\gamma$, where $s\geq 4$ is sufficiently large so that $\min\{m',m'-m\}\geq 2$ in \eqref{eq:kam_new}. Let $L^\u$ be an $(N-1)$-dimensional manifold close to a strong-unstable leaf $\ell^\uu$ of $W^\u_\loc(\gamma)$ and  $L^\s$ be an $(N-1)$-dimensional manifold  close to a strong-stable leaf $\ell^\ss$ of $W^\s_\loc(\gamma)$ such that $W^\s(\Lambda)\cap L^\u \neq \emptyset$ and $W^\u(\Lambda)\cap L^\s\neq \emptyset$ by the blender property of Definition~\ref{defi:sympblen_conn}. 
We first show that each of these two intersections can be unfolded within one-parameter families  (Lemmas~\ref{lem:unfold1} and~\ref{lem:unfold2}). After that, we  consider their simultaneous unfolding in Proposition~\ref{prop:unfoldble}, proving the first claim of Proposition~\ref{prop:persis}. The second claim then follows from a general bifurcation result for non-transverse intersections with the invariant manifolds of a  partially-hyperbolic saddle (see Lemma~\ref{lem:model}).

Consider any family $\{f_\eps \}\subset \symp^s(\mathcal{M})$ with $f_0=f$, where $f_\eps$ is jointly $C^s$ with respect to variables and parameters. Let
$\{L^{\u}_\eps\}$ and $\{L^{\s}_\eps\}$ be two $C^1$ families  of $(N-1)$-dimensional $C^1$ manifolds with $L^{\u}_0=L^{\u}$ and $L^{\s}_0=L^{\s}$.  Recall that $\Lambda$ and $\gamma$ persist for all small $\eps$ and we omit their dependence on $\eps$ for simplicity.
The genericity condition of Proposition~\ref{prop:persis} will be formulated with
the signed distance $\Delta^\u$ between $L^\u_\eps$ and  $W^\s_\loc(\gamma)$, and the signed distance $\Delta^\s$ between $L^\s_\eps$ and  $W^\u_\loc(\gamma)$. These distances have simple expressions in the   Fenichel coordinates. In these coordinates (see Section~\ref{sec:innermap}), we have  for all small $\eps$  that
\begin{equation}\label{eq:prelimiaries}
\begin{aligned}
&\ell^\uu=\{r=0,\pp=\pp^\uu,x=0,y\in D^\uu\},\\
&\ell^\ss=\{r=0,\pp=\pp^\ss,x\in D^\ss,y=0\},\\
&W^\s_\loc(\gamma)=\{r=0,\pp\in \mathbb{S}^1, x\in D^\ss, y=0\},\\
&W^\u_\loc(\gamma)=\{r=0, \pp\in \mathbb{S}^1,x=0, y\in D^\uu\},\\
&W^\s_\loc(\A)=\{(r,\pp)\in\A, x\in D^\ss, y=0\},\\
&W^\u_\loc(\A)=\{(r,\pp)\in\A, x=0, y\in D^\uu\},
\end{aligned}
\end{equation}
for some $\pp^\uu,\pp^\ss\in \mathbb{S}^1$ and some closed discs $D^\uu,D^\ss\subset\mathbb{R}^{N-1}$ containing $0$.  Since $\ell^\uu$ is transverse to $W^\s_\loc(\A)$ by the partial hyperbolicity, up to decreasing $\delta'$,  $L^\u_\eps$ intersects $W^\s_\loc(\A)$ transversely at some point $Q^\u_\eps$ for every small $\eps$. It is then  clear from \eqref{eq:prelimiaries} that $\Delta^\u$ is equal to the $r$-coordinate of $Q^\u_\eps$.  Similarly, $L^\s_\eps$ intersects $W^\u_\loc(\A)$ transversely at some point $Q^\s_\eps$, whose $r$-coordinate gives $\Delta^\s$.

\subsubsection{First claim of Proposition~\ref{prop:persis}}

First let $\{f_\eps\}$, $\{L^{\u}_\eps\}$ and $\{L^{\s}_\eps\}$ be one-parameter families. Let $L^\u$ and $L^\s$ be $\delta'$-$C^1$-close to $\ell^\ss$ and $\ell^\uu$ for some $\delta'\in(0,\delta)$, where $\delta$ is given by the blender property of Definition~\ref{defi:sympblen_conn}. By definition, there exist $P_1,P_2\in \Lambda$ such that $L^\u\cap W^\s(P_1)\neq \emptyset$ and $L^\s\cap W^\s(P_2)\neq \emptyset$. 

\begin{lem}\label{lem:unfold1}
The following hold for all sufficiently small $\delta'$:
\begin{itemize}[nosep]
\item if ${d\Delta^\u}/{d\eps}|_{\eps=0} \neq 0$, then  the intersection $L^\u\cap W^\s(P_1)$ unfolds with a non-zero velocity as $\eps$ varies;
\item if  ${d\Delta^\s}/{d\eps}|_{\eps=0} \neq 0$, then  the intersection $L^\s\cap W^\s(P_2)$ unfolds with a non-zero velocity as $\eps$ varies.
\end{itemize}
\end{lem}

We postpone the proof of this lemma to Section~\ref{sec:prooflemunfold}.

\begin{lem}\label{lem:unfold2}
The following hold for all sufficiently small $\delta'$ and  any point $P\in \Lambda$:
\begin{itemize}[nosep]
\item   if ${d\Delta^\u}/{d\eps}|_{\eps=0} \neq 0$, then there exist $\eps\to 0$  for which $L^\u_{\eps}\cap W^\s(P_\eps)\neq \emptyset$ and the intersection unfolds with a non-zero velocity as $\eps$ varies;
\item   if  ${d\Delta^\s}/{d\eps}|_{\eps=0} \neq 0$, then there exist $\eps\to 0$  for which $L^\s_{\eps}\cap W^\u(P_\eps)\neq \emptyset$ and the intersection unfolds with a non-zero velocity as $\eps$ varies.
\end{itemize}
\end{lem}
\begin{proof}
First note that, since  $L^\u$ is close to $\ell^\uu$, the partial hyperbolicity  of Definition~\ref{defi:sympblen_conn} implies that the tangent space at each point of $L^\u$ is transverse to $E^\ss\oplus E^\mathrm{ws} \oplus E^\mathrm{wu}$. Since $T W^\s(P')\subset E^\ss\oplus E^\mathrm{ws}$, it  follows that the  unfolding of the intersection $L^\u\cap W^\s(P')$ given by  Lemma~\ref{lem:unfold1} must have non-zero velocity in the direction of $E^\mathrm{wu}$.  This immediately gives the first claim of the lemma  since $W^\s(P)$ accumulates on $W^\s(P')$ in the $C^1$-topology  ($\Lambda$ is a hyperbolic basic set). The second claim follows from a completely parallel argument.
\end{proof}

Now let  the families $\{f_\eps\},\{L^\u_\eps\},\{L^\s_\eps\}$ have at least two parameters. Write $\eps=(\mu,\nu,\eps')$, where $\mu,\nu\in \mathbb{R}$ are any two components of $\eps$ and $\eps'$ denotes the remaining components. We have the following detailed version of the first claim of Proposition~\ref{prop:persis}:
\begin{prop}\label{prop:unfoldble}
The following holds for all sufficiently small $\delta'$: if 
\begin{equation}\label{eq:gcpersis}
\det \left.\frac{\p(\Delta^\u,\Delta^\s)}{\p (\mu,\nu)}\right|_{\eps=0}\neq 0,
\end{equation}
then one can find a neighborhood $\mathcal{E}$ of $\eps=0$ such that, for  any points $P,P'\in\Lambda$,  there exists a dense subset $ \hat{\mathcal{E}}_{P,P'}\subset \mathcal{E}$  such that, for every $\eps\in \hat{\mathcal{E}}_{P,P'}$,
$$
L^\u_{\eps}\cap W^\s(P_{\eps})\neq \emptyset,\qquad
L^\s_{\eps}\cap W^\u(P'_{\eps})\neq \emptyset.
$$
Moreover, these two intersections unfold independently with respect to $\mu$ and $\nu$, that is, for every $\eps_*\in \hat{\mathcal{E}}_{P,P'}$, there exist
\begin{itemize}[nosep]
\item a smooth function $\mu=\hat\mu_*(\nu,\eps')$ defined in a neighborhood of $(\nu_*,\eps'_*)$   such that $\mu_*=\hat\mu_*(\nu_*,\eps'_*)$ and  $L^\u_{\eps}$ intersects $W^\s(P_{\eps})$  if and only if $\eps = (\hat\mu_*(\nu,\eps'),\nu,\eps')$, and
\item a smooth function $\nu=\hat\nu_*(\mu,\eps')$ defined in a neighborhood of $(\mu_*,\eps'_*)$   such that $\nu_*=\hat\nu_*(\mu_*,\eps'_*)$ and  $L^\s_{\eps}$ intersects $W^\u_{\eps}(P'_\eps)$  if and only if  $\eps = (\mu,\hat\nu_*(\mu,\eps'),\eps')$.
\end{itemize}
\end{prop}

\begin{proof}
Let ${\mathcal{E}}$ be the set of $\eps$ values for which $\det\frac{\p(\Delta^\u,\Delta^\s)}{\p (\mu,\nu)}\neq 0$, and $L^\u_\eps$ and $L^\s_\eps$ are $\delta'$-close to $\ell^\uu$ and $\ell^\ss$, respectively.  Take any small $\eps_0=(\mu_0,\nu_0,\eps'_0)\in {\mathcal{E}}$. To prove the proposition, it suffices to  find $\eps_*$ that is arbitrarily close to $\eps_0$ and associated with the required functions.

By assumption, we can take $\mu=\Delta^\u$ and $\nu=\Delta^\s$. Applying the first claim of Lemma~\ref{lem:unfold2} to the family $\{f^1_{\mu}\}:=\{f_{\mu+\mu_0,\nu_0,\eps'_0}\}$, we find by the implicit function theorem a  set $\mathcal{E}^\u$ of $\mu$ values converging to 0, and, for every $\mu_*\in\mathcal{E}^\u$, a smooth function $\mu=\hat\mu_*(\nu,\eps')$  defined near $(\nu_0,\eps'_0)$ 
such that $\mu_*=\hat\mu_*(\nu_0,\eps'_0)$ and $L^\u_{\eps}\cap W^\s(P_\eps)\neq \emptyset$ for $\eps=(\hat\mu_*(\nu,\eps'),\nu,\eps')$.

Next consider the family $\{f^2_{\nu}\}:=\{f_{\hat\mu_*(\nu+\nu_0,\eps'_0),\nu+\nu_0,\eps'_0}\}$ for the function $\hat\mu_*$  associated with some $\mu_*\in \mathcal{E}^\u$. Since $d \Delta^\s/d\mu= d \nu/d\mu=  0$, we can apply the second claim of Lemma~\ref{lem:unfold2} to this family, yields a  set $\mathcal{E}^\s$ of $\nu$ values converging to 0, and, for every $\nu_*\in\mathcal{E}^\s$, a smooth function $\nu=\hat\nu_*(\mu,\eps')$ defined near $(\hat\mu_*(\nu_*,\eps'_0),\eps'_0)$ with $\nu_*=\hat\nu_*(\hat\mu_*(\nu_*,\eps'_0),\eps'_0)$ such that $L^\s_{\eps}\cap W^\u(P'_\eps)\neq \emptyset$ for $\eps=(\mu,\hat\nu_*(\mu,\eps'),\eps')$ and this intersection disappears as $\eps$ varies. The proposition follows by taking $\eps_*=(\hat\mu_*(\nu_*,\eps'_0),\nu_*,\eps'_0)$ with sufficiently small $\mu_*$ and $\nu_*$.
\end{proof}

\subsubsection{Proof of Lemma~\ref{lem:unfold1}}\label{sec:prooflemunfold}
Recall that $Q^\u_\eps$ and $Q^\s_\eps$ are the intersection points of $L^\u_\eps\cap W^\s_\loc(\A)$ and $L^\s_\eps\cap W^\u_\loc(\A)$, respectively.
Let $r^\u(\eps)$ and $r^\s(\eps)$ be the $r$-coordinates of  $Q^\u_\eps$ and $Q^\s_\eps$. 
Denote $r^{\u/\s}_*:=r^{\u/\s}(0)$.
The two assumptions  of Lemma~\ref{lem:unfold1}  imply that we can take
$$\eps=r^\u-r^\u_* 
\quad\mbox{and}\quad
\eps=r^\s-r^\s_*,$$
respectively. 
 Due to the symmetry of the problem, it suffices to prove the first claim of Lemma~\ref{lem:unfold2}. 

\noindent\textbf{(1) Iterations of $L^{\u}_\eps$.}  Recall that the $C^1$-closeness between the manifolds and leaves are understood as the $C^1$-closeness of the corresponding embeddings. Thus, since $\eps=r^\u-r^\u_*$, the defining function of $L^\u_\eps$ has the following form:
\begin{equation}\label{eq:L^ufam}
\begin{aligned}
r_0=\eps+r^\u_*+ \ell^\u_1(y_0,\eps),\qquad
\pp_0=\ell^\u_2(y_0,\eps),\qquad
x_0= \ell^\u_3(y_0,\eps),
\end{aligned}
\end{equation}
where  $|r^\u_*|<\delta$  and  $\ell_i^{\u}$   $(i=1,2,3)$ are $C^1$ functions defined on $D^{\uu}$ and satisfying for all sufficiently small $\eps$ that $\ell_1^{\u}(0,\eps)\equiv 0$ and $\|\ell_i^{\u}\|_{C^1}<2\delta$. 

Let us use  Lemma~\ref{lem:derisyst2}  to find a formula for $T_0^{k'}(L^\u_\eps)$ with ${k'}=o(\delta^{-1})$. Recall that we assumed $\min\{m'-m,m\}\geq 2$.
Combining \eqref{eq:L^ufam} with the last equation in \eqref{eq:derisyst2:4}, yields $y_0=\hat h_4(y_{k'})$ for some function $\hat h_4$ with $\|\hat h_4\|_{C^{1}}=O(\lambda^{k'})$. Substituting this together with~\eqref{eq:L^ufam} into the remaining equations in  \eqref{eq:derisyst2:4} and using \eqref{eq:derisyst2:1}, we readily find the defining function of $T_0^{k'}(L^\u_\eps)$ as (after renaming $r_k,\pp_k,x_k,y_k$ to $r,\pp,x,y$)
\begin{equation}\label{eq:T^kW^u}
r= \hat \ell^\u_1(y,\eps),\qquad
\pp= {k'}(\rho+r^\u_*) + {k'} \eps + \hat \ell^\u_2(y,\eps),\qquad
x= \hat \ell^\u_3(y,\eps),
\end{equation}
where the functions $\hat \ell^\u_i$ are defined for $y\in D^\uu$ and satisfy
\begin{equation}\label{eq:T^kW^uDeri}
\begin{aligned}
&\hat \ell^\u_{i}=O(\delta),\qquad
 \dfrac{\p\hat \ell^\u_{i}}{\p y}=O(\lambda^\frac{{k'}}{3}),\qquad 
 \dfrac{\p\hat \ell^\u_{1,3}}{\p \eps}=O(\delta) ,\qquad
 \dfrac{\p\hat \ell^\u_{2}}{\p \eps}=O({k'} \delta),
\end{aligned}
\end{equation}
for all sufficiently small $\eps$.  

\noindent\textbf{(2) Unfolding of the intersections.} By the blender property of Definition~\ref{defi:sympblen_conn}, up to an  integer to $k'$, there exists $P_1\in \Lambda\cap V^\u$  such that $T_0^{k'}(L^\u)\cap W^\s_\loc(P_1)\neq \emptyset.$
In what follows, we compute the distance, denoted by $\Delta_\eps$, between $T_0^{k'}(L^\u_\eps)$ and $ W^\s_\loc(P_{1,\eps})$ near the intersection point for small $\eps$.

First note from \eqref{eq:prelimiaries} and the local transversality  of Definition~\ref{defi:sympblen_conn} that, for all small $\eps$, the  local manifold $ W^\s_\loc(P_{1,\eps})$ is  given by $ \pp=w^\s_1(r,x,\eps)$ and $y=w^\s_2(r,x,\eps)$
for some functions $w^\s_{1,2}$  with uniformly bounded first derivatives. We straighten it by setting
$$
\pp^\new =\pp - w^\s_1(r,x,\eps),\qquad
y^\new = y - w^\s_2(r,x,\eps).
$$
In the new coordinates, we have $W^\s_\loc(P_{1,\eps})=\{\pp=0,y=0\}$. The iterates $T_0^{k'}(L^\u_\eps)$ still takes the  form  \eqref{eq:T^kW^u}, where the functions $\hat \ell^\u_i$ are different but have the same estimates as in \eqref{eq:T^kW^uDeri}. It is then obvious that
$$
\Delta_\eps ={k'}(\rho+r^\u_*) + {k'}\eps + \hat \ell^\u_2(0,\eps),
$$
Since $\p\hat \ell^\u_2/\p \eps$ is uniformly bounded for all choices of the local manifolds in $V^\u$, one has  that $\p \Delta_\eps/\p \eps $ is arbitrarily large with sufficiently large $k'$. In particular, the intersection unfolds with a non-zero velocity as $\eps$ varies. The proof of the first claim of Lemma~\ref{lem:unfold1} is now complete.

\subsubsection{Second claim of Proposition~\ref{prop:persis}}
We prove a general statement that implies the second claim. Let $g$ be a diffeomorphism of a $d$-dimensional manifold for some integer $d\geq 4$. Let $P$ be a periodic point of $g$ with multipliers $\lambda_1,\dots,\lambda_{i^\s}$ and $\gamma_1,\dots,\gamma_{i^\u}$, where $i^\s+j^\u=d$, satisfying $\lambda_1,\gamma_1\in\mathbb{R}$ and
$$
|\lambda_{i^\s}|\leq \dots \leq |\lambda_2|<|\lambda_1|<1< |\gamma_1|<|\gamma_2|\leq \dots \leq |\gamma_{i^\u}|.
$$
Thus, the orbit $\mathcal{O}(P)$ is a hyperbolic invariant set with a partially-hyperbolic structure.
There is an invariant splitting $E^{\ss}\oplus E^{\mathrm{ws}}\oplus E^{\mathrm{wu}}\oplus E^{\uu}$ in a small neighborhood $V$ of $\mathcal{O}(P)$, where $\dim E^{\mathrm{ss}} =d^\s-1$, $\dim E^{\mathrm{uu}}=d^\u-1$ and $\dim E^{\mathrm{ws}} =\dim E^{\mathrm{wu}}=1$. Let $\mathcal{C}^\ss$ be the strong-stable cone field containing $E^\ss$ and $\mathcal{C}^\uu$ be the strong-unstable cone field containing $E^\uu$. 

We consider a two-parameter family $\{g_{\mu,\nu}\}$ such that $g_0=g$ and $g_{\mu,\nu}$ is jointly $C^3$ with respect to variables and parameters.  Let $\{L^\s_{\mu,\nu} \subset V\}$ be a family of   $(d^\s-1)$-dimensional $C^1$ manifolds which are tangent to $\mathcal{C}^\ss$, and $\{L^\u_{\mu,\nu} \subset V\}$ be a family of   $(d^\u-1)$-dimensional $C^1$ manifolds which are tangent to $\mathcal{C}^\uu$. Suppose   $W^\s(P)\cap L^\u\neq \emptyset$ and  $W^\u(P)\cap L^\s\neq \emptyset$ at $(\mu,\nu)=0$.
\begin{lem}\label{lem:model} 
If the two intersections at $(\mu,\nu)=0$ unfold independently with respect to $\mu$ and $\nu$, then there exists a sequence $\{(\mu_j,\nu_j)\}$  converging to 0 such that $L^\u_{\mu_j,\nu_j}\cap L^\s_{\mu_j,\nu_j}\neq \emptyset$ at $(\mu,\nu)=(\mu_j,\nu_j)$ for each $j$.
\end{lem} 

\begin{proof}[Proof of the second claim of Proposition~\ref{prop:persis}]
Let the families $\{f_\eps\},\{L^\u_\eps\},\{L^\s_\eps\}$ be as in Proposition~\ref{prop:unfoldble}.
Let $P$ be any periodic point of the blender $\Lambda$, and let the set  $\hat{\mathcal{E}}_{P,P}$ be given by Proposition~\ref{prop:unfoldble} with setting $P'=P$. Take any $\eps_*=(\mu_*,\nu_*,\eps'_*)\in \hat{\mathcal{E}}_{P,P}$ and an integer $n$ large enough so that $f^{n}(L^\u_{\eps_*})\cap W^\s_\loc(P_{\eps_*})\neq \emptyset$ and $f^{-n}(L^\s_{\eps_*})\cap W^\u_\loc(P_{\eps_*})\neq \emptyset$. Since $L^\u$ is close to a strong-unstable leaf, it is tangent to the strong-unstable cone field $\mathcal{C}^\uu$ associated with  $\Lambda$, and the same holds for $f^{n}(L^\u_{\eps_*})$ since $\mathcal{C}^\uu$ is forward-invariant. Similarly, $f^{-n}(L^\s_{\eps_*})$ is tangent to the strong-stable cone field $\mathcal{C}^\uu$ associated with $\Lambda$. 
The second claim of Proposition~\ref{prop:persis}  follows by applying Lemma~\ref{lem:model} to the families $\{g_{\mu,\nu}\}:=\{f_{\mu+\mu_*,\nu+\nu_*,\eps'_*}\}$, $\{L^\u_{\mu,\nu}\}:=\{f^{n}(L^\u_{\mu+\mu_*,\nu+\nu_*,\eps'_*})\}$, $\{L^\s_{\mu,\nu}\}:=\{f^{-n}(L^\s_{\mu+\mu_*,\nu+\nu_*,\eps'_*})\}$.
\end{proof}

\begin{proof}[Proof of Lemma~\ref{lem:model}]
By \citep[Lemma 6]{GonShiTur:08}, there exist $C^r$ ($C^{r-2}$ with respect to parameters) coordinates $(u,v,x,y)\in \mathbb R \times\mathbb R \times\mathbb R ^{d^\ss} \times\mathbb R ^{d^\uu}$ in a small neighborhood $U$ of $P$ such that, for all small $\mu$ and $\nu$, the local objects are straightened (we omit the subscripts for continuations): $W^\s_{\loc}(P)=\{v=0,y=0\}$, $ W^\u_{\loc}(P)=\{u=0,x=0\}$, the strong-stable leaves in $W^\s_{\loc}(P)$ are given by $\{u={\const},v=0,y=0\}$, and the strong-unstable leaves in $W^\u_{\loc}(P)$ are given by $\{u=0,v={\const},x=0\}$. Denote $F:=g^{\mathrm{per}(P)}|_U$. Lemma  7 of \cite{GonShiTur:08} also shows that these coordinates can be chosen such that, for any $(u_0,v_0,x_0,y_0)\in V$, one has $(u_k,v_k,x_k,y_k)=F^k(u_0,v_0,x_0,y_0)$ if and only if 
\begin{equation}\label{eq:model:1}
\begin{aligned}
&u_k = \lambda_1^k u_0 +h_1(u_0,v_k,x_0,y_k,\mu,\nu),\qquad
v_0 = \gamma_1^{-k} v_k +h_2(u_0,v_k,x_0,y_k,\mu,\nu),\\
&x_k = h_3(u_0,v_k,x_0,y_k,\mu,\nu),\qquad
y_0 = h_4(u_0,v_k,x_0,y_k,\mu,\nu),
\end{aligned}
\end{equation}
where $\|h_{1,3}\|=o(\lambda_1^k)$ and $\|h_{2,4}\|=o(\gamma_1^{-k})$.

Up to replacing $L^\u$ and $L^\s$  with some forward and, respectively, backward iterates, we have $L^\u\subset U$ and $L^\s\subset U$. Thus, since $L^\u$ and $L^\s$ are tangent to $\mathcal{C}^\ss$ and $\mathcal{C}^\uu$, respectively, their  defining functions  are given by
\begin{equation*}
\begin{aligned}
L^\u:& \quad u=\xi_1(y,\mu,\nu),\qquad
&&v=\xi_2(y,\mu,\nu),\qquad
&&x=\xi_3(y,\mu,\nu),\\
L^\s:& \quad u=\eta_1(x,\mu,\nu),\qquad
&&v=\eta_2(x,\mu,\nu),\qquad
&&y=\eta_3(x,\mu,\nu),
\end{aligned}
\end{equation*}
for some smooth functions $\xi$ and $\eta$ with bounded first derivatives. The assumption of the independent unfolding means that  
\begin{equation}\label{eq:model:2}
\xi_2(0,0,0)=0,\quad 
\eta_1(0,0,0)=0,\quad 
\det\left.\dfrac{\partial (\xi_2,\eta_1)}{\partial (\mu,\nu)}\right|_{(x,y,\mu,\nu)=0}\neq 0.
\end{equation}

In what follows, we solve the system  consisting  of \eqref{eq:model:1} and the following equations:
\begin{equation}\label{eq:model:3}
\begin{aligned}
&  u_0=\xi_1(y_0,\mu,\nu),\qquad
&&v_0=\xi_2(y_0,\mu,\nu),\qquad
&&x_0=\xi_3(y_0,\mu,\nu),\\
&  u_k=\eta_1(x_k,\mu,\nu),\qquad
&&v_k=\eta_2(x_k,\mu,\nu),\qquad
&&y_k=\eta_3(x_k,\mu,\nu).
\end{aligned}
\end{equation}
The found value of $(\mu,\nu)$ for each fixed $k$ corresponds to the parameter value for which $F^k(L^\u)\cap L^\s\neq \emptyset$. We will show that those values tend to 0 as $k\to\infty$ and hence prove the Lemma.

Substituting the equations for $u_0,x_0,v_k,y_k$ in \eqref{eq:model:3} into the $x_k$-equation in \eqref{eq:model:1}, yields $x_k=o(\lambda_1^k)$ as a function of $(y_0,\mu,\nu)$. Similarly, substituting those equations into the $y_0$-equation in \eqref{eq:model:1}, yields $y_0=o(\gamma_1^{-k})$ also as a function of $(x_k,\mu,\nu)$. Combining the two newly obtained expressions, one finds  
\begin{equation}\label{eq:model:4} 
x_k=o(\lambda_1^k),\qquad y_0=o(\gamma_1^{-k})
\end{equation}
as functions of $(\mu,\nu)$. Substituting \eqref{eq:model:4} and  the equations for $u_0,x_0,v_k,y_k$ in \eqref{eq:model:3} into the $u_k$- and $v_0$-equations in \eqref{eq:model:1}, yields
\begin{equation}\label{eq:model:5} 
u_k=O(\lambda_1^k), \qquad
v_0=O(\gamma_1^{-k}),
\end{equation}
where the right hand sides are functions of $(\mu,\nu)$. 

On the other hand, by the first two equations in \eqref{eq:model:2} and \eqref{eq:model:4}, we can write the $v_0$- and $u_k$-equations in \eqref{eq:model:3} as
\begin{align*}
v_0&=\dfrac{\partial \xi_2(0)}{\partial \mu}\mu + \dfrac{\partial \xi_2(0)}{\partial \nu}\nu +o(\gamma_1^{-k})+O(\mu^2+\nu^2)\\
u_k&=\dfrac{\partial \eta_1(0)}{\partial \mu}\mu + \dfrac{\partial \eta_1(0)}{\partial \nu}\nu +o(\lambda_1^k)+O(\mu^2+\nu^2),
\end{align*}
where the right-hand sides are functions of $(\mu,\nu)$. It then follows from the last equation in  \eqref{eq:model:2} that   the system consisting of the above two equations and \eqref{eq:model:5} admits solutions of the form $\mu=O(\lambda_1^k)$ and $\nu=O(\gamma_1^{-k})$ for all sufficiently large $k$. Taking $k\to\infty$ gives the desired sequence of parameter values. 
\end{proof}

\section{General homoclinic orbits: proof of Theorem~\ref{thm:main_blender_allcases}}\label{sec:thmA}
We will use Theorem~\ref{thm:main_2punfolding} to prove Theorem~\ref{thm:main_blender_allcases}. For that, we  first 
make $\gamma$ KAM non-degenerate and
create a transverse homoclinic to $\gamma$, and then obtain a quadratic homoclinic tangency from it.

\begin{prop}\label{prop:A1}
Let $f\in \symp^s(\M)$,  $\gamma$ be a whiskered $C^s$ torus 
with an irrational rotation number, $\Gamma$ be a homoclinic orbit 
of $\gamma$, and $\hat V$ be any small neighborhood of
$\orb\cup \Gamma$, as in Theorem~\ref{thm:main_blender_allcases}. In any neighborhood of $f$ in $\symp^s(\M)$, there exists 
a map $g$ (of class $C^\infty$ if $s<\infty$) such that $g$ has a non-degenerate whiskered KAM-torus $\gamma_g$ with a partially-hyperbolic transverse homoclinic orbit $\tilde{\Gamma}\subset \hat V$. In the case of infinite regularity ($s=\infty,\omega$) we can choose 
$\gamma_g=\gamma$, and for finite $s$ the torus 
$\gamma_g$ is $C^\infty$ and $C^s$-close to $\gamma$. 
\end{prop}

\begin{prop}\label{prop:A2}
Let $g\in \symp^s(\mathcal{M})$, $s=\infty,\omega$, 
has a non-degenerate whiskered KAM-torus $\gamma$ of class $C^s$, with a partially-hyperbolic transverse homoclinic orbit $\tilde\Gamma$, as in Proposition~\ref{prop:A1}.  Then, arbitrarily $C^s$-close to $g$ there exists $\hat g\in \symp^s(\mathcal{M})$ such that the KAM-continuation of $\gamma$ has, in $\hat V$, an orbit $\hat \Gamma$ of a hyperbolic quadratic homoclinic tangency.
\end{prop}

The proofs of these two propositions are given in Sections~\ref{sec:propA1} and~\ref{sec:propA2}, respectively.

\begin{proof}[Proof of Theorem~\ref{thm:main_blender_allcases}]
We first apply Proposition~\ref{prop:A1} and then Proposition~\ref{prop:A2} to obtain the map $\hat g$.  
Since the corresponding KAM-continuation of $\gamma$ is periodic
under $\hat g$, the map $\hat g$ is exact in a small neighborhood of
$\orb$.
The theorem then  follows by  embedding $\hat g$ into any proper unfolding family of Theorem~\ref{thm:main_2punfolding}, case (3); the existence of such families is shown in Section~\ref{sec:fam3}.
\end{proof}

\subsection{Proof of Proposition~\ref{prop:A1}}\label{sec:propA1}

The desired map $g$ is obtained by applying a sequence of perturbations. 
First, we prove in Section \ref{sec:propA1.trans} the following
\begin{lem}\label{lem:steptrans}
Arbitrarily close in $\symp^s(\M)$ to the map $f$ of Proposition~\ref{prop:A1} there exists a map, for which $\gamma$ remains a whiskered torus with the same rotation number, and the homoclinic orbit $\Gamma$ to $\gamma$ becomes transverse. 
\end{lem}
We add the perturbation given by this lemma to the original map $f$; by the transversality, the homoclinic orbit persists at all further perturbations.

At the next step, we focus on the case of finite  $s$ and prove the following result in Section~\ref{sec:reduction}:
\begin{lem}\label{lem:Cfinite}
Let a map $f\in \symp^s(\M)$, $2\leq s <\infty$, have 
a whiskered torus $\gamma$ with an irrational rotation number.
Then arbitrarily close to $f$ in $\symp^s(\M)$ there exists a $C^\infty$ map $\tilde f$ with a $C^\infty$ whiskered torus $\tilde \gamma$, 
$C^s$-close to $\gamma$, with the same irrational rotation number. 
\end{lem}

Recall that the stable and unstable manifolds of a whiskered torus are given by the strong-stable and strong-unstable foliations, which depend continuously on the map in the $C^1$ topology. Since $\tilde f$ and $\tilde \gamma$ are close to $f$ and $\gamma$,  it follows that $\tilde \gamma$  must also have a transverse homoclinic orbit near the one of $\gamma$.
Thus, this lemma reduces the finite smoothness case to the $C^\infty$ one, and we may continue under the assumption that $s=\infty$ or $s=\omega$. In Section \ref{sec:propA1.ph}, we prove

\begin{lem}\label{lem:stepph}
Let $s=\infty,\omega$, and let a map in $\symp^s(\M)$ have a whiskered torus $\gamma$ with a transverse homoclinic orbit $\Gamma$. Then there exists an arbitrarily small perturbation in $\symp^s(\M)$ such that $\gamma$ remains a whiskered torus with the same rotation number and the homoclinic orbit becomes partially-hyperbolic, in the sense that it satisfies condition \eqref{eq:condition:trans1}.
\end{lem}

Like the hyperbolicity, the partial hyperbolicity is an open property, so
it persists when we apply the final perturbation, which makes the curve $\gamma$ KAM non-degenerate, thus giving us the sought map $g$ of Proposition~\ref{prop:A1}. This  perturbation exists according to the following
\begin{lem}\label{lem:stepkam}
Let $s=\infty,\omega$, and let a map in $\symp^s(\M)$ have a whiskered torus $\gamma\in C^s$ with an irrational rotation number. Then there exists an arbitrarily small perturbation in 
$\symp^s(\M)$ that makes $\gamma$ a non-degenerate whiskered KAM-torus.
\end{lem}
The proof is given in Section \ref{sec:propA1.kam}.

\subsubsection{Transversality of $\Gamma$: proof of Lemma~\ref{lem:steptrans}}\label{sec:propA1.trans}
We start with a preliminary result (Lemma~\ref{lem:genearaltrans}) on the intersection of Lagrangian manifolds, and then find the perturbed map with a transverse homoclinic to $\gamma$ for the smooth and real-analytic cases separately.

\noindent\textbf{(1) Intersection of Lagrangian manifolds.}
Given a pair of smooth manifolds that intersect at a point $P$, we say that the intersection has {\em corank $c$} if the  tangent spaces of the two manifolds at $P$ intersect over a $c$-dimensional subspace; the case $c=0$ corresponds to the transverse intersection.

Recall that we say that a family $\{f_\eps\}$ of maps is $C^k$ if   $f_\eps$ is a $C^k$ function of variables and parameters. For the sake of this proof, we also say that a family is $C^{k_1,k_2}$ if the
map and its derivatives up to the order $k_2$ with respect to the variables have continuous derivatives up to the order $k_1$ with respect to the parameters. So, we say two families  are $C^{1,1}$-close if the maps and their first derivatives with respect to the variables
are uniformly close along with their first derivatives with respect to the parameters. Two families of embedded manifolds are $C^{1,1}$-close when their embedding maps are $C^{1,1}$-close.
 
\begin{lem}\label{lem:genearaltrans}
Let $W_1$ and $W_2$ be a pair of Lagrangian  manifolds in $\M$, and let $f\in\symp^s(\M)$ with $s\geq 1$. Assume  $f(W_1)$ intersects $W_2$ at some point $P$ such that $f(P)\neq P$, and the intersection is non-transverse, of corank $c>0$. Take any small neighborhood $U(P)$ of $P$. Then there exists a $C^\infty$ family $\{H^0_{\varepsilon,\mu,\nu}\}$ of functions supported in $U(P)$,
with parameters $\eps\in \mathbb R^1, \mu\in \mathbb R^N, \nu \in \mathbb R^N$, such that $H^0_{0,0,0}=0$ and the following holds:
\begin{itemize}[nosep]
\item  for any family $\{G_{\varepsilon,\mu,\nu}\}$ of symplectic maps which is $C^{1,1}$-close to the family of the time-1 maps $G^0_{\varepsilon,\mu,\nu}$ defined by the Hamiltonians $H^0_{\varepsilon,\mu,\nu}$, and
\item for any families $\{W_{1,\eps,\mu,\nu}\}$ and $\{W_{2,\eps,\mu,\nu}\}$ of smoothly embedded $N$-dimensional discs that are $C^{1,1}$-close to the constant families given by the unmoving original manifolds $W_1$ and $W_2$,
\end{itemize}
there exist uniquely defined smooth functions $\mu(\eps)$ and $\nu(\eps)$ such that the manifold $G_{\varepsilon,\mu(\eps),\nu(\eps)}\circ f(W_{1,\varepsilon,\mu(\eps),\nu(\eps)})$ at each small $\eps$ has an intersection with $W_{2,\varepsilon,\mu(\eps),\nu(\eps)}$ at the point $P$, with corank strictly smaller than $c$.
\end{lem}
The proof of this lemma is given in Section~\ref{sec:lemtrans}.

\noindent\textbf{(2) Smooth case.}
Take the map $f$ of Proposition~\ref{prop:A1}. Let $P$ be a homoclinic point and $W_1$ and $W_2$ be small path-connected pieces
of $W^{\u}(\gamma)$ and, respectively, $W^{\s}(\gamma)$ such that
$f^{-1}(P)\in W_1$ and $P\in W_2$. Suppose
the intersection of $f(W_1)$ and $W_2$ is non-transverse, of corank $c>0$.

Let $G^0_{\varepsilon,\mu,\nu}$ be given by Lemma~\ref{lem:genearaltrans}. By construction $G^0_{\varepsilon,\mu,\nu}=\mathrm{id}$ outside $U(P)$, so the maps $f_{\varepsilon,\mu,\nu} := G_{\varepsilon,\mu,\nu}\circ f$ are equal to $f$ outside a small neighborhood of $f^{-1}(P)$. Therefore, the whiskered torus $\gamma$ persists for the family $\{f_{\varepsilon,\mu,\nu}\}$; its rotation number does not change, and the local stable and unstable manifolds do not move. 

By the definition of $W^\u(\gamma)$, the small piece $W_1$ around the point $f^{-1}(P)$ is obtained by iterating by $f$
a small piece of $W^{\u}_\loc(\gamma)$ finitely many times. Since these iterations lie outside the supports of $H^0_{\varepsilon,\mu,\nu}$ when $U(P)$ is taken sufficiently small, the manifold $W_1$ remains
a piece of $W^\u(\gamma)$ for every small $(\varepsilon,\mu,\nu)$. 
Similarly, $W_2$ remains a piece of $W^\s(\gamma)$. Therefore, Lemma~\ref{lem:genearaltrans} gives us smooth functions $\mu(\eps)$ and $\nu(\eps)$ such that $W^\u(\gamma)$ and $W^\s(\gamma)$ 
intersect at the point $P$ at $(\mu,\nu)=(\mu(\varepsilon),\nu(\eps))$, and the corank of the intersection is less than $c$ for $\varepsilon\neq 0$. 

Note that $f_{0,0,0} =  f$ since $G^0_{0,0,0}=\mathrm{id}$. Consequently, $f_{0,0,0}(W_1)$ intersects $W_2$ at the point $P$, implying that 
$\mu(0)=0$, $\nu(0)=0$ by the uniqueness of $(\mu(\varepsilon),\nu(\varepsilon))$.
It follows that the map $f_{\varepsilon,\mu(\varepsilon),\nu(\varepsilon)}$ for an arbitrarily small $\eps\neq 0$ has  $\gamma$ as a whiskered torus, with the rotation number unchanged, and
$W^\u(\gamma)$ has an orbit of homoclinic intersection (close to the original orbit $\Gamma$) with $W^\s(\gamma)$ of corank strictly smaller than $c$. 

Take $\eps$ small, so $f_{\varepsilon,\mu(\varepsilon),\nu(\varepsilon)}$ is a sufficiently small perturbation of $f$ in $C^s$.
If the corresponding intersection is still non-transverse, we
repeat the above procedure, obtaining a homoclinic intersection of strictly smaller corank, and so on, until we arrive at corank $0$ after finitely many steps. This gives the desired map with a transverse homoclinic to $\gamma$ in the smooth case. 

\noindent\textbf{(3) Real-analytic case.} In order to conduct the same arguments as above, we find a $C^\omega$ family 
$\{H_{\eps,\mu,\nu}\}$ of Hamiltonian functions such that
\begin{enumerate}[nosep]
\item the corresponding $C^\omega$ family $\{G_{\eps,\mu,\nu}\}$ of time-1 maps 
 is $C^2$-close to the family 
$\{G^0_{\eps,\mu,\nu}\}$ of Lemma~\ref{lem:genearaltrans};
\item $H_{0,0,0}=0$, hence $G_{0,0,0}=\mathrm{id}$;
\item the maps $f_{\eps,\mu,\nu}:=G_{\varepsilon,\mu,\nu}\circ f$
are equal to identity on $\orb$.
\end{enumerate}

The last condition implies that the whiskered torus $\gamma$ persists with 
the same rotation number for all $(\eps,\mu,\nu)$. Moreover, by the continuous dependence, in $C^2$, of $W^{\u/\s}_\loc(\gamma)$ on the map, it also implies that  the speed, with which $W^{\u/\s}_\loc(\gamma)$ can move as the parameters change, is small (it is zero for the family 
$\{G^0_{\eps,\mu,\nu}\circ f\}$). As a result, we have families of smooth embedded discs $W_{1,\varepsilon,\mu,\nu} \subset W^\u(\gamma)$ and 
$W_{2,\varepsilon,\mu,\nu} \subset W^\s(\gamma)$ which are $C^{1,1}$-close to the constant families given by the original discs $W_1$ and $W_2$, respectively. Then Lemma~\ref{lem:genearaltrans} gives us
a family of $C^\omega$ perturbations $f_{\eps,\mu(\eps),\nu(\eps)}$
for which $W^\u(\gamma)$ and $W^\s(\gamma)$ have an intersection of corank less than $c$, and
the desired map with a transverse homoclinic (i.e., a corank-$0$ intersection) can be found by finitely many applications of such perturbations. 

To construct the family $\{H_{\eps,\mu,\nu}\}$, we define a function $h\in C^\omega (\M)$ satisfying
\begin{equation}\label{eq:propA1:1}
h|_\orb= 0,\qquad
\nabla h|_\orb= 0,\qquad h(P)=1;
\end{equation}
the existence of $h$ follows from Cartan's Theorem B \citep{Car:53,Car:57}. Next, we take any $C^\omega$ family of Hamiltonians $H^1_{\varepsilon,\mu,\nu}$  which is $C^{1,2}$-close to the family of
Hamiltonians that equal to 
$\frac{1}{h} H^0_{\varepsilon,\mu,\nu}$ in $U(P)$ and to $0$ outside of $U(P)$ (this is well-defined as $ H^0_{\varepsilon,\mu,\nu}$ is compactly supported in $U$ and $h\neq 0$ in $U$). Since $H^0_{0,0,0}=0$, we can take $H^1_{\varepsilon,\mu,\nu}$ such that 
$H^1_{0,0,0}=0$.
Now notice that the family $\{H_{\eps,\mu,\nu}\}:=\{h H^1_{\eps,\mu,\nu}\}$ satisfies Conditions 1-3.

\subsubsection{Partial hyperbolicity of $\Gamma$: proof of Lemma \ref{lem:stepph}}\label{sec:propA1.ph}

We restrict the consideration to the $C^\infty$ and $C^\omega$ cases, and assume that the homoclinic orbit $\Gamma$ is transverse. 

We consider a one-parameter $C^s$ $(s=\infty,\omega)$ family of symplectic maps $f_\eps$ such that $f_0$ is the original map and $\gamma$ remains the whiskered torus with the same rotation number for all small $\eps$. The transverse homoclinic intersection persists at all small $\eps$.
We take a pair of homoclinic points $M^+\in W^\s_\loc(\gamma)$
and $M^-\in W^\u_\loc(\gamma)$, so that $M^+=T_1 (M^-)$ where $T_1$ is the transition map near $\Gamma$, and show that the family can be chosen such that the angle between $T_1(\ell^\uu)$ and $W^\s(\A)$ at the point $M^+$ and the angle between $T^{-1}_1(\ell^\ss)$ and $W^\u(\A)$ at the point $M^-$ (see Section \ref{sec:phorbit}) change with non-zero velocity as the parameter $\eps$ varies through zero. This ensures the required  partial hyperbolicity for arbitrarily small $\eps$, thus proving  Lemma~\ref{lem:stepph}.

\noindent\textbf{(1) Conditions of the partial hyperbolicity for a transverse homoclinic.} 
We first consider Fenichel coordinates in a neighborhood of $\gamma$ (see Section~\ref{sec:local}), and denote these coordinates near $M^-$
as $(r,\pp,x,y)$ and near $M^+$ as $(\tilde r, \tilde \pp, \tilde x, \tilde y)$.
We have 
\begin{equation}\label{wulocf}
W^\u_\loc(\gamma)=\{r=0,x=0\},\qquad W^\s_\loc(\gamma)=\{\tilde r=0,\tilde y=0\}.
\end{equation}
Let $M^-=(0,\pp^-,0,y^-)$ and $M^+=(0,\tilde \pp^+,\tilde x^+,0)$.
We also have
\begin{equation}\label{lulocf}
\ell^\uu=\{r=0,\pp=\pp^-,x=0\},\qquad \ell^\ss=\{\tilde r=0,\tilde \pp=\tilde\pp^+,\tilde y=0\},
\end{equation}
\begin{equation}\label{wclocf}
W^\s_\loc(\A)=\{\tilde y=0\},\qquad
W^\u_\loc(\A)=\{x=0\},
\end{equation}
where $\ell^\uu$ is the strong-unstable leaf through $M^-$ and $\ell^\ss$ is the strong-stable leaf through $M^+$. Note that the Fenichel coordinates have only finite regularity, which can be assumed as high as we want though.

We make an additional coordinate transformation as given
by the following
\begin{lem}\label{lem:strsymp}
There exists a sufficiently smooth change of coordinates in a small neighborhood $U^+$ of $M^+$ and a small neighborhood $U^-$ if $M^-$ such that 
\begin{itemize}[nosep]
\item the symplectic form becomes
$\Omega|_{U^-}=d r\wedge d \pp + d x \wedge d  y$ and $\Omega|_{U^+}=d\tilde r\wedge d\tilde \pp +  d \tilde x \wedge d \tilde y$;
\item the coordinates of the points $M^\pm$ do not change;
\item $W^\s_{\loc}(\gamma)$, $\ell^\ss$ and $W^\u_{\loc}(\gamma)$, $\ell^\uu$ remain straightened near $M^+$ and, respectively, $M^-$, i.e., equations  
\eqref{wulocf} and \eqref{lulocf} hold in $U^\pm$;
\item the tangents to $W^\s_\loc(\A)$ at $M^+$ and to $W^\u_\loc(\A)$ at $M^-$ remain $\{\tilde y=0\}$ and, respectively, $\{x=0\}$.
\end{itemize}
\end{lem}
\begin{proof} Since $W^{\u/\s}_\loc(\gamma)$ are Lagrangian manifolds, the existence of a smooth coordinate transformation which makes the symplectic form standard near $M^-$ and $M^+$ is given by the Darboux-Weinstein theorem \cite{Wei:71}. Moreover, this transformation is identity on these manifolds, so equations \eqref{wulocf} and \eqref{lulocf} stay the same in $U^\pm$, as required. 

Equations \eqref{wclocf} for the local manifolds of $\A$ may change. However, 
observe that $W^\u_\loc(\A)$ must be $\Omega$-orthogonal to $\ell^\uu$
at $M^-$ (the form $\Omega$ is invariant, so $\Omega|_{M^-}(e^{\cu},e^{\uu})$ can only be zero for all vectors $e^{\cu}$ tangent to  $W^\u_\loc(\A)$ and $e^{\uu}$ tangent to $\ell^\uu$ at $M^-$, as the backward iterations are exponentially contracting in the strong-unstable directions and non-expanding in the central direction). For the standard
symplectic form and $\ell^\uu$ given by \eqref{lulocf}, this means that 
$W^\u_\loc(\A)$ is tangent to $\{x=0\}$ at $M^-$, as required. The tangency 
of $W^\u_\loc(\A)$ to $\{\tilde y=0\}$ at $M^+$ follows in the same way.
\end{proof}

The transition map $T_1$ near the homoclnic orbit $\Gamma$ is given by \eqref{eq:T1}, where all coefficients are smooth functions of $\eps$.
The transversality of the homoclinic orbit $\Gamma$ is equivalent to
the transversality of $T_1(W^\u_{\loc}(\gamma))$ with $W^\s_{\loc}(\gamma)$ at the point $M^+$. By \eqref{wulocf}, this means that the linear equation
$$\D T_1|_{M^-}(0,\Delta \pp,0,\Delta y) = (0,\Delta \tilde\pp,\Delta \tilde x,0)$$
has only the zero solution (here $(0,\Delta \pp,0,\Delta y)$ is a vector tangent to $W^\u_\loc(\gamma)$ at $M^-$, and $(0,\Delta \tilde\pp,\Delta \tilde x,0)$ is a vector tangent to $W^\s_\loc(\gamma)$ at $M^+$). 
By \eqref{eq:T1}, this means that $\det \begin{psmallmatrix}\hat a_{11} &\hat a_{13} \\ \hat a_{41} & \hat a_{43}\end{psmallmatrix}\neq 0$. Hence, we can rewrite \eqref{eq:T1} as 
\begin{equation}\label{eq:propA1.ph:1}
\begin{aligned}
\varphi-\varphi^- &= a_{11} r + a_{12} \tilde r + a_{13} x +  a_{14}\tilde y+\dots, \\
\tilde \varphi - \varphi^+ &= a_{21}  r +  a_{22}\tilde r +  a_{23} x +  a_{24}\tilde y+\dots, \\
\tilde x-x^+ &= a_{31}  r +  a_{32} \tilde r +  a_{33} x +  a_{34}\tilde y+\dots, \\
 y-y^- &= {a}_{41}  r +  a_{42}\tilde r  +  a_{43} x +  a_{44}\tilde y+\dots, 
\end{aligned}
\end{equation}
where $a_{ij}$ are some $\eps$-dependent coefficients, and the dots denote higher order terms. 

As $W^\s(\A)$ is tangent to $\{\tilde y=0\}$ at $M^+$, the transversality between $T_1(\ell^\uu)$ and $W^\s(\A)$ at $M^+$ is equivalent, by \eqref{lulocf}, to the non-existence of a non-zero solution to $(\Delta \tilde r,\Delta \tilde\pp,\Delta \tilde x,0)=\D T|_{M^-} (0,0,0,\Delta y)$, which in turn is equivalent to $a_{12}\neq 0$. Similarly, the transversality between $T_1^{-1}(\ell^\ss)$ and $W^\u(\A)$ at $M^-$ is equivalent to $a_{21}\neq 0$.

Since $T_1$ preserves the standard symplectic form, it is easy to see that $a_{12}=-a_{21}$. Thus, to prove Lemma \ref{lem:stepph}, we need
to show that the family $f_\eps$ can be constructed in such a way that
if $a_{21} = 0$ at $\eps=0$, then
$$\frac{d a_{21}}{d\eps} \neq 0.$$ 

\noindent\textbf{(2) Construction of the family of perturbations.} 
Since $T_1$ is injective, we find from \eqref{eq:propA1.ph:1} that $\det \begin{psmallmatrix} a_{21} & a_{23} \\  a_{31} &  a_{33}\end{psmallmatrix}\neq 0$. Hence, if $a_{21}=0$, then there exists $i$ such that the $i$-th component $a^i_{31}$ of the vector $a_{31}$ is non-zero. Define 
$$H = - \tilde r\tilde x_i\cdot \chi,$$ where $\chi$ is a $C^\infty$ bump function supported in $U^+$ and equal to $1$ in a neighborhood of $M^+$.
The time-$\eps$ map $G_\eps$ of the  Hamiltonian flow defined by $H$ is 
identity outside $U^+$, whereas near $M^+$ it acts on the $\tilde \pp$ coordinate as
\begin{equation}\label{pptlac}
\tilde\pp \mapsto \tilde \pp +\eps \tilde x_i.
\end{equation}

Now consider the maps $G_\eps \circ f_0$. By construction, each of these maps restricts to the identity on $\orb$, so the whiskered torus $\gamma$ and its local stable and unstable manifold do not move with $\eps$, and we have
from \eqref{pptlac} that
$a_{21}(\eps)=\eps a^i_{31}(0)$ for the corresponding transition map 
$T_{1,\eps}=G_\eps \circ T_1$. Thus,
\begin{equation}\label{eq:propA1.ph:3}
\dfrac{d a_{21}}{d\eps}= a^i_{31}(0)\neq 0.
\end{equation}

Note that  the Hamiltonian $H$ has only finite smoothness since the coordinates we are working with have only finite smoothness in general.
Therefore, we replace $H$ by $\hat H\in C^s(\M)$ such that it is close to $H$ in a sufficiently high regularity class ($C^3$ is enough) and vanishes on $\orb$ with its first derivatives. Namely, we take
$\hat H= h \cdot H^1$, where $h\in C^s(\M)$ satisfies \eqref{eq:propA1:1} (with $P=M^+$) and 
$H^1\in C^s(\M)$ is close, with derivatives up to a sufficiently high order, to  
$\frac{1}{h} H$ in $U^+$ and to $0$ outside of $U^+$. The existence of such $h$ is obvious if $s=\infty$, and is given by Cartan's Theorem B  if $s=\omega$.

Let $\hat G_\eps$ be the time-$\eps$ map of $\hat H$. The maps $f_\eps=\hat G_\eps\circ f_0$ have $\gamma$ as a whiskered torus, with its rotation number unchanged. The manifolds $W^{\s/\u}_\loc(\gamma)$ and the
leaves $\ell^{\ss/\uu}$ move slowly as $\eps$ changes, so the inequality \eqref{eq:propA1.ph:3} holds when $\hat H$ is sufficiently close to $H$. It follows that, for every sufficiently small $\eps\neq 0$, the homoclinic orbit $\Gamma$ is partially hyperbolic, as required.

\subsubsection{KAM properties of $\gamma$: proof of Lemma~\ref{lem:stepkam}}\label{sec:propA1.kam}
Let $s=\infty,\omega$, and let $f\in\symp^s(\M)$ have a whiskered torus $\gamma$ of class $C^s$. Our goal now is to prove that $\gamma$ can be made a non-degenerate whiskered KAM-torus by a $C^s$-small perturbation.

There exist $C^\s$ coordinates $(r,\pp, x, y)$ in a small tubular neighborhood $V$ of $\gamma$, disjoint from $\orb\setminus \gamma$, such that 
$\gamma=\{r=0,x=0,y=0\}$. Since the symplectic form vanishes on $\Gamma$, we can apply the Darboux-Weinstein theorem to make the symplectic form standard by a $C^s$ change of coordinates which keeps
the equation of $\gamma$. So, we can assume
$$\Omega|_V= dr\wedge d\pp + dx\wedge dy.$$ 

Let $H^0\in C^s(\M)$ satisfy
$$\left.\dfrac{\p H^0}{\p r}\right|_\gamma=1,\qquad
\left.\dfrac{\p H^0}{\p (\pp,x,y)}\right|_\gamma=0,$$
and let $H^0$ vanish with the first derivatives at the points of 
$\orb\setminus\gamma$. The existence of such function in the real-analytic case is given by Cartan's Theorem B.

Every component of $\orb$ is invariant with respect to the time-$\tau$ map $G_\tau$ of $H^0$; we have $G_\tau=\mathrm{id}$ on $\orb\setminus\gamma$
and  
\begin{equation}\label{gtaumap}
G_\tau|_\gamma: \pp\mapsto \pp+ \eps.
\end{equation}
Thus, $\gamma$ remains a whiskered torus for the maps $f_\tau=G_\tau\circ f$, and $f_\tau^\per = G_\tau\circ f^\per$ on 
$\gamma$. Since ${\p G_\tau|_\gamma}/{\p \tau}  >0$ and the rotation number of $f^\per|_\gamma$ is irrational, the rotation number of 
$G_\tau\circ f^\per|_\gamma$ is strictly monotone at $\tau=0$ (see e.g. \citep[Proposition 11.1.9]{KatHas:95}). Thus, one can find arbitrarily small $\tau$ for which the rotation number of 
$f_\tau^\per|_\gamma$ is Diophantine.

Fix  such a parameter value $\tau=\tau^*$. To complete the proof of Lemma~\ref{lem:stepkam}, we need to add a $C^s$-small perturbation to 
$f^* :=f_{\tau^*}$ which would ensure the twist condition. 

By normal hyperbolicity, the cylinder $\A$ persists for the map $f^*$,
see Section~\ref{sec:innermap}. We may choose the coordinates such that 
$\A=\{x=0,y=0\}$ and $\Omega|_A$ is standard (see \eqref{eq:strait} and \eqref{8222}). 
By the Darboux-Weinstein theorem, we can also assume that the coordinates $(r,\pp,x,y)$ are chosen such that $\Omega$ is standard in a neighborhood of $\A$. Note that these coordinates have only finite (as large as we need) smoothness.

As discussed in Section~\ref{sec:quadraunfold}, the rotation number $\rho^*$ being Diophantine implies that there exist smooth symplectic coordinates for which 
$(f^*)^\per|_{\A}:(r,\pp)\mapsto (\bar r,\bar \pp)$ takes the form  \eqref{eq:kam_low}, i.e., $(f^*)^\per|_\gamma$ is the rotation $\pp \mapsto \pp + \rho^*$. In this case, the twist condition is equivalent to the non-vanishing of the integral 
$$
I:=\int_{\pp\in\mathbb{S}^1} \frac{\p\bar \pp(0,\pp)}{\p r} d\pp.
$$

We now consider the Hamiltonian function 
$$
H(r,\pp,x,y)= - \frac{r^2}{2} \cdot \chi,
$$
where $\chi$ is a $C^\infty$ function supported in $V$ and equal to $1$ in a neighborhood of $\gamma$. Its time-$\tau$ map $G_\tau$ is identity outside 
$V$ and it is given by
\begin{equation}\label{eq:propA1.kam:2}
(r,\pp,x,y)\mapsto (r,\pp+ \tau r, x,y)
\end{equation}
in a neighborhood of $\gamma$. Thus, for the map $(G_\tau \circ f^*)^\per$, for all small $\tau$, the curve $\gamma$ stays invariant, with the same rotation number $\rho^*$, and the manifold $\A$ is locally invariant. Moreover, we have
$$\frac{dI}{d\tau}=1.$$

Now consider  a function $\hat H \in C^s(\M)$ (recall $s=\infty, \omega$) which is close to $H$ with derivatives up to a sufficiently high order and has the first derivatives vanishing at the points of $\orb\setminus \gamma$. Let $\hat G_\tau$ be  the time-$\tau$ map of the Hamiltonian flow defined by $\hat H$, and denote $\hat f_\tau:=\hat G_\tau\circ f^*$.
By construction,   the curve $\gamma$
remains invariant for $\hat f_\tau^\per$, with the same rotation number  (Diophantine),
and $dI/d\tau$ remains positive. Thus, the map $\hat f_\tau$ for a sufficiently small
$\tau$ is the sought small perturbation of $f$ in $\symp^s(\M)$.

To construct the Hamiltonian $\hat H$, we take a function $h\in C^s(\M)$ such that
it vanishes with the first derivatives at the points of $\orb$, except for $\p h/\p r$
which we require to be non-zero at the points of $\gamma$ and zero on 
$\orb\setminus\gamma$ (using again Cartan's Theorem B in the $C^\omega$ case).
Then, take any $h_1\in C^s(\M)$  that is sufficiently close, with derivatives up to a sufficiently high order, to the $C^\infty$ function
$
-  r \cdot \left( \frac{\p h}{\p r}(0,\pp,0,0)\right)^{-1} \cdot \chi.
$
Finally, we define $\hat H = h \cdot h_1$. This concludes the proof of Lemma~\ref{lem:stepkam}.

\subsubsection{Case of finite smoothness: proof of Lemma~\ref{lem:Cfinite} (new)}\label{sec:reduction}
Before starting the proof,  we  mention a difficulty that arises when we require the $C^s$ closeness of the perturbation $\tilde f$ to $f$: we cannot directly work with the coordinates where the curve $\gamma$ is straightened and the symplectic form is standard, since such coordinates are $C^{s-1}$ and hence cannot yield  $C^s$ approximations. To overcome this, we   consider a family of $C^\infty$  coordinates  where $\gamma$ is nearly straightened.

\noindent\textbf{(1) A family of the Darboux-Weinstein coordinates.}
Let $V$ be a small tubular neighborhood of the whiskered torus
$\gamma$, disjoint from $\orb\setminus\gamma$. Take a continuous (in the $C^s$ topology) family $\{\gamma_\delta\}$ of curves 
in $ V$ such that $\gamma_0=\gamma$ and $\gamma_{\delta>0} $ are $C^\infty$.

For each small $\delta\geq 0$, we can introduce $C^s$ coordinates $(\pp, u)\in \mathbb{S}^1\times \mathbb{R}^{2N-1}$ in $V$ such that the curve 
$\gamma_\delta$ is straightened, i.e., $\gamma_\delta = \{u=0\}$; for $\delta>0$ these coordinates are $C^\infty$.
In these coordinates, the symplectic form $\Omega$ has coefficients of class $C^{s-1}$  for $\delta=0$, and of class $C^\infty$ for $\delta>0$. Like in Section~\ref{sec:propA1.kam}, we apply the Darboux-Weinstein theorem to make an additional coordinate transformation that keeps $\gamma_\delta=\{(r,x,y)=0\}$ and brings the symplectic form to 
\begin{equation}\label{eq:omegaeq}
\Omega|_V= dr\wedge d\pp + dx\wedge dy,
\end{equation}
where $r\in \mathbb{R}^1$ and $(x,y)\in \mathbb{R}^{N-1}\times\mathbb{R}^{N-1}$. 
 Let us denote  by $\mathcal{DW}_\delta:V\to \mathbb{R}^{2N}$ the charts corresponding to these $\delta$-dependent coordinates. 
 By construction, $\mathcal{DW}_0$ is $C^{s-1}$  and $\mathcal{DW}_{\delta>0}$ are $C^\infty$, and the family $\{\mathcal{DW}_\delta\}$ is continuous in the $C^{s-1}$ topology.
 
Since the whiskered torus $\gamma$ is $C^{s}$-close to the approximating curve $\gamma_\delta$, we can write the equation of $\gamma$ in the chart $\mathcal{DW}_\delta$ as 
\begin{equation}\label{eq:gammaeq}
(r,x,y)=(r_\delta(\pp), x_\delta(\pp), y_\delta(\pp)),
\end{equation}
 where the functions $r_\delta, x_\delta, y_\delta$ are of class $C^s$ and they tend to $0$ in the $C^{s-1}$ topology as $\delta\to 0$. (Note that \eqref{eq:gammaeq} is the equation of the same curve $\gamma$ in different charts, not the equation of $\gamma_\delta$.)

\noindent\textbf{(2) The candidate curves for $\tilde \gamma$.}
Let $\{f_\nu\}$ be a continuous (in the $C^s$ topology) family of symplectic maps, exact in the neighborhood   of $\orb$, 
such that $f_0=f$ and $f_{\nu>0}$ are $ C^\infty$. The existence of such approximating family follows from
the Zehnder's symplectic approximation theorem \citep{Zeh:76}; the exactness is automatic as the proof in \citep{Zeh:76} is based on the use of locally supported generating functions.

In the chart $\mathcal{DW}_\delta$, we write the maps $f_\nu^\per|_V$ in the form 
$(r,\pp,x,y)\mapsto(\bar r, \bar \pp,\bar x, \bar y)$ with	
\begin{equation}\label{pqrs1}\begin{array}{l}
\bar r = p_{\delta,\nu}(r, \pp, x,y),\qquad
\bar \pp = q_{\delta,\nu}(r, \pp, x,y), \qquad (\bar x, \bar y) = z_{\delta,\nu}(r, \pp, x,y),\end{array}
\end{equation}
where the functions $p_{\delta,\nu}, q_{\delta,\nu}, z_{\delta,\nu}$ are $C^\infty$ for $\delta>0, \nu>0$, and depend continuously on $\delta$ and $\nu$ in the $C^{s-1}$ topology (in the $C^s$ topology for $\delta>0$).

For $\delta=0$ and $\nu=0$, we have $\gamma=\{r=0,x=0,y=0\}$. Since $\gamma$ is invariant under the map $f^\per$ (which is given by \eqref{pqrs1} at $\delta=0,\nu=0$), the curve $\{x=0,y=0\}$ is  invariant under the auxiliary map 
\begin{equation}\label{eq:auxi}
\bar \pp=q_{0,0}(0,\pp,x,y),\qquad (\bar x,\bar y)=z_{0,0}(0,\pp,x,y).
\end{equation}
We claim that this curve is {\em normally hyperbolic}. To see this, we note that, in the directions transverse to $\gamma$, the map $f^\per$ is exponentially contracting and expanding in $W^\s(\gamma)$ and $W^\u(\gamma)$, respectively. Thus, it suffices to show that the tangent vectors  to $W^{\s/\u}_\loc(\gamma)$ at every point of 
$\gamma$ lie in $\{\Delta r=0\}$. This is immediate from \eqref{eq:omegaeq} and the fact that  these tangent vectors must be $\Omega$-orthogonal to the tangent vectors to $\gamma$  (since $W^{\s/\u}_\loc(\gamma)$ is Lagrangian).

By the $C^1$ persistence of normally-hyperbolic curves \cite{Fen:71}, the invariant curve $\{x=0,y=0\}$ admits a continuation for any small perturbation of  the map \eqref{eq:auxi}. We will use this continuation to construct  candidates for the sought curve $\tilde \gamma$ of Lemma~\ref{lem:Cfinite}.

For $\delta>0$, we  approximate function $r_\delta$ in \eqref{eq:gammaeq} by $C^\infty$ functions $r_{\delta,\nu}$ such that $r_{\delta,\nu}\to r_\delta$ in the $C^s$ topology as $\nu\to 0$; so we can take $r_{\delta,0}=r_\delta$. We then consider the following perturbation of \eqref{eq:auxi}:
\begin{equation}\label{auxm}
\bar\pp=  q_{\delta,\nu}(r_{\delta,\nu}(\pp),\pp,x,y)+\eps, \qquad
(\bar x, \bar y) = z_{\delta,\nu}(r_{\delta,\nu}(\pp),\pp,x,y).
\end{equation}
This map has, for every small $\eps$ and $\delta>0,\nu>0$, a uniquely defined, normally-hyperbolic invariant curve 
$$\{(x,y)=(\tilde x_{\delta,\nu}(\pp,\eps),\tilde y_{\delta,\nu}(\pp,\eps))\},$$
where the $C^\infty$
functions $\tilde x_{\delta,\nu}$ and $\tilde y_{\delta,\nu}$ tend to zero, as $\delta,\nu \to 0$, along with their derivatives up to order $s-1\geq 1$ with respect to $\pp$ and $\eps$. We cannot guarantee the convergence in higher regularity because the chart $\mathcal{DW}_0$ is $C^{s-1}$, but the $C^1$ converge is enough for us (see step (4)). Note that the range of $\eps$ values,
for which this invariant curve persists, does not shrink as $\delta,\nu\to 0$, and, for each fixed $\delta>0$, the dependence of  the functions
$\tilde x_{\delta,\nu}$ and $\tilde y_{\delta,\nu}$ on $\nu$ is continuous in the $C^s$ topology by the choice of $\{f_\nu\}$.

On the other hand, since $\gamma$ is the invariant curve of \eqref{pqrs1} for $\nu=0$, the curve
$\{(x,y)=(x_\delta(\pp),y_\delta(\pp))\}$ (see \eqref{eq:gammaeq}) is invariant under
the map \eqref{auxm} at $\eps=0$ and $\nu=0$.
It then follows from  the uniqueness that 
\begin{equation}\label{zdl0}
\tilde x_{\delta,0}(\pp,0)= x_\delta(\pp), \qquad \tilde y_{\delta,0}(\pp,0)=y_\delta(\pp).
\end{equation}
Note also that the invariance of $\gamma$ under \eqref{pqrs1} at 
$\nu=0$ gives us
\begin{equation}\label{pdl0}
p_{\delta,0}(r_\delta(\pp), \pp, \tilde x_{\delta,0}(\pp,0),\tilde y_{\delta,0}(\pp,0))
= r_\delta(q_{\delta,0}(r_\delta(\pp), \pp, \tilde x_{\delta,0}(\pp,0),\tilde y_{\delta,0}(\pp,0))).
\end{equation}

Define
\begin{align}
\Phi_{\delta,\nu}(\pp,\eps)&=q_{\delta,\nu}(r_{\delta,\nu}(\pp), \pp, \tilde x_{\delta,\nu}(\pp,\eps), \tilde y_{\delta,\nu}(\pp,\eps))+\eps,\label{restr}\\
P_{\delta,\nu}(\pp,\eps) &= p_{\delta,\nu}(r_{\delta,\nu}(\pp), \pp, \tilde x_{\delta,\nu}(\pp,\eps),\tilde y_{\delta,\nu}(\pp,\eps)) -
r_{\delta,\nu}(\Phi_{\delta,\nu}(\pp,\eps)).\nonumber
\end{align}
It is easy to check that for every $\eps$ the curve 
$$\tilde \gamma_{\delta,\nu,\eps}=\{r=r_{\delta,\nu}(\pp),(x,y)=(\tilde x_{\delta,\nu}(\pp,\eps),\tilde y_{\delta,\nu}(\pp,\eps))\}$$
is invariant under the map $F_{\delta,\nu,\eps}$ given by 
\begin{equation}\label{repsf}
\begin{aligned}
&\bar r = p_{\delta,\nu}(r, \pp, x,y) - P_{\delta,\nu}(\pp,\eps),\\
&\bar \pp = q_{\delta,\nu}(r, \pp, x,y)+\eps, \qquad
(\bar x, \bar y) = z_{\delta,\nu}(r, \pp, x,y),
\end{aligned}
\end{equation}
and the restriction of this map to $\tilde \gamma_{\delta,\nu,\eps}$ is given by
$$\pp \mapsto \Phi_{\delta,\nu}(\pp,\eps).$$

By construction, the  curve $\tilde \gamma_{\delta,\nu,\eps}$ is $C^\infty$ for $\delta>0$ and $\nu>0$. Moreover, it depends on $\nu$ and  $\eps$ continuously in the $C^s$ topology for $\delta>0$.  This implies that $\tilde \gamma_{\delta,\nu,\eps}$ is  $C^s$-close to $\gamma$ for small $\nu>0$ and $\eps>0$, since by \eqref{zdl0}  the curve $\tilde \gamma_{\delta,0,0}$ coincides with $\gamma$. In what follows, we show that  $\tilde \gamma_{\delta,\nu,\eps}$ is the sought whiskered torus $\tilde \gamma$  for some choice of the parameters.

\noindent\textbf{(3) Globalization of the map $F_{\delta,\nu,\eps}$.} We now construct a map $f_{\delta,\nu,\eps}\in \symp^\infty(\M)$, $C^s$-close to $f$, such that $f_{\delta,\nu,\eps}^\per$ restricts to $F_{\delta,\nu,\eps}$ near $\gamma$. 
 By \eqref{repsf}, 
the map $F_{\delta,\nu,\eps}$   is the composition $G_{\delta,\nu,\eps}\circ R_\eps \circ f_\nu^\per$ where $R_\eps$ is the rotation of angle $\eps$ and
$$G_{\delta,\nu,\eps}: (r, \pp, x,y) \mapsto 
(r -P_{\delta,\nu}\circ\Phi_{\delta,\nu}^{-1}(\pp,\eps), \pp, x, y),$$
where  $\Phi_{\delta,\nu}:\pp\mapsto\bar\pp$ is given by \eqref{restr}.
Note that 
\begin{equation}\label{htpdll}
P_{\delta,0}\circ\Phi_{\delta,0}^{-1}(\pp,0)=0
\end{equation}
by (\ref{pdl0}). The map
$G_{\delta,\nu,\eps}$ preserves the standard symplectic form in $V$, so $F_{\delta,\nu,\eps}$ is a symplectic map. Since $F_{\delta,\nu,\eps}$ has
an invariant curve $\gamma_{\delta,\nu,\eps}$, it is exact. Consequently, the map
$G_{\delta,\nu,\eps}$ is exact too, which implies
\begin{equation}\label{eq:exactcon}
\int_0^1 P_{\delta,\nu}\circ\Phi_{\delta,\nu}^{-1}(t,\eps) dt=0.
\end{equation}
Now we define the Hamiltonian function 
$H_{\delta,\nu,\eps}$, which 
is equal to zero outside of $V$ and is given by 
$$H_{\delta,\nu,\eps} (r,\pp,x,y) = - \xi(r,\pp,x,y) \cdot \int_0^\pp P_{\delta,\nu}\circ\Phi_{\delta,\nu}^{-1}(t,\eps) d t$$
in the chart $\mathcal{DW}_\delta$ iin $V$; here $\xi$ is some 
 $C^\infty$ bump function, 1-periodic in $\pp$, equal to $1$ in a neighborhood $V^\prime$ of $\gamma$ (independent of $\delta$, $\nu$, and $\eps$) and equal to zero at $\partial V$. By \eqref{eq:exactcon}, the function $H_{\delta,\nu,\eps}$
is 1-periodic in $\pp$, i.e., it is well-defined. Note that
$H_{\delta,0,0}=0$ by \eqref{htpdll}.

The time-1 map of $H_{\delta,\nu,\eps}$ coincides with 
$G_{\delta,\nu,\eps}$
in $V^\prime$ and is identity outside of $V$. So, if we define the map $f_{\delta,\nu,\eps}$ by the composition of this time-1 map, the time-1 map of the Hamiltonian given by $(r,\pp,x,y)\mapsto -\eps r \xi$ in $V$ and $0$ outside of $V$ (it gives the rotation $R_\eps$ when restricted to $V^\prime$), and the map $f_\nu$, then $f_{\delta,\nu,\eps}^\per$ is given by \eqref{repsf} in $V^\prime$.

Summarizing, we have constructed symplectic maps $f_{\delta,\nu,\eps}$ of $\M$ such that $f_{\delta,0,0}=f$ for all small $\delta$; for each fixed $\delta$ and $\nu$ the one-parameter family 
$\{f_{\delta,\nu,\eps}\}_\eps$ is at least $C^{s-1}$, and it is $C^s$ for $\delta> 0$ and $C^\infty$ for $\delta>0,\nu>0$; for each fixed $\delta>0$ these families depend continuously on $\nu$ in $C^s$. 
Thus, for every small $\delta>0,\nu>0,\eps$, the map  $f_{\delta,\nu,\eps}$ belongs to $\symp^\infty(\M)$, is $C^s$-close to $f$, and has $\tilde\gamma_{\delta,\nu,\eps}$ as a whiskered torus with period $\per$.

\noindent\textbf{(4) Modification of the rotation number of $\tilde \gamma_{\delta,\nu,\eps}$.}  By \eqref{restr}, the restriction $f^\per_{\delta,\nu,\eps}|_{\tilde\gamma_{\delta,\nu,\eps}}$ is given by $\pp\mapsto \Phi_{\delta,\nu}(\pp,\eps)$. 
We have
$$\Phi_{0,0}(\pp,\eps)=q_0(0, \pp, 0,0)+\eps,$$
so ${\p\Phi_{0,0}}/{\p \eps} =1 $. Since $\Phi_{\delta,\nu}$
depends on $\delta$ and $\nu$ continuously in $C^1$, we have
${\p\Phi_{\delta,\nu}}/{\p \eps} > 0$ for all small $\delta$ and $\nu$. This implies that the rotation number 
$\rho_{\delta,\eps,\nu}$ of $f^\per_{\delta,\nu,\eps}|_{\gamma_{\delta,\nu,\eps}}$ is a monotone function of $\eps$, and strictly monotone when it takes irrational values.
In particular, it is strictly monotone at $\eps=0,\nu=0$, because $f^\per_{\delta,0,0}|_{\tilde\gamma_{\delta,0,0}} = f^\per|_\gamma$ by the construction, and the rotation number $\rho(\gamma)$ of $f^\per|_\gamma$ is irrational.

Therefore, for every fixed small $\delta>0$ the rotation number
$\rho_{\delta,0,\eps}$ runs an interval around $\rho(\gamma)$ as $\eps$
runs a small interval around zero. By the continuous dependence of the rotation number of parameters, the same holds true for every sufficiently small $\nu$. This allows us, for this value of $\delta>0$,
to find arbitrarily small $\nu>0$ and 
$\eps$ such that $\rho_{\delta,\nu,\eps}=\rho(\gamma)$. The map
$f_{\delta,\nu,\eps}$ for these parameter values is the sought map
$\tilde f$ of Lemma~\ref{lem:Cfinite}: it is $C^\infty$ because $\nu>0$,
and it is $C^s$-close to $f_{\delta,0,0}=f$, as required.

\subsection{Proof of Proposition~\ref{prop:A2}}\label{sec:propA2}
The transverse intersection between $W^\s(\gamma)$ and $W^\u(\gamma)$ at $\Gamma$ implies, by Lemma~\ref{lem:incli}, the existence of infinitely many transverse homoclinic orbits in $\hat V$; we choose one and denote it by $\tilde \Gamma$. Let $S$ and $\tilde S$ be the scattering maps defined along $\Gamma$ and $\tilde \Gamma$, respectively.  By Lemma~\ref{lem:equitangency}, both $S(\gamma)$ and $\tilde S^{-1}(\gamma)$ intersect $\gamma$ transversely (see Figure~\ref{fig:transverse}(a)).

 Recall that $\A$ is the normally-hyperbolic $g^\per$-invariant cylinder that contains $\gamma$. The KAM-curve $\gamma$ and the cylinder $\A$ persist at all $C^s$-small perturbations, where we now have $s=\infty,\omega$.
\begin{lem} \label{lem:g_eps}
Let $\mathcal{U}$ be any neighborhood of $g$ in $\symp^s(\M)$, 
There exists a one-parameter $C^s$ family $\{g_\nu\}\subset\mathcal{U}$  such that the maps $g_\nu$ have a generic saddle-center periodic orbit smoothly dependent on $\nu$,
and the following holds. There are two  points $P_1,P_2\in \A$ that belong to this orbit such that
$P_1\in S(\gamma)$ for all small $\nu$ and $P_2\in \tilde S^{-1}(\gamma)$ for $\nu=0$, and the distance between   
$\tilde S^{-1}(\gamma)$ and $P_2$ changes with a non-zero velocity as the parameters $\nu$ varies.
\end{lem}

\begin{figure}[!h]
\begin{center}
\includegraphics[scale=1.1]{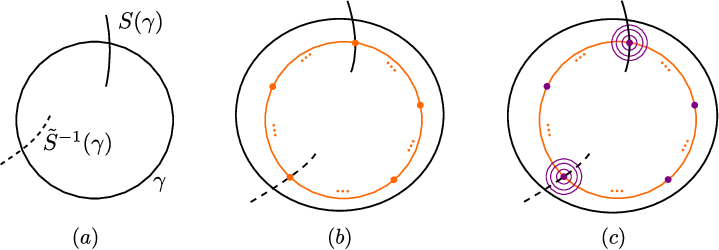}
\end{center}
\caption{(a) Two transverse homoclinic intersections. (b) A new whiskered torus filled with periodic points and transverse to $S(\gamma)$ and $\tilde S^{-1}(\gamma)$. (c) Creating saddle-center periodic points.}
\label{fig:transverse}
\end{figure}

The lemma is proved in \ref{sec:g_eps} and it implies Proposition~\ref{prop:A2} as follows.
Since the saddle-center periodic orbit is  generic, the points $P_1$ and $P_2$ in this orbit are encircled by KAM-curves in $\A$ that accumulate on these points (see condition \ref{word:C1} of Section~\ref{sec:localSC} and the discussion that follows). An illustration is given in Figure~\ref{fig:transverse}(c). For every KAM-curve $\gamma_1$ around $P_1$, there is  a KAM-curve $\gamma_2$ around $P_2$ satisfying 
\begin{equation}\label{eq:A2:2}
\gamma_2=g_0^n(\gamma_1)
\end{equation}
for some integer $n>0$. We have that $P_1\in S(\gamma)$, so $S(\gamma)$ intersects transversely all the KAM-curves around $P_1$. It then follows from Lemma~\ref{lem:incli0} that $W^\u(\gamma)$ accumulates on the unstable manifold  $W^\u(\gamma_1)$ for any $\gamma_1$ around $P_1$, and hence on the unstable manifold $W^\u(\gamma_2)$ by \eqref{eq:A2:2}.

Since $P_2\in \tilde S^{-1}(\gamma)$ at $\nu=0$, an arbitrarily small change in $\nu$ which makes $\tilde S^{-1}(\gamma)$ move can create a quadratic tangency with a KAM-curve 
$\gamma_2$ around $P_2$. By Lemma~\ref{lem:equitangency}, this corresponds to a partially-hyperbolic quadratic tangency between $W^\s(\gamma)$  and $W^\u(\gamma_2)$. It then follows from the accumulation of $W^\u(\gamma)$ on $W^\u(\gamma_2)$ that the sought quadratic homoclinic tangency for $\gamma$ is created by an additional move of $\tilde S^{-1}(\gamma)$ (i.e., by an arbitrarily small change in $\nu$).

\subsubsection{Proof of Lemma \ref{lem:g_eps}}\label{sec:g_eps}
We will first embed $g$ into a family that controls the positions of $S(\gamma)$ and $\tilde S^{-1}(\gamma)$. Based on this family, we then create a saddle-center periodic orbit with the desired properties.

\noindent\textbf{(1) Moving $S(\gamma)$ and $\tilde S^{-1}(\gamma)$.}
Take two homoclinic points $M_1\in W^\s_\loc(\gamma)\cap \Gamma$ and $M_2\in W^\u_\loc(\gamma)\cap \tilde\Gamma$. In the Fenichel coordinates $(r,\pp,x,y)$  (see \eqref{eq:strait}--\eqref{eq:strait3}), we have $M_1=(0,\pp'_1,x'_1,0)$ and $M_2=(0,\pp'_2,0,y'_2)$ for some $\pp'_{1,2},x'_1,y'_2$.
 Lemma~\ref{lem:equitangency} shows that $S(\gamma)$ and $\tilde S^{-1}(\gamma)$ intersect  $\gamma$  transversely at the points $Q_1:=\pi^\s(M_1)=(0,\pp'_1)$ and $Q_2:=\pi^\u(M_2)=(0,\pp'_2)$, respectively. Since the whiskered KAM-torus $\gamma$ persists and the Fenichel coordinates depend continuously on the map, the values $\pp'_i$ have well-defined continuations for  small  perturbations. 

We first construct a perturbation that controls $\pp'_1$.  
Since $W^\s_\loc(\gamma)$ is Lagrangian, by the Darboux-Weinstein Theorem, there exists a symplectic change of coordinates
in a small neighborhood $U$ of $M_1$ that restricts to the identity on $W^\s_\loc(\gamma)$ and makes the symplectic form standard. In the new coordinates,  we have   $W^\s_\loc(\gamma)=\{\tilde r=0,\tilde y=0\}$  and $\Omega|_U=d \tilde r\wedge d\tilde \pp +d\tilde x \wedge d\tilde y$. The  transition map $T_1$ defined by $\Gamma$ takes the  form  \eqref{eq:T1_xform_n} with 
$\ell=0$ (since $\Gamma$ is a transverse homoclinic) and  
$T_1(W^\u_{\loc}(\gamma))$  is given by 
\begin{equation}\label{eq:A2:3a}
\begin{aligned}
\tilde r &=   \dd (\tilde \varphi - \varphi'_1) +   a_{14}\tilde y+O((\tilde \pp - \pp'_1)^2+\tilde y^2), \\
\tilde x-x'_1 &=   a_{32} (\tilde \varphi - \varphi'_1) +  a_{34}\tilde y+O((\tilde \pp - \pp'_1)^2+\tilde y^2),
\end{aligned}
\end{equation}
where $\beta\neq 0$. Consider the Hamiltonian function
$
H^1_{\eps_1}(\tilde r,\tilde \pp,\tilde x,\tilde y)=-\eps_1 r\cdot \chi(\tilde r,\tilde \pp,\tilde x,\tilde y),
$
where $\chi$  is a $C^{\infty}$ bump function supported in a neighborhood of $M^+$ and equal to $1$ in some smaller neighborhood of $M^+$.
The corresponding time-1 map  $G^1_{\eps_1}$ near $M^+$ takes the form
$(\tilde r,\tilde \pp,\tilde x,\tilde y)\mapsto (\tilde r+\eps_1,\tilde \pp,\tilde x,\tilde y)$. By construction, the transition map for   $G^1_{\eps_1}\circ g$   is $T_{1,\eps_1}=G^1_{\eps_1} \circ T_1$. As a result, $T_{1,\eps_1}(W^\u_{\loc})$ is given by  the same formula \eqref{eq:A2:3a}, with replacing $\pp'_1$ by
$$\varphi'_1(\eps_1)=\varphi'_1(0) + \eps_1.
$$
Obviously, $\pp'_1(\eps)$ is the $\pp$-coordinate of $Q_1(\eps)$, and we have $d\pp'_1/d\eps_1\neq 0$. 
A similar  construction near $M_2$, applied to the inverse map, gives a function $H^2_{\eps_2}$ whose time-1 map  changes  $\pp'_2(\eps)$ --  the $\pp$-coordinate of $Q_2(\eps)$ -- with non-zero velocity.  We take $H_\eps=H^1_{\eps_1}+H^2_{\eps_2}$, which is  supported near $M_1$ and $M_2$.   By construction, with  $G_\eps=G^2_{\eps_2}\circ G^1_{\eps_1}$ being the corresponding time-1 map, the values $\pp'_i$  for   the map $G_{\eps}\circ g$   change  independently  as parameters vary, that is, 
\begin{equation}\label{eq:A2:3}
\det \left.\dfrac{\p(\pp'_1,\pp'_2)}{\p(\eps_1,\eps_2)}\right|_{\eps=0}\neq 0.
\end{equation}
Recall that the Fenichel coordinates have only finite smoothness (though arbitrarily high). We further consider  a $C^s$ family $ \{H'_\eps\}\subset C^\s(\M)$, which satisfies $H'_0=0$ and  is   close to $\{H_{\eps}\}$ with sufficiently many derivatives so that \eqref{eq:A2:3} remains valid. We still denote by $G_{\eps}$ the  time-1 map of the Hamiltonian flow defined by $H'_\eps$.

\noindent\textbf{(2) Creating candidate periodic orbits.}
By the KAM property, the whiskered torus $\gamma$ persists for small perturbations, where the continuation is  the unique invariant curve  near $\gamma$ that has the same Diophantine rotation number $\rho(\gamma)$.
 Let us denote by $\gamma_\eps$   the KAM continuation of $\gamma$ for the map $G_{\eps}\circ g$.
Take a sufficiently small tubular neighborhood $V$ of $\gamma_\eps$, disjoint from $\mathcal{O}(\gamma_\eps)\setminus \gamma_\eps$.  For each small $\eps$, consider $C^\s$ coordinates $(r,\pp, x, y)$ in  $V$ such that $\Omega|_V= dr\wedge d\pp + dx\wedge dy$ and $\gamma_\eps$ coincides with the set $c:=\{r=0,x=0,y=0\}$. Moreover, we can assume that $(G_\eps\circ g)^\per|_c$ is a rotation of angle $\rho(\gamma)$ (by \citep{Yoc:84}).  Let $\mathcal{C}_\eps: V\to \mathbb{R}^{2N}$ be the   chart  corresponding to the $\eps$-dependent coordinates.
In what follows, we construct perturbations $g_{\eps,\tau}$ for which $c$ becomes  filled with  periodic points, while the continuation $\gamma_{\eps,\tau}$ of $\gamma_\eps$ moves away from $c$ as $\tau$ varies.

As shown in  Section~\ref{sec:propA1.kam}, for each $\eps$, there exists a Hamiltonian $ H^0_\eps\in C^\s(\M)$ such that the time-$\tau$ map  $\tilde G_{\eps,\tau}$ of its flow keeps $\mathcal{O}(\gamma_\eps)$ invariant, acts as
identity on $\mathcal{O}(\gamma_\eps)\setminus\gamma_\eps$ and as a rotation of angle $\tau $ on $\gamma_\eps$, see \eqref{gtaumap}. By construction, the Hamiltonian $H^0_\eps$  is $C^s$ in $\eps$.

Since the set $c$ coincides with $\gamma_\eps$ in the chart $\mathcal{C}_\eps$,  it remains invariant with respect to 
\[
g_{\eps,\tau}:=\tilde G_{\eps,\tau}\circ  G_{\eps}\circ g,
\]
and $g_{\eps,\tau}^\per|_c$ is a rotation of  angle
$\rho_c=\rho(\gamma)+\tau.$
In particular, we can find arbitrarily small $\tau^*\neq 0$ and  $q\in \mathbb{N}$ such that $\rho_c=p/q$ for some $p\in\mathbb{Z}$. In this case, the whiskered torus $c$ is filled with parabolic periodic points of $g_{\eps,\tau^*}|_c$.

Recall that the KAM continuation $\gamma_{\eps,\tau}$ is the unique invariant curve of $g^\per_{\eps,\tau}$ near $\gamma_\eps$ that has the same Diophantine rotation number $\rho(\gamma)$. Thus, we have $\gamma_{\eps,\tau}\neq c$ in the chart $\mathcal{C}_\eps$ for $\tau\neq 0$, and $c$ is $C^s$-close to $\gamma_{\eps,\tau}$. In particular, the curve $c$ coexists with $\gamma_{\eps,\tau}$ at the $\tau$ values for which $c$ is filled with parabolic periodic points.On the other hand,  since the cylinder $\A_{\eps,\tau}$ contains all   orbits of $g^\per_{\eps,\tau}$  that lie entirely in $V$ (up to shrinking $V$ if necessary), we have $c\subset \A_{\eps,\tau}$. It then follows from   
the assumption that $\gamma$ intersects transversely $S(\gamma)$ and $\tilde S^{-1}(\gamma)$ that $c$ intersects transversely $S(\gamma_{\eps,\tau})$ and $\tilde S^{-1}(\gamma_{\eps,\tau})$   for all small $\eps$ and $\tau$.  

Denote by $P_1$ and $P_2$ the  intersection points of $c$ with $S(\gamma_{\eps,\tau})$ and $\tilde S^{-1}(\gamma_{\eps,\tau})$,  respectively, 
and by $\pp_1$ and $\pp_2$ their $\pp$-coordinates.  By construction,  $P_i=Q_i$ and $\pp_i=  \pp'_i$ at $\tau=0$.
It  follows from \eqref{eq:A2:3}  and the fact $\tilde G_{\eps,0}=\mathrm{id}$ that
$$
\det \left.\dfrac{\p(\pp_1,\pp_2)}{\p(\eps_1,\eps_2)}\right|_{\eps=0}\neq 0.
$$
Therefore,  there exist arbitrarily small $(\eps^*,\tau^*)$   such that $\rho_c=p/q$   and $(\pp_1-\pp_2)=p'/q$ for some  $p,p'\in \mathbb{Z}$, implying that $P_1$ and $P_2$ belong to the same parabolic periodic orbit of $g_{\eps^*,\tau^*}$.
Moreover, as $\eps_1$ and $\eps_2$ vary from $\eps_1^*$ and $\eps_2^*$, the distances between $S(\gamma_{\eps,\tau})$ and $P_1$ and, respectively,
between $\tilde S^{-1}(\gamma_{\eps,\tau})$ and $P_2$ change with non-zero velocity.

\noindent\textbf{(3) Creating the saddle-center periodic orbit.}
By the partial hyperbolicity, the periodic orbit that contains the points $P_{1}$ and $P_2$ has $N-1$ multipliers outside and $N-1$ multipliers inside the unit circle. At $(\eps,\tau)=(\eps^*,\tau^*)$, the remaining two multipliers $\lambda_1$ and $\lambda_2$ are equal to $1$. The goal now is to modify these two multipliers so that the periodic orbit becomes a 
saddle-center.

It is a general principle that, given the
germ at a point $P$ of any  $C^\infty$ symplectic map, there exists a $C^\infty$ Hamiltonian function such the $k$-jet  (the Taylor polynomial of degree $k$ at $P$) of the germ coincides with the $k$-jet at $P$ of the time-1 map of the corresponding Hamiltonian flow. Moreover, the Hamiltonian depends continuously on the germ. 
So, given any close to identity symplectic matrix, there exists a close to identity Hamiltonian
$H$, supported in a small neighborhood of $P_1$, such that its time-1 map has the point $P_1$ fixed, and the derivative at $P_1$ is given by this matrix. Thus, we can choose $H$ close to identity such that $P_1$ and $P_2$ remain
in the same periodic orbit of $G\circ g_{\eps^*,\tau^*}$ and the multipliers $\lambda_{1,2}$ of this orbit become complex
(here $G$ is the time-1 map of $H$). Thus, the orbit of $P_1$ becomes a saddle-center. In the same way, one makes   the first Birkhoff coefficient non-zero so that this saddle-center is generic (i.e., the twist condition in \ref{word:C1} is satisfied, see Section~\ref{sec:localSC} for details). 

Since  $G$ is locally supported, the intersections of $S(\gamma_{\eps^*,\tau^*})$ with $P_1$ and $\tilde S^{-1}(\gamma_{\eps^*,\tau^*})$ with $P_2$ do not disappear when $g_{\eps^*,\tau^*}$ is modified to $G\circ g_{\eps^*,\tau^*}$. 
Thus, in the $C^\infty$ case,   the one-parameter family $\{G\circ g_{\eps_1^*,\eps_2^*+\nu,\tau^*}\}$  is  the sought family $\{g_\nu\}$ of Lemma~\ref{lem:g_eps}

In the $C^\omega$ case, we  replace $H$ by any sufficiently close real-analytic approximation. Note that the non-zero speed of the motions of $S(\gamma_{\eps,\tau})$ and $\tilde S^{-1}(\gamma_{\eps,\tau})$ relative to $P_1$ and $P_2$, respectively, ensures $P_1\in S(\gamma_{\eps^{**},\tau^{**}})$ and $P_2\in \tilde S^{-1}(\gamma_{\eps^{**},\tau^{**}})$ for some $(\eps^{**},\tau^{**})$ close to $(\eps^{*},\tau^{*})$. Moreover, using Cartan's Theorem B again, we can choose this approximating Hamiltonian such that its 4-jets at the points of the orbit of $P_1$ coincide with those of $H$, so that $P_1$ and $P_2$
remain in the same generic saddle-center periodic orbit (the multipliers of a periodic orbit are determined by the 1-jet of the map, and the first Birkhoff coefficient of the saddle-center is determined by the 3-jet). The sought family is then given by   $G\circ g_{\eps_1^{**},\eps_2^{**}+\nu,\tau^{**}}$ (where $G$ is the time-1 map corresponding to the real-analytic approximation of $H$).

\section{Symplectic blenders near saddle-center periodic points}\label{sec:SC}
In this section, we prove Theorems~\ref{thm:main_SC_para} and~\ref{thm:main_SCblender}, following the lines sketched  in Section~\ref{sec:introSC}.  

\subsection{Problem setting: genericity conditions and unfolding families}\label{sec:para_SC}

\subsubsection{Local dynamics}\label{sec:localSC} 
Let $f\in \symp^s(\mathcal{M})$  have a saddle-center periodic point $O$, and let  $\lambda_i$ and $\lambda^{-1}_i$ $(i=1,\dots,N)$ be the multipliers of $O$:
\begin{equation}\label{eq:Orho}
\lambda_{1}=e^{ i\rho},\quad\lambda_{2}=e^{ -i\rho},\quad
|\lambda_i|<1,\; i=2,\dots,N,
\end{equation}
for some $\rho\in(0,\pi).$
By assumption the point $O$ has a two-dimensional,  locally-invariant, normally-hyperbolic, symplectic center manifold $W^\c(O)$  in some small neighborhood $V$ of $O$. The discussion  in Section~\ref{sec:whiskeredtori} on the normally-hyperbolic cylinder $\A$ also applies to $W^\c(O)$. In particular, $W^\c(O)$ is the intersection of the $(n+1)$-dimensional center-stable manifold $W^\cs(O)$ and center-unstable manifold $W^\cu(O)$; all these manifolds are of  class $C^{s'}$, where $s'=s$ if $s$ is finite and $s'$ can be arbitrarily large if $s=\infty,\omega$.

 We define the local map as  $T_0:=f^{\mathrm{per}(O)}|_{V}$. There exist {\em Birkhoff coordinates} (see e.g. \citep[Appendix 7]{Arn:78}) $(u,v)\in \mathbb{R}^2$ such that $O$ is at the origin, and, in the polar coordinates 
  $(r,\pp)\in \mathbb{R}\times\mathbb{S}^1$ given by
\begin{equation}\label{eq:bcoor} 
u=\sqrt{2r}\cos\pp,\qquad v=\sqrt{2r}\sin\pp,
\end{equation}  
the inner map $F:=T_0|_{W^\c(O)}$ takes the form 
\begin{equation}\label{eq:localmap_SC_old}
\begin{aligned}
\bar r = r + O(r^m),\qquad
\bar \varphi = \varphi + p(r) + O(r^m),
\end{aligned}
\end{equation} 
where $p(r)$ is a polynomial of $r$ with $p(0)=\rho$ and $\deg p(r)=m-1$, and the terms $O(r^m)$ are functions of $r,\varphi$ with period 1 in $\pp$. Here $m$ can be taken  as large as we want if $s$ is sufficiently large.  
In these coordinates the symplectic form restricted to $W^\c(O)$ keeps the standard form 
\begin{equation}\label{eq:OmegaW^c}
du\wedge dv=dr\wedge d\pp. 
\end{equation}

We impose the following genericity condition:

\noindent\textbf{(}\setword{\textbf{C1}}{word:C1}\textbf{)} The argument $\rho$ is irrational and the twist condition is satisfied, i.e., $p'(0)\neq 0$.

\noindent By Moser \citep{Mos:73}, this condition, plus that the inner map $F$ is at least $C^4$, implies the existence of  a large measure Cantor set  of KAM-curves (see Definition \ref{defi:kam}) on $W^\c(O)$. Let $\mathcal{G}$ be a subset that contains all KAM-curves whose rotation numbers are $(c,\tau)$-Diophantine for some fixed $c,\tau$ (see \eqref{eq:diophantine}).
The KAM-curves in $\mathcal{G}$ satisfy the density and persistence properties as discussed in Section~\ref{sec:kam_strait}; moreover, in the current case,   the set $\bigcup_{\gamma\in\mathcal{G}}\gamma$ has density 1 at $O$.

Consider a family $\{f_\eps\}\subset \symp^s(\mathcal{M})$ with $f_0=f$, where $f_\eps$ is jointly $C^s$ with respect to variables and parameters.
Like in Section~\ref{sec:quadraunfold}, where we make the cylinder $\A$ invariant by bounding it with  two KAM-curves, we now take a small neighborhood $\mathbb{D}$ of $O$ in $ W^\c(O)$ that is bounded by an arbitrary KAM-curve in $\mathcal{G}$. Then $\mathbb{D}$ is a normally-hyperbolic {\em invariant} manifold lying in the neighborhood $V$, and it persists for all small $\eps$.  
It possesses  an $(N+1)$-dimensional stable manifold $W^\s(\DD)\subset W^\cs(O)$ and an $(N+1)$-dimensional unstable manifold $W^\u(\DD)\subset W^\cu(O)$, foliated by a strong-stable foliation $\mathcal{F}^{\ss}$ and, respectively, a strong-unstable foliation $\mathcal{F}^{\uu}$, with $N$-dimensional leaves.  When $V$ is sufficiently small, both manifolds are of class $C^{s'}$, and both foliations are $C^{s'-1}$ with $C^{s'}$ leaves.

As in Section~\ref{sec:innermap}, we  introduce $C^{s'-1}$ Fenichel coordinates $(r,\pp,x,y)\in \mathbb R\times \mathbb R\times \mathbb{R}^{N-1} \times \mathbb{R}^{N-1}$ such that $O_\eps$  is at the origin  and  the local manifolds are  straightened: 
 \begin{equation}\label{eq:strait2}
 \begin{array}{l}
 \mathbb D\subset \{x=y=0\},\qquad W^\s_{\loc}(\mathbb D)=\{y=0\},\qquad W^\u_{\loc}(\mathbb D)=\{x=0\},\\
W^\s_{\loc}(O)=\{r=\pp= y=0\},\qquad W^\u_{\loc}(O)=\{r=\pp= x=0\},\\
\ell^{\ss}=\{(r,\pp)={\const},y=0\},\qquad \ell^{\uu}=\{(r,\pp)={\const},x=0\},
 \end{array}
 \end{equation} 
where $\ell^{\ss}$ and $\ell^\uu$ denote leaves of $\mathcal{F}^{\ss}$ and $\mathcal{F}^{\uu}$.  Note that the whole structure persists for all small $\eps$, provided that $f$ is sufficiently smooth (so that $s'$ is large).

By P\"oschel \citep{Pos:82}, all the KAM-curves in $\mathcal{G}$ can be straightened simultaneously for all small $\eps$, as in Section~\ref{sec:kam_strait}. We will consider Fenichel coordinates with this additional property. In summary, after normalizing the linear term of the polynomial $p$ in \eqref{eq:localmap_SC_old},  the inner map  assumes the form (cf.~\eqref{eq:kam_new})
\begin{equation}\label{eq:localmap_SC}
\bar r= r + \xi(r,\pp,\eps),\qquad
\bar \varphi =\varphi +\rho(\eps)+ r+ \sum_{i=2}^{m-1}p_i(\eps)r^i+ \eta(r,\pp,\eps),
\end{equation}
where $r>0$, $\rho(\eps)$ is the argument of the central multiplier $\lambda_{1,\eps}$ of $O$, and the functions $\xi$ and $\eta$ satisfy   \eqref{eq:kam_new_deri} and \eqref{8241}. 
Denote 
\begin{equation}\label{eq:localpoly}
p(r,\eps):=\rho(\eps)+ r+ \sum_{i=2}^{m-1}p_i(\eps)r^i.
\end{equation}
Condition \eqref{8241} implies that every KAM-curve $\gamma\in \mathcal{G}$ is given by $\{r=r_{\gamma,\eps}\}$ for some constant $r_{\gamma,\eps}>0$ depending on $\gamma$ and $\eps$, and 
the restriction $F|_{\gamma_\eps}$ is given by
\begin{equation}\label{eq:kam_onit_SC}
\begin{aligned}
\bar r = r_{\gamma,\eps} ,\qquad
\bar \varphi = \varphi + p(r_{\gamma,\eps},\eps),
\end{aligned}
\end{equation}
where  $p(r_{\gamma,\eps},\eps)$ is equal to the rotation number $\rho_\gamma$ of $\gamma$. 
 Solving $p(r_{\gamma,\eps},\eps)=\rho_\gamma=p(r_{\gamma,0},0)$, we obtain
\begin{equation}\label{eq:rotnr}
\begin{aligned}
r_{\gamma,\eps}=r_{\gamma,0} - (\rho(\eps) -\rho(0) ) +  h(r_{\gamma,0},\eps),
\end{aligned}
\end{equation}
where 
$$
h = O(r^2_{\gamma,0} \cdot \eps),\qquad
\dfrac{\p h}{\p r}=O(r_{\gamma,0}\cdot \eps),\qquad
\dfrac{\p h}{\p \eps}=O(r^2_{\gamma,0}).
$$

\begin{rem}
The difference between the equations \eqref{eq:localmap_SC} and \eqref{eq:kam_new} is that $\rho(\eps)$ is now the argument of the central multipliers of $O$, rather than the rotation number of a KAM-curve, so it can change as parameters vary.
\end{rem}

\subsubsection{Homoclinic orbit and scattering map}\label{sec:SCsattering}
We now assume that the saddle-center $O$ has a homoclinic orbit $\Gamma$ satisfying the following partial hyperbolicity condition:

\noindent\textbf{(}\setword{\textbf{C2}}{word:C2}\textbf{)} For any point $M\in \Gamma$, 
\begin{equation}\label{eq:condition:trans2}
 T_M W^\s(O) \oplus  T_M W^\u(O) \oplus  T_M (W^\s(\mathbb{D})\cap W^\u(\mathbb{D})) = \mathbb{R}^{2N}.
\end{equation}
\noindent This condition is essentially the same as \eqref{eq:condition:trans1}, allowing us to define the scattering map as in Section~\ref{sec:scattering}. Specifically, we  take two homoclinic points $M^-\in W^\u_{\loc}(O)\cap\Gamma$ and $M^+\in W^\s_{\loc}(O)\cap\Gamma$ such that $f^n(M^-)=M^+$ for some $n\in\mathbb{N}$; then we take their  neighborhoods $\Sigma^-(M^-)$ and $\Sigma^+(M^+)$ inside the two-dimensional intersection $W^\u(\mathbb{D})\cap W^\s(\mathbb{D})$, and define the holonomy maps   $\pi^\s:\Sigma^+ \to \mathbb{D}$ and $ \pi^\u:\Sigma^- \to \mathbb{D}$ along the leaves of $\mathcal{F}^\ss$ and $\mathcal{F}^\uu$, respectively; finally, with the transition map $T_1:=f^n$, we define the scattering map $S$ by~\eqref{eq:scattering_3}. As before, $S$ is a symplectic $C^{s'-1}$-diffeomorphism.

When the system is perturbed, the homoclinic orbit $\Gamma$ may disappear, but the scattering map is still defined. So, we can write $S$ as 
\begin{equation}\label{eq:scattering_SC}
\begin{aligned}
\bar u &= \mu+ b_{11}u + b_{12}v +\dots,\qquad
\bar v &=\nu+ b_{21}u + b_{22}v +\dots,
\end{aligned}
\end{equation} 
where $S(O)=(\mu,\nu)$ and the dots denote the higher order terms. In particular, $\mu=\nu=0$ when the homoclinic orbit exists. The last genericity condition is

\noindent \textbf{(}\setword{\textbf{C3}}{word:C3}\textbf{)} The  matrix $L:=\D S(O)= \begin{pmatrix}
b_{11}&b_{12}\\b_{21}&b_{22}
\end{pmatrix}$ is not a rotation.

\noindent This condition does not depend on the choice of  coordinates, since the only coordinate transformations which keep $F$ in the Birkhoff normal form are rotations. In the proof we will consider the  coordinates given by the following 

\begin{lem}\label{lem:rotaton}
If condition~\ref{word:C3} is satisfied, then the Fenichel coordinates can be chosen such that  ${b}^2_{11}+{b}^2_{12}\neq {b}^2_{21}+{b}^2_{22}$.
\end{lem}
\begin{proof}
It suffices to show that 
 there exists a rotation matrix $R$ such that in the new coordinates $(\hat u,\hat v)^\T=R\cdot (u,v)^\T$, the elements of $\hat L =R L R^{-1} $ satisfy
$\hat{b}^2_{11}+\hat{b}^2_{12}\neq \hat{b}^2_{21}+\hat{b}^2_{22}$,
or, equivalently, the two row vectors of $RLR^{-1}$ have different lengths. Assume the contrary.
Since the right multiplication by $R^{-1}$ only rotates the rows  of $RL$, we thus  suppose that the row vectors  of $RL$  have the same length for every rotation matrix $R$. A direct computation then shows that we must  have ${b}^2_{11}+{b}^2_{12}={b}^2_{22}+{b}^2_{22}$ and $b_{11}b_{21}+b_{12}b_{22}=0$. This implies that $L L^\T=c\cdot \mathrm{id}$ for some constant $c$. Since $\det L=1$ (being symplectic), $c=1$ and hence $L$ is a rotation, a contradiction.
\end{proof} 

\begin{rem}\label{rem:LTL}
Since $L^\T L$ is symplectic,  symmetric and positive definite, it has two positive real eigenvalues  $\lambda$ and $\lambda^{-1}$ with $\lambda\geq 1$. Condition~\ref{word:C3} means $\lambda>1$; otherwise $L^\T L$, being diagonalizable, would be the identity matrix, implying that $L$ is a rotation.
\end{rem}

\subsubsection{Unfolding families for Theorems~\ref{thm:main_SC_para} and \ref{thm:main_SCblender}}\label{sec:properunfolding}
Recall that $\rho$ is the argument of the central multipliers of $O$ (see \eqref{eq:Orho}), and $L$ is the matrix in condition \ref{word:C3}. 
Consider the following three unfolding families (we denote by $\mathrm{tr}$  the trace of a matrix):

\noindent \textbf{(}\setword{\textbf{H1}}{word:H1}\textbf{)}
A  family $\{f_\eps\}$ with at least 4 parameters    is called a {\em proper unfolding family}  if  the matrix
$ \left.\frac{\partial(\rho, \tr L^\T L, \mu,\nu)}{\partial \eps}\right|_{\eps=0}
$
has full rank, where $\mu$ and $\nu$ are the splitting parameters of the homoclinic orbit $\Gamma$ and they are defined by \eqref{eq:scattering_SC}. 

\noindent \textbf{(}\setword{\textbf{H2}}{word:H2}\textbf{)}
A family $\{f_\eps\}$ with at least 2 parameters is called  {\em tangency-unfolding}  if
the matrix
$ \left.\frac{\partial(\rho,\tr L^\T L)}{\partial \eps}\right|_{\eps=0}
$
has full rank, and $\mu=\nu=0$ for all $\eps$ (i.e., the homoclinic orbit persists). 

\noindent \textbf{(}\setword{\textbf{H3}}{word:H3}\textbf{)}
A  family  $\{f_\eps\}$ with   at least 3 parameters is called   {\em homoclinic-unfolding}   if the matrix
$ \left.\frac{\partial(\rho,\mu,\nu)}{\partial \eps}\right|_{\eps=0}
$
has full rank.

We will show that, at $\eps=0$, there exists a KAM-curve $\gamma\in \mathcal{G}$ having heteroclinic tangencies with two other KAM-curves in  $\mathcal{G}$ (see Lemma~\ref{lem:tangency_KAM}), and then prove that
changing $\rho$ and $\tr L^\T L$ unfolds independently these heteroclinic tangencies and hence produces two homoclinic tangencies of $\gamma$ (see Lemma~\ref{lem:tangenfamily}). As a result, case (2) of Theorem~\ref{thm:main_2punfolding} applies and gives a symplectic blender  connected to $\gamma$, proving Theorem~\ref{thm:main_SCblender}. 
In Section~\ref{sec:homounfold}, we use the found  blender to connect the manifolds of $O$ by changing $\mu$ and $\nu$ within a homoclinic-unfolding family, proving Theorem~\ref{thm:main_SC_para}.

\subsection{Heteroclinic tangencies between KAM-curves}\label{sec:tangenunfold}
Here, we prove
\begin{lem}\label{lem:tangency_KAM}
Given any $\delta>0$, there exist KAM-curves $\gamma,\gamma_{1},\gamma_2\in \mathcal{G}$ such that  the radius of $\gamma$ is smaller than $\delta$ and  $S(\gamma)$ is quadratically tangent to $\gamma_1$ and $\gamma_2$. That is, by Lemma~\ref{lem:equitangency}, there are 
partially-hyperbolic quadratic  tangencies between $W^\u(\gamma)$ and $W^\s_{\loc}(\gamma_1)$, and between $W^\u(\gamma)$ and $W^\s_{\loc}(\gamma_2)$.
\end{lem}

It is proved by Lerman and Markova \citep{LerMar:15} that when conditions \ref{word:C1}--\ref{word:C3} are satisfied, the image $S(\gamma)$ of each KAM-curve $\gamma\in\mathcal{G}$ is an ellipse that intersects $\gamma$ transversely at four points (see Figure~\ref{fig:chain}(a)); so, $\gamma$ has exactly four primary homoclinic orbits. In fact, there exists $\kappa>0$ such that, for all small $r$, the image of the circle of radius $r$ intersects transversely every circle of radii between $(1-\kappa) r$ and $(1+\kappa) r$. Since $\bigcup_{\gamma\in\mathcal{G}}\gamma$ has Lebesgue density 1 at $O$, this implies the following

\begin{lem}\label{lem:chain}
Up to shrinking $\DD$, for any two KAM-curves $\gamma, \gamma'\in \mathcal{G}$, there exist two heteroclinic chains $\{\gamma^+_i\in \mathcal{G}\}_{i=1}^{n^+}$  and $\{\gamma^-_i\in \mathcal{G}\}_{i=1}^{n^-}$, with $\gamma^+_1=\gamma^-_1=\gamma$ and $\gamma^+_{n^+}=\gamma^-_{n^-}= \gamma'$, such that $S(\gamma^+_i)\pitchfork \gamma^+_{i+1}\neq \emptyset$ for $i=1,\dots,n^+-1$, and $S^{-1}(\gamma^-_i)\pitchfork \gamma^-_{i+1}\neq \emptyset$ for $i=1,\dots,{n^-}-1$.
\end{lem}
\begin{figure}[!h]
\begin{center}
\includegraphics[scale=1.4]{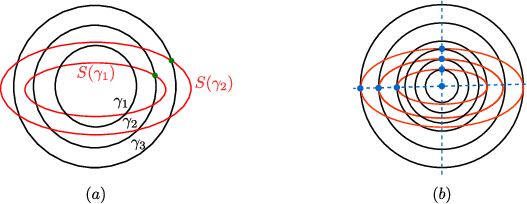}
\end{center}
\caption{(a) A heteroclinic chain connecting $\gamma_1$ and $\gamma_3$ via $S$. (b) The circles are tangent to their images (in red) along two smooth curves (in blue).}
\label{fig:chain}
\end{figure}

Using \eqref{eq:bcoor}, we  rewrite the scattering map \eqref{eq:scattering_SC}  at $\eps=0$ in the polar coordinates as $S:(r,\pp)\mapsto (\bar r,\bar \pp)$, where
\begin{equation}\label{eq:scat_polar}
\bar r = r\left((b_{11}\cos\pp +b_{12}\sin\pp)^2+(b_{21}\cos\pp +b_{22}\sin\pp)^2+O(\sqrt{r})\right).
\end{equation}
Since  all  KAM-curves   in  $\mathcal{G}$ are straightened,  the circles centered at $O$ form a smooth foliation $\mathcal{F}$ of $\mathbb{D}$ that contains $\mathcal{G}$.

\begin{lem}\label{lem:tangency_2family}
If  condition \ref{word:C3} is satisfied, then given the circle $C_r\in\mathcal{F}$ with a sufficiently small radius $r$, the image $S(C)$ is quadratically tangent to two circles with the radii
\begin{equation}\label{eq:radii}
r^+=\lambda  r+O(r^\frac{3}{2}) \quad \mbox{and}\quad r^-=\lambda^{-1}  r+ O(r^\frac{3}{2}),
\end{equation}
where $\lambda>1$ is the maximal eigenvalue of $L^\T L$.
The two corresponding sets of the tangency points are smooth curves $\ell^+$ and $\ell^-$ which intersect transversely at $O$ and intersect 
transversely the circles in $\mathcal{F}$ (see Figure \ref{fig:chain}(b)).
\end{lem}

\begin{proof}
The image $S(C_r)$ has a quadratic  tangency to $\mathcal{F}$ if and only if there exists $\varphi_r\in \mathbb{S}^1$ such that  $\partial \bar r(r,\varphi_r)/\partial\varphi=0$ and $\partial^2 \bar r(r,\varphi_r)/\partial \varphi^2\neq 0$. 
Let us rewrite  \eqref{eq:scat_polar} as
\begin{equation}\label{eq:g1+g2}
\bar r =r g(r,\pp),
\end{equation}
where
\begin{equation}\label{eq:gnew}
g(r,\pp)=(b_{11}\cos\pp +b_{12}\sin\pp)^2+(b_{21}\cos\pp +b_{22}\sin\pp)^2 + O(\sqrt{r})).
\end{equation}
Obviously, the sought tangencies are given by non-degenerate extrema of $g$. Thus, it suffices to show that 
\begin{equation}\label{eq:g1}
g_0(\pp)=(b_{11}\cos\pp +b_{12}\sin\pp)^2+(b_{21}\cos\pp +b_{22}\sin\pp)^2
\end{equation}
has a non-degenerate minimum and a non-degenerate maximum on $[0,\pi)$, and the corresponding minimal and maximal values are 
$\lambda^{-1}$ and $\lambda$.

The existence of the non-degenerate extrema becomes obvious once we
rewrite \eqref{eq:g1} as
$$
g_0(\varphi)= \dfrac{1}{2}(b^2_{11}+b^2_{12}+b^2_{21}+b^2_{22})+\dfrac{1}{2}(b^2_{11}+b^2_{21}-b^2_{12}-b^2_{22})\cos 2\pp+(b_{11}b_{12}+b_{21}b_{22})\sin 2\pp.$$
By \eqref{eq:gnew}, the sought curves $\ell^\pm$ are the images by $S$ of the graphs of some functions 
$$
\pp=\pp^\pm + O(\sqrt{r}),\qquad
\pp=\pp^\pm +\pi + O(\sqrt{r}),
$$
where $\pp^\pm$ are the two extremal points of $g_0$ on $[0,\pi)$.

To evaluate the corresponding extremal values, one just notes 
from \eqref{eq:bcoor} that
\[
g_0(\pp)= \frac{1}{u^2+v^2} (u,v)\cdot L^\T L \cdot(u,v)^\T,
\]
and hence the minimum and maximum of $g_0$ equal to the eigenvalues of $L^\T L$. 
\end{proof}

By the above lemma, the image $S(\gamma)$ of a KAM-curve $\gamma\in \mathcal{G}$ has a quadratic tangency with KAM-curves $\gamma_{1,2} \in \mathcal{G}$ when $S(\gamma)\cap \ell^+ \cap\gamma_1\neq \emptyset$ and $S(\gamma)\cap \ell^- \cap \gamma_2 \neq \emptyset$. So,  denoting $G:=\bigcup_{\gamma\in \mathcal{G}} \gamma$ and $S(G):=\bigcup_{\gamma\in \mathcal{G}} S(\gamma)$,
Lemma~\ref{lem:tangency_KAM} is reduced to the claim that
the set 
$\ell^+\cap S(G)$ has an non-empty intersection with $\ell^+\cap G$
arbitrarily close to $O$, and the set $\ell^-\cap S(\gamma)$ has an non-empty intersection with $\ell^-\cap G $ arbitrarily close to $O$. This, in turn, follows since 
$G$ and $S(G)$ have Lebesgue density 1 at $O$,
thus completing the proof of Lemma~\ref{lem:tangency_KAM}.

\subsection{Creation of coexisting quadratic homoclinic tangencies: proof of Theorem~\ref{thm:main_SCblender}}
Since the property of being a symplectic blender connected to some whiskered KAM-torus is $C^1$-open (in the space of symplectic diffeomorphisms), to prove Theorem~\ref{thm:main_SCblender}, it suffices to consider a tangency-unfolding family   $\{f_\eps\}$ with two parameters. In this case, the first requirement in \ref{word:H2} reads
\begin{equation}\label{eq:gcfamily_SC}
 \det\left.\dfrac{\partial(\rho,\tr L^\T L)}{\partial \eps}\right|_{\eps=0}\neq 0.
\end{equation}

Recall that every KAM-curve in $\mathcal{G}$ persists  in the sense that every  symplectic map close to $f$ has a KAM-curve with the same rotation number. Thus, we say that a quadratic tangency between $S(\gamma_1)$ and $\gamma_2$ persists under a perturbation if the continuations of the two curves in the perturbed system have a quadratic tangency close to the original tangency point.  
We  show below that the two heteroclinic tangencies between KAM-curves given by Lemma~\ref{lem:tangency_KAM}  are unfolded independently within the family $\{f_{\eps}\}$: there exist two smooth curves in a small neighborhood of 0 in the parameter plane such that they intersect transversely at $0$ and that each tangency persists on one of the curves  and disappears elsewhere in the neighborhood.

\begin{lem}\label{lem:tangenfamily}
If  $\mathbb D$ is sufficiently small, then for any KAM-curves $\gamma,\gamma_{1},\gamma_2$ given by Lemma \ref{lem:tangency_KAM}, the family $\{f_{\eps}\}$ unfolds independently  the   tangencies between $S(\gamma)$ and   $\gamma_{1},\gamma_{2}$.
\end{lem}

\begin{proof}
Denote by $\gamma_{\eps}$ and $\gamma_{i,\eps}$ ($i=1,2$) the continuations of the corresponding KAM-curves (so, $\gamma_0=\gamma$ and $\gamma_{i,0}=\gamma_i)$, and by $r_{*,\eps}$ and $r_{i,\eps}$ their radii. Denote by  $\lambda_\eps$ and $\lambda_\eps^{-1}$  the two eigenvalues of $L^\T L$.
Since we have $S(O)=O$ for a tangency-unfolding family, formulas in 
 \eqref{eq:radii} hold for all small $\eps$. Hence, $\gamma_{\eps}$ is tangent to two circles with radii
\begin{equation*}
r^+_{*,\eps}=\lambda_\eps  r_{*,\eps}+O(r_{*,\eps}^\frac{3}{2}),\qquad  r_{*,\eps}^-=\lambda_\eps^{-1}  r_{*,\eps}+ O(r_{*,\eps}^\frac{3}{2}),
\end{equation*}
where $r^+_{*,0}=r_{1,0}$ and $r^-_{*,0}=r_{2,0}$. We define
\begin{equation}\label{eq:unfold2:1}
\Delta_1(\eps)=r^+_{*,\eps}-r_{1,\eps},\qquad
\Delta_2(\eps)=r^-_{*,\eps}-r_{2,\eps}.
\end{equation}
By construction, for each $i$, $S(\gamma_{\eps})$ is tangent to $\gamma_{i,\eps}$ if and only if $\Delta_i(\eps)=0$. In what follows we prove  
\begin{equation*}
\det \left.\dfrac{\p(\Delta_1,\Delta_2)}{\p(\eps_1,\eps_2)}\right|_{\eps = 0}\neq 0,
\end{equation*}
which immediately gives the lemma.

Condition \eqref{eq:gcfamily_SC} allows us to take
\begin{equation}\label{eq:unfold2:2}
\eps_1=\rho-\rho_0,\qquad
\eps_2=\tr L^\T L -\tr L_0^\T L_0,
\end{equation} 
where $\rho_0$ and $L_0$ denote  the corresponding values for $f_0=f$.
It  follows from \eqref{eq:rotnr} that for all small $\eps$
\begin{equation}\label{eq:radius0}
r_{\alpha,\eps}=r_{\alpha,0}-\eps_1+O(r_{\alpha,0}^2 \cdot \eps   ), \qquad \alpha=*,1,2.
\end{equation}

We next compute the derivatives involved in the Jacobin matrix.
Since $\tr L^\T L = \lambda_\eps + \lambda^{-1}_\eps$, the choice of $\eps_2$ implies that $A:=(\p\lambda_\eps /\p \eps_2)|_{\eps=0}$ is non-zero. Since ${\p \tr L^\T L }/{\p \eps_1}=0$, we must have  $(\p\lambda_\eps /\p \eps_1)|_{\eps=0}=0$, for otherwise $\lambda_0^2=1$ contradicting condition \ref{word:C3} (see Remark~\ref{rem:LTL}). This together with \eqref{eq:unfold2:1} and the above expressions for $r_{*,\eps}$ and $r_{i,\eps}$ leads to 
\begin{align*}
\det \left.\dfrac{\p(\Delta_1,\Delta_2)}{\p (\eps_1,\eps_2)}\right|_{\eps=0}
&=
\left|
\begin{matrix}
(1-\lambda_0)+O(\sqrt{r_{*,0}})&
A r_* +O(r_{*,0}^2)\\
(1-\lambda_0^{-1})+O(\sqrt{r_{*,0}})&
-\dfrac{A}{\lambda_0^2} r_*+ O(r_{*,0}^2)
\end{matrix}
\right|\\
&=-\dfrac{(\lambda_0-1)^2}{\lambda_0^2} A r_{*,0}+ O(r_{*,0}^\frac{3}{2}).
\end{align*}
Since $\lambda_0> 1$ by Remark~\ref{rem:LTL}, this determinant is non-zero for all sufficiently small $r_{*,0}$. 
\end{proof}

\begin{proof}[Proof of Theorem~\ref{thm:main_SCblender}]
Lemma~\ref{lem:tangenfamily} together with Lemma~\ref{lem:equitangency} implies that the family $\{f_\eps\}$ unfolds independently the partially-hyperbolic quadratic tangencies between $W^\u(\gamma)$ and $W^\s(\gamma_1)$ and between $W^\u(\gamma)$ and $W^\s(\gamma_2)$. On the other hand, by Lemma~\ref{lem:chain}, $\gamma$ can be connected to $\gamma_{1,2}$ by the $S^{-1}$ heteroclinic chains of KAM-curves; then, applying Lemma~\ref{lem:incli}  repeatedly shows that $W^\s(\gamma)$ accumulates on $W^\s(\gamma_1)$ and $W^\s(\gamma_2)$ in the $C^{m^*}$ topology, where $m^*\geq 2$ when $f_\eps$ are sufficiently smooth.  It follows that an arbitrarily small change of $\eps$ creates two  quadratic homoclinic tangencies of $\gamma$, which, by construction, unfold independently under further change of $\eps$. It is clear that an independent unfolding of two homoclinic tangencies  is generic in the sense of~\eqref{eq:unfold2quadra}.

Note that the partial hyperbolicity condition~\eqref{eq:condition:trans2} remains valid  for all orbits sufficiently close to $\mathcal{O}(O)\cup \Gamma$; in particular, it holds for the homoclinic tangencies to $\gamma$ which we just  obtained. We thus proved

\begin{prop}\label{prop:tangenfamily}
Let  $\{f_\eps\}$ be a two-parameter tangency-unfolding family. Given any $\delta>0$, there exist a KAM-curve $\gamma\in\mathcal{G}$ with radius smaller than $\delta$ and  a sequence $\{\eps_j\}$ converging to $\eps=0$ such that the continuation of $\gamma$ for each $f_{\eps_j}$ has two partially-hyperbolic quadratic homoclinic tangencies that unfold independently.
\end{prop}

Theorem~\ref{thm:main_SCblender} then follows  from Proposition~\ref{prop:tangenfamily} and case (2) of Theorem~\ref{thm:main_2punfolding}.
\end{proof}

\subsection{Persistent saddle-center homoclinics: proof of Theorem~\ref{thm:main_SC_para}\label{sec:homounfold}}

We will connect the manifolds of $O$ by using a symplectic blender found from   Theorem~\ref{thm:main_SCblender}.
The key step   is to show that, within a homoclinic-unfolding family, one can create a heteroclinic connection between $O$ and the whiskered KAM-torus that is connected to a symplectic blender.

By condition~\ref{word:H3}, for a $d$-parameter homoclinic-unfolding family, we can take 
\begin{equation}\label{eq:eps}
\eps=(\Delta \rho,\mu,\nu,\eps'),
\end{equation}
where $\Delta \rho=\rho_\eps-\rho_0$  and  $\eps'\in \mathbb{R}^{d-3}$. 

\begin{lem}\label{lem:Otogamma}
For any homoclinic-unfolding family $\{f_\eps\}$ and any given KAM-curve $\gamma\in \mathcal{G}$, there exists a sequence $\{\eps_j\}$ converging to $0$ such that $W^\u(O_{\eps_j}) \cap W^\s(\gamma_{\eps_j})\neq \emptyset$ and $W^\s(O_{\eps_j})\cap W^\u(\gamma_{\eps_j})\neq \emptyset$. Moreover, the two  intersections unfold independently as $\mu$ and $\nu$ vary from $\mu_j$ and $\nu_j$.
\end{lem}

\begin{proof}
Recall that we consider the  coordinates where the KAM-curves in $\mathcal{G}$ are circles (see Section~\ref{sec:localSC}). 
Take any sequence $\{\gamma_j\in \mathcal{G}\}$ of KAM-curves   with radii $r_j\to 0$. Denote the continuations of  $\gamma$ and $\gamma_j$ by $\gamma_\eps$ and $\gamma_{j,\eps}$. 
By Lemma~\ref{lem:chain}, there exists for each large $j$ and  all small $\eps$  heteroclinic chains  from $\gamma_{j,\eps}$ to $\gamma_{\eps}$ via $S$ and $S^{-1}$.    Applying Lemma~\ref{lem:incli} to these chains, we see that $W^\s(\gamma_\eps)$ and $W^\u(\gamma_\eps)$ accumulate on $W^\s_{\loc}(\gamma_{j,\eps})=\{r=r_{j,\eps},y=0\}$ and, respectively, on $W^\u_{\loc}(\gamma_{j,\eps})=\{r=r_{j,\eps},x=0\}$,
with their  first derivatives with respect to variables and parameters. Thus, to prove the lemma, it sufficies to show the existence of $\eps_j\to 0$ for which
 \begin{equation}\label{eq:Otogamma0}
W^\u(O)\cap W^\s_\loc(\gamma_{j,\eps_j})\neq\emptyset,\qquad
W^\s(O)\cap W^\u_\loc(\gamma_{j,\eps_j})\neq\emptyset,
\end{equation}
and the two intersections unfold independently as $\mu$ and $\nu$ vary from $\mu_j$ and $\nu_j$.

It further suffices to find $\eps_j\to 0$ for which both $S(O)$ and $S^{-1}(O)$ lie in $ \gamma_{j,\eps_j}$, and, moreover, they can be independently moved away from  $ \gamma_{j,\eps_j}$  by changing $\mu$ and $\nu$.
A direct computation using \eqref{eq:scattering_SC} shows that $S^{-1}(O)=(b_{12}\nu -b_{22}\mu, b_{21}\mu - b_{11}\nu)+O(\mu^2+\nu^2)$. So, by \eqref{eq:bcoor}, the $r$-coordinates of $S(O)$ and $S^{-1}(O)$ are
\begin{equation}\label{eq:rurs}
r^\u =\dfrac{1}{2}(\mu^2 + \nu^2),\quad
r^\s =\dfrac{1}{2}((b_{22}^2+b_{21}^2)\mu^2+(b_{12}^2+b_{11}^2)\nu^2-2(b_{12}b_{22}+b_{21}b_{11})\mu\nu)+\dots,
\end{equation}
where the dots denote cubic and higher order terms.
Let $r_{j,\eps}$ be the  radius of $\gamma_{j,\eps}$, which by \eqref{eq:rotnr} is given by
\begin{equation}\label{eq:radius2}
r_{j,\eps}=r_{j,0}-\Delta \rho + h(r_{j,0},\eps),
\end{equation}
where $h=O(r^2_{j,0}\cdot \eps)$. We need to find for each large $j$  a parameter value $\eps_j=(\Delta\rho_j,\mu_j,\nu_j,0)$  such that 
\begin{equation}\label{eq:Otogammanew:1}
r^\u=r^\s=r_{j,\eps_j},\qquad
\left.\det\dfrac{\p(r^\u,r^\s)}{\p (\mu,\nu)}\right|_{\eps=\eps_j}\neq 0.
\end{equation}
Since $\det\frac{\p(r^\u,r^\s)}{\p (\mu,\nu)}=(b_{12}^2+b_{11}^2-b_{22}^2-b_{21}^2)\mu\nu+\dots$, it follows from  Lemma~\ref{lem:rotaton} that this determinant is non-zero 
for all sufficiently small non-zero $\mu$ an $\nu$. Thus, it further suffices to find $\eps_j$ with $\mu_j\neq 0$ and $\nu_j\neq 0$ that solves \eqref{eq:Otogammanew:1}.

Observe that, for every sufficiently large $j$, the equation 
$$\Delta \rho - h(r_{j,0},\Delta \rho,\mu,\nu,0)=0$$
 admits a solution $\Delta\rho=\hat \rho_j(\mu,\nu)$ defined near 0. Substituting this into \eqref{eq:radius2}, we see that to solve $r^\u=r^\s=r_{j,\eps_j}$ it suffices to find $\mu_j$ and $\nu_j$ such that
$r^\u=r_{j,0}$ and $r^\s=r_{j,0}$, and then set $\eps_j=(\hat\rho_j(\mu_j,\nu_j),\mu_j,\nu_j,0)$. By \eqref{eq:rurs}, it is enough to take $\mu_j =\sqrt{2r_{j,0}}\cos\theta$ and $\nu_j =\sqrt{2r_{j,0}}\sin\theta$ for some $\theta$ satisfying $\theta\bmod \pi\notin \{0,\pi/2\}$ and solving $r^\s=r_{j,0}$, that is,
\begin{align*}
&\dfrac{1}{2}(b_{22}^2+b_{21}^2-b_{12}^2-b_{11}^2)\cos 2\theta
- (b_{12}b_{22}+b_{21}b_{11})\sin 2\theta \\
&\qquad\qquad\qquad\qquad\qquad\qquad = 1 - \dfrac{1}{2}(b_{22}^2+b_{21}^2+b_{12}^2+b_{11}^2)+O(\sqrt{r_{j,0}}).
\end{align*}
It is easy to see that this equation admits such a solution  when $b_{22}^2+b_{21}^2+b_{12}^2+b_{11}^2=\tr L^\T L>2$, which holds automatically by Remark~\ref{rem:LTL}.
\end{proof}

\begin{prop}\label{prop:homofamily}
Let $f\in\symp^s(\mathcal{M})$, $s=\infty,\omega$, have a generic saddle-center point $O$ with a generic homoclinic orbit and a symplectic blender  connected to some KAM-torus $\gamma_*\in \mathcal{G}$ that is sufficiently close to $O$. 
For any homoclinic-unfolding family $\{f_\eps\}$, in any neighborhood of $\eps=0$ in the parameter space there exists an open set $\mathcal{E}$ such that the set of $\eps$  for which the continuation of the saddle-center  has a homoclinic orbit is dense in $\mathcal{E}$. 
\end{prop}

\begin{proof}
Let $\eps_j$ be given by Lemma~\ref{lem:Otogamma} such that, at $\eps=\eps_j$,  a piece $W^\u$ of $W^\u(O)$  intersects $W^\s(\gamma_*)$ and a piece $W^\s$ of $W^\s(O)$  intersects $W^\u(\gamma_*)$. Due to the partial hyperbolicity, up to replacing $W^\u$ by $f^n(W^\u)$ with a sufficiently large $n$, we can assume that $W^\u$ is  $C^1$-close to a local unstable leaf of $W^\u_\loc(\gamma_*)$. Similarly,  $W^\s$ can be taken  $C^1$-close to a local stable leaf of $W^\s_\loc(\gamma_*)$. Since these intersections unfold independently by Lemma~\ref{lem:Otogamma}, the condition \eqref{eq:gcpersis} of Proposition~\ref{prop:persis} is fulfilled with  $L^\u_\eps:=W^\u_{\eps+\eps_j}$ and $L^\s_\eps:=W^\s_{\eps+\eps_j}$.  The lemma then follows  immediately from the second claim of Proposition~\ref{prop:persis}.
\end{proof}

\begin{proof}[Proof of Theorem~\ref{thm:main_SC_para}]
Note that a proper unfolding family (see \ref{word:H1}) restricted to $\mu=\nu=0$ is tangency-unfolding (see \ref{word:H2}), so we can apply Theorem~\ref{thm:main_SCblender}  to obtain a sequence $\eps_j\to 0$ such that $f_{\eps_j}$ has a symplectic blender connected to a KAM-curve, which can be arbitrarily close to $O$. 
Note that the homoclinic to $O$ persists since we took $\mu=\nu=0$.
Because the genericity conditions \ref{word:C1}--\ref{word:C3} are all robust under small symplectic perturbations, they are  satisfied for all $f_{\eps_j}$ with $j$ sufficiently large. 
  Moreover, since $\{f_\eps\}$ is  a homoclinic-unfolding family, the family $\{f_{j,\eps}\}_\eps:=\{f_{\eps+\eps_j}\}$ for every large $j$ is homoclinic-unfolding as well (note that $\{f_{j,\eps}\}$ is the same family $\{f_\eps\}$ with the origin in the parameter space shifted to $\eps_j$). Thus, we can apply Proposition~\ref{prop:homofamily} to each family $\{f_{j,\eps}\}$ to obtain the sought sequence $\mathcal{E}_{j}$ of Theorem~\ref{thm:main_SC_para}. 
\end{proof}

\appendix

\renewcommand{\theHsection}{A\arabic{section}}
\setcounter{lema}{0}                       
\renewcommand\thelema{A\arabic{lema}}   

\setcounter{equation}{0}
\renewcommand{\theequation}{A\arabic{equation}}
\renewcommand{\thesubsection}{A\arabic{subsection}}
\section*{Appendix}

\subsection{Approximation lemmas}
\subsubsection{Iterations of the local map $T^k_0$: proof of Lemma~\ref{lem:derisyst}}\label{sec:T0k}

We  extend  $T_0$ to the map $(r,\pp,\eps,x,y)\to(\bar r,\bar \pp,\bar \eps,\bar x,\bar y)$, where $\eps$ are now additional variables under the action of identity map. Let us denote $w:=(r,\pp,\eps)$. Then  the cross-form \eqref{eq:T0_xform} of the new $T_0$ can  be written as
\begin{equation}\label{eq:T0_xform123}
\begin{aligned}
\bar w &=F(w)+g(w,x,\bar y),\\
\bar x &=  g_3(w,x,\bar y),\qquad
 y =  g_4(w,x,\bar y),	
\end{aligned}
\end{equation}
where $F(w)=(F_1(w),F_2(w),\eps)$ and $g(w,x,y)=(g_1(w,x,y), g_2(w,x,y),0)$, and the functions satisfy
\begin{equation}\label{eq:ggg123}
\begin{aligned}
g(w,0,\bar y)\equiv 0, \quad g(w,x,0) \equiv 0, \quad
g_3(w,0,\bar y)\equiv 0,\quad g_4(w,x,0)\equiv 0,\\
\left\|\dfrac{\partial F }{\partial w}\right\|<\cc, \qquad
\left\|\left(\dfrac{\partial F}{\partial w}\right)^{-1}\right\|<\cc, \qquad
\left\| \dfrac{\partial g_3}{\partial x} \right\| < \lambda, \qquad \left\| \dfrac{\partial g_4}{\partial \bar y} \right\| < \lambda.
\end{aligned}
\end{equation}

 Denote by $V_\delta$ the $\delta$-neighborhood of the cylinder $\A$.
\begin{lema}[{\citep[Lemma 1]{GelTur:17}}]\label{lem:GelTur:17}
 If \eqref{eq:ggg123} holds, then there exists $\delta_0$ such that for any $\delta\in(0,\delta_0)$ and any $k\geqslant 0$ the following results hold:
\begin{enumerate}
\item Any orbits of length $k$ with $(w_j,x_j,y_j)=T^j_0(w_0,x_0,y_0)\in V_\delta$ for $j=0,\dots,k$ satisfy  
\begin{equation}\label{eq:gt17:1}
\begin{array}{c}
\|x_j\|\leqslant \delta\lambda^j, \quad \|y_i\|\leqslant \delta\lambda^{k-j},\quad
 \|w_j-F^j(w_0)\|\leqslant \delta(\cc\lambda)^{k/2},
\end{array}
\end{equation}
\item The orbit $(w_j,x_j,y_j)$ is determined uniquely for any given $(w_0,x_0,y_k)$ provided that $\|x_0\|,\|y_k\|<\delta$ and $w_0\in \A$. In particular, $w_k,x_k,y_0$ are smooth functions of $w_0,x_0,y_k$.
\end{enumerate}
\end{lema}

This lemma gives the $C_0$ estimates in Lemma~\ref{lem:derisyst}. To find estimates for higher derivatives, the  idea  is to consider a new  system which is the combination of \eqref{eq:T0_xform123} and its formal first  derivatives. We verify condition \eqref{eq:intro:1} for this system so that Lemma~\ref{lem:GelTur:17} is applicable, giving the $C^0$ estimates for the new system and hence $C^1$ estimates for the original system. This procedure can be repeated to achieve the estimates of derivatives up the order equal to $s'-1$ (where $s'$ is the smoothness of our system in Fenichel coordinates).
Let us start to prove the estimates for the first derivatives in the Lemma~\ref{lem:derisyst}.

\noindent\textbf{(1) Systems with formal derivatives.} 
Consider the map 
$$\hat T_0:(w,x,y,w',x',y')\mapsto (\bar w,\bar x,\bar y,\bar w',\bar x',\bar y')$$
 given in the cross-form consisting of \eqref{eq:T0_xform123} and
\begin{equation}\label{eq:derisyst:3.0}
\begin{aligned}
\bar w'&=\dfrac{\partial F(w)}{\partial w}w'
+\left(\dfrac{\partial g(w,x,y)}{\partial w}w'
+\dfrac{\partial g(w,x,y)}{\partial x}x' +\dfrac{\partial g(w,x,y)}{\partial \bar y}\bar y'\right)\\
&=:F'(w,w')+g'(w,w',x,x',\bar y,\bar y'),\\
\bar x'&=\dfrac{\partial g_3(w,x,y)}{\partial w}w'
+\dfrac{\partial g_3(w,x,y)}{\partial x}x' +\dfrac{\partial g_3(w,x,y)}{\partial \bar y}\bar y'\\
&=:g'_3(w,w',x,x',\bar y,\bar y'),\\
y'&=\dfrac{\partial g_4(w,x,y)}{\partial w}w'
+\dfrac{\partial g_4(w,x,y)}{\partial x}x' +\dfrac{\partial g_4(w,x,y)}{\partial \bar y}\bar y'\\
&=:g'_4(w,w',x,x',\bar y,\bar y').
\end{aligned}
\end{equation}
Denote
\begin{equation}\label{eq:boldcoor}
\begin{aligned}
&\bm w=(w,w'),\quad \bm x=(x,x'),\quad \bm y=(y,y'),\\
&\bm F=(F,F'),\quad \bm g=(g,g'),\quad \bm g_i=(g_i,g'_i)\quad i=3,4.
\end{aligned}
\end{equation}
Then the map $\hat T_0$ assumes the form $(\bm w,\bm x,\bm y )\mapsto(\bar {\bm w},\bar{\bm x},\bar{\bm y} )$. Since $w=(r,\pp,\eps)$ with $(r,\pp)$ lying in the cylinder $\A$ and $\eps$ small, the variable $\bm w$ belongs to some closed bounded region $\hat A\subset \mathbb R^6$. Recall that $V$ is the neighborhood where \eqref{eq:T0_xform_deri} holds. Let $\hat V\subset\mathbb R^{4N+2} $ be a small neighborhood of $\hat A$ such that, the $(w,x,y)$ components of any point in $\hat V$  lie in $V$.
\begin{claim*}
There exists a norm $\|\cdot\|_\circ$ such that the following holds in $\hat V$:
\begin{equation}\label{eq:derisyst:3}
\begin{array}{c}
\bm g(\bm w,0,\bar{\bm y})\equiv 0,\quad \bm g(\bm w,\bm x,0)\equiv 0,\quad
\bm g_3(\bm w,0,\bar{\bm y})\equiv 0,\quad \bm g_4(\bm w,\bm x,0)\equiv 0,\\[5pt]
\left\|\dfrac{\partial\bm F}{\partial \bm w}\right\|_\circ<\cc, \qquad
\left\|\left(\dfrac{\partial\bm F}{\partial \bm w}\right)^{-1}\right\|_\circ<\cc, \qquad
\left\| \dfrac{\partial \bm g_3}{\partial \bm x} \right\|_\circ < \lambda, \qquad \left\| \dfrac{\partial \bm g_4}{\partial \bar{\bm y}} \right\|_\circ < \lambda.
\end{array}
\end{equation}
\end{claim*}
We postpone the proof of the claim until we prove the lemma for case $i=1$.

\noindent\textbf{(2) First derivatives with respect to $w_0$.} 
Let $V_\delta$ be the neighborhood where Lemma~\ref{lem:GelTur:17} holds. Consider any orbit segment $\{(w_j,x_j,y_j)\}_{j=0}^k\subset V_\delta$, where $w_j=(w^1_j,w^2_j,w^3_j)=(r_j,\pp_j,\eps)$. For any $i=1,2,3$,  we take
\begin{equation}\label{eq:derisyst:3.1}
\begin{array}{l}
w'_j=\dfrac{\partial w_j}{\partial w^i_0},\qquad
x'_j=\dfrac{\partial x_j}{\partial w^i_0},\qquad 
 y'_{j-1}=\dfrac{\partial y_{j}}{\partial w^i_0},\qquad (j=1,\dots,k)\\
w'_0=
\begin{cases}
(1,0,0) &\quad \mbox{if}\quad i=1,\\
(0,1,0) &\quad \mbox{if}\quad i=2,\\
(0,0,1) &\quad \mbox{if}\quad i=3,
\end{cases}
\qquad x'_0=0,\quad y'_k =0.
\end{array}
\end{equation}

One easily sees that this gives an orbit segment $\{(\bm  w_j,\bm x_j,\bm y_j)\}_{j=0}^k$ of $\hat T_0$. The identities for $g_{3,4}$ in \eqref{eq:ggg123} show that $x'_j = O(\delta)$ and $y'_j=O(\delta)$ for any orbit segment in  $V_\delta$. Thus, $\{(\bm  w_j,\bm x_j,\bm y_j)\}_{j=0}^k$ lies in  some $\hat \delta$-neighbourhood of $\hat A$ satisfying $\hat \delta\to 0$ as $\delta\to 0$. By \eqref{eq:derisyst:3}, we can apply Lemma \ref{lem:GelTur:17} to $\hat T_0$ with this orbit segment. 
 The first two estimates in \eqref{eq:gt17:1} give
$$
\left\|\bm x_k\right\|_\circ\leqslant\hat \delta \lambda^k
\quad \mbox{and}\quad 
\left\|\bm y_0\right\|_\circ\leqslant\hat \delta \lambda^k,
$$
which, by \eqref{eq:boldcoor}, implies
\begin{equation}\label{eq:derisyst:3.2}
\left\|\dfrac{\partial x_k}{\partial w^i_0}\right\|_\circ\leqslant\hat \delta \lambda^k
\quad \mbox{and}\quad 
\left\|\dfrac{\partial y_0}{\partial w^i_0}\right\|_\circ\leqslant\hat \delta \lambda^k,
\end{equation}
for any $\hat \delta\in(0,\hat \delta_0)$ with $\hat \delta_0$ given by Lemma \ref{lem:GelTur:17}.

By the last estimate in \eqref{eq:gt17:1}, we have
\begin{equation*}\label{eq:derisyst:4}
\|(w_k, w'_k)-\bm F^k(w_0, w'_0)\|_\circ\leqslant \hat \delta(\cc\lambda)^{\frac{k}{2}}.
\end{equation*}
We claim that
\begin{equation}\label{eq:derisyst:5}
\bm F^k(w_0, w'_0)=\left(F^k(w_0),\dfrac{\partial F^k(w_0)}{\partial r}\right).
\end{equation}
This together with \eqref{eq:derisyst:3.1}  yields
\begin{equation}\label{eq:derisyst:6}
\left\| \dfrac{\partial w_k}{\partial w^i_0}
-\dfrac{\partial F^k(w_0)}{\partial w^i_0}\right\|_\circ \leqslant \hat \delta(\cc\lambda)^{\frac{k}{2}}.
\end{equation}
Since all norms in finite dimensional spaces are equivalent, \eqref{eq:derisyst:3.2} and \eqref{eq:derisyst:6} give the desired estimates for the first derivatives with respect to $r_0,\pp_0,\eps$.

Let us now prove the claim by induction on $k$. 
The initial case for $k=1$ is immediate. We assume that the claim holds for $k-1$. By \eqref{eq:derisyst:3.0}, we can then write 
\begin{align*}
\bm F^k(w_0, w'_0) &= \left(F^k(w_0),F'(\bm F^{k-1}(w_0, w'_0) )\right)\\
&=\left(F^k(w_0),F'\left(\left(F^{k-1}(w_0),\dfrac{\partial F^{k-1}(w_0)}{\partial w^i}\right) \right)\right)\\
&=\left(F^k(w_0),\dfrac{\p F(F^{k-1}(w_0))}{\p w}\dfrac{\partial F^{k-1}(w_0)}{\partial w^i}\right),\\
\end{align*}
which equals \eqref{eq:derisyst:5}. The claim is proven.


\noindent\textbf{(3) First derivatives with respect to $x_0$ and $y_k$.} We denote $$x:=(x^1,\dots,x^{N-1}),$$ and find the derivatives with respect to each component $x^i$ with $i=1,\dots,N-1$. Take
\begin{equation}\label{eq:derisyst:9}
\begin{array}{l}
w'_j=\dfrac{\partial w_j}{\partial x^i_0},\quad
x'_j=\dfrac{\partial x_j}{\partial x^i_0},\quad 
 y'_{j-1}=\dfrac{\partial y_{j}}{\partial x^i_0},\quad (j=1,\dots,k)\\
w'_0=0,\quad x'_0=(0,\dots,1,\dots,0),\quad y'_k =0,
\end{array}
\end{equation}
where the $i$-th component of $x'_0$ is 1. Obviously, this also gives an orbit segment of $\hat T_0$. In order to use the arguments in step (2), we need  the segment to lie in a small neighborhood of $\hat A$ (whose size is of order $O(\delta)$). It is not automatic in this case, since  $(x'_0)^i=1$ (and $x'_j$ are just bounded).  However, one sees  from \eqref{eq:T0_xform123} that $\bar w,\bar x, y$ do not depend on $w',x',\bar y'$ and
from \eqref{eq:derisyst:3.0}  that  $\bar w',\bar x', y'$ are linear in $w',x',\bar y'$. Hence, after the rescaling
$$(r',\varphi',x',\bar y')^{\mathrm{new}}=L^{-1}\hat \delta (r',\varphi',x',\bar y')$$
with $L$ being the upper bound of all $\|x'_j\|$, the orbit segment belongs to a $(L^{-1}\hat \delta)$-neighbourhood of $\hat A$.

 Similar argument also applies to the orbit segment
\begin{equation*}\label{eq:derisyst:10}
\begin{array}{l}
w'_j=\dfrac{\partial w_j}{\partial y^i_k},\quad
x'_j=\dfrac{\partial x_j}{\partial y^i_k},\quad 
 y'_{j-1}=\dfrac{\partial y_{j}}{\partial y^i_k},\quad (j=1,\dots,k)\\

w'_0=0,\quad y'_k =(0,\dots,1,\dots,0),
\end{array}
\end{equation*}
where the $i$-th component of $y'_k$ is 1. Therefore, applying Lemma~\ref{lem:GelTur:17} to the above two orbit segments with each $i$, we get
\begin{equation*}\label{eq:derisyst:11}
\left\|\dfrac{\partial (x_k,y_0)}{\partial (x_0,y_k)}\right\|_\circ\leqslant L\lambda^k,\quad
\left\|\dfrac{\partial w_k}{\partial (x_0,y_k)}-\dfrac{\partial F^k(w_0)}{\partial (x_0,y_k)}\right\|_\circ \leqslant L(\cc\lambda)^{\frac{k}{2}}.
\end{equation*}
Since $F$ is independent of $x_0$ and $y_k$, the last inequality implies
$$
\left\|\dfrac{\partial w_k}{\partial (x_0,y_k)}\right\|_\circ \leqslant L(\cc\lambda)^{\frac{k}{2}}.
$$
We thus obtained all required estimates for the derivatives with respect to $(x_0,y_k)$.

\noindent\textbf{(4) Proof of \eqref{eq:derisyst:3}.}
First note that the identities in \eqref{eq:ggg123} imply
\begin{equation*}
\begin{aligned}
&g(w,x,\bar y)=x\bar y h(w,x,\bar y), \\
&g_3(w,x,\bar y)=xh_3(w,x,\bar y),\quad g_4(w,x,\bar y)=\bar y h_4(wx,\bar y),
\end{aligned}
\end{equation*}
for some smooth functions $h$ and $h_{1,2}$. Hence, we can rewrite the functions $g'$ and $g'_{3,4}$ in \eqref{eq:derisyst:3.0} as
\begin{align*}
g' &= x\bar y\dfrac{\partial h}{\partial w}w'
+\bar y\dfrac{\partial xh}{\partial x}x' +x\dfrac{\partial \bar y h}{\partial \bar y}\bar y'\quad
g'_3=x\dfrac{\partial h_3}{\partial r}w'
+\dfrac{\partial xh_3}{\partial x}x' +x\dfrac{\partial  h_3}{\partial \bar y}\bar y',\\
g'_4&=\bar y\dfrac{\partial h_4}{\partial r}w'
+\bar y\dfrac{\partial h_4}{\partial x}x' +\dfrac{\partial \bar y h_4}{\partial \bar y}\bar y',
\end{align*}
which immediately lead to the first line of \eqref{eq:derisyst:3}. 

We proceed to estimate the derivatives. By \eqref{eq:derisyst:3.0} we have
\begin{equation*}\label{eq:derisyst:12}
\dfrac{\partial \bm g_3}{\partial \bm x}
=\dfrac{\p(g_3,g'_3)}{\p(x,x')}
=\begin{pmatrix}
\dfrac{\partial g_3}{\partial x} & 0\\[15pt]
\dfrac{\partial g_3'}{\partial x} & \dfrac{\partial g_3}{\partial x} 
\end{pmatrix}
\quad\mbox{and}\quad
\dfrac{\partial \bm g_4}{\partial \bar{\bm y}}
=\dfrac{\p(g_4,g'_4)}{\p(\bar y,\bar y')}
=\begin{pmatrix}
\dfrac{\partial g_4}{\partial \bar y} & 0\\[15pt]
\dfrac{\partial g_4'}{\partial \bar y} & \dfrac{\partial g_4}{\partial \bar y}
\end{pmatrix}.
\end{equation*}
By the second line of \eqref{eq:ggg123}, one can find $\delta_1>0$ such that
\begin{equation*}\label{eq:derisyst:13}
\left\|\dfrac{\partial  g_3}{\partial x}\right\|+{\delta_1}<\lambda
\quad\mbox{and}\quad
\left\|\dfrac{\partial  g_4}{\partial \bar y}\right\|+{\delta_1}<\lambda,
\end{equation*}
Denote 
$$C_1 = \sup_{(\hat r,\hat\varphi,\hat x,\hat y)\in \hat V}\left\|\dfrac{\partial g_3'}{\partial x}\right\|+\left\|\dfrac{\partial  g_3}{\partial  x}\right\|
+
\left\|\dfrac{\partial g_4'}{\partial \bar y}\right\|+\left\|\dfrac{\partial  g_4}{\partial \bar y}\right\|.
$$

Let  $(\Delta x,\Delta x')$ be any vector in the tangent space of the $\hat x$-space. We see that the norm defined by
$$\|(\Delta x,\Delta x')\|_1:=\|\Delta x\|+\delta_1 C^{-1}_1\|\Delta x'\|$$
satisfies
\begin{align*}
\left\|
\dfrac{\partial \bm g_3}{\partial \bm x}
\begin{pmatrix}
\Delta x \\ \Delta x'
\end{pmatrix}
\right\|_1
&=
\left\|
\begin{pmatrix}
\dfrac{\partial g_3}{\partial x}\Delta x \\[15pt]
\dfrac{\partial g_3'}{\partial x}\Delta x + \dfrac{\partial g_3}{\partial x} \Delta x'
\end{pmatrix}
\right\|_1 \\
&\leq
\left\|\dfrac{\partial  g_3}{\partial x}\right\| \|\Delta x\|
+\dfrac{\delta_1}{C_1} \left\|\dfrac{\partial g_3'}{\partial x}\right\| \|\Delta x\|
+\dfrac{\delta_1}{C_1}
\left\|\dfrac{\partial  g_3}{\partial x}\right\|\|\Delta x'\|\\
&<
\lambda (\|\Delta x\|+\dfrac{\delta_1}{C_1}\|\Delta x'\|)
=\lambda\|(\Delta x,\Delta x')\|_1,
\end{align*}
where the last inequality follows from the choice of $\delta_1$ and $C_1$. Similarly, one has
$$
\left\|
\dfrac{\partial \bm g_3}{\partial \bm x}
\begin{pmatrix}
\Delta x \\ \Delta x'
\end{pmatrix}
\right\|_1
<\lambda \|(\Delta y,\Delta y')\|_1.
$$
We thus have
\begin{equation*}\label{eq:g_3}
\left\|\dfrac{\p \bm g_3}{\p \bm x}\right\|_1<\lambda
\quad\mbox{and}\quad
\left\|\dfrac{\p \bm g_4}{\p \bar{\bm y}}\right\|_1<\lambda.
\end{equation*}

By \eqref{eq:derisyst:3.0}, we also have
\begin{align*}
\dfrac{\p \bm F}{\p \bm w}
=
\dfrac{\partial(F,F')}{\partial ( w,w')}
=\begin{pmatrix}
\dfrac{\partial F}{\partial w} & 0\\[15pt]
\dfrac{\partial F'}{\partial w} & \dfrac{\partial F}{\partial w} 
\end{pmatrix}
=:\begin{pmatrix}
A &0\\B &A
\end{pmatrix},
\end{align*}
whose inverse is given by
$$\begin{pmatrix}
A^{-1} &0\\-A^{-1}BA^{-1} & A^{-1}
\end{pmatrix}.$$
Similarly, we consider the norm defined by 
$$\|(\Delta w,\Delta w')\|_2:=\|(\Delta w)\|+\delta_2 C^{-1}_2\|(\Delta w')\|,$$
where $\delta_3$ satisfies
\begin{equation*}\label{eq:derisyst:15}
\left\|\dfrac{\partial F}{\partial w}\right\|+\delta_2<\cc 
\quad\mbox{and}\quad
\left\|\left(\dfrac{\partial F}{\partial w}\right)^{-1}\right\|+\delta_2<\cc,
\end{equation*}
and
$$C_2 = \sup_{(\hat r,\hat\varphi,\hat x,\hat y)\in \hat V}\|B\|+\|A\|+\|A^{-1}BA^{-1}\|+\|A^{-1}\|.$$
The same computation as above shows that
$$
\left\|\dfrac{\p \bm F}{\p \bm w}\right\|_2 < \cc
\quad\mbox{and}\quad
\left\|\left(\dfrac{\p \bm F}{\p \bm w}\right)^{-1}\right\|_2 < \cc.
$$


Thus, the desired norm in \eqref{eq:derisyst:3} can be defined as following:  
\begin{equation*}
\|(\Delta u, \Delta u')\|_{\circ}=\|\Delta u\|+\delta C^{-1}\| \Delta u'\|, \quad  u\in\{w, x, y\}, 
\end{equation*}
where $\delta=\min\{\delta_1,\delta_2\}$ and $C=\max\{C_1,C_2\}$.

\noindent\textbf{(5) Higher derivatives.} 
Now let $\hat T_0$ be given by \eqref{eq:derisyst:3.0} with the prime coordinates defined in \eqref{eq:derisyst:3.1}. By  replacing the role of $T_0$ with $\hat T_0$ in the above arguments, we obtain the desired estimates for the first derivatives of $\p(w_k,x_k,y_0)/\p(w_0,x_0,y_k)$ with respect to $w^i_0$.  Similarly, with the prime coordinates in   \eqref{eq:derisyst:9}, we obtain the derivatives with respect to $x_0,y_k$. This concludes the case of second derivatives. Repeating this procedure up to $m'-1$ gives the lemma, where $m'$ is the smoothness of our system in Fenichel coordinates. (Since  Lemma~\ref{lem:GelTur:17} requires the system to be at least $C^1$, our approach only works derivatives up to order $m'-1$.)

\subsubsection{Iterations of the inner map $F^k$: proof of Lemma~\ref{lem:F^k}}\label{sec:lem:F^k}

We first prove an auxiliary result. Take any $\delta>0$, any positive integers $n_1,n_2,d$. Consider any map $\hat F: \mathbb R^{n_1}\times \mathbb R^{n_2}\to \mathbb R^{n_1}\times \mathbb R^{n_2} $ of the following form:
\begin{equation}\label{eq:Fhat0}
\bar u = u + \delta^{d-1}p(u,v),\qquad \bar v = v+ v^*+\delta \hat q(u) + \delta^{d}q(u,v),
\end{equation}
where $v^*\in \mathbb{R}^{n_2}$ is constant, $\hat q$ is $C^1$ with bounded $C^1$ norm, and $p$ and $q$ are bounded continuous functions.

\begin{lema}\label{lem:Fhat}
The $k$-th iteration map $\hat{F}^k : (u_0,v_0)\mapsto (u_k,v_k)$ is given by
\begin{equation}\label{eq:Fhat1}
u_k = u_0 + k\delta^{d-1}p_k(u_0,v_0),\qquad v_k = v_0+ k v^*+ k\delta\hat q(u_0) + k^2\delta^{d}q_k(u_0,v_0),
\end{equation}
for some continuous functions $p_k$ and $q_k$ bounded uniformly for all $k$.
\end{lema}

\begin{proof}
Take any $C_1$ and $C_2$ such that $\|p\|<C_1, \|q\|<C_2, C_2\|\hat p\|_{C^1}< C_1$.
We prove the lemma by induction on $k$. 
The elementary step is automatic. We now assume the validity of \eqref{eq:Fhat1} with $\|p_k\|<C_1$ and $\|q_k\|<C_2$.

The $(k+1)$-iteration is found from
 \eqref{eq:Fhat0} and \eqref{eq:Fhat1} as
\begin{equation}\label{eq:00}
\begin{aligned}
u_{k+1} &= u_k +  p(u_k,v_k)= u_0 +k\delta^{d-1}p_k(u_0,v_0) + \delta^{d-1} p(u_k,v_k),\\
v_{k+1} &= v_k + \hat q(u_k)+
\delta^{d}q(u_k,v_k)\\
&= v_0+ (k+1)\delta\hat q(u_0) +  k^2\delta^{d}q_k(u_0,v_0) + O(\hat q') k\delta^d p_k(u_0,v_0)+ \delta^{d} q(u_k,v_k).
\end{aligned}
\end{equation}
Denote 
\begin{equation}\label{eq:000}
\begin{aligned}
 p_{k+1}(u_0,v_0)&:= \dfrac{1}{k+1}  (kp_k(u_0,v_0) + p(u_k,v_k)),\\
 q_{k+1}(u_0,v_0)&:=\dfrac{1}{(k+1)^2} (k^2 q_k(u_0,v_0) + O(\hat q')\cdot k p_k(u_0,v_0)+  q(u_k,v_k)).
\end{aligned}
\end{equation}
Since $C_2\|\hat p\|_{C^1}< C_1$, it is immediate that $\|p_{k+1}\|<C_1$ and $\|q_{k+1}\|<C_2$. This completes the induction and hence the proof of the lemma.
\end{proof}

Recall that $m'\geq m\geq 2$ in (\ref{eq:kam_new}) and (\ref{eq:kam_new_deri}) and $m^*=\min\{m,m'-m\}$. Take any small $\delta>0$ such that \eqref{eq:kam_new} and \eqref{eq:kam_new_deri} are valid  for $[-2\delta,2\delta]$. We consider  the coordinates
\begin{equation}\label{eq:Fk1}
r= \delta r^{\mathrm{new}}.
\end{equation} 
\begin{lema}
For all sufficiently small $\delta$ and  all $r^{\mathrm{new}}\in [-3/2,3/2]$, the formula for the iteration $F^k$ assumes the form
\begin{equation}\label{eq:Fk1a}
\begin{aligned}
r^{\mathrm{new}}_k&=  r^{\mathrm{new}}_0 + k\delta^{m-1}p_k(r^{\mathrm{new}}_0,\pp_0,\eps),\\
 \varphi_k &=\varphi_0 +k\rho+  k\delta\hat q(r^{\mathrm{new}}_0,\eps) r+  k^2\delta^m q_k(r^{\mathrm{new}}_0,\pp_0,\eps),
\end{aligned}
\end{equation}
where $\hat q(r^{\mathrm{new}},\eps)=r^{\mathrm{new}} + \delta^{-1}\hat\rho(\delta r^{\mathrm{new}},\eps)$, and
 $p_k$ and $q_k$ are uniformly bounded in the $C^{m^*}$ topology for all sufficiently small $\delta$ and $k=o(\delta^{1-m})$. 
\end{lema}

Lemma~\ref{lem:F^k} follows immediately by taking  $\delta=r_0$ and
  $r^{\mathrm{new}}=1$ (so $r_0^\mathrm{new}=1$ and $r_k^\mathrm{new}=r_0^{-1}r_k$), and setting $\xi_k(r_0,\pp_0,\eps)=kr_0^{m}p_k(1,\pp_0,\eps)$ and  $\eta_k(r_0,\pp_0,\eps)=kr_0^{m}q_k(1,\pp_0,\eps)$.

\begin{proof}
Let us drop the superscript.  The inner map $F$ given by  \eqref{eq:kam_new} assumes the following  form in the coordinates \eqref{eq:Fk1}:
\begin{equation}\label{eq:Fk2}
\bar r=  r + \delta^{m-1}p(r,\eps,\pp),\qquad
\bar \eps =\eps,\qquad
\bar \varphi =\varphi +\rho+ \delta \hat q(r,\eps) r+ \delta^m q(r,\eps,\pp),
\end{equation}
where we  also consider  parameters $\eps$ as  variables, and
\begin{equation}\label{eq:Fk2a}
p(r,\eps,\pp)=\delta^{-m}\xi(\delta r,\eps,\pp),\qquad
q(r,\eps,\pp)=\delta^{-m}\eta(\delta r,\eps,\pp).
\end{equation}
By \eqref{eq:kam_new_deri}, we have
\begin{equation}\label{eq:Fk2b}
\|p,q,\hat q\|_{C^{m^*}}<C,
\end{equation}
for some constant $C>0$ uniformly for all small $\delta$.

After taking $u=(r,\eps)$ and $v=\pp$, the map $F$ has same form as \eqref{eq:Fhat0}. By \eqref{eq:kam_new_deri},  $p,q,\hat q$ are bounded for $r\in [-2,2]$, uniformly for all sufficiently small $\delta$. Take any continuous extension $F_\mathrm{ext}$ of $F$ to $\mathbb{R}^3$  such that it still has the form \eqref{eq:Fk2} with \eqref{eq:Fk2b} satisfied and $F_\mathrm{ext}=F$ for $r\in [-5/3,5/3]$. Then applying Lemma~\ref{lem:Fhat} to $F_\mathrm{ext}$, we find the desired formula  \eqref{eq:Fk1a} and the  $C^1$ boundedness of $p_k$ and $q_k$ for the iteration $F^k_\mathrm{ext}$, uniformly for all small $\delta$. Finally, since $k=o(\delta^{1-m})$, the estimates for $r_k$ implies that $r_k=r_0+o(1)_{\delta\to 0}$, and hence $F^k_\mathrm{ext}=F^k$ when $r\in [-3/2,3/2]$ for all sufficiently small $\delta$.

Let us proceed to find  derivatives of $p_k$ and $q_k$, which is done by the same method as in the proof of Lemma~\ref{lem:derisyst}. We first deal with the first derivative with respect to $r$. Consider the map $\hat F:(r,r',\eps,\pp,\pp')\mapsto (\bar r,\bar r',\bar \eps,\bar \pp,\bar \pp)$ defined by \eqref{eq:Fk2} and 
\begin{align*}
\bar r' &= r' + \delta^{m-1} \dfrac{\p p(r,\eps,\pp)}{\p r}r'+\delta^{m-1} \dfrac{\p p(r,\eps,\pp)}{\p \pp}\pp'
=:r' + \delta^{m-1}p'(r,r',\eps,\pp,\pp'),\\
\bar \pp' &= \pp' +\delta \dfrac{\p \hat q(r,\eps)}{\p r} r' 
+\delta^{m} \dfrac{\p q(r,\eps,\pp)}{\p r}r'+\delta^{m} \dfrac{\p q(r,\eps,\pp)}{\p \pp}\pp'\\
&=:\pp' +\delta \hat q'(r,r',\eps) + \delta^{m} q'(r,r',\eps,\pp,\pp').
\end{align*}
By \eqref{eq:Fk2b}, the functions $p',q',\hat q'$ are uniformly bounded for all small $\delta$.
Arguing as in the previous paragraph, we take $u=(r,r',\eps)$ and $v=(\pp,\pp')$ and apply Lemma~\ref{lem:Fhat} to an  extension $\hat F_{\mathrm{ext}}$ of $\hat F$, which, by the choice of $k$, can be taken such that $\hat F^k_{\mathrm{ext}}=\hat F^k$ for  $r\in [-3/2,3/2]$. 

Therefore, we have
$$r'_k=r_0 + k\delta^{m-1}p'_k,\qquad
\pp'_k =\pp_0+ k\delta\hat q' + k^2\delta^{m}q'_k,
$$
where $p'_k$ and $q'_k$ are functions of $(r_0,r'_0,\eps,\pp_0,\pp'_0)$, uniformly bounded for all small $\delta$ and  $k<\delta^{1-m}$. This gives the desired estimates for $\p r_k/\p r_0$ and $\p \pp_k/\p \pp_0$ since $\{(r_j,r'_j,\eps,\pp_j,\pp'_j)\}$ with 
$$
r'_j=\dfrac{\p r_j}{\p r_0},\qquad
\pp'_j=\dfrac{\p \pp_j}{\p r_0},\qquad j=1,\dots,k
r'_0=1,\qquad \pp'_0=0,
$$
is the orbit segment corresponding to $\hat F_k(r_0,r'_0,\eps,\pp_0,\pp'_0)=(r_k,r'_k,\eps,\pp_k,\pp'_k)$.

The  first derivative with respect to $\pp_0$ is obtained in the same way, and also for the higher derivatives up to order $m'$ (see step (5) of the proof of Lemma~\ref{lem:derisyst}). 
\end{proof}

\subsection{Examples of unfolding families for Theorem~\ref{thm:main_2punfolding} and Theorem~\ref{thm:main_SC_para}}
\label{sec:unfoldingfam}

We will construct each unfolding family in the form $\{f_\eps\}:=\{G_\eps\circ f\}$, where $G_\eps$, with $G_0=\mathrm{id}$, is the composition of  time-1 maps of certain $C^\infty$ Hamiltonian flows, each depending on finitely many parameters. This guarantees that the maps $f_\eps$ are symplectic and, in the context of Theorem~\ref{thm:main_2punfolding},   meet the local exactness requirement. Note that any unfolding family for Theorem~\ref{thm:main_SC_para} also works for Theorem~\ref{thm:main_SCblender}.

The unfolding family can be made real analytic whenever $f$ is real analytic. Indeed, it suffices to take real-analytic approximations of the $C^\infty$ families of Hamiltonian functions that define the above Hamiltonian flows.  Since all unfolding conditions are $C^1$-open in the space of families, they also hold for the family $\{\tilde G_\eps\circ f\}$, where $\tilde G_\eps$ is the   composition of  time-1 maps of the Hamiltonian flows associated with these approximations. Moreover,  since the $C^\infty$ families of Hamiltonian functions vanish at $\eps=0$ by construction, the real-analytic approximations can be chosen to vanish at $\eps=0$ as well, ensuring $\tilde G_0=\mathrm{id}$. The desired real-analytic family is thus given by $\{\tilde G_\eps\circ f\}$.

In the end of this section, we also construct perturbations that realize the genericity conditions of Theorem~\ref{thm:main_SC_para}, in the same way as above.
 
\subsubsection{Case (1) of Theorem~\ref{thm:main_2punfolding}}\label{sec:fam1}

Recall that the transition map $T_1:(r,\pp,x,y)\mapsto (\tilde r,\tilde \pp,\tilde x,\tilde y)$  takes a small neighborhood of $M^-\in W^\u_\loc(\gamma)$ to a small neighborhood of $M^+\in W^\s_\loc(\gamma)$. In the Fenichel coordinates, where $W^\u_\loc(\gamma)=\{r=0,x=0\}$ and $W^\s_\loc(\gamma)=\{\tilde r=0,\tilde y=0\}$, the map  $T_1$ takes the form   \eqref{eq:T1_xform_n}, with setting $\ell=2$ there. The image $T_1(W^\u_{\loc}(\gamma))$  is given by 
\begin{equation}\label{eq:fam1}
\begin{aligned}
\tilde r =   \dd (\tilde \varphi - \varphi^+)^3 +   a_{14}\tilde y+\dots, \qquad
\tilde x-x^+ =   a_{32} (\tilde \varphi - \varphi^+) +  a_{34}\tilde y+\dots,
\end{aligned}
\end{equation}
where $M^+=(0,\pp^+,x^+,0)$. To obtain condition~\eqref{eq:unfoldcubic} (see also \eqref{eq:T1_cubic}), it suffices to construct a two-parameter family $\{f_\eps\}$ such that $W^\s_\loc(\gamma)=\{\tilde r=0,\tilde y=0\}$ for all $\eps$ and the $\tilde r$-component of the defining function of $T_{1,\eps}(W^\u_{\loc}(\gamma))$, denoted by $\tilde r = w^\u(\tilde \pp, \tilde y,\eps)$  satisfies
$$
\left.\dfrac{\p w^\u}{\p\eps_1}\right|_{M^+,\eps=0}\neq 0,\qquad
\left.\dfrac{\p^2 w^\u}{\p\tilde\pp \p\eps_2}\right|_{M^+,\eps=0}\neq 0.
$$
In what follows, we obtain $f_\eps$ as localized perturbations  that change the shape and position of $T_1(W^\u_{\loc}(\gamma))$.

Since $W^\s_\loc(\gamma)$ is Lagrangian, by \citep{Wei:71}  there exists a symplectic change of coordinates
in a small neighborhood $U$ of $M^+$ that restricts to the identity on $W^\s_\loc(\gamma)$  and brings the symplectic form $\Omega$ to the standard form. In the new coordinates (which we still denote by $\tilde r,\tilde \pp,\tilde x,\tilde y$),   $W^\s_\loc(\gamma)=\{\tilde r=0,\tilde y=0\}$, $\Omega|_U=d \tilde r\wedge d\tilde \pp +d\tilde x \wedge d\tilde y$, and the  map $T_1$ remains the same form with
$T_1(W^\u_{\loc}(\gamma))$   given by \eqref{eq:fam1}.
Let $\{g_\eps\}$ be any two-parameter family  of $C^{\infty}$ functions $g_\eps:\tilde \pp\mapsto (\tilde r,\tilde x)$ with $g_0\equiv 0$. Following the construction in \citep[Section 2.3]{GonTurShi:07}, we define the Hamiltonian function
\begin{equation}\label{eq:hamil}
H_\eps(\tilde r,\tilde \pp,\tilde x,\tilde y)=-\chi(\tilde r,\tilde \pp,\tilde x,\tilde y) \cdot \int^{\tilde \pp}_{\pp^+} g_\eps(s) ds,
\end{equation}
where $\chi$  is a $C^{\infty}$ bump function which is supported in $U$ and equal to 1 near $M^+$.
One readily finds that the time-1 map $G_\eps$ of the Hamiltonian flow near $M^+$ takes the form
$$(\tilde r,\tilde \pp,\tilde x,\tilde y)\mapsto (\tilde r+g_\eps(\tilde \pp),\tilde \pp,\tilde x,\tilde y).$$

For each $\eps $  the perturbation $G_\eps\circ f$ coincides with $f$ outside a small neighborhood of $f^{-1}(M^+)$. Hence, the transition map for $G_\eps\circ f$ is $T_{1,\eps}:=(G_\eps\circ f)^n=G_\eps \circ f^n=G_\eps \circ T_1$. As a result, $T_{1,\eps}(W^\u_{\loc})$ is given by 
$$\begin{aligned}
\tilde r =  g_\eps(\tilde \pp)+  \dd (\tilde \varphi - \varphi^+)^3 +   a_{14}\tilde y+\dots, \qquad
\tilde x-x^+ =   a_{32} (\tilde \varphi - \varphi^+) +  a_{34}\tilde y+\dots,
\end{aligned}
$$
Taking $g_\eps(\pp)=\eps_1+\eps_2(\pp-\pp^+)$ yields the  desired family $\{f_\eps\}:=\{G_\eps\circ f\}$. 

\subsubsection{Case (2) of Theorem~\ref{thm:main_2punfolding}}\label{sec:fam2}

We now have two orbits of quadratic tangencies. Let $T_1$ be the transition map defined along one of these orbits.
Setting $\ell=1$ in \eqref{eq:T1_xform_n}, we find  the image  $T_1(W^\u_{\loc}(\gamma))$ as
\begin{equation*}
\begin{aligned}
\tilde r =   \dd (\tilde \varphi - \varphi^+)^2 +   a_{14}\tilde y+\dots, \qquad
\tilde x-x^+ =   a_{32} (\tilde \varphi - \varphi^+) +  a_{34}\tilde y+\dots,
\end{aligned}
\end{equation*}
The same procedure as in case (1), with $g(\tilde \pp)=\eps_1$ in \eqref{eq:hamil},  gives a family $\{H'_{\eps_1}\}$ of locally supported Hamiltonian functions such that the perturbation obtained by post-composing the corresponding time-1 maps with $f$ unfolds the tangency under consideration.  Applying the construction again to   the second orbit of tangency, now using   $g(\tilde \pp)=\eps_2$ in \eqref{eq:hamil}, we find  another family $\{H''_{\eps_2}\}$ by which we unfold this tangency. Since the two families are localized in different places, the time-1 maps $G_\eps$ of  the Hamiltonians $H_\eps=H''_{\eps_2}\circ H'_{\eps_1}$ give the desired  family  $\{f_\eps\}:=\{G_{\eps}\circ f\}$ satisfying \eqref{eq:unfold2quadra}.

%

\subsubsection{Case (3) of Theorem~\ref{thm:main_2punfolding}}\label{sec:fam3}

We first construct a one-parameter family that changes the hyperbolicity coefficient $\alpha$ in \eqref{eq:unfoldquadra}. By Lemma~\ref{lem:scattering_deri} and \eqref{eq:a11}, it suffices to change the coefficient $\tilde a_{22}$ in \eqref{eq:T1_general} (here we do not need to use the fact that the homoclinic orbit is non-transverse). Formula \eqref{eq:T1_general} is written in the Fenichel coordinates where $W^\u_\loc(\gamma)=\{r=0,x=0\}$ and $W^\s_\loc(\gamma)=\{\tilde r=0,\tilde y=0\}$. 
One sees that   $\D T_1$ takes  the vector $v=(0,1,0,\tilde a_{42})$ in the tangent space of $W^\u_\loc(\gamma)$ at $M^-$ to the vector $(0, \tilde a_{22},\tilde a_{32},0)$ in the   tangent space of $ W^\s_\loc(\gamma)$ at $M^+$ , i.e., $\tilde a_{22}$ is the $\tilde \pp$-component of the vector $\D T_1 v$. 
Evidently, for any diffeomorphism $G$ satisfying $G|_{W^\u_\loc(\gamma)}=\mathrm{id}$, the $\tilde \pp$-component of $\D(G\circ T_1)v$ remains $\tilde a_{22}$.
It follows that the coefficient $\tilde a_{22}$  is independent of the choice of coordinates near $M^+$ that restrict to the identity on $W^\s_\loc(\gamma)$.
In particular, we  can use the  standard coordinates near $M^+$   used in Section~\ref{sec:fam1}, and modify  $\tilde a_{22}$ by using the Hamiltonian function
\begin{equation*}
H'_{\eps_1}(\tilde r,\tilde \pp,\tilde x,\tilde y)=-\eps_1\tilde r \tilde \pp\cdot \chi(\tilde r,\tilde \pp,\tilde x,\tilde y).
\end{equation*}
Indeed, since the corresponding time-1 map  $G'_{\eps_1}$ restricts to $(\tilde r,\tilde \pp,\tilde x,\tilde y)\mapsto (\tilde r e^{\eps_1},\tilde \pp e^{-\eps_1},\tilde x,\tilde y)$ near $M^+$ and the transition map for  $G'_{\eps_1}\circ f$ is $T_{1,\eps_1}=G'_{\eps_1}\circ T_1$, the $\tilde \pp$-component of the vector $\D T_{1,\eps_1} v$ is $\tilde a_{22,\eps_1}=\tilde a_{22} e^{-\eps_1}$.

To unfold the homoclinic tangency of $\gamma$, we  construct as in case (2) a family $\{H''_{\eps_2}\}$ of Hamiltonian functions supported near $M^+$, whose corresponding time-1 map $G''_{\eps_2}$ restricts to $(\tilde r,\tilde \pp,\tilde x,\tilde y)\mapsto (\tilde r +\eps_2,\tilde \pp ,\tilde x,\tilde y)$ near $M^+$, such that the tangency between $G''_{\eps_2}\circ T_{1}(W^\u_\loc(\gamma))$ and $W^\s_\loc(\gamma)$ unfolds generically.

Now consider the perturbation $G''_{\eps_2}\circ G'_{\eps_1}\circ f $. Its transition map is $T_{1,\eps_1,\eps_2}=G''_{\eps_2}\circ   G'_{\eps_1}\circ T_1$. 
By construction, the tangency between $ T_{1,\eps_1,\eps_2}(W^\u_\loc(\gamma))$ and $W^\s_\loc(\gamma)$ persists for all $\eps_1$ and unfolds generically when $\eps_2$ varies. Moreover, since the differential of $G''_{\eps_2}$ is the identity, $\D T_{1,\eps_1,\eps_2} v=\D T_{1,\eps_1} v$, and hence $\tilde a_{22,\eps_1,\eps_2}=\tilde a_{22} e^{-\eps_1}$ changes as $\eps_1$ varies. Evidently, $\{G''_{\eps_2}\circ G'_{\eps_1}\circ f \} $ is the desired unfolding family.

\subsubsection{Unfolding families for Theorem \ref{thm:main_SC_para}}\label{sec:famC}

We will construct four one-parameter perturbations that  independently modify  the quantities $\rho, \tr L^\T L, \mu,\nu $  in \ref{word:H1}.

\noindent\textbf{(1) Modification of  $\rho$.}
Recall that $\rho$ is the argument of the central multipliers $e^{\pm i\rho}$ of the saddle-center $O$. Since $W^\c(O)$ is symplectic and $\Omega|_{W^\c(O)}=du\wedge dv$, one can find standard coordinates $(u,v,x,y)\in \mathbb{R}\times \mathbb{R}\times \mathbb{R}^{N-1}\times \mathbb{R}^{N-1}$ in a small neighborhood  $U$ of $O$ such that $O$ is at the origin, $W^\c(O)=\{x=0,y=0\}$ and $\Omega|_U=du \wedge dv+ dx \wedge dy$. A further symplectic change of coordinates on $W^\c(O)$ puts $\D f^\tau|_{W^\c(O)}$  at $O$ in  the  form
$
\begin{psmallmatrix}
\cos\rho & -\sin\rho\\
\sin\rho & \cos\rho
\end{psmallmatrix},
$
where $\tau$ denotes the period of $O$.  Now consider the Hamiltonian function
$$H'_{\eps_1}=- \dfrac{\eps_1}{2\tau} (u^2+v^2)\cdot \chi_1(u,v,x,y),$$
where  $ \chi_1$ is a bump function supported near $O$.  The corresponding time-1 map $G_{\eps_1}$ restricts to a rotation of angle $\eps_1/\tau$ on $W^\c(O)$. As a result, $O$ is a periodic point of $G_{\eps_1}\circ f$  with  period $\tau$ and central multipliers $e^{\pm i(\eps_1+\rho)}$.\\

\noindent\textbf{(2) Modification of $\tr L^\T L$.}
Recall that $L=\D S(O)$ and the scattering map $S=\pi^\s\circ T_1|_{\Sigma^-}\circ (\pi^\u)^{-1}$ is defined in Section~\ref{sec:SCsattering}.
 Here the transition map $T_1$ takes the two-dimensional disc $\Sigma^+\subset W^\cu_\loc(O)\cap W^\cs(O)$ containing $M^-$ to the two-dimensional disc $\Sigma^-\subset W^\cs_\loc(O)\cap W^\cu(O)$ containing $M^+$, and $\pi^\u$ and $\pi^\s$ are the holonomy maps of $\Sigma^-$ and $\Sigma^+$, respectively.

We use the coordinates \eqref{eq:strait2}, where the two holonomy maps are identities, so 
\begin{equation}\label{eq:famC:0}
L=\D T_1|_{\Sigma^-}(M^-).
\end{equation}
 By \eqref{eq:bcoor}, we replace the central polar coordinates $r,\pp$ by Cartesian coordinates $u,v$.
Since $\Sigma^\pm$ are transverse to the foliations, we can use $(u,v)$ as coordinates of $\Sigma^\pm$. Note also that $M^+=(0,0,x^+,0)$ and $M^-=(0,0,0,y^-)$, and hence the map $T_1|_{\Sigma^-}$ takes the form
\begin{equation}\label{eq:famC:1}
\begin{pmatrix}
u\\ v
\end{pmatrix}
\mapsto
\begin{pmatrix}
b_{11} u + b_{12} v +\dots\\
b_{12} u + b_{22} v +\dots
\end{pmatrix},
\end{equation}
where the dots denote higher order terms.

By \eqref{eq:OmegaW^c}, we have  $\Omega_{\Sigma^+}=(\pi^{\s})^* \Omega_{W^\c(O)}=d u \wedge d  v$. So, there exists a symplectic change of coordinates in a small neighborhood $U$ of $\Sigma^-$ that restricts to the identity on $\Sigma^-$ and brings the symplectic form in the new coordinates $( u,  v, x ,  y)$ to $\Omega_U=d u \wedge d  v + d x \wedge d  y$. In particular,  formula \eqref{eq:famC:1} for $T_1|_{\Sigma^-}$ remains the same.

Define the Hamiltonian function
\begin{equation*}
H''_{\eps_2}( u, v, x, y)=-\eps_2 u  v\cdot \chi_2( u, v, x, y) .
\end{equation*}
where $\chi_2$ is a bump function supported near $M^+\in \Sigma^{+}$. Let $\{G''_{\eps_2}\}$ be the  time-1 maps of the corresponding flows. Then the transition map  for  $G''_{\eps_2}\circ G'_{\eps_1}\circ f$ is $T_{1,\eps_1,\eps_2}=G'_{\eps_2}\circ T_1$ (recall that the supports of $\chi_1$ and $\chi_2$ are disjoint), given by 
\begin{equation}\label{eq:famC:1a}
\begin{pmatrix}
u\\ v
\end{pmatrix}
\mapsto
\begin{pmatrix}
e^{\eps_2} b_{11} u + e^{\eps_2} b_{12} v +\dots\\
e^{-\eps_2} b_{12} u +e^{-\eps_2} b_{22} v +\dots
\end{pmatrix}.
\end{equation}
By \eqref{eq:famC:0}, we have  $L_{\eps_2}=\D T_{1,\eps_2}|_{\Sigma^-}(M^-)$ and hence
$$\dfrac{d\tr L_{\eps_2}^\T L_{\eps_2}}{d\eps_2}=e^{2\eps_2} (b^2_{11}+b^2_{21})- e^{2\eps_2}(b^2_{12}+ b^2_{22}).$$
Arguing as in the proof of Lemma~\ref{lem:rotaton} (by considering columns and $LR^{-1}$ instead of rows and $RL$), we can find Fenichel coordinates  where $b^2_{11}+b^2_{21}- (b^2_{12}+ b^2_{22})\neq 0$, so  $\{G'_{\eps_2}\circ f\}$ is the desired perturbation family that modifies $\tr L^\T L$.\\

\noindent\textbf{(3) Modification of $\mu$ and $\nu$.} 
Recall that $(\mu,\nu)$ are the  coordinates of $S(O)$ (see Section~\ref{sec:properunfolding}). In the Fenichel coordinates where $\pi^\u$ and $\pi^\s$ are identities, we have $(u,v)=T_{1,\eps}|_{\Sigma^-}(0,0)$.
Define 
$$
H'''_{\eps_3,\eps_4}(u,v,x,y)=
(\eps_4 u -\eps_3 v) \cdot
\chi_3(u,v,x,y),
$$
where $\chi_3$ is supported near $M^+$.   The restriction $G'''_{\eps_3,\eps_4}|_{\Sigma^-}$ of  the     corresponding time-1 map is given by 
$
(u,v)\mapsto (u+ \eps_3,
v+\eps_4).
$
Since $G'_{\eps_1}$ is supported near $O$, it follows that the transition map for $ G'''_{\eps_3,\eps_4}\circ G''_{\eps_2}\circ G'_{\eps_1}\circ f$ is $T_{1,\eps}=G'''_{\eps_3,\eps_4}\circ  G''_{\eps_2}\circ T_{1}$, and its restriction to $\Sigma^-$ is given by 
\begin{equation}\label{eq:famC:2}
\begin{pmatrix}
u\\ v
\end{pmatrix}
\mapsto
\begin{pmatrix}
\eps_3+ e^{\eps_2} b_{11} u + e^{\eps_2} b_{12} v +\dots\\
\eps_4+e^{-\eps_2} b_{12} u +e^{-\eps_2} b_{22} v +\dots
\end{pmatrix}.
\end{equation}
One then easily verifies that $\{G'''_{\eps_3,\eps_4}\circ G''_{\eps_2}\circ G'_{\eps_1}\circ f\}$  is the desired unfolding family.

\subsubsection{Genericity conditions of Theorem~\ref{thm:main_SC_para}}
\label{sec:gc2}
For each genericity condition we construct a family of smooth Hamiltonian functions such that this condition is satisfied by  the perturbation $G_\eps\circ f$ for every $\eps\neq 0$, where $G_\eps$ is the corresponding time-1 map. Applying these perturbations successively, we obtain all the required conditions. 
When $f$ is real analytic, the involved perturbations 
can be made real analytic by replacing  the   Hamiltonian functions with their real analytic  approximations.

The irrationality of $\rho$ in condition~\ref{word:C1} is obtained as in step~(1) of Section~\ref{sec:famC}. Regarding the second part of condition~\ref{word:C1} (the twist condition), we proceed as in the whiskered torus setting. Specifically, we consider standard coordinates $(u,v,x,y)$ in a small neighborhood of $W^\c(O)$, extending the standard coordinates $(u,v)$ on $W^\c(O)$. Consider the Hamiltonian function 
$$H_\eps( u, v, x, y)= \frac{1}{64} \eps(u^2+v^2)^2 \cdot \chi( u, v, x, y),$$ where $\chi$ is a bump function supported in a small neighborhood of $O$ (chosen to be disjoint from $f(O)$ if $\mathrm{per}(O)>1$). The Hamiltonian flow restricted to $W^\c(O)\setminus\{O\}$, written in the polar coordinates \eqref{eq:bcoor}, is given by $\dot r=0,\dot\pp = \eps r$, which  extends  analytically to the whole of $W^\c(O)$. As a result, for  every $\eps\neq 0$, post-composing the time-1 map $G_\eps$ with $f$ yields  the perturbation $G_\eps\circ f$ for which  $O$ remains a saddle-center periodic point and satisfies the twist condition.

Condition~\ref{word:C2} is essentially the same as the partial hyperbolicity in Theorem~\ref{thm:main_2punfolding}. Condition~\ref{word:C3} can be dealt with by the perturbation in step~(2) of Section~\ref{sec:famC}, since \eqref{eq:famC:1a} shows that $L_\eps$ becomes hyperbolic for $\eps\neq 0$ and therefore cannot be a rotation.

\subsection{Intersection of Lagrangian manifolds: proof of Lemma~\ref{lem:genearaltrans}}\label{sec:lemtrans}
Since $W_2$ is Lagrangian, there are local coordinates 
 $(u,v)\in \mathbb{R}^N\times \mathbb{R}^N$ near $P$ such that $W_2$ is straightened, i.e., $W_2=\{ v=0\}$, and the symplectic form $\Omega$ is given\footnote{
This can be  done, for example, by following the construction of  Darboux coordinates in \citep[Section~43]{Arn:78}: one introduces $N$ pairs of conjugate coordinates inductively, and the fact that $W^\s_\loc(\gamma)$ is Lagrangian ensures that at each step one can  choose a  coordinate in $W^\s_\loc(\gamma)$ as the base for defining its conjugate coordinate.\label{fn:ph}
}
by $du\wedge dv$. Similarly, there are coordinates $(\tilde u, \tilde v)\in \mathbb{R}^N\times \mathbb{R}^N$ near $f^{-1}(P)$ such that $W_1=\{\tilde u=0\}$ and $\Omega = d\tilde u \wedge d\tilde v$. We can also assume that $P=(0,0)$ and $f^{-1}(P)$ are at the origin of the corresponding coordinate systems.

Let $(u,v)\in \mathbb{R}^N\times \mathbb{R}^N$ be the Darboux coordinates near $P$, so the symplectic form $\Omega$ is given by $du\wedge dv$. We can choose the coordinates such that $P$ is at the origin and $W_2$ is
tangent at $P$ to $\{v = Q u\}$ with some
$N\times N$ matrix $Q$. Since $W_2$ is Lagrangian, $Q$ is symmetric, implying that the coordinate transformation
$(u,v)\mapsto (u, v - Qu)$ is symplectic. We make this transformation, so $W_2$ becomes tangent to $\{v=0\}$ in these coordinates. Similarly, there are coordinates $(\tilde u, \tilde v)\in \mathbb{R}^N\times \mathbb{R}^N$ near $f^{-1}(P)$ such that $\Omega = d\tilde u \wedge d\tilde v$, $f^{-1}(P)=(0,0)$ and $W_1$
is tangent to $\{\tilde u=0\}$.

Let us write  $f$ near $f^{-1}(P)$ as
\begin{equation}\label{eq:transfn}
u = A \tilde u + B \tilde v+\dots,\qquad
v = C\tilde u + D \tilde v +\dots,
\end{equation}
where the dots denote terms of order higher than $1$, and $A,B,C,D$ are $N\times N$ matrices. Since $T_{f^{-1}(P)}W_1=\{\tilde u=0\}$, the tangent vectors in $T_P f(W_1)$  are of the form
$$u = B \tilde v,\qquad v = D\tilde v.$$
Since $T_P W_2=\{v=0\}$, the assumption that the intersection has corank $c$ means that
$${\rm rank} D = N-c.$$
Choose $N-c$ linearly independent rows 
$d_{i_1}, \dots, d_{i_{N-c}}$ in $D$; the rest are linear combinations of these.

Since $f$ is a diffeomorphism, $\det 
\begin{psmallmatrix}
A &B\\C&D
\end{psmallmatrix}
\neq 0$. Therefore, there exist $N$ linearly independent rows in the $2N\times N$ matrix 
$\begin{psmallmatrix}
B\\D
\end{psmallmatrix}$. 
One can then find $i$ such that the row $b_i$ of the matrix $B$ is not a linear combination of the rows $d_{i_1}, \dots, d_{i_{N-c}}$.

There are only two possibilities:
\begin{itemize}[nosep]
\item[(1)] either we can choose the row $b_i$ such that $i\not\in\{i_1, \dots, i_{N-c}\}$, so the corresponding row  $d_i$ is a linear combination of $d_{i_1}, \dots, d_{i_{N-c}}$,
\item[(2)] or $i\in\{i_1, \dots, i_{N-c}\}$ and there exists $j\not\in\{i_1, \dots, i_{N-c}\}$ such that
both the rows $b_j$ and $d_j$ are linear combinations of   $d_{i_1}, \dots, d_{i_{N-c}}$.
\end{itemize}

Denote by $\delta_{k_1k_2}$ the matrix whose only non-zero entry, which is equal to 1, is at the intersection of the $k_1$-th row and the $k_2$-th column. 
Define
\begin{equation}\label{deps}
D_\varepsilon = D + \varepsilon (\delta_{ij}+\delta_{ji})B,
\end{equation}
where we put $j=i$ in  case (1).
By construction, we have
\begin{equation}\label{rankdeps}
{\rm rank} D_\varepsilon = N - c+1
\end{equation}
at $\varepsilon\neq 0$. Indeed, by the choice of $i$ and $j$, the $(N-c+1)$ linearly independent rows are
$d_{i_1,\varepsilon}, \dots, d_{i_{N-c},\varepsilon}, d_{j,\varepsilon}$, where, in case (1),  $d_{i_n,\eps}=d_{i_n}$ for $n \in \{1,\dots,N-c\}$ and $d_{j,\varepsilon}=d_{i,\eps}= d_i + 2\eps b_i$, and, in case (2), $d_{i_n,\eps}=d_{i_n}$ for $n \in \{1,\dots,N-c\}\setminus\{i\}$, $d_{i,\eps}=d_i+ \eps b_j$ and $d_{j,\varepsilon}= d_j + \eps b_i$.
Therefore, one can find a $(N-c+1)\times (N-c+1)$ matrix $\hat D_\varepsilon$ from these rows such that
\begin{equation}\label{rankmon}
\left.\dfrac{d}{d\eps}\det \hat D_\varepsilon \right|_{\eps=0} \neq 0.
\end{equation}

Consider the Hamiltonian function
$$\tilde H_{\eps,\mu,\nu}=(-\sum_{k=1}^N\nu_k u_k + \sum_{k=1}^N\mu_k v_k - \varepsilon u_i u_j)\cdot \chi(u,v),$$
where $\chi$ is a $C^\infty$ bump function, supported in $U$ and equal to $1$ in a small neighborhood of zero.
The time-1 map $\tilde G_{\eps,\mu,\nu}:(u,v) \mapsto (\bar u,\bar v)$ for this Hamiltonian is, for small $(u,v)$, given by
\begin{equation}\label{case1uvmn}
\bar u = u + \mu, \qquad \bar v = v + \nu + 
\eps (\delta_i +\delta_j) (u + \frac{1}{2}\mu).
\end{equation}
Thus, we have $\bar P=P$ at $(\mu,\nu)=0$, and
\begin{equation}\label{munundg}
\det (\partial_{\mu,\nu}(\bar P-P)) = 1.
\end{equation}
In particular, $G_{\eps,\mu,\nu}\circ f (W_1)$ intersect $W_2$ at the point $P$ at 
$(\mu,\nu)=0$ for all $\varepsilon$. By (\ref{case1uvmn}),
the matrix $D$ in (\ref{eq:transfn}) is replaced by $D_\varepsilon$ 
given by (\ref{deps}). So, by (\ref{rankmon}), the intersection 
of $G_{\eps,\mu,\nu}\circ f (W_1)$ and $W_2$ has corank less than $c$ at 
$\eps\neq 0$.

Since conditions (\ref{rankdeps})  and (\ref{munundg}) persist
at small perturbations, the lemma follows by taking $H^0_{\eps,\mu,\nu}=\tilde H_{\eps,\mu,\nu}$ whenever $\tilde H^0_{\eps,\mu,\nu}$ is jointly $C^\infty$ with respect to variables and parameters; this is the case when $W_2$, and hence the coordinates $(u,v)$, are $C^\infty$. When $\tilde H_{\eps,\mu,\nu}$ has only  finite smoothness, it suffices to take $\{H^0_{\eps,\mu,\nu}\}$ to be any $C^\infty$ family of Hamiltonians supported in $U(P)$ that is sufficiently close to $\{\tilde H_{\eps,\mu,\nu}\}$.

\subsection{Translating the results to Hamiltonian dynamics}\label{sec:conti}
The main  results of Sections~\ref{sec:intro} and \ref{sec:intro2} can be translated to the continuous-time setting.

\subsubsection{Creation of blenders}
Let $\mathcal{M}'$ be a $2(N+1)$-dimensional ($N\geq 2$) symplectic manifold and consider any  Hamiltonian $H\in C^{s+1}(\mathcal{M}')$. The manifold  $\mathcal{M}'$ is foliated by  energy levels -- the $(2N+1)$-dimensional level sets of $H$, which are invariant under the flow of the system.  Assume that the Hamiltonian flow has, in some $(2N+1)$-dimensional energy level $H=h_0$, a two-dimensional whiskered torus $\tau\cong\mathbb{T}^2$, that is, there exists an  $2N$-dimensional cross-section $V\subset H^{-1}(h_0)$ such that its intersection with $\tau$  is a one-dimensional invariant whiskered torus $\gamma$ of  the Poincar\'e return map $T_0$ of $V$.

The map $T_0$ is a $C^s$ symplectic diffeomorphism to its image and we can consider the same local objects as in the discrete-time case:  the two-dimensional invariant cylinder $\A\subset V$ containing $\gamma$, and the strong-stable and strong-unstable foliations $\mathcal{F}^{\ss}$ and $\mathcal{F}^{\uu}$ in $V$ whose leaves comprise the local invariant manifolds of $\A$. In particular, the leaves through $\gamma$ form the $N$-dimensional local stable and unstable manifolds $W^\s_{\loc}(\gamma)$ and $W^\u_{\loc}(\gamma)$. The trajectories of the Hamiltonian flow starting from these two manifolds give the global invariant manifolds $W^\s(\tau)$ and  $W^\u(\tau)$, which are both $(N+1)$-dimensional.

We next define partially-hyperbolic homoclinics analogously to Definition~\ref{defi:phorbit}.
Let us assume that $W^\s(\tau)$ intersects $W^\u(\tau)$   along some orbit $\Gamma$. 
We take two points $M^-\in W^\u_{\loc}(\gamma)$ and $M^+\in W^\s_{\loc}(\gamma)$ from $\Gamma$. The flow near the homoclinic orbit defines the transition map $T_1$ from a small neighborhood $\Pi^-\subset V$ of $M^-$ to a small neighborhood $\Pi^+\subset V$ of $M^+$. Denote by $\ell^{\uu}$ the strong-unstable leaf through $M^-$ and by $\ell^{\ss}$ the strong-unstable leaf through $M^+$. Then homoclinic orbit $\Gamma$ is {\em partially hyperbolic} if  $T_1(\ell^{\uu})$ and $\ell^{\ss}$ satisfy condition \eqref{eq:condition:trans1} at the point $M^+$. 

The partial hyperbolicity implies that  $W^\u(\tau)\cap W^\s_{\loc}(\A)$ near $M^+$ and $W^\s(\tau)\cap W^\u_{\loc}(\A)$ near $M^-$ are  two-dimensional discs transverse to $\mathcal{F}^{\ss}$ and  $\mathcal{F}^{\uu}$, respectively. As a result, one can define the holonomy maps $\pi^\s$ and $\pi^\u$ from these discs to $\A$, and hence the scattering map  ${S}=\pi^\s\circ T_1\circ (\pi^\u)^{-1}$.  
The notion of contraction/expansion of homoclinic tangencies is defined in the same way as  in the discrete-time case, as well as the order of tangencies  (see Section~\ref{sec:transtangency}). 

The maps $T_0$ and $T_1$   have exactly the same  properties as  in the discrete-time case. As explained below Theorem~\ref{thm:main_blender_cubic_nonpara}, the search of blenders is based on  analysing these two maps, we have the following analogue of Theorem~\ref{thm:main_blender_cubic_nonpara}:
\begin{manualtheorem}{G$^\prime$}
Let $H\in C^4(\mathcal{M})$, and let the corresponding Hamiltonian flow  have a two-dimensional whiskered torus of class $C^2$ in some energy level $H=h_0$, which has a 2-flat homoclinic tangency. If $T_0|_{\tau \cap V}$ for some cross-section $V\subset H^{-1}(h_0)$ is $C^1$-conjugate to an irrational rotation, then,  given any neighborhood $V'$ of the homoclinic orbit and $\tau$, there exists a blender  connected via $V'$ to $\tau$, center-stable if the tangency is contracting and center-unstable if expanding.
\end{manualtheorem}

Here a blender of the flow is a blender of the first-return map $T$ of $V$  defined as $T=T_1$ in $\Pi^-$ and $T=T_0$ outside a small neighborhood of the closure of $\Pi^-$. 
The connection means that the blender of the induced map is connected to $\gamma=\tau \cap V$.

Let us now  discuss the perturbative results. 
We  further assume  that $H$ is $C^\infty$  and $\gamma$ is a one-dimensional whiskered KAM-torus of the Poincar\'e map $T_0$ in the sense of Definition~\ref{defi:kamtori}. 
In this case $\tau$, consisting of the orbits through $\gamma$, is  a two-dimensional KAM-torus of the Hamiltonian flow, which  persists
 for all  Hamiltonians close to $H$ and all energy levels close to $H=h_0$.
We embed the Hamiltonian $H$ into a one-parameter family $\{H_{\eps}\in C^{\infty}(\mathcal{M}')\}$ with $H_0=H$. We take the value of energy as a second parameter for the local and transition maps, i.e., there is a two-parameter family of first-return maps
 $T_{\hat\eps}$.

We note that the conditions for the three  unfolding families of Theorem~\ref{thm:main_2punfolding}  are all formulated in terms of the transition maps, as given by \eqref{eq:unfoldcubic}, \eqref{eq:unfold2quadra} and  \eqref{eq:unfoldquadra} (see also the end of Section~\ref{sec:transtangency}). Thus, those conditions also make sense for the family $\{T_{\hat\eps}\}$. 
Since the proof of Theorem~\ref{thm:main_2punfolding} is solely based on the analysis of the corresponding family of the local map and  transition map(s) and these maps have the same properties as in the discrete-time case, we obtain
\begin{manualtheorem}{B$^\prime$}\label{thm:B'}
Let $\{H_{\eps}\in C^{\infty}(\mathcal{M}')\}$ be any one-parameter  family such that 
 the  two-parameter family $\{T_{\hat\eps}\}$ unfolds 
 \begin{itemize}[nosep]
\item one partially-hyperbolic cubic homoclinic tangency of $\tau$ in the sense of \eqref{eq:unfoldcubic}, or
\item two partially-hyperbolic quadratic homoclinic tangencies of $\tau$ in the sense of \eqref{eq:unfold2quadra} , or
\item  one partially-hyperbolic quadratic homoclinic tangency of $\tau$ in the sense of  \eqref{eq:unfoldquadra}.
 \end{itemize}
 Let   $V'$ be any neighborhood of the homoclinic orbit and $\tau$.
Then there exist $\eps$ arbitrarily close to $0$ and  $h$ close $h_0$ such that the Hamiltonian flow of $H_\eps$ has a symplectic blender connected to the continuation of $\tau\subset H^{-1}(h)$ via $V'$. 
\end{manualtheorem}

As in the discrete-time case, the theorem also applies to families of sufficiently high regularity.
Note that the exactness assumption in Theorem~\ref{thm:main_2punfolding} is not required here (see Remark~\ref{rem:exact2}).
Again, the connection means that the blender of the induced map is connected to the continuation\footnote{The continuation of the torus $\tau$ is uniquely defined, whereas that of $\gamma$ depends on the choice of the cross-section.} of $\gamma$ in the sense of Definition~\ref{defi:sympblen_conn}.

Similarly, the arguments used to prove Theorem~\ref{thm:main_blender_allcases} can be translated  to the case of Hamiltonian flows. Thus, implementing the same series of perturbations we constructed to derive Theorem~\ref{thm:main_blender_allcases}  from  Theorem~\ref{thm:main_2punfolding}, we can obtain from Theorem~\ref{thm:B'} the following
\begin{manualtheorem}{A$^\prime$}\label{thm:A'} 
Let $H\in C^{s}(\mathcal{M}')$,  $s=3,\dots, \infty,\omega$, and let the corresponding Hamiltonian flow have a two-dimensional whiskered torus $\tau$ of class $C^s$ with a homoclinic orbit $\Gamma$ in some energy level $H=h_0$. Assume the rotation number of $\tau$  is irrational. Given any neighborhood $\hat V$ of $\Gamma\cup\tau$, there exists $H'\in C^s(\M')$, arbitrarily $C^s$-close to $H$, such that the flow of $H'$ has a symplectic blender for every energy value $h$ close to $h_0$, which is connected to a non-degenerate whiskered KAM-torus $\tau_{h}$ of class $C^s$, arbitrarily $C^s$-close to $\tau$. When the smoothness $s$ is finite,  both  $H'$ and $\tau_{h}$ can be taken $C^\infty$.
\end{manualtheorem}

\subsubsection{Saddle-center homoclinics} 
Let $H\in C^{s+1}(\mathcal{M}')$ and  the corresponding system have a periodic orbit $L$ in some energy level $H=h_0$. Take a small $2N$-dimensional cross-section $V\subset H^{-1}(h_0)$ to $L$. The flow of the system restricted to the energy level $H=h_0$ defines the Poincar\'e return map $T_0$, which is a locally defined $C^s$ symplectic diffeomorphism of $V$. The intersection point $O=L\cap V$ is a fixed point of $T_0$. The periodic orbit $L$ is called a {\em saddle-center} when $O$ is a saddle-center for $T_0$. Note that the map $T_0$ plays the same role as the local map near the saddle-center periodic point in the discrete-time case. On $V$ there exist $(N-1)$-dimensional local stable and unstable manifolds $W^\s_{\loc}(O)$ and $W^\u_{\loc}(O)$. The orbits of the Hamiltonian flow starting at $W^\s_{\loc}(O)$ and $W^\u_{\loc}(O)$ form the global  stable and, respectively, unstable invariant manifolds of $L$. Assume that $W^\s(L)$ intersects $W^\u(L)$ along a homoclinic orbit $\Gamma$. Take two points of the intersection of $\Gamma$ with $V$: $M^-\in W^\u_{\loc}(O)$ and $M^+\in W^\s_{\loc}(O)$. The flow near $\Gamma$ defines the transition map $T_1$ from a small neighborhood of $M^-$ to a small neighborhood of $M^+$ in $V$.

We say that $L$ and $\Gamma$ are generic if the pair of symplectic maps $(T_0,T_1)$ satisfies the genericity conditions \ref{word:C1}--\ref{word:C3} of Section \ref{sec:para_SC}, where $T_1$ is used to define the scattering map. As before,  for a   family $\{H_{\eps}\}$ with $H_0=H$, we can take the value of the energy $h$ as an extra parameter so that it generates a   family  of pairs of symplectic maps $\{(T_{0,\hat\eps},T_{1,\hat\eps})\}$, where   $\hat\eps=(\eps,h-h_0)$. We say that the family $\{H_{\eps}\}$ is a {\em proper unfolding} of $\Gamma$ if the corresponding family $\{(T_{0,\hat\eps},T_{1,\hat\eps})\}$ satisfies condition~\ref{word:H1} of Section~\ref{sec:para_SC}, and it is a {\em tangency unfolding} if condition~\ref{word:H2} is satisfied. Applying the proofs of Theorem~\ref{thm:main_SC_para} and Theorem~\ref{thm:main_SCblender}  to the family of pairs $\{(T_{0,\hat\eps},T_{1,\hat\eps})\}$, we obtain the continuous-time versions of these theorems.

\begin{manualtheorem}{C$^\prime$}\label{thm:C'} 
Let the Hamiltonian system defined by  $H\in C^{\infty}(\mathcal{M}')$  have a generic saddle-center periodic orbit with a generic homoclinic orbit. For any proper unfolding family $\{H_{\eps}\}$,  there exists a sequence $\{\mathcal{E}_j\}$ of open sets in the  $(\eps,h)$-space converging to $(0,h-h_0)$ such that  the parameter values  for which the continuation of the saddle-center  has a homoclinic orbit are dense in $\bigcup_j \mathcal{E}_j$.
\end{manualtheorem}

\begin{manualtheorem}{D$^\prime$}\label{thm:D'} 
Let the Hamiltonian system defined by  $H\in C^{\infty}(\mathcal{M}')$  have a generic saddle-center periodic orbit with a generic homoclinic orbit. For any tangency-unfolding family $\{H_{\eps}\}$,  there exist $\eps$ arbitrarily close to $0$ and  $h$ close $h_0$ such that the Hamiltonian flow of $H_\eps$ has a symplectic blender connected to some two-dimensional KAM-torus   $ \tau\in H^{-1}(h)$
\end{manualtheorem}

By the discussion above Corollary~\ref{cor:SC_nonpara}, the above theorems immediately imply

\begin{manualcor}{E$^\prime$}\label{cor:E'} 
Let $H\in C^{s}(\mathcal{M'},\mathbb R) $,  $s=2,\dots,\infty,\omega$, have a   saddle-center periodic orbit  $L$ with a  homoclinic orbit $\Gamma$, and let $ V'$ be any neighborhood of $L\cup\Gamma$. Then
 there exist a map $H'\in C^{s}(\mathcal{M'},\mathbb R)$, arbitrarily close to $H$, a $C^2$ neighborhood $\mathcal{U}\subset C^{s}(\mathcal{M'},\mathbb R)$ of $H'$ and a $C^s$-dense subset $\mathcal{U}'$ of $\mathcal{U}$ such that 
\begin{itemize}[nosep]
\item the system corresponding to $H'$ has a symplectic blender $\Lambda\subset V'$,
\item $W^\u(L_{F})\cap W^\s(\Lambda_{F})\neq \emptyset$ and  $W^\s(L_{F})\cap W^\u(\Lambda_{F})\neq \emptyset$ for every $F\in\mathcal{U}$, and
\item $L_F$ has a homoclinic orbit in some energy level for every $F\in \mathcal{U}'$.
\end{itemize}
\end{manualcor}

\bibliographystyle{myplainnat}
\bibliography{references}

\end{document}